\documentclass[english,leqno]{article}

\usepackage{latexsym}
\usepackage{linearb}
\usepackage[LGR, OT1]{fontenc}
\usepackage{amsmath,amsthm}
\usepackage{amssymb}
\usepackage{MnSymbol}
\usepackage{graphicx}
\usepackage{marvosym}
\usepackage{eufrak}
\usepackage[margin=1.25in,dvips]{geometry}
\usepackage{color}
\usepackage{comment}
\usepackage{mathrsfs}\usepackage[LGR,T1]{fontenc}
\usepackage[latin9]{inputenc}
\usepackage{geometry}
\usepackage{babel}
\usepackage{enumitem}
\usepackage{amsthm}
\usepackage{amstext}
\usepackage{amssymb}
\usepackage{esint}
\usepackage[unicode=true,pdfusetitle,
 bookmarks=true,bookmarksnumbered=false,bookmarksopen=false,
 breaklinks=false,pdfborder={0 0 1},backref=false,colorlinks=false]
 {hyperref}
\makeatletter

\DeclareRobustCommand{\greektext}{%
  \fontencoding{LGR}\selectfont\def\encodingdefault{LGR}}
\DeclareRobustCommand{\textgreek}[1]{\leavevmode{\greektext #1}}
\DeclareFontEncoding{LGR}{}{}
\DeclareTextSymbol{\~}{LGR}{126}

\numberwithin{equation}{section}
\numberwithin{figure}{section}
\usepackage{enumitem}		
\newcommand{\lyxaddress}[1]{
\par {\raggedright #1
\vspace{1.4em}
\noindent\par}
}
  \theoremstyle{remark}
  \newtheorem*{rem*}{\protect\remarkname}
 \theoremstyle{definition}
 \newtheorem*{defn*}{\protect\definitionname}
  \theoremstyle{plain}
  \newtheorem{prop}{\protect\propositionname}[section]
  \theoremstyle{plain}
  \newtheorem{lem}{\protect\lemmaname}[section]
  \theoremstyle{plain}
  \newtheorem{cor}{\protect\corollaryname}[section]

\newtheorem{innercustomTheorem}{Theorem}
\newenvironment{customTheorem}[1]
  {\renewcommand\theinnercustomTheorem{#1}\innercustomTheorem}
  {\endinnercustomTheorem}

\makeatother

\usepackage{babel}
  \providecommand{\definitionname}{Definition}
  \providecommand{\lemmaname}{Lemma}
  \providecommand{\propositionname}{Proposition}
  \providecommand{\remarkname}{Remark}
\providecommand{\corollaryname}{Corollary}

\begin{document}
\title{Superradiant instabilities for short-range non-negative\\ potentials on Kerr spacetimes and applications}

\author{Georgios Moschidis}

\maketitle

\lyxaddress{Princeton University, Department of Mathematics, Fine Hall, Washington
Road, Princeton, NJ 08544, United States, \tt gm6@math.princeton.edu}
\begin{abstract}
In \cite{Shlap}, Shlapentokh-Rothman established that the wave equation
$\square_{g_{M,a}}\text{\textgreek{y}}=0$ on subextremal Kerr spacetimes
$(\mathcal{M}_{M,a},g_{M,a})$, $0<|a|<M$, does not admit real mode
solutions. This is a highly non-trivial result, in view of the phenomenon
of superradiance, i.\,e.~the fact that the stationary Killing field
$T$ fails to be causal on the horizon $\mathcal{H}$. The analogue
of this result fails for long-range perturbations of the wave equation,
such as the Klein--Gordon equation $\square_{g_{M,a}}\text{\textgreek{y}}-\text{\textgreek{m}}^{2}\text{\textgreek{y}}=0$
for $\text{\textgreek{m}}>0$, as was shown by Shlapentokh-Rothman
in \cite{Shlapentokh-Rothman2013}. The question naturally arises
whether the absence of real modes persists under the addition of an
arbitrary short-range non-negative potential $V$ to the wave equation
or under changes of the metric $g_{M,a}$ in the far away region of
$\mathcal{M}_{M,a}$ (retaining the causality of $T$ there). 

In this paper, we answer the above question in the negative in both
cases. First, for any $0<|a|<M$, we establish the existence of real
mode solutions $\text{\textgreek{y}}$ to equation $\square_{g_{M,a}}\text{\textgreek{y}}-V\text{\textgreek{y}}=0$,
for a suitably chosen time-independent real potential $V$ with compact
support in space, satisfying the sign condition $V\ge0$. Exponentially
growing modes are also obtained after perturbing the potential $V$.
Then, as an application of the above results, we construct a family
of spacetimes $(\mathcal{M}_{M,a},g_{M,a}^{(def)})$ which are compact
in space perturbations of $(\mathcal{M}_{M,a},g_{M,a})$, have the
same symmetries as $(\mathcal{M}_{M,a},g_{M,a})$ and moreover admit
real and exponentially growing mode solutions to equation $\square_{g}\text{\textgreek{y}}=0$.
The nature of our construction forces, however, the spacetimes $(\mathcal{M}_{M,a},g_{M,a}^{(def)})$
to contain stably trapped null geodesics. We also construct a more
complicated family of spacetimes $(\mathcal{M}_{0},g_{M,a}^{(h)})$
admitting real and exponentially growing mode solutions to the wave
equation, on which the trapped set is normally hyperbolic, at the
expense of $g_{M,a}^{(h)}$ having conic asymptotics. 

The above results are in contrast with the case of stationary asymptotically
flat (or conic) spacetimes $(\mathcal{M},g)$ with a globally timelike
Killing field $T$, where the absence of real modes for equation $\square_{g}\text{\textgreek{y}}-V\text{\textgreek{y}}=0$
is immediate. On such spacetimes, this fact gives a useful continuity
criterion for showing stability for a smooth family of equations $\square_{g}\text{\textgreek{y}}-V_{\text{\textgreek{l}}}\text{\textgreek{y}}=0$,
with $\text{\textgreek{l}}\in[0,1]$ and $V_{0}=0$: It suffices to
bound the resolvent for frequencies in a neighborhood of $\text{\textgreek{w}}=0$
for all $\text{\textgreek{l}}\in[0,1]$. We show explicitly that this
criterion fails on Kerr spacetime, by constructing a potential $V$
so that for the smooth family of equations $\square_{g_{M,a}}\text{\textgreek{y}}-\text{\textgreek{l}}V\text{\textgreek{y}}=0$,
$\text{\textgreek{l}}\in[0,1]$, a real mode first appears at $\text{\textgreek{w}}=\text{\textgreek{w}}_{0}\in\mathbb{R}\backslash\{0\}$
for $\text{\textgreek{l}}=\text{\textgreek{l}}_{0}\in(0,1]$. 
\end{abstract}
\tableofcontents{}

\section{\label{sec:Introduction}Introduction}

The celebrated Kerr family of spacetimes $(\mathcal{M}_{M,a},g_{M,a})$,
first discovered in 1963 (see \cite{Kerr1963}), is a 2-parameter
family of solutions to the vacuum Einstein equations 
\begin{equation}
Ric_{\text{\textgreek{m}\textgreek{n}}}=0,
\end{equation}
parametrised by the \emph{mass} $M$ and the \emph{angular momentum
per unit mass} $a$. In the Boyer--Lindquist coordinate chart $(t,r,\text{\textgreek{j}},\text{\textgreek{f}}):\mathcal{M}_{M,a}\rightarrow\mathbb{R}\times(r_{+},+\infty)\times\mathbb{S}^{2}$,
the metric $g_{M,a}$ takes the form 
\begin{align}
g_{M,a}=-\big( & 1-\frac{2Mr}{\text{\textgreek{r}}^{2}}\big)dt^{2}-\frac{4Mar\sin^{2}\text{\textgreek{j}}}{\text{\textgreek{r}}^{2}}dtd\text{\textgreek{f}}+\frac{\text{\textgreek{r}}^{2}}{\text{\textgreek{D}}}dr^{2}+\label{eq:KerrMetric}\\
 & +\text{\textgreek{r}}^{2}d\text{\textgreek{j}}^{2}+\sin^{2}\text{\textgreek{j}}\frac{\text{\textgreek{P}}}{\text{\textgreek{r}}^{2}}d\text{\textgreek{f}}^{2},\nonumber 
\end{align}
where 
\begin{gather}
\text{\textgreek{r}}^{2}=r^{2}+a^{2}\cos^{2}\text{\textgreek{j}},\label{eq:Coeff1}\\
\text{\textgreek{D}}=(r-r_{+})\cdot(r-r_{-}),\label{eq:Coeff2}\\
r_{\pm}=M\pm\sqrt{M^{2}-a^{2}},\label{eq:Coeff3}\\
\text{\textgreek{P}}=(r^{2}+a^{2})^{2}-a^{2}\sin^{2}\text{\textgreek{j}}\text{\textgreek{D}}.\label{eq:Coeff4}
\end{gather}
 The Schwarzschild metric corresponds to (\ref{eq:KerrMetric}) for
$a=0$.

In the so called \emph{subextremal} parameter range $0\le|a|<M$,
the maximal extension $(\widetilde{\mathcal{M}}_{M,a},\tilde{g}_{M,a})$
of the Kerr spacetime $(\mathcal{M}_{M,a},g_{M,a})$, first constructed
by Carter in \cite{Carter1968a}, has two asymptotically flat ends,
and contains a black hole and a white hole region, which are bounded
by a future and a past event horizon $\mathcal{H}^{+}$ and $\mathcal{H}^{-}$
respectively. The union $\mathcal{H}=\mathcal{H}^{+}\cup\mathcal{H}^{-}$
is the event horizon of $(\widetilde{\mathcal{M}}_{M,a},\tilde{g}_{M,a})$,
while the intersection $\mathcal{H}^{+}\cap\mathcal{H}^{-}$ (which
is non empty) is the so called \emph{bifurcation sphere}. In this
extension, $(\mathcal{M}_{M,a},g_{M,a})$ is identified with the domain
of outer communications of one of the two asymptotically flat ends
of $(\widetilde{\mathcal{M}}_{M,a},\tilde{g}_{M,a})$. 

The wave equation 
\begin{equation}
\square_{g_{M,a}}\text{\textgreek{y}}=0\label{eq:WaveEquation}
\end{equation}
on $(\mathcal{M}_{M,a},g_{M,a})$, for $0\le|a|<M$, has been extensively
studied. A first result relevant to the stability (i.\,e.~boundedness
and decay) properties of equation (\ref{eq:WaveEquation}) in the
full subextremal range $0\le|a|<M$ is the proof by Whiting, in \cite{Whiting1989},
that equation (\ref{eq:WaveEquation}) does not admit exponentially
growing mode solutions, i.\,e.~solutions $\text{\textgreek{y}}$
of the form 
\begin{equation}
\text{\textgreek{y}}(t,r,\text{\textgreek{j}},\text{\textgreek{f}})=e^{-i\text{\textgreek{w}}t}\text{\textgreek{y}}_{\text{\textgreek{w}}}(r,\text{\textgreek{j}},\text{\textgreek{f}})\label{eq:ModeIntroduction}
\end{equation}
such that $Im(\text{\textgreek{w}})>0$, with $\text{\textgreek{y}}$
being smooth up to $\mathcal{H}^{+}\backslash\mathcal{H}^{-}$ and
having finite energy on the $\{t=const\}$ slices. This result was
extended by Shlapentokh-Rothman in \cite{Shlap}, where the non-existence
of outgoing real mode solutions of (\ref{eq:WaveEquation}) on $(\mathcal{M}_{M,a},g_{M,a})$,
$0\le|a|<M$, was established. Recall that a solution $\text{\textgreek{y}}$
to (\ref{eq:WaveEquation}) is called an outgoing real mode solution
if it is of the form (\ref{eq:ModeIntroduction}) with $\text{\textgreek{w}}\in\mathbb{R}\backslash\{0\}$,
such that $\text{\textgreek{y}}$ is smooth up to $\mathcal{H}^{+}\backslash\mathcal{H}^{-}$
and has finite energy flux through a hyperboloidal hypersurface $\mathcal{S}$
terminating at future null infinity and intersecting $\mathcal{H}^{+}$
transversally (satisfying also $\mathcal{S}\cap\mathcal{H}^{-}=\emptyset$),
but has infinite energy flux through the $\{t=const\}$ hypersurfaces;
see \cite{Shlap} (or Section \ref{sub:Modes}) for more details. 

\begin{figure}[h] 
\centering 
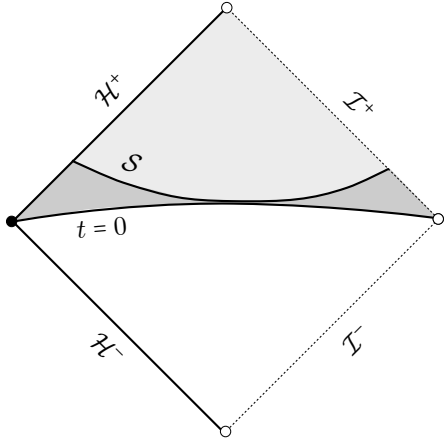 
\caption{A hyperboloidal hypersurface $\mathcal{S}\subset\mathcal{M}_{M,a}$ terminating at $\mathcal{I}^{+}$ and intersecting $\mathcal{H}^{+}$ transversally, such that $\mathcal{S}\cap\mathcal{H}^{-}=\emptyset$ (depicted above intersected with the 1+1 dimensional slice $\{\theta=\pi/2,\phi=0\}\subset\mathcal{M}_{M,a}$). An outgoing real mode solution $\text{\textgreek{y}}$ to equation \eqref{eq:WaveEquation}  has finite energy flux through  $\mathcal{S}$, but infinite energy flux through the hypersurface $\{t=0\}$.} 
\end{figure}

Notice that, in view of the fact that the stationary Killing field
$T=\partial_{t}$ fails to be causal on $\mathcal{H}^{+}$ when $a\neq0$,
the results of \cite{Whiting1989,Shlap} are highly non-trivial in
this case (unlike the Schwarzschild case $a=0$, where they would
simply follow from the energy identity for $T$; see Section \ref{sub:Superradiance}).
The absence of an everywhere causal Killing field on $(\mathcal{M}_{M,a},g_{M,a})$
when $a\neq0$ gives rise to the phenomenon of \emph{superradiance}
for equation (\ref{eq:WaveEquation}), which we will discuss in more
detail in Section \ref{sub:Superradiance}. 

The results of \cite{Shlap} were then used in \cite{DafRodSchlap},
where quantitative decay estimates for solutions $\text{\textgreek{y}}$
to (\ref{eq:WaveEquation}) on $(\mathcal{M}_{M,a},g_{M,a})$ for
$0\le|a|<M$ were obtained. For earlier stability results in the Schwarzschild
case (i.\,e.~when $a=0$) and the very slowly rotating Kerr case
(i.\,e.~when $|a|\ll M$), see \cite{KayWald,DafRod1,DafRod2,DafRod4,BlueSof1,BlueSterb}
and \cite{DafRod5,DafRod6,DafRod9,TatToh1,AndBlue1} respectively.
It should be noted, also, that the techniques employed in \cite{Shlap}
are robust enough to yield a mode stability statement for \emph{small
short-range} potential perturbations of equation (\ref{eq:WaveEquation}).

In contrast to the above results, superradiance-related instability
phenomena come into the picture in the case of \emph{long-range} perturbations
to equation (\ref{eq:WaveEquation}), showing that the mode stability
results of \cite{Whiting1989,Shlap} can not be extended to include
this case: In \cite{Shlapentokh-Rothman2013}, Shlapentokh-Rothman
constructed exponentially growing mode solutions (as well as real
mode solutions) to the Klein--Gordon equation 
\begin{equation}
\square_{g_{M,a}}\text{\textgreek{y}}-\text{\textgreek{m}}^{2}\text{\textgreek{y}}=0\label{eq:KleinGordon}
\end{equation}
on $(\mathcal{M}_{M,a},g_{M,a})$ with $a,\text{\textgreek{m}}\neq0$.
This result was anticipated by the heuristics of \cite{Damour1976,Zouros1979,Detweiler1980}
(see also the numerics of \cite{Dolan2007,Dolan2013} and references
therein).

The question naturally arises, therefore, whether the mode stability
results of \cite{Shlap} can be extended to include \underline{\emph{large short-range}}
deformations of equation (\ref{eq:WaveEquation}), in the form of
either potential perturbations or metric deformations retaining the
causal character of $T$ at each point. In this paper, we will provide
a negative answer to this question in both cases.

\subsection{Theorem 1: An instability result for $\square\text{\textgreek{y}}-V\text{\textgreek{y}}=0$
on Kerr spacetimes}

Our first result will show that the mode stability statement of \cite{Whiting1989,Shlap}
can \emph{not} be extended to include the case when an arbitrary term
of the form $-V\text{\textgreek{y}}$, with $V\ge0$ compactly supported
in the region $\{r\gg1\}$, is added to the wave equation (\ref{eq:WaveEquation})
on $(\mathcal{M}_{M,a},g_{M,a})$ for any $a$ in the subextremal
range $0<|a|<M$. In particular, we will establish the following:

\begin{customTheorem}{1}[short version] \label{Theorem 1} For any
$0<|a|<M$, any $\text{\textgreek{w}}_{R}\in\mathbb{R}\backslash\{0\}$,
any $r_{0}\gg1$ large in terms of $\text{\textgreek{w}}_{R}$ and
$a$ and any $0\le\text{\textgreek{w}}_{I}\ll1$ small in terms of
$\text{\textgreek{w}}_{R}$ and $a$, there exists a $V:\mathcal{M}_{M,a}\rightarrow[0,+\infty)$
compactly supported in the region $\{r\ge r_{0}\}$ and satisfying
$\partial_{t}V=\partial_{\text{\textgreek{f}}}V=0$, such that the
equation 
\begin{equation}
\square_{g_{M,a}}\text{\textgreek{y}}-V\text{\textgreek{y}}=0\label{eq:WavePotentialInrtro}
\end{equation}
admits an outgoing mode solution with frequency parameter $\text{\textgreek{w}}_{R}+i\text{\textgreek{w}}_{I}$.
In particular, (\ref{eq:WavePotentialInrtro}) admits a real mode
solution with $\text{\textgreek{w}}_{I}=0$, and an exponentially
growing mode solution with $\text{\textgreek{w}}_{I}>0$.

\end{customTheorem}

For a more detailed statement of Theorem \hyperref[Theorem 1]{1},
see Section \ref{sec:ProofOfMainThm}. For the definition of an outgoing
mode solution, see Section \ref{sub:Modes}. 

Note, in contrast, that, on a stationary and asymptotically flat spacetime
with an everywhere causal Killing field (e.\,g.~Schwarzschild exterior),
adding a potential $V$ with the properties described in Theorem \hyperref[Theorem 1]{1}
to the wave equation does not ``destroy'' the decay properties of
the corresponding solutions. It is due to the phenomenon of superradiance
(more precisely, the fact that $\mathcal{H}^{+}\cap\{g(T,T)>0\}\neq\emptyset$)
that a construction of a real mode solution to equation (\ref{eq:WavePotentialInrtro})
in the case $a\neq0$ is possible. For a discussion on the role of
superradiance as a mechanism of instability, see Section \ref{sub:Superradiance}\emph{.}

\subsection{Theorem 2: An instability result for $\square\text{\textgreek{y}}=0$
on short-range deformations of the Kerr metric}

As an application of Theorem \hyperref[Theorem 1]{1}, we will infer
the existence of real and exponentially growing modes for the wave
equation associated to stationary and compactly supported in space
deformations $g_{M,a}^{(def)}$ of the Kerr metric $g_{M,a}$ with
the same ergoregion as $g_{M,a}$:

\begin{customTheorem}{2}[short version] \label{Theorem 2} For any
$0<|a|<M$, any $\text{\textgreek{w}}_{R}\in\mathbb{R}\backslash\{0\}$,
any $0\le\text{\textgreek{w}}_{I}\ll1$ and any $r_{0}\gg1$, there
exists a stationary and axisymmetric Lorentzian metric $g_{M,a}^{(def)}$
on $\mathcal{M}_{M,a}$, coinciding with $g_{M,a}$ outside the region
$\{r_{0}\le r\le r_{0}+C\}$ (for some $C\gg1$) and satisfying $g_{M,a}^{(def)}(\partial_{t},\partial_{t})<0$
in $\{r_{0}\le r\le r_{0}+C\}$, such that the wave equation 
\begin{equation}
\square_{g_{M,a}^{(def)}}\text{\textgreek{y}}=0\label{eq:WavePerturbedMetricIntro}
\end{equation}
completely separates in the Boyer--Lindquist coordinate chart and
admits an outgoing mode solution with frequency parameter $\text{\textgreek{w}}_{R}+i\text{\textgreek{w}}_{I}$.
In particular, (\ref{eq:WavePerturbedMetricIntro}) admits a real
mode solution with $\text{\textgreek{w}}_{I}=0$, and an exponentially
growing mode solution with $\text{\textgreek{w}}_{I}>0$.

\end{customTheorem}

For a more detailed statement of Theorem \hyperref[Theorem 2]{2},
see Section \ref{sec:ProofOfCorolary}.

\subsection{\label{sub:IntroAsymptoticallyConic}An example with normally hyperbolic
trapping}

The spacetimes $(\mathcal{M}_{M,a},g_{M,a}^{(def)})$ of Theorem \hyperref[Theorem 2]{2}
possess the same ergoregion structure as $(\mathcal{M}_{M,a},g_{M,a})$,
but our specific construction forces the structure of the trapped
set to be different. In particular, the region $r_{0}\le r\le r_{0}+C$
of $(\mathcal{M}_{M,a},g_{M,a}^{(def)})$ contains stable trapped
null geodesics. As a consequence, the behaviour of high frequency
solutions to equations (\ref{eq:WaveEquation}) and (\ref{eq:WavePerturbedMetricIntro})
is substantially different. 

Despite this aspect of the spacetimes $(\mathcal{M}_{M,a},g_{M,a}^{(def)})$,
in general there is no reason for a connection to exist between the
structure of the trapped set outside the ergoregion (which manifests
itself in the high frequency behaviour of solutions to (\ref{eq:WavePerturbedMetricIntro}))
and superradiance-related mode-instabilities (which is a purely fixed
frequency phenomenon). In order to better clarify the irrelevance
of the structure of trapping to the existence of a real mode solution
to (\ref{eq:WavePerturbedMetricIntro}), it would be preferable to
have an example of a spacetime $(\mathcal{M},g)$ possessing a ``nice''
trapped set and at the same time admitting real or exponentially growing
modes. In Section \ref{sec:DeformedSpacetimes}, we construct a spacetime
$(\mathcal{M}_{0},g_{M,a}^{(h)})$ which has the symmetries of the
Kerr exterior $(\mathcal{M}_{M,a},g_{M,a})$, such that the trapped
set of $(\mathcal{M}_{0},g_{M,a}^{(h)})$ is normally hyperbolic and
the wave equation 
\begin{equation}
\square_{g_{M,a}^{(h)}}\text{\textgreek{y}}=0\label{eq:WaveNormallyHyperbolicIntro}
\end{equation}
admits an outgoing real mode solution (which in turn yields, after
a suitable perturbation, an exponentially growing mode solution).
However, our specific construction forces the spacetime $(\mathcal{M}_{0},g_{M,a}^{(h)})$
to be asymptotically conic, instead of asymptotically flat like the
Kerr exterior $(\mathcal{M}_{M,a},g_{M,a})$. For more details regarding
this technically involved (in comparison to the proof of Theorem \hyperref[Theorem 2]{2})
construction, see Section \ref{sec:DeformedSpacetimes} and Theorem
\hyperref[Theorem 3]{3}.

\subsection{\label{sub:Superradiance}Discussion of the role of superradiance
as a mechanism of instability }

In this section, we will try to put into some context the phenomenon
of superradiance and its relation to instability results concerning
the wave equation 
\begin{equation}
\square_{g}\text{\textgreek{y}}=0\label{eq:WaveEquationIntro}
\end{equation}
on a general class of spacetimes $(\mathcal{M},g)$. 

Let $(\mathcal{M},g)$ be a globally hyperbolic, stationary and asymptotically
flat spacetime with stationary Killing field $T$, possibly bounded
by an event horizon $\mathcal{H}$.%
\footnote{The discussion in this section also applies to asymptotically conic
spacetimes, as the ones discussed in Section \ref{sub:IntroAsymptoticallyConic}.%
} For the purposes of this section, and in analogy with the properties
of the subextremal Kerr exterior $(\mathcal{M}_{M,a},g_{M,a})$, the
spacetime $(\mathcal{M},g)$ will be called \emph{superradiant} if
$T$ fails to be causal everywhere on $\mathcal{M}$. In the case
when $T$ is everywhere timelike on $\mathcal{M}\backslash\mathcal{H}$,
the spacetime $(\mathcal{M},g)$ will be called \emph{non-superradiant}.%
\footnote{Note that, under the above definitions, there exist spacetimes which
are neither superradiant nor non-superradiant: For instance, stationary
spacetimes $(\mathcal{M},g)$ on which $T$ is everywhere causal and
identically null on an open set $\mathcal{U}\subset\mathcal{M}$,
or the spacetimes considered in \cite{Eperon2016}, do not fall in
either category.%
} It can be readily shown that a spacetime $(\mathcal{M},g)$ as above
is superradiant if and only if there exist solutions $\text{\textgreek{y}}$
to (\ref{eq:WaveEquationIntro}) such that their $T$-energy flux
through future null infinity $\mathcal{I}^{+}$ is greater than their
$T$-energy flux through a Cauchy hypersurface $\text{\textgreek{S}}$
of $\mathcal{M}$ (for the definition of the $T$-energy flux, see
below; for the definition of $\mathcal{I}^{+}$ on a general asymptotically
flat spacetime, see \cite{Moschidisc}). 

We will now proceed to examine some general mode stability statements
for equation (\ref{eq:WaveEquationIntro}) on non-superradiant spacetimes
and then highlight the failure of these statements on superradiant
spacetimes. To this end, it will be convenient to distinguish among
superradiant spacetimes the ones having a non-empty future event horizon
$\mathcal{H}^{+}$ satisfying $\mathcal{H}^{+}\cap\{g(T,T)>0\}\neq\emptyset$
(such as Kerr exterior spacetime $(\mathcal{M}_{M,a},g_{M,a})$ when
$a\neq0$).

\subsubsection{\label{sub:NonSuperradiant}Mode stablity results on non-superradiant
spacetimes}

Let $(\mathcal{M},g)$ be a spacetime as above which is non-superradiant,
and let $V:\mathcal{M}\rightarrow\mathbb{R}$ be a smooth function,
having compact support in space and satisfying $T(V)=0$. In this
case, equation 
\begin{equation}
\square_{g}\text{\textgreek{y}}-V\text{\textgreek{y}}=0,\label{eq:PotentialWaveGeneralIntro}
\end{equation}
satisfies the following mode stability statement for real and non-zero
frequencies:

\begin{enumerate}

\item[1. a)] No outgoing real mode solutions at a non-zero real frequency
parameter $\text{\textgreek{w}}$ exist for equation (\ref{eq:PotentialWaveGeneralIntro}). 

\item[$\hphantom{1. }$ b)] No $L^{2}$ ``eigenfunctions'' (i.\,e.~solutions
$\text{\textgreek{y}}$ of the form (\ref{eq:ModeIntroduction}) such
that $\int_{t=0}|\text{\textgreek{y}}|^{2}<+\infty$) at a non-zero
real frequency parameter $\text{\textgreek{w}}$ exist for equation
(\ref{eq:PotentialWaveGeneralIntro}). 

\end{enumerate}

The above statement can be inferred from the fact that the energy
flux associated to the vector field $T$, which is obtained by integrating
the divergence-free current
\begin{equation}
\mathscr{E}_{\text{\textgreek{m}}}[\text{\textgreek{y}}]=J_{\text{\textgreek{m}}}^{T}(\text{\textgreek{y}})-V|\text{\textgreek{y}}|^{2}g_{\text{\textgreek{m}\textgreek{n}}}T^{\text{\textgreek{n}}}\label{eq:TenergyCurrent}
\end{equation}
over a chosen causal hypersurface (see Section \ref{sec:ConstantsAndCurrents}
for our notations on vector field currents), is positive definite
both on $\mathcal{H}^{+}$ and on future null infinity $\mathcal{I}^{+}$.
Thus, integrating the identity 
\begin{equation}
\nabla^{\text{\textgreek{m}}}\mathscr{E}_{\text{\textgreek{m}}}[\text{\textgreek{y}}]=0
\end{equation}
 over suitable subregions of $\mathcal{M}$ (yielding an identity
for the $T$-energy flux of $\text{\textgreek{y}}$ on the associated
boundary hypersurfaces, the so called $T$-\emph{energy identity})
yields that a real mode solution of (\ref{eq:PotentialWaveGeneralIntro})
must have vanishing radiation field on $\mathcal{I}^{+}$and that
it must actually be an $L^{2}$ eigenfunction for (\ref{eq:PotentialWaveGeneralIntro}).
However, using a unique continuation argument (similar to the ones
appearing in \cite{Rellich1943,Odeh1965}, or using the Carleman estimates
of \cite{Rodnianski2011,Moschidisb}), we can infer that any solution
$\text{\textgreek{y}}$ to (\ref{eq:PotentialWaveGeneralIntro}) of
the form (\ref{eq:ModeIntroduction}) such that $\int_{t=0}|\text{\textgreek{y}}|^{2}<+\infty$
vanishes identically on $\mathcal{M}$; see also \cite{Shlap}. 

Assuming, moreover, that the potential function $V$ in (\ref{eq:PotentialWaveGeneralIntro})
satisfies the non-negativity condition $V\ge0$, one can readily obtain
the following results in addition to 1.a--1.b, yielding the full ``mode
stability'' statement for equation (\ref{eq:PotentialWaveGeneralIntro}):

\begin{enumerate}

\item[2.] Equation (\ref{eq:PotentialWaveGeneralIntro}) does not
have a zero eigenvalue or a zero resonance.%
\footnote{Equation (\ref{eq:PotentialWaveGeneralIntro}) is said to admit a
zero eigenvalue if there exists a solution $\text{\textgreek{y}}$
to (\ref{eq:PotentialWaveGeneralIntro}) which is smooth up to $\mathcal{H}$
(if non-empty) and satisfies $T\text{\textgreek{y}}=0$ and $||\text{\textgreek{y}}||_{L^{2}(\text{\textgreek{S}})}<+\infty$,
where $\text{\textgreek{S}}$ is a Cauchy hypersurface of $(\mathcal{M},g)$.
If $\text{\textgreek{y}}$ satisfies $T\text{\textgreek{y}}=0$, $||\partial\text{\textgreek{y}}||_{L^{2}(\text{\textgreek{S}})}<+\infty$
but $||\text{\textgreek{y}}||_{L^{2}(\text{\textgreek{S}})}=+\infty$,
then $\text{\textgreek{y}}$ is called a zero resonance. %
}

\item[3.] Equation (\ref{eq:PotentialWaveGeneralIntro}) does not
admit any exponentially growing mode solutions, since the energy norm
for solutions to (\ref{eq:PotentialWaveGeneralIntro}) associated
to the $T$-energy flux is positive definite and conserved. 

\end{enumerate}

In particular, these results are in contrast with the situation in
Theorem \hyperref[Theorem 1]{1}.

\subsubsection{\label{sub:SuperradiantFriedman}Superradiant spacetimes: the case
$\mathcal{H}^{+}\cap\{g(T,T)>0\}=\emptyset$}

Let us now examine the behaviour of solutions to (\ref{eq:PotentialWaveGeneralIntro})
in the case $(\mathcal{M},g)$ has a non-empty ergoregion (i.\,e.~$\{g(T,T)>0\}\neq\emptyset$)
with the property that either $\mathcal{H}^{+}=\emptyset$, or $\mathcal{H}^{+}\neq\emptyset$
and $\mathcal{H}^{+}\cap\{g(T,T)>0\}=\emptyset$ (note that the Kerr
exterior spacetime $(\mathcal{M}_{M,a},g_{M,a})$, $a\neq0$, does
\underline{not} satisfy this property). The condition $\mathcal{H}^{+}\cap\{g(T,T)>0\}=\emptyset$
implies the positivity of the flux of (\ref{eq:TenergyCurrent}) through
$\mathcal{H}^{+}$. 

On such a spacetime $(\mathcal{M},g)$, there exist smooth solutions
$\text{\textgreek{y}}$ to equation (\ref{eq:WaveEquationIntro})
with compactly supported initial data, such that the energy of $\text{\textgreek{y}}$
grows to infinity as time increases; see \cite{Friedman1978,Moschidis}.
While the proof of \cite{Moschidis} does not yield the existence
of exponentially growing mode solutions to equation (\ref{eq:WaveEquationIntro})
on $(\mathcal{M},g)$, it is reasonable to expect that, in general,
such a mode exists.%
\footnote{We should note that it is not at all clear if there exists a spacetime
$(\mathcal{M},g)$ with a non-empty ergoregion and no event horizon,
such that $(\mathcal{M},g)$ does not admit an exponentially growing
mode. A positive answer to this question would be particularly interesting,
as it would show that mode stability results for equation (\ref{eq:WaveEquation})
on superradiant spacetimes can coexist with instability results in
physical space. %
} Thus, it is in general expected that the mode stability statements
2 and 3 (concerning the non-existence of mode solutions with a frequency
parameter $\text{\textgreek{w}}$ satisfying $\text{\textgreek{w}}=0$
or $Im(\text{\textgreek{w}})>0$) fail in this case, even under the
sign condition $V\ge0$ for the potential term in equation (\ref{eq:WaveEquationIntro}).
However, such a spacetime $(\mathcal{M},g)$ always satisfies Statement
1.a, i.\,e.~$(\mathcal{M},g)$ does not admit an outgoing real mode
solution to (\ref{eq:PotentialWaveGeneralIntro}) in this case. This
fact can be inferred as follows: the $T$-energy identity and the
positivity of the flux of (\ref{eq:TenergyCurrent}) through $\mathcal{H}^{+}$
imply that any solution $\text{\textgreek{y}}$ to (\ref{eq:PotentialWaveGeneralIntro})
which is of the form (\ref{eq:ModeIntroduction}) for some $\text{\textgreek{w}}\in\mathbb{R}\backslash\{0\}$
and satisfies 
\begin{equation}
\lim_{r\rightarrow+\infty}\big(\partial_{r}\text{\textgreek{y}}-i\text{\textgreek{w}}\text{\textgreek{y}}\big)=0
\end{equation}
in each asymptotically flat end of $(\mathcal{M},g)$ has necessarily
vanishing $T$-energy flux through future null infinity $\mathcal{I}^{+}$.
Therefore, since $\text{\textgreek{w}}\neq0$, it can be readily verified
that, for any Cauchy hypersurface $\text{\textgreek{S}}$ of $(\mathcal{M},g)$:
\begin{equation}
\int_{\text{\textgreek{S}}}|\text{\textgreek{y}}|^{2}<+\infty.\label{eq:L^2eigenfunction}
\end{equation}
Thus, equation (\ref{eq:PotentialWaveGeneralIntro}) in this case
does not admit an outgoing real mode solution. 
\begin{rem*}
In general, we can not exclude the existence of $L^{2}$ ``eigenfunctions''
at a non-zero real frequency parameter $\text{\textgreek{w}}$ for
spacetimes $(\mathcal{M},g)$ as above, i.\,e.~Statement 1.b might
not hold. However, the conditions (\ref{eq:L^2eigenfunction}) and
$\text{\textgreek{w}}\in\mathbb{R}\backslash\{0\}$ imply, through
a suitable unique continuation argument that can be obtained by adapting
the Carleman-type estimates of \cite{Moschidisb} (or the estimates
of Section 6 of \cite{Moschidis}), that any $L^{2}$ ``eigenfunction''
$\text{\textgreek{y}}$ will be identically $0$ in the connected
component of $\mathcal{M}\backslash\{g(T,T)>0\}$ which contains the
asymptotically flat region of $\mathcal{M}$. Thus, under a stronger
unique continuation assumption for equation (\ref{eq:PotentialWaveGeneralIntro})
in a neighborhood of the ergoregion (satisfied, for instance, when
both $(\mathcal{M},g)$ and the potantial $V$ are analytic), the
statement 1.b can also be established. 
\end{rem*}

\subsubsection{\label{sub:SuperradiantNotFriedman}Superradiant spacetimes: the
case $\mathcal{H}^{+}\cap\{g(T,T)>0\}\neq\emptyset$}

Let us finally examine the case when $(\mathcal{M},g)$ has a non-empty
ergoregion and a non-empty future event horizon, with the property
$\mathcal{H}^{+}\cap\{g(T,T)>0\}\neq\emptyset$ (such as the Kerr
exterior spacetime $(\mathcal{M}_{M,a},g_{M,a})$ when $a\neq0$).
In this case, the energy identity of $T$ no longer yields a positive
energy flux through $\mathcal{H}^{+}$, and, thus, the vanishing of
the flux of a real mode solution $\text{\textgreek{y}}$ through $\mathcal{I}^{+}$
can not be inferred as before. Therefore, on such a spacetime, the
aforementioned argument leading to the non-existence of real mode
solutions $\text{\textgreek{y}}$ to equation (\ref{eq:PotentialWaveGeneralIntro})
no longer applies, and the real mode stability statement 1.a might,
in general, fail (in addition to the statements 1.b, 2 and 3). In
particular, Theorems \hyperref[Theorem 1]{1} and \hyperref[Theorem 2]{2}
provide examples of spacetimes $(\mathcal{M},g)$ and potentials $V$
such that the mode stability statement 1.a for equation (\ref{eq:PotentialWaveGeneralIntro})
fails.

In view of the aforementioned discussion, the proof of \cite{Shlap},
that equation (\ref{eq:WaveEquation}) on $(\mathcal{M}_{M,a},g_{M,a})$
(with $0<|a|<M$) does not admit real mode solutions, is highly non-trivial,
and relies on the specific algebraic structure of $(\mathcal{M}_{M,a},g_{M,a})$.
According to Theorem \hyperref[Theorem 1]{1}, this structure is ``destroyed''
by adding a compactly supported non-negative potential term $V\text{\textgreek{y}}$
to equation (\ref{eq:WaveEquation}), a modification that would be
completely harmless in the non-superradiant case (preserving the mode
stability properties of (\ref{eq:WaveEquationIntro}) in the case
where $T$ is everywhere causal on $\mathcal{M}$, as explained in
Section \ref{sub:NonSuperradiant}).

\subsection{\label{sub:SpectralConsequences}Local energy decay for the family
$\square_{g}\text{\textgreek{y}}-V_{\text{\textgreek{l}}}\text{\textgreek{y}}=0$
via continuity in $\text{\textgreek{l}}\in[0,1]$}

As a final example of the differences between superradiant and non-superradiant
spacetimes concerning the behaviour of solutions to equation (\ref{eq:PotentialWaveGeneralIntro}),
we will state below a simple continuity criterion for integrated local
energy decay for a family of wave equations with potential on non-superradiant
spacetimes, which utterly fails in the presence of superradiance.

As we showed in Section \ref{sub:NonSuperradiant}, on a globally
hyperbolic, stationary and asymptotically flat spacetime $(\mathcal{M},g)$,
possibly bounded by an event horizon $\mathcal{H}$ and possessing
a Killing field $T$ which is everywhere timelike on $\mathcal{M}\backslash\mathcal{H}$,
the mode stability statements 1.a and 1.b of Section \ref{sub:NonSuperradiant}
hold. This fact gives rise to the following zero-frequency continuity
criterion for integrated local energy decay for the family of equations
\begin{equation}
\square_{g}\text{\textgreek{y}}-V_{\text{\textgreek{l}}}\text{\textgreek{y}}=0,\label{eq:FamilyEquations}
\end{equation}
$\text{\textgreek{l}}\in[0,1]$, with $V_{\text{\textgreek{l}}}:\mathcal{M}\rightarrow\mathbb{R}$
having compact support in space, satisfying $T(V_{\text{\textgreek{l}}})=0$
and depending smoothly on $\text{\textgreek{l}}$: 

\medskip{}

\noindent \emph{Zero-frequency continuity criterion for integrated
local energy decay:} Provided the family (\ref{eq:FamilyEquations})
satisfies an integrated local energy decay estimate (possibly with
loss of derivatives) when $\text{\textgreek{l}}=0$, and the resolvent
of $\square-V_{\text{\textgreek{l}}}$ can be bounded for frequencies
$\text{\textgreek{w}}$ in a neighborhood of $0$ for all $\text{\textgreek{l}}\in[0,1]$,
then (\ref{eq:FamilyEquations}) satisfies a similar integrated local
energy decay estimate for all $\text{\textgreek{l}}\in[0,1]$. 

\medskip{}

\noindent See Section \ref{sec:Criterion} for a more detailed statement
of the above criterion, as well as for a definition of the resolvent
operator in this setting. 

Let us now examine whether the above criterion can be extended to
the case when $(\mathcal{M},g)$ is superradiant. As we did in Sections
\ref{sub:SuperradiantFriedman} and \ref{sub:SuperradiantNotFriedman},
we have to differentiate between the cases $\{g(T,T)>0\}\cap\mathcal{H}^{+}=\emptyset$
and $\{g(T,T)>0\}\cap\mathcal{H}^{+}\neq\emptyset$. In view of \cite{Friedman1978}
(see also our forthcoming \cite{Moschidis}), in the case when $\{g(T,T)>0\}\neq\emptyset$
and $\mathcal{H}^{+}=\emptyset$ (or at least $\{g(T,T)>0\}\cap\mathcal{H}^{+}=\emptyset$),
there exist solutions $\text{\textgreek{y}}$ to equation (\ref{eq:WaveEquationIntro})
which are not square integrable in time. Thus, in order to study conditions
yielding an integrated local energy decay estimate for families of
equations of the form (\ref{eq:FamilyEquations}) on superradiant
spacetimes, it is necessary to restrict to spacetimes $(\mathcal{M},g)$
with $\mathcal{H}^{+}\neq\emptyset$ satisfying $\{g(T,T)>0\}\cap\mathcal{H}^{+}\neq\emptyset$,
such as the subextremal Kerr spacetime $(\mathcal{M}_{M,a},g_{M,a})$
with $a\neq0$. 

In view of Theorem \hyperref[Theorem 1]{1}, the aforementioned zero-frequency
continuity criterion fails in the case of $(\mathcal{M}_{M,a},g_{M,a})$
with $a\neq0$. In particular, there exists a suitable potential $V$
on $\mathcal{M}_{M,a}$ so that, for the family 
\begin{equation}
\square_{g_{M,a}}\text{\textgreek{y}}-\text{\textgreek{l}}V\text{\textgreek{y}}=0\label{eq:FamilyKerr}
\end{equation}
with $\text{\textgreek{l}}\in[0,1]$, a real mode solution at some
non-zero frequency $\text{\textgreek{w}}_{0}$ first appears for $\text{\textgreek{l}}=\text{\textgreek{l}}_{0}\in(0,1]$,
while the resolvent operator associated to (\ref{eq:Family Kerr})
remains bounded for frequencies close to $0$ for all $\text{\textgreek{l}}\in[0,1]$.
Note that, when $\text{\textgreek{l}}$ is close to $0$, the robustness
of the method of \cite{Shlap} implies that (\ref{eq:FamilyKerr})
does not admit any real mode solutions. See Section \ref{sec:Criterion}
for more details.

\subsection{Outline of the paper}

This paper is organised as follows:

In Section \ref{sec:ConstantsAndCurrents}, we will introduce the
conventions on the notations of constants and vector fields that will
be used throughout the paper.

In Section \ref{sec:KerrProperties}, we will review the basic properties
of the subextremal Kerr spacetimes $(\mathcal{M}_{M,a},g_{M,a})$,
including the separability of the wave operator $\square_{g_{M,a}}$,
and we will introduce the notion of an outgoing mode solution of the
wave equation.

In Section~\ref{sec:ProofOfMainThm}, we will provide a more detailed
statement and the proof of Theorem \hyperref[Theorem 1]{1}. Similarly,
a more detailed statement and the proof of Theorem \hyperref[Theorem 2]{2}
will be presented in Section \ref{sec:ProofOfCorolary}.

In Section \ref{sec:DeformedSpacetimes}, we will construct a class
of asymptotically conic spacetimes with normally hyperbolic trapping,
admitting outgoing real or exponentially growing mode solutions to
the wave equation.

Finally, in Section \ref{sec:Criterion} we will present a zero-frequency
continuity criterion for decay for the family of equations (\ref{eq:FamilyEquations})
on non-superradiant spacetimes, and we will then examine its failure
in the presence of superradiance. To this end, we will also present
a definition of the resolvent operator for the wave equation on a
general class of spacetimes. We should note that Section \ref{sec:Criterion}
can be read independently of Sections \ref{sec:ProofOfCorolary} and
\ref{sec:DeformedSpacetimes}.

\subsection{\label{sub:Acknowledgements}Acknowledgements}

I would like to express my gratitude to my advisor Mihalis Dafermos
for providing me with comments, ideas and assistance while this paper
was being written. I would also like to thank Igor Rodnianski for
many insightful comments and suggestions. Finally, I would like to
thank Yakov Shlapentokh-Rothman and Stefanos Aretakis for many helpful
conversations.

\section{\label{sec:ConstantsAndCurrents}Conventions on constants and vector
field currents}

We will adopt the same conventions for denoting constants, volume
forms and vector field currents as was done in \cite{Moschidisb,Moschidisc}.

In particular, capital letters (e.\,g.~$C$) will be used to denote
``large'' constants, typically appearing on the right hand side
of inequalities, while lower case letters (e.\,g. $c$) will be used
to denote ``small'' constants. The same characters will be frequently
used to denote different constants. The dependence of these constants
on various unfixed parameters will be usually made explicit through
the use of appropriate subscripts. 

The notation $f_{1}\lesssim f_{2}$ for two real functions $f_{1},f_{2}$
will be used to imply that there exists some $C>0$, such that $f_{1}\le Cf_{2}$.
We will also write $f_{1}\sim f_{2}$ when we can bound $f_{1}\lesssim f_{2}$
and $f_{2}\lesssim f_{1}$, while $f_{1}\ll f_{2}$ will be equivalent
to the statement that $\frac{|f_{1}|}{|f_{2}|}$ can be bounded by
some sufficiently small (depending on the context) constant $c>0$.
In each of the aforementioned cases, the dependence of the implicit
constants on any unfixed parameters will be clear from the context.
For any function $f:\mathcal{A}\rightarrow[0,+\infty)$ on some set
$\mathcal{A}$, $\{f\gg1\}$ will denote the subset $\{f\ge C\}$
of $\mathcal{A}$ for some constant $C\gg1$.

The natural volume form $dg$ associated to the metric $g$ of a Lorentzian
manifold $(\mathcal{M}^{d+1},g)$ is defined, in any local coordinate
chart $(x^{0},x^{1},x^{2},\ldots x^{d})$, as: 
\[
dg=\sqrt{-det(g)}dx^{0}\cdots dx^{d}.
\]
We will often omit the notation for $dg$ in the expression of integrals
over measurable subsets of $\mathcal{M}$. The same rule will apply
when integrating over any spacelike hypersurface $\mathcal{S}$ of
$(\mathcal{M},g)$ using the natural volume form of its induced (Riemannian)
metric.

We will also frequently use the language of currents and vector field
multipliers for the wave equation: On any Lorentzian manifold $(\mathcal{M},g)$,
associated to the wave operator 
\begin{equation}
\square_{g}=\frac{1}{\sqrt{-det(g)}}\partial_{\text{\textgreek{a}}}\Big(\sqrt{-det(g)}\cdot g^{\text{\textgreek{a}\textgreek{b}}}\partial_{\text{\textgreek{b}}}\Big)
\end{equation}
 is a $(0,2)$-tensor called the \emph{energy momentum tensor} $Q$.
For any smooth function $\text{\textgreek{y}}:\mathcal{M}\rightarrow\mathbb{C}$,
the energy momentum tensor $Q$ evaluated at $\text{\textgreek{y}}$
is given by the expression

\begin{equation}
Q_{\text{\textgreek{a}\textgreek{b}}}(\text{\textgreek{y}})=\frac{1}{2}\Big(\partial_{\text{\textgreek{a}}}\text{\textgreek{y}}\cdot\partial_{\text{\textgreek{b}}}\bar{\text{\textgreek{y}}}+\partial_{\text{\textgreek{b}}}\bar{\text{\textgreek{y}}}\cdot\partial_{\text{\textgreek{a}}}\text{\textgreek{y}}\Big)-\frac{1}{2}\big(\partial^{\text{\textgreek{g}}}\text{\textgreek{y}}\cdot\partial_{\text{\textgreek{g}}}\bar{\text{\textgreek{y}}}\big)g_{\text{\textgreek{a}\textgreek{b}}}.
\end{equation}
For any continuous and piecewise $C^{1}$ vector field $X$ on $\mathcal{M}$,
the following associated currents can be defined almost everywhere
on $\mathcal{M}$:

\begin{equation}
J_{\text{\textgreek{m}}}^{X}(\text{\textgreek{y}})=Q_{\text{\textgreek{m}\textgreek{n}}}(\text{\textgreek{y}})X^{\text{\textgreek{n}}},
\end{equation}
\begin{equation}
K^{X}(\text{\textgreek{y}})=Q_{\text{\textgreek{m}\textgreek{n}}}(\text{\textgreek{y}})\nabla^{\text{\textgreek{m}}}X^{\text{\textgreek{n}}}.
\end{equation}
 The following divergence identity then holds almost everywhere on
$\mathcal{M}$:

\begin{equation}
\nabla^{\text{\textgreek{m}}}J_{\text{\textgreek{m}}}^{X}(\text{\textgreek{y}})=K^{X}(\text{\textgreek{y}})+Re\Big\{(\square_{g}\text{\textgreek{y}})\cdot X\bar{\text{\textgreek{y}}}\Big\}.
\end{equation}

\section{\label{sec:KerrProperties}Basic properties of the Kerr spacetimes}

In this Section, we will provide an overview of the basic properties
of the subextremal Kerr spacetimes that will be used throughout the
rest of the paper.

\subsection{\label{sub:CoordinateCharts}Special coordinate charts and Killing
fields}

In the present paper, we will work on subextremal Kerr exterior spacetimes
$(\mathcal{M}_{M,a},g_{M,a})$, where $g_{M,a}$ is given in the Boyer--Lindquist
$(t,r,\text{\textgreek{j}},\text{\textgreek{f}})$ chart by the expression
(\ref{eq:KerrMetric}). Keeping the notation introduced in Section
\ref{sec:Introduction}, we will denote with $(\widetilde{\mathcal{M}}_{M,a},\tilde{g}_{M,a})$
the maximally extended Kerr spacetime, while $\mathcal{H}^{+}$ and
$\mathcal{H}^{-}$ will denote the future and past event horizons
associated to one (fixed) asymptotically flat end of $(\widetilde{\mathcal{M}}_{M,a},\tilde{g}_{M,a})$.
The original Kerr spacetime $(\mathcal{M}_{M,a},g_{M,a})$ will be
identified with the domain of outer communications of this asymptotically
flat end of $(\widetilde{\mathcal{M}}_{M,a},\tilde{g}_{M,a})$. 

We will work entirely on $(\mathcal{M}_{M,a},g_{M,a})$, but we will
need to consider functions on $\mathcal{M}_{M,a}$ which are regular
up to $\mathcal{H}^{+}\backslash\mathcal{H}^{-}$. Since the Boyer--Lindquist
coordinate chart can not be extended smoothly beyond $\mathcal{H}^{+}\backslash\mathcal{H}^{-}$
(in view of the fact that the expression (\ref{eq:KerrMetric}) becomes
singular as $r\rightarrow r_{+}$), we will sometimes use the \emph{Kerr
star }coordinate chart $(t^{*},r,\text{\textgreek{j}},\text{\textgreek{f}}^{*})$
on $\mathcal{M}_{M,a}$, which is regular up to $\mathcal{H}^{+}\backslash\mathcal{H}^{-}$.
The new $t^{*},\text{\textgreek{f}}^{*}$ coordinate variables in
this chart are given by the expressions
\[
t^{*}=t+\int_{2r_{+}}^{r}\frac{x^{2}+a^{2}}{(x-r_{+})(x-r_{-})}\, dx
\]
and
\[
\text{\textgreek{f}}^{*}=\text{\textgreek{f}}+\int_{2r_{+}}^{r}\frac{a}{(x-r_{+})(x-r_{-})}\, dx.
\]
In Kerr star coordinates $(t^{*},r,\text{\textgreek{j}},\text{\textgreek{f}}^{*})$,
$\mathcal{H}^{+}\backslash\mathcal{H}^{-}$ corresponds to $\mathbb{R}\times\{r_{+}\}\times\mathbb{S}^{2}$,
where $\mathbb{S}^{2}$ is parametrised by $(\text{\textgreek{j}},\text{\textgreek{f}}^{*})$
in the standard way.

The vector fields $T=\partial_{t}$ and $\text{\textgreek{F}}=\partial_{\text{\textgreek{f}}}$
(in the Boyer--Lindquist coordinate chart) are Killing fields for
$g_{M,a}$, and extend to smooth Killing fields across $\mathcal{H}^{+}\backslash\mathcal{H}^{-}$
(and, in fact, on the whole of $(\widetilde{\mathcal{M}}_{M,a},\tilde{g}_{M,a})$).
The hypersurface $\mathcal{H}^{+}\backslash\mathcal{H}^{-}$ is a
non-degenerate Killing horizon of $(\widetilde{\mathcal{M}}_{M,a},\tilde{g}_{M,a})$,
with associated Killing field 
\begin{equation}
K=T-\frac{a}{2Mr_{+}}\text{\textgreek{F}}.\label{eq:HawkingField}
\end{equation}
We will call the vector field $T$ the \emph{stationary} Killing field
of $(\mathcal{M}_{M,a},g_{M,a})$, while $\text{\textgreek{F}}$ will
be called the \emph{axisymmetric }Killing field.

\subsection{\label{sub:Seperation}Separation of the wave equation on Kerr spacetimes}

A remarkable feature of the Kerr metric is that the wave equation
(\ref{eq:WaveEquation}) on $\mathcal{M}_{M,a}$ is separable. The
separability of the wave equation on $(\mathcal{M}_{M,a},g_{M,a})$
was dicovered by Carter in \cite{Carter1968} and is a consequence
of the fact that in addition to the Killing fields $T$, $\text{\textgreek{F}}$,
the Kerr metric also possesses an additional Killing tensor $K_{\text{\textgreek{a}}\text{\textgreek{b}}}$
(as explained by Walker and Penrose in \cite{Walker1970}). 

In particular, for any $\text{\textgreek{w}}\in\mathbb{C}$ and any
$m\in\mathbb{Z}$, equation (\ref{eq:WaveEquation}) admits formal
solutions of the form
\begin{equation}
\text{\textgreek{y}}(t,r,\text{\textgreek{j}},\text{\textgreek{f}})=e^{-i\text{\textgreek{w}}t}e^{im\text{\textgreek{f}}}S(\text{\textgreek{j}})R(r),\label{eq:ProductSolution}
\end{equation}
 where $S(\text{\textgreek{j}})$ satisfies (for some $\text{\textgreek{l}}\in\mathbb{C}$)
\begin{equation}
-\frac{1}{\sin\text{\textgreek{j}}}\frac{d}{d\text{\textgreek{j}}}\big(\sin\text{\textgreek{j}}\frac{dS}{d\text{\textgreek{j}}}\big)+\big(\frac{m^{2}}{\sin^{2}\text{\textgreek{j}}}-a^{2}\text{\textgreek{w}}^{2}\cos^{2}\text{\textgreek{j}}\big)S=\text{\textgreek{l}}S\label{eq:FormalAngularOde}
\end{equation}
and $R(r)$ satisfies 
\begin{equation}
\text{\textgreek{D}}\frac{d}{dr}\big(\text{\textgreek{D}}\frac{dR}{dr}\big)+\big(a^{2}m^{2}-4Mar\text{\textgreek{w}}m+(r^{2}+a^{2})^{2}\text{\textgreek{w}}^{2}-\text{\textgreek{D}}(\text{\textgreek{l}}+a^{2}\text{\textgreek{w}}^{2})\big)R=0.\label{eq:RadialODE}
\end{equation}

When $\text{\textgreek{w}}\in\mathbb{R}$, equation (\ref{eq:FormalAngularOde}),
combined with the boundary condition 
\begin{equation}
S(\text{\textgreek{j}})\mbox{ is bounded at }\text{\textgreek{j}}=0,\text{\textgreek{p}},
\end{equation}
 defines a well-posed Sturm--Liouville problem, with a set of eigenfunctions
$\{S_{\text{\textgreek{w}}ml}\}_{l\ge|m|}$ forming an orthonormal
basis of $L^{2}(\sin\text{\textgreek{j}}d\text{\textgreek{j}})$,
with corresponding real eigenvalues $\{\text{\textgreek{l}}_{\text{\textgreek{w}}ml}\}_{l\ge|m|}$.
In particular, when $\text{\textgreek{w}}=0$, $S_{0ml}$ are the
standard spherical harmonics, and $\text{\textgreek{l}}_{0ml}=l(l+1)$.
The construction of the pair $(S_{\text{\textgreek{w}}ml},\text{\textgreek{l}}_{\text{\textgreek{w}}ml})$
can also be performed, via perturbation theory, in the case when $\text{\textgreek{w}}$
is complex with $|Im(\text{\textgreek{w}})|\ll1$. See \cite{Shlap,DafRodSchlap}
for more details.

Let us define, for convenience, the auxiliary radial function $r_{*}=r_{*}(r):(r_{+},+\infty)\rightarrow(-\infty,+\infty)$
by solving 
\begin{equation}
\frac{dr_{*}}{dr}=\frac{r^{2}+a^{2}}{\text{\textgreek{D}}}.\label{eq:AuxiliaryR*}
\end{equation}
Having defined $(S_{\text{\textgreek{w}}ml},\text{\textgreek{l}}_{\text{\textgreek{w}}ml})$
as above for $(\text{\textgreek{w}},m,l)\in\mathbb{C}\times\mathbb{Z}\times\mathbb{Z}_{\ge|m|}$
with $|Im(\text{\textgreek{w}})|\ll1$, after setting 
\begin{equation}
u_{\text{\textgreek{w}}ml}(r_{*})=(r^{2}+a^{2})^{1/2}R(r)\label{eq:RescaledRadialFunction}
\end{equation}
and $\text{\textgreek{l}}=\text{\textgreek{l}}_{\text{\textgreek{w}}ml}$
in equation (\ref{eq:RadialODE}), we infer that 
\begin{equation}
u_{\text{\textgreek{w}}ml}^{\prime\prime}+\big(\text{\textgreek{w}}^{2}-V_{\text{\textgreek{w}}ml}\big)u_{\text{\textgreek{w}}ml}=0,\label{eq:SeperatedODE}
\end{equation}
where $^{\prime}$ denotes differentiation with respect to $r_{*}$
and 
\begin{equation}
V_{\text{\textgreek{w}}ml}=\frac{4Mram\text{\textgreek{w}}-a^{2}m^{2}+\text{\textgreek{D}}(\text{\textgreek{l}}_{\text{\textgreek{w}}ml}+a^{2}\text{\textgreek{w}}^{2})}{(r^{2}+a^{2})^{2}}+\frac{\text{\textgreek{D}}}{(r^{2}+a^{2})^{4}}\big(a^{2}\text{\textgreek{D}}+2Mr(r^{2}-a^{2})\big).\label{eq:KerrPotential}
\end{equation}

The aforementioned separability is preserved after adding to the wave
equation (\ref{eq:WaveEquation}) a term of the form $-\frac{(r^{2}+a^{2})^{2}}{(r-r_{+})(r-r_{-})\text{\textgreek{r}}^{2}}V(r)\cdot\text{\textgreek{y}}$
for some smooth function $V:(r_{+},+\infty)\rightarrow\mathbb{R}$.
In particular, the equation 
\begin{equation}
\square_{g_{M,a}}\text{\textgreek{y}}-\frac{(r^{2}+a^{2})^{2}}{(r-r_{+})(r-r_{-})\text{\textgreek{r}}^{2}}V(r)\text{\textgreek{y}}=0\label{eq:KerrWaveWithPotential}
\end{equation}
 admits for any $\text{\textgreek{w}}\in\mathbb{C}$ and any $m\in\mathbb{Z}$
solutions of the form
\begin{equation}
\text{\textgreek{y}}(t,r,\text{\textgreek{j}},\text{\textgreek{f}})=e^{-i\text{\textgreek{w}}t}e^{im\text{\textgreek{f}}}S(\text{\textgreek{j}})R(r),\label{eq:ProductSolution-1}
\end{equation}
 with $S$ as before and $R$ solving 
\begin{equation}
\text{\textgreek{D}}\frac{d}{dr}\big(\text{\textgreek{D}}\frac{dR}{dr}\big)+\big(a^{2}m^{2}+(r^{2}+a^{2})\text{\textgreek{w}}^{2}-\text{\textgreek{D}}(\text{\textgreek{l}}+a^{2}\text{\textgreek{w}}^{2})-(r^{2}+a^{2})^{2}V\big)R=0.\label{eq:RadialODE-1}
\end{equation}
In the case $|Im(\text{\textgreek{w}})|\ll1$ (and with $(S_{\text{\textgreek{w}}ml},\text{\textgreek{l}}_{\text{\textgreek{w}}ml})$,
$l\ge|m|$, as before), setting $u_{\text{\textgreek{w}}ml}(r_{*})\doteq(r^{2}+a^{2})^{1/2}R(r)$
in equation (\ref{eq:RadialODE-1}) for $\text{\textgreek{l}}=\text{\textgreek{l}}_{\text{\textgreek{w}}ml}$,
we obtain: 
\begin{equation}
u_{\text{\textgreek{w}}ml}^{\prime\prime}+\big(\text{\textgreek{w}}^{2}-V_{\text{\textgreek{w}}ml}-V\big)u_{\text{\textgreek{w}}ml}=0.\label{eq:SeperatedODEWithPotential}
\end{equation}

\subsection{\label{sub:Modes}Fourier decomposition and mode solutions}

Any smooth function $\text{\textgreek{Y}}:\mathcal{M}_{M,a}\rightarrow\mathbb{C}$
which is square integrable in the $t$ variable can be represented
as 
\[
\text{\textgreek{Y}}(t,r,\text{\textgreek{j}},\text{\textgreek{f}})\doteq\sum_{(m,l)\in\mathbb{Z}\times\mathbb{Z}_{\ge|m|}}\int_{-\infty}^{\infty}e^{-i\text{\textgreek{w}}t}e^{im\text{\textgreek{f}}}S_{\text{\textgreek{w}}ml}(\text{\textgreek{j}})\text{\textgreek{Y}}_{\text{\textgreek{w}}ml}(r)\, d\text{\textgreek{w}}
\]
for some $\text{\textgreek{Y}}_{\text{\textgreek{w}}ml}:(r_{+},+\infty)\rightarrow\mathbb{C}$
which is square integrable in $\text{\textgreek{w}}$ and square summable
in $(m,l)$. 

For any function $F$ with compact support in $\mathcal{M}_{M,a}$,
any solution $\text{\textgreek{y}}$ to equation
\begin{equation}
\square_{g_{M,a}}\text{\textgreek{y}}=F
\end{equation}
which is smooth up to $\mathcal{H}^{+}\backslash\mathcal{H}^{-}$
and is supported on the set $\{t\ge t_{F}\}$, where 
\begin{equation}
t_{F}=\inf\big\{ t_{0}\in\mathbb{R}\,|\,\{t=t_{0}\}\cap supp(F)\neq\emptyset\},
\end{equation}
is square integrable in time (according to the polynomial decay estimates
established in \cite{DafRodSchlap}) and can thus be decomposed as
\begin{equation}
\text{\textgreek{y}}(t,r,\text{\textgreek{j}},\text{\textgreek{f}})=\sum_{(m,l)\in\mathbb{Z}\times\mathbb{Z}_{\ge|m|}}\int_{-\infty}^{\infty}e^{-i\text{\textgreek{w}}t}e^{im\text{\textgreek{f}}}S_{\text{\textgreek{w}}ml}(\text{\textgreek{j}})R_{\text{\textgreek{w}}ml}(r)\, d\text{\textgreek{w}},
\end{equation}
 with $u_{\text{\textgreek{w}}ml}$ (defined in terms of $R_{\text{\textgreek{w}}ml}$
as in (\ref{eq:RescaledRadialFunction})) satisfying for almost all
$\text{\textgreek{w}}\in\mathbb{R}$: 
\begin{equation}
\begin{cases}
u_{\text{\textgreek{w}}ml}^{\prime\prime}+\big(\text{\textgreek{w}}^{2}-V_{\text{\textgreek{w}}ml}\big)u_{\text{\textgreek{w}}ml}=(r^{2}+a^{2})^{-3/2}\text{\textgreek{D}}(\text{\textgreek{r}}^{2}F)_{\text{\textgreek{w}}ml}\\
u_{\text{\textgreek{w}}ml}^{\prime}-i\text{\textgreek{w}}u_{\text{\textgreek{w}}ml}\rightarrow0\mbox{ as }r_{*}\rightarrow+\infty\\
u_{\text{\textgreek{w}}ml}^{\prime}+i\big(\text{\textgreek{w}}-\frac{am}{2Mr_{+}}\big)u_{\text{\textgreek{w}}ml}\rightarrow0\mbox{ as }r_{*}\rightarrow-\infty.
\end{cases}\label{eq:ODEForTheFourierTransform}
\end{equation}
 See \cite{DafRodSchlap} for more details.

We will now define the notion of an outgoing mode solution to equation
(\ref{eq:KerrWaveWithPotential}):
\begin{defn*}
Let $V:(r_{+},+\infty)\rightarrow\mathbb{\mathbb{C}}$ be any smooth
and compactly supported function. A solution $\text{\textgreek{y}}$
to equation 
\begin{equation}
\square_{g_{M,a}}\text{\textgreek{y}}-\frac{(r^{2}+a^{2})^{2}}{(r-r_{+})(r-r_{-})\text{\textgreek{r}}^{2}}V(r)\text{\textgreek{y}}=0\label{eq:WaveEquationPotential}
\end{equation}
 on $(\mathcal{M}_{M,a},g_{M,a})$ will be called an \emph{outgoing
mode solution} (or simply a mode solution) with parameters $(\text{\textgreek{w}},m,l)\in\big(\mathbb{C}\backslash\{0\}\big)\times\mathbb{Z}\times\mathbb{Z}_{\ge|m|}$,
$0\le Im(\text{\textgreek{w}})\ll1$, if $\text{\textgreek{y}}$ is
of the form 
\begin{equation}
\text{\textgreek{y}}(t,r,\text{\textgreek{j}},\text{\textgreek{f}})=e^{-i\text{\textgreek{w}}t}e^{im\text{\textgreek{f}}}S_{\text{\textgreek{w}}ml}(\text{\textgreek{j}})\cdot R_{\text{\textgreek{w}}ml}(r),\label{eq:ModeSolution}
\end{equation}
 with $u_{\text{\textgreek{w}}ml}$ (defined in terms of $R_{\text{\textgreek{w}}ml}$
as in (\ref{eq:RescaledRadialFunction})) satisfying 
\begin{equation}
\begin{cases}
u_{\text{\textgreek{w}}ml}^{\prime\prime}+\big(\text{\textgreek{w}}^{2}-V_{\text{\textgreek{w}}ml}-V\big)u_{\text{\textgreek{w}}ml}=0\\
e^{-i\text{\textgreek{w}}r_{*}}\big(u_{\text{\textgreek{w}}ml}^{\prime}-i\text{\textgreek{w}}u_{\text{\textgreek{w}}ml}\big)\rightarrow0\mbox{ as }r_{*}\rightarrow+\infty\\
e^{i\big(\text{\textgreek{w}}-\frac{am}{2Mr_{+}}\big)r_{*}}\big(u_{\text{\textgreek{w}}ml}^{\prime}+i\big(\text{\textgreek{w}}-\frac{am}{2Mr_{+}}\big)u_{\text{\textgreek{w}}ml}\big)\rightarrow0\mbox{ as }r_{*}\rightarrow-\infty.
\end{cases}\label{eq:ODEForMode}
\end{equation}
\end{defn*}
\begin{rem*}
The above definition of a mode solution also extends to time frequencies
$\text{\textgreek{w}}$ with $Im(\text{\textgreek{w}})<0$, provided
the boundary conditions at $r_{*}=\pm\infty$ in (\ref{eq:ODEForMode})
are replaced by the equivalent (in the $Im(\text{\textgreek{w}})\ge0$
case) conditions $u_{\text{\textgreek{w}}ml}\sim e^{i\text{\textgreek{w}}r_{*}}$
and $u_{\text{\textgreek{w}}ml}\sim e^{-i\big(\text{\textgreek{w}}-\frac{am}{2Mr_{+}}\big)r_{*}}$
as $r_{*}\rightarrow+\infty,-\infty$ respectively. Such solutions
to the wave equation (\ref{eq:WaveEquation}) are known in the physics
literature as \emph{quasinormal modes}, and have been studied extensively
(see e.\,g.~\cite{Chandrasekhar1991}). For a definition of quasi-normal
modes on more general spacetimes, see \cite{Dyatlov2011,Dyatlov2015,Warnick2015}.
Notice, also, that it is straightforward to extend the above definition
of an outgoing mode solution to any metric $g$ on $\mathcal{M}_{M,a}$,
for which $\square_{g}$ completely separates in the Boyer--Lindquist
coordinate chart (such as the metrics constructed in Sections \ref{sec:ProofOfCorolary}--\ref{sec:DeformedSpacetimes}).
\end{rem*}
The notion of an outgoing mode solution can be introduced for more
general stationary and asymptotically flat spacetimes: Let$(\mathcal{M},g)$
be a smooth and globally hyperbolic spacetime, which is stationary
(with stationary Killing field $T$), asymptotically flat (with the
asymptotics described by Assumption 1 in Section 2.1.1 of \cite{Moschidisb})
and possibly bounded by a future event horizon $\mathcal{H}^{+}$
and a past event horizon $\mathcal{H}^{-}$ (for the relevant definitions,
see Section 2.1.1 of \cite{Moschidisb}). Let $t:\mathcal{M}\rightarrow\mathbb{R}$
satisfy $T(t)=1$, with $\{t=0\}$ being a Cauchy hypersurface of
$(\mathcal{M},g)$, and let $N$ be a globally timelike vector field
on $(\mathcal{M},g)$ satisfying $[T,N]=0$ and $N=T$ in the asymptotically
flat region of $(\mathcal{M},g)$. Finally, let $\mathcal{S}\subset\mathcal{M}$
be a smooth, inextendible spacelike hyperboloidal hypersurface satisfying
$\mathcal{S}\cap\mathcal{H}^{-}=\emptyset$, intersecting $\mathcal{H}^{+}$
transversally (in the case $\mathcal{H}^{+}\neq\emptyset$) and terminating
at future null infinity $\mathcal{I}^{+}$ (see Section 3.1 of \cite{Moschidisc}).
We can then introduce the following defintion:
\begin{defn*}
A smooth solution $\text{\textgreek{y}}$ to equation 
\begin{equation}
\square_{g}\text{\textgreek{y}}-V\text{\textgreek{y}}=0,
\end{equation}
for some smooth $V:\mathcal{M}\rightarrow\mathbb{C}$ satisfying $T(V)=0$,
is called an outgoing mode solution with frequency parameter $\text{\textgreek{w}}\in\mathbb{C}\backslash\{0\}$,
$Im(\text{\textgreek{w}})\ge0$, if $\text{\textgreek{y}}$ has the
following properties:

\begin{enumerate}

\item The function $\text{\textgreek{y}}$ is of the form 
\begin{equation}
\text{\textgreek{y}}=e^{-i\text{\textgreek{w}}t}\text{\textgreek{y}}_{\text{\textgreek{w}}}
\end{equation}
with $T(\text{\textgreek{y}}_{\text{\textgreek{w}}})=0$. 

\item We have 
\begin{equation}
\int_{\mathcal{S}}J_{\text{\textgreek{m}}}^{N}(\text{\textgreek{y}})n_{\mathcal{S}}^{\text{\textgreek{m}}}<+\infty,\label{eq:FiniteEnergyHyperboloids}
\end{equation}
where $n_{\mathcal{S}}$ is the future directed unit normal to $\mathcal{S}$.

\item In the case $\text{\textgreek{w}}\in\mathbb{R}\backslash\{0\}$,
we have 
\begin{equation}
\int_{\{t=0\}}J_{\text{\textgreek{m}}}^{N}(\text{\textgreek{y}})n^{\text{\textgreek{m}}}=+\infty,
\end{equation}
where $n$ is the future directed unit normal to $\{t=0\}$.

\end{enumerate}
\end{defn*}
We should remark that (\ref{eq:FiniteEnergyHyperboloids}) is a condition
on both the regularity of $\text{\textgreek{y}}$ near $\mathcal{H}^{+}\backslash\mathcal{H}^{-}$
and the decay properties of the first derivatives of $\text{\textgreek{y}}$
near $\mathcal{I}^{+}$.

Notice that, specialising the latter definition to the case of Kerr
exterior spacetime $(\mathcal{M}_{M,a},g_{M,a})$, a mode solution
$\text{\textgreek{y}}$ with parameters $(\text{\textgreek{w}},m,l)$
on $(\mathcal{M}_{M,a},g_{M,a})$, according to our former definition,
is automatically an outgoing mode solution with frequency parameter
$\text{\textgreek{w}}$, according to the latter definition. See also
\cite{Shlap} for the relation between the outgoing radiation condition
at $r=+\infty$ and condition (\ref{eq:FiniteEnergyHyperboloids}).

\subsection{\label{sub:SuperradiantRegime}The superradiant frequency regime}

The \emph{superradiant frequency regime }of $(\mathcal{M}_{M,a},g_{M,a})$,
consists of those frequency triads $(\text{\textgreek{w}},m,l)\in\mathbb{R}\times\mathbb{Z}\times\mathbb{Z}_{\ge|m|}$
for which the limits 
\begin{equation}
\mathcal{F}_{\pm}[u_{\text{\textgreek{w}}ml}]=\lim_{r_{*}\rightarrow\pm\infty}\pm Im\big(\text{\textgreek{w}}u_{\text{\textgreek{w}}ml}^{-1}\cdot u_{\text{\textgreek{w}}ml}^{\prime}\big)\label{eq:FluxesForSuperradiance}
\end{equation}
 for any non-zero function $u_{\text{\textgreek{w}}ml}$ satisfying
(\ref{eq:ODEForTheFourierTransform}) have opposite sign.%
\footnote{Note that the limits (\ref{eq:FluxesForSuperradiance}) differ from
the frequency seperated $T$-energy fluxes at $r_{*}=\pm\infty$ by
a factor of $|u_{\text{\textgreek{w}}ml}|^{-2}$.%
} In view of the boundary conditions of (\ref{eq:ODEForTheFourierTransform})
at $r_{*}=\pm\infty$, it readily follows that the superradiant frequency
regime has the following form: 
\begin{equation}
\mathfrak{A}_{sup}^{(a,M)}=\Big\{(\text{\textgreek{w}},m,l)\in\mathbb{R}\times\mathbb{Z}\times\mathbb{Z}_{\ge|m|}\,\big|\,\text{\textgreek{w}}\big(\text{\textgreek{w}}-\frac{am}{2Mr_{+}}\big)<0\Big\}.\label{eq:SuperradiantRegimeKerr}
\end{equation}
For a more detailed discussion about the structure of the superradiant
frequency regime (\ref{eq:SuperradiantRegimeKerr}), as well as its
relation to energy estimates for solutions to (\ref{eq:WaveEquation}),
see \cite{Shlap} and references therein.

\section{\label{sec:ProofOfMainThm}Proof of Theorem \hyperref[Theorem 1]{1}}

In this Section, we will provide a more detailed statement and the
proof of Theorem \hyperref[Theorem 1]{1}.

\subsection{Detailed statement and sketch of the proof of Theorem \hyperref[Theorem 1]{1}}

A more detailed statement of Theorem \hyperref[Theorem 1]{1} is the
following:

\begin{customTheorem}{1}[detailed version] \label{Theorem 1 detailed}
For any $0<|a|<M$ and any frequency triad $(\text{\textgreek{w}}_{R},m,l)\in(\mathbb{R}\backslash\{0\})\times\mathbb{Z}\times\mathbb{Z}_{\ge|m|}$
in the superradiant regime (\ref{eq:SuperradiantRegimeKerr}), there
exist constants $C_{\text{\textgreek{w}}_{R}ml}>r_{+}$ and $C_{\text{\textgreek{w}}_{R}ml}^{(0)}>0$
depending only on $\text{\textgreek{w}}_{R},m,l$ (as well as the
Kerr parameters $M,a$) such that for any $r_{0}>C_{\text{\textgreek{w}}_{R}ml}$
and any $\text{\textgreek{w}}_{I}\ge0$ sufficiently small in terms
of $(\text{\textgreek{w}}_{R},m,l,r_{0})$, there exists a smooth
function $V:(r_{+},\infty)\rightarrow[0,2\text{\textgreek{w}}_{R}^{2}]$
supported on $\{r_{0}\le r\le r_{0}+C_{\text{\textgreek{w}}_{R}ml}^{(0)}\}$
such that the equation 
\begin{equation}
\square_{g_{M,a}}\text{\textgreek{y}}-\frac{(r^{2}+a^{2})^{2}}{(r-r_{+})(r-r_{-})\text{\textgreek{r}}^{2}}V(r)\cdot\text{\textgreek{y}}=0
\end{equation}
 on Kerr spacetime $(\mathcal{M}_{M,a},g_{M,a})$ admits an outgoing
mode solution with parameters $(\text{\textgreek{w}}_{R}+i\text{\textgreek{w}}_{I},m,l)$. 

\end{customTheorem}

The proof of Theorem \hyperref[Theorem 1 detailed]{1} will be presented
in the following sections. In particular, we will prove the following
stronger proposition, which immediately yields Theorem \hyperref[Theorem 1 detailed]{1}
as a special case:
\begin{prop}
\label{prop:StrongerProposition}For any $0<|a|<M$, any frequency
triad $(\text{\textgreek{w}}_{R},m,l)\in(\mathbb{R}\backslash\{0\})\times\mathbb{Z}\times\mathbb{Z}_{\ge|m|}$
in the superradiant regime (\ref{eq:SuperradiantRegimeKerr}), any
smooth function $\mathcal{G}:\mathbb{R}^{4}\rightarrow\mathbb{C}$
such that $\mathcal{G}[\text{\textgreek{w}}_{R},0,\cdot,\cdot]\equiv0$
and $\mathcal{G}[\text{\textgreek{w}}_{R},\cdot,0,\cdot]\equiv0$,
there exist constants $C_{\text{\textgreek{w}}_{R}ml}>r_{+}$ and
$C_{\text{\textgreek{w}}_{R}ml}^{(0)}>0$ , depending on $\text{\textgreek{w}}_{R},m,l,M,a$,
such that, for any $r_{0}>C_{\text{\textgreek{w}}_{R}ml}$, any $0<\text{\textgreek{e}}_{1}\le1$
and any $0\le\text{\textgreek{w}}_{I}\ll1$ sufficiently small in
terms of $\text{\textgreek{w}}_{R},m,l,M,a,r_{0},\text{\textgreek{e}}_{1}$
and $\mathcal{G}$, there exists a smooth real valued function $V:(r_{+},+\infty)\rightarrow[0,+\infty)$
supported on $\{r_{0}\le r\le r_{0}+C_{\text{\textgreek{w}}_{R}ml}^{(0)}\}$
and satisfying 
\begin{equation}
V(r)+\frac{4Mram\text{\textgreek{w}}_{R}-a^{2}m^{2}+\text{\textgreek{D}}(\text{\textgreek{l}}_{\text{\textgreek{w}}_{R}ml}+a^{2}\text{\textgreek{w}}_{R}^{2})}{(r^{2}+a^{2})^{2}}\le(1+\text{\textgreek{e}}_{1})\text{\textgreek{w}}_{R}^{2},\label{eq:UpperBoundPrecise}
\end{equation}
such that equation 
\begin{equation}
u^{\prime\prime}+\Big((\text{\textgreek{w}}_{R}+i\text{\textgreek{w}}_{I})^{2}-V_{(\text{\textgreek{w}}_{R}+i\text{\textgreek{w}}_{I})ml}(r_{*})+\mathcal{G}\big(\text{\textgreek{w}}_{R},\text{\textgreek{w}}_{I},V(r(r_{*})),r_{*}\big)-V(r(r_{*}))\Big)u=0,\label{eq:SeperatedODEWithPotential-1}
\end{equation}
where $V_{\text{\textgreek{w}}ml}(r_{*})=V_{\text{\textgreek{w}}ml}(r(r_{*}))$
is given by (\ref{eq:KerrPotential}), admits a solution $u$ satisfying
the following boundary conditions at $r_{*}=\pm\infty$:
\begin{gather}
\lim_{r_{*}\rightarrow+\infty}\Big(e^{-i(\text{\textgreek{w}}_{R}+i\text{\textgreek{w}}_{I})r_{*}}\big(u^{\prime}-i(\text{\textgreek{w}}_{R}+i\text{\textgreek{w}}_{I})u\big)\Big)=0,\label{eq:BoundaryCondition+Infty}\\
\lim_{r_{*}\rightarrow-\infty}\Big(e^{i\big(\text{\textgreek{w}}_{R}+i\text{\textgreek{w}}_{I}-\frac{am}{2Mr_{+}}\big)r_{*}}\big(u^{\prime}+i\big(\text{\textgreek{w}}_{R}+i\text{\textgreek{w}}_{I}-\frac{am}{2Mr_{+}}\big)u\big)\Big)=0.\label{eq:BoundaryCondition-Infty}
\end{gather}

\end{prop}
Notice that, in view of the separation of the wave equation on $(\mathcal{M}_{M,a},g_{M,a})$
and the definition of a mode solution (see Sections \ref{sub:Seperation},
\ref{sub:Modes} respectively), Theorem \hyperref[Theorem 1 detailed]{1}
is obtained after setting $\mathcal{G}\equiv0$ and $\text{\textgreek{e}}_{1}=1$
in the statement of Proposition \ref{prop:StrongerProposition}. Choosing
the function $\mathcal{G}$ and the parameter $\text{\textgreek{e}}_{1}$
in a different way will be useful in the proof of Theorem \hyperref[Theorem 2]{2}.

The proof of Proposition \ref{prop:StrongerProposition} will occupy
Sections \ref{sub:UsefulLemma} and \ref{sub:Proof}. We will now
proceed to sketch the main ideas of the proof.

\paragraph*{Sketch of the proof of Proposition \ref{prop:StrongerProposition}.}

The proof of Proposition \ref{prop:StrongerProposition} will be mainly
based on the following straightforward observation, which we state
here in the form of a lemma:
\begin{lem}
\label{lem:ConservationOfEnergyFlux}For any function $w:\mathbb{R}\rightarrow\mathbb{C}$
satisfying an equation of the form 
\begin{equation}
\frac{d^{2}w}{dx^{2}}+\text{\textgreek{W}}\cdot w=0
\end{equation}
 for some real function $\text{\textgreek{W}}:\mathbb{R}\rightarrow\mathbb{R}$,
the quantity $Im\big(\frac{dw}{dx}\cdot\bar{w}\big)$ is constant
in $x$.
\end{lem}
Notice that in the case of the separated equation (\ref{eq:SeperatedODEWithPotential}),
the conserved quantity $Im\big(\frac{du_{\text{\textgreek{w}}ml}}{dr_{*}}\cdot\bar{u}_{\text{\textgreek{w}}ml}\big)$
is proportional to the frequency seperated $T$-energy current (see,
e.\,g.~\cite{DafRodSchlap}).

As a first step towards establishing Proposition \ref{prop:StrongerProposition},
we will show that for any $\text{\textgreek{w}}\in\mathbb{R}\backslash\{0\}$
and any two solutions $u_{1},u_{2}:[a,b]\rightarrow\mathbb{C}$ of 

\begin{equation}
\frac{d^{2}u}{dx^{2}}+\text{\textgreek{w}}^{2}u=0\label{eq:BasicODE-1}
\end{equation}
 satisfying 
\begin{equation}
Im(u_{1}^{\prime}\bar{u}_{1})=Im(u_{2}^{\prime}\bar{u}_{2}),\label{eq:SameConstantOfMotion}
\end{equation}
assuming also that $b-a$ is large enough, there exists a piecewise
constant potential $V:[a,b]\rightarrow[0,\text{\textgreek{w}}^{2}]$
which is identically $0$ in a neighborhood of $x=a,b$, such that
the modified equation 
\begin{equation}
\frac{d^{2}u}{dx^{2}}+\big(\text{\textgreek{w}}^{2}-V\big)u=0\label{eq:BasicODE-1-1}
\end{equation}
admits a solution $u$ such that $u=u_{1}$ in a neighborhood of $x=a$
and $u=e^{i\text{\textgreek{j}}_{2}}u_{2}$ in a neighborhood of $x=b$,
for some suitable $\text{\textgreek{j}}_{2}\in[0,2\text{\textgreek{p}})$.
See Lemma \ref{eq:RoughProblem} for more details. Note that, in view
of Lemma \ref{lem:ConservationOfEnergyFlux}, the condition (\ref{eq:SameConstantOfMotion})
is necessary for the existence of such a function $V$. 

The proof of Lemma \ref{eq:RoughProblem} will be based on the observation
that the functions $u_{1},u_{2}:[a,b]\rightarrow\mathbb{C}$ trace
out two ellipses $\mathcal{C}_{1},\mathcal{C}_{2}$ in the complex
plane, having the same orientation in view of (\ref{eq:SameConstantOfMotion}).
This fact will be used to show that, for a suitable value of $\text{\textgreek{j}}_{2}$
and a well chosen closed interval $[x_{1},x_{2}]\subset(a,b)$, the
curve $u:[a,b]\rightarrow\mathbb{C}$ defined so that
\begin{enumerate}
\item $u=u_{1}$ on $[a,x_{1}]$,
\item $u=e^{i\text{\textgreek{j}}_{2}}u_{2}$ on $[x_{2},b]$ and
\item $u=\text{\textgreek{e}}_{\text{\textgreek{j}}_{2}}$ on $[x_{1},x_{2}]$,
where $\text{\textgreek{e}}_{\text{\textgreek{j}}_{2}}:[x_{1},x_{2}]\rightarrow\mathbb{C}$
is a parametrization of one of the common tangent lines of the ellipse
$\mathcal{C}_{1}$ and the rotated ellipse $e^{i\text{\textgreek{j}}_{2}}\mathcal{C}_{2}$
\end{enumerate}
is a $C^{1}$ and piecewise $C^{2}$ function from $[a,b]$ to $\mathbb{C}$.
In particular, $u$ will satisfy an ode of the form (\ref{eq:BasicODE-1-1}).

The second step in the proof of Proposition \ref{prop:StrongerProposition}
will consist of showing the following: Even if equation (\ref{eq:BasicODE-1})
is perturbed by introducing a small imaginary part for $\text{\textgreek{w}}$
(as well as other small complex valued terms), and the function $V$
is in addition required to be smooth, one can still obtain a suitable
potential $V$ and a smooth solution $u$ of the (perturbed version
of) equation (\ref{eq:BasicODE-1-1}), such that $u$ induces any
chosen initial data on $a$ and $b$ which are close enough to the
initial data induced by $u_{1}$ and $\text{\textgreek{l}}u_{2}$,
respectively, for some suitable $\text{\textgreek{l}}\in\mathbb{C}\backslash\{0\}$.
See Lemma \ref{lem:PerturbationOfTheSimplerODE} for more details. 

Finally, the proof of Proposition \ref{prop:StrongerProposition}
will be completed in Section \ref{sub:Proof}, roughly along the following
lines: We will first choose the functions $u_{1}$ and $u_{2}$ (appearing,
for instance, in the statement of Lemma \ref{lem:PerturbationOfTheSimplerODE})
so that they satisfy the boundary conditions (\ref{eq:BoundaryCondition+Infty})
and (\ref{eq:BoundaryCondition-Infty}), respectively, for $\text{\textgreek{w}}_{I}=0$,
normalised so that (\ref{eq:SameConstantOfMotion}) also holds (this
is possible only in the case when the frequency triad $(\text{\textgreek{w}}_{R},m,l)$
lies in the superradiant regime (\ref{eq:SuperradiantRegimeKerr})).
Then, the proof will follow by applying Lemma \ref{lem:PerturbationOfTheSimplerODE}
for equation (\ref{eq:SeperatedODEWithPotential-1}) on the interval
$\{r_{0}\le r\le r_{0}+C_{\text{\textgreek{w}}_{R}ml}^{(0)}\}$, for
a well chosen perturbation of the initial data induced by $u_{1},u_{2}$
at $r=r_{0},r=r_{0}+C_{\text{\textgreek{w}}_{R}ml}^{(0)}$ respectively.

\subsection{\label{sub:UsefulLemma}The main lemmas}

In this section, we will state and prove some lemmas concerning the
behaviour of solutions to the ordinary differential equation (\ref{eq:BasicODE-1})
with equal constants of motion. These lemmas lie at the heart of the
proof of Proposition \ref{prop:StrongerProposition}, and some of
the technical details involved in their proof will be needed in the
constructions of Section \ref{sec:DeformedSpacetimes}.
\begin{lem}
\label{lem:RoughPotential}Let $\text{\textgreek{w}},L_{0}\in\mathbb{R}\backslash\{0\}$
and $C_{0}>0$, and let $a<b$ be two real numbers such that $b-a>\frac{2\text{\textgreek{p}}}{\text{\textgreek{w}}}+\frac{\text{\textgreek{p}}C_{0}^{2}}{|L_{0}|}+2$.
Let $u_{1},u_{2}:[a,b]\rightarrow\mathbb{C}$ be two solutions of
the ordinary differential equation 
\begin{equation}
\frac{d^{2}u}{dx^{2}}+\text{\textgreek{w}}^{2}u=0\label{eq:BasicODE}
\end{equation}
with equal and non-zero constant of motion (see Lemma \ref{lem:ConservationOfEnergyFlux})
\begin{equation}
Im\big(\frac{du_{1}}{dx}\cdot\bar{u}_{1}\big)=Im\big(\frac{du_{2}}{dx}\cdot\bar{u}_{2}\big)=L_{0}\neq0,\label{eq:EqualAngularMomentum}
\end{equation}
satisfying also 
\begin{equation}
\sup_{[a,b]}\big(|u_{1}|+|u_{2}|\big)\le C_{0},\label{eq:SupBound}
\end{equation}
such that $\frac{u_{1}}{u_{2}}$ is not a constant function. Then
for any $x_{0}\in[a,b]$ such that $[x_{0},x_{0}+\frac{2\text{\textgreek{p}}}{\text{\textgreek{w}}}+\frac{\text{\textgreek{p}}C_{0}^{2}}{|L_{0}|}]\subset(a,b)$,
there exists a $C^{1}$ and piecewise $C^{2}$ function $\tilde{u}:[a,b]\rightarrow\mathbb{C}$
satisfying for some suitable $\text{\textgreek{j}}_{2}\in[0,2\text{\textgreek{p}})$
and $x_{1},x_{2}\in[x_{0},x_{0}+\frac{2\text{\textgreek{p}}}{\text{\textgreek{w}}}+\frac{\text{\textgreek{p}}C_{0}^{2}}{|L_{0}|}]$
with $x_{1}<x_{2}$: 
\begin{equation}
\begin{cases}
\frac{d^{2}\tilde{u}}{dx^{2}}+\big(\text{\textgreek{w}}^{2}-\tilde{V}(x)\big)\tilde{u}=0\\
\tilde{u}=u_{1}\mbox{ on }[a,x_{1}]\\
\tilde{u}=e^{i\text{\textgreek{j}}_{2}}u_{2}\mbox{ on }[x_{2},b],
\end{cases}\label{eq:RoughProblem}
\end{equation}
where the piecewise constant function $\tilde{V}:\mathbb{R}\rightarrow[0,\text{\textgreek{w}}^{2}]$
is defined in terms of $x_{1},x_{2}$ as: 
\begin{equation}
\tilde{V}(x)=\begin{cases}
0, & x\in\mathbb{R}\backslash[x_{1},x_{2}]\\
\text{\textgreek{w}}^{2}, & x\in[x_{1},x_{2}].
\end{cases}\label{eq:RoughPotential-1}
\end{equation}
\end{lem}
\begin{proof}
For $j=1,2$, we can decompose 
\begin{equation}
u_{j}(x)=A_{j}e^{i\text{\textgreek{w}}x}+B_{j}e^{-i\text{\textgreek{w}}x}\label{eq:Ellipses}
\end{equation}
 for some $A_{j},B_{j}\in\mathbb{C}$. Since $b-a>\frac{2\text{\textgreek{p}}}{\text{\textgreek{w}}}$
and the constant of motion $Im\big(\frac{du_{j}}{dx}\cdot\bar{u}_{j}\big)$
for $u_{j}$ is non zero, the image of $u_{j}$ is an ellipse $\mathcal{C}_{j}$
in the plane of the complex numbers, with its center at the origin.
In view of the equality (\ref{eq:EqualAngularMomentum}), we have
\begin{equation}
|A_{1}|^{2}-|B_{1}|^{2}=|A_{2}|^{2}-|B_{2}|^{2}=\text{\textgreek{w}}^{-1}L_{0}.\label{eq:EqualAngularMomentumForEllipseCoefficients}
\end{equation}
Thus, in view of (\ref{eq:EqualAngularMomentumForEllipseCoefficients})
and the fact that $u_{1}/u_{2}$ was assumed to not be identically
constant, we infer that the semi-major and semi-minor axes of $\mathcal{C}_{1},\mathcal{C}_{2}$
have different length, i.\,e. 
\begin{equation}
|A_{1}|+|B_{1}|\neq|A_{2}|+|B_{2}|\mbox{ and }|A_{1}|-|B_{1}|\neq|A_{2}|-|B_{2}|.\label{eq:UnequalAxes}
\end{equation}
Therefore, for any $\text{\textgreek{j}}\in[0,2\text{\textgreek{p}})$,
the ellipses $\mathcal{C}_{1}$ and $e^{i\text{\textgreek{j}}}\mathcal{C}_{2}$
(where $e^{i\text{\textgreek{j}}}\mathcal{C}_{2}$ is the rotation
of $\mathcal{C}_{2}$ by a $\text{\textgreek{j}}$-angle around the
origin) intersect at four distinct points which vary smoothly with
$\text{\textgreek{j}}$. Moreover, the orientation induced on $\mathcal{C}_{1},\mathcal{C}_{2}$
by their parametrization through $u_{j}:[a,b]\rightarrow\mathcal{C}$
is the same, and, in particular, it is clockwise if $L_{0}<0$ and
counter-clockwise if $L_{0}>0$. 

\begin{figure}[h] 
\centering 
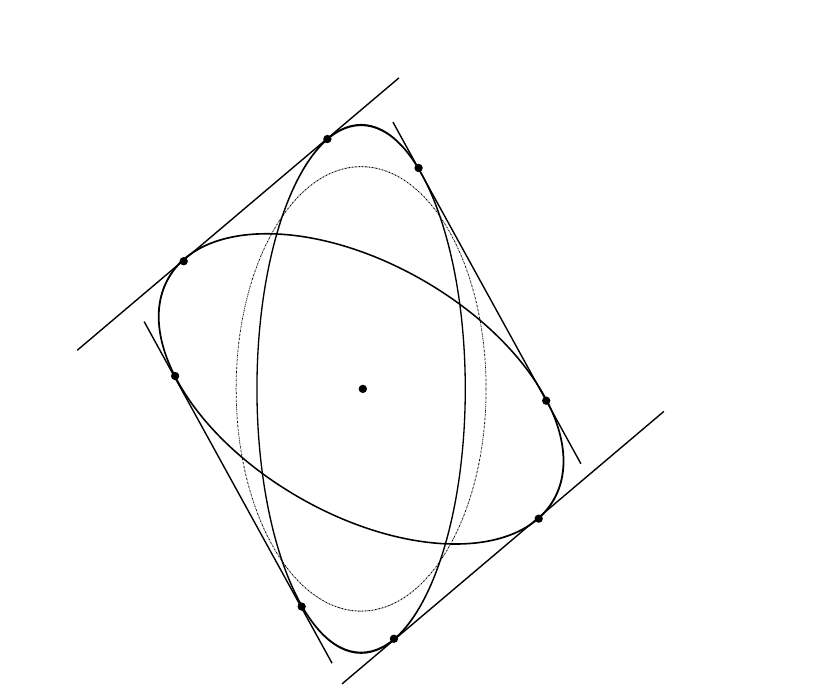 
\caption{The two ellipses $\mathcal{C}_{1}$ and $e^{i \text{\textgreek{j}}} \mathcal{C}_{2}$, with their four common tangents $\text{\textgreek{e}}^{(j)}_{\text{\textgreek{j}}}$ and the associated tangent points $P^{(j)}_{k,\text{\textgreek{j}}}$.} 
\end{figure}

In view of (\ref{eq:EqualAngularMomentumForEllipseCoefficients})
and (\ref{eq:UnequalAxes}), for any $\text{\textgreek{j}}\in[0,2\text{\textgreek{p}}]$
there exist four straight lines $\text{\textgreek{e}}_{\text{\textgreek{j}}}^{(j)}$,
$j=1,\ldots,4$, which lie in the exterior of $\mathcal{C}_{1}$ and
$e^{i\text{\textgreek{j}}}\mathcal{C}_{2}$ and are tangent to both,
and these lines depend smoothly on $\text{\textgreek{j}}$, satisfying
also $\text{\textgreek{e}}_{0}^{(j)}=\text{\textgreek{e}}_{2\text{\textgreek{p}}}^{(j)}$
(notice that the smooth dependence of $\text{\textgreek{e}}_{\text{\textgreek{j}}}^{(j)}$
on $\text{\textgreek{j}}$ is a consequence of (\ref{eq:UnequalAxes})).
We will parametrize these lines as $\text{\textgreek{e}}_{\text{\textgreek{j}}}^{(j)}:\mathbb{R}\rightarrow\mathbb{C}$,
with 
\begin{equation}
\text{\textgreek{e}}_{\text{\textgreek{j}}}^{(j)}(x)=a_{\text{\textgreek{j}}}^{(j)}x+b_{\text{\textgreek{j}}}^{(j)}\label{eq:Parametrization Line}
\end{equation}
for some $a_{\text{\textgreek{j}}}^{(j)},b_{\text{\textgreek{j}}}^{(j)}\in\mathbb{C}$
(that will be fixed more precisely shortly), depending smoothly on
$\text{\textgreek{j}}$. 

Let $P_{1,\text{\textgreek{j}}}^{(j)}$ and $P_{2,\text{\textgreek{j}}}^{(j)}$
be the points where $\text{\textgreek{e}}_{\text{\textgreek{j}}}^{(j)}$
meets $\mathcal{C}_{1}$ and $e^{i\text{\textgreek{j}}}\mathcal{C}_{2}$
respectively. Without loss of generality, we will assume that $|A_{1}|+|B_{1}|>|A_{2}|+|B_{2}|$
in (\ref{eq:UnequalAxes}).%
\footnote{In case $|A_{1}|+|B_{1}|>|A_{2}|+|B_{2}|$, the same analysis goes
through with the roles of $u_{1}$ and $u_{2}$ interchanged.%
} Then, it follows that, if $\text{\textgreek{e}}_{max}^{(0)}$ is
the straight line in $\mathbb{C}$ defined by the major semi-axis
of $\mathcal{C}_{1}$, no one of the points $P_{k,\text{\textgreek{j}}}^{(j)}$
($k=1,2$, $j=1,\ldots,4$) belongs to $\text{\textgreek{e}}^{(0)}$
for any $\text{\textgreek{j}}\in[0,2\text{\textgreek{p}}]$. Notice
also that $P_{k,\text{\textgreek{j}}}^{(j)}$ vary smoothly with $\text{\textgreek{j}}$
in view of (\ref{eq:UnequalAxes}) and (\ref{eq:EqualAngularMomentumForEllipseCoefficients})
(and $P_{k,0}^{(j)}=P_{k,2\text{\textgreek{p}}}^{(j)}$). Thus, for
any $k=1,2$ and $j=1,\ldots,4$, the point $P_{k,\text{\textgreek{j}}}^{(j)}$
remains on the same side of $\text{\textgreek{e}}^{(0)}$ for all
$\text{\textgreek{j}}\in[0,2\text{\textgreek{p}}]$. 

Let $x_{1,\text{\textgreek{j}}}^{(j)},x_{2,\text{\textgreek{j}}}^{(j)}\in\mathbb{R}$
be the unique points for which $\text{\textgreek{e}}_{\text{\textgreek{j}}}^{(j)}(x_{1,\text{\textgreek{j}}}^{(j)})=P_{1,\text{\textgreek{j}}}^{(j)}$
and $\text{\textgreek{e}}_{\text{\textgreek{j}}}^{(j)}(x_{2,\text{\textgreek{j}}}^{(j)})=P_{2,\text{\textgreek{j}}}^{(j)}$,
$j=1,\ldots,4$. Let us also define $y_{1,\text{\textgreek{j}}}^{(j)}$
to be the unique point in $[x_{0},x_{0}+\frac{2\text{\textgreek{p}}}{\text{\textgreek{w}}})$
such that $u_{1}(y_{1,\text{\textgreek{j}}}^{(j)})=P_{1,\text{\textgreek{j}}}^{(j)}.$
Since $\text{\textgreek{e}}_{\text{\textgreek{j}}}^{(j)}$ is tangent
to $\mathcal{C}_{1}$ at $P_{1,\text{\textgreek{j}}}^{(j)}$, we have
\begin{equation}
\text{\textgreek{e}}_{\text{\textgreek{j}}}^{(j)}(x_{1,\text{\textgreek{j}}}^{(j)})=u_{1}(y_{1,\text{\textgreek{j}}}^{(j)})\label{eq:IntersectionPoint}
\end{equation}
 and 
\begin{equation}
\frac{d\text{\textgreek{e}}_{\text{\textgreek{j}}}^{(j)}}{dx}(x_{1,\text{\textgreek{j}}}^{(j)})=\text{\textgreek{l}}_{1,\text{\textgreek{j}}}^{(j)}\frac{du_{1}}{dx}(y_{1,\text{\textgreek{j}}}^{(j)})\label{eq:TangencyPoint}
\end{equation}
for some $\text{\textgreek{l}}_{1,\text{\textgreek{j}}}^{(j)}\in\mathbb{R}\backslash\{0\}$.
We will uniquely fix the the parametrization (\ref{eq:Parametrization Line})
of $\text{\textgreek{e}}_{\text{\textgreek{j}}}^{(j)}$ by requiring
that $x_{1,\text{\textgreek{j}}}^{(j)}=y_{1,\text{\textgreek{j}}}^{(j)}$
and $\text{\textgreek{l}}_{1,\text{\textgreek{j}}}^{(j)}=1$. 

Notice that for two values of $j$ (say $j=1,3$), we have $x_{1,\text{\textgreek{j}}}^{(j)}<x_{2,\text{\textgreek{j}}}^{(j)}$
for all $\text{\textgreek{j}}\in[0,2\text{\textgreek{p}}]$, i.\,e.~in
this parametrization of $\text{\textgreek{e}}_{\text{\textgreek{j}}}^{(j)}$,
$P_{1,\text{\textgreek{j}}}^{(j)}$ lies before $P_{2,\text{\textgreek{j}}}^{(j)}$.
From now on, we will work only with the family of lines $\text{\textgreek{e}}_{\text{\textgreek{j}}}^{(1)}$
and the family of points $P_{1,\text{\textgreek{j}}}^{(1)}$ and $P_{2,\text{\textgreek{j}}}^{(1)}$.
In view of (\ref{eq:EqualAngularMomentum}), (\ref{eq:SupBound})
and the fact that $\text{\textgreek{l}}_{1,\text{\textgreek{j}}}^{(1)}=1$
in (\ref{eq:TangencyPoint}), as well as the fact that the angle formed
by the polar rays through the points $P_{1,\text{\textgreek{j}}}^{(1)}$
and $P_{2,\text{\textgreek{j}}}^{(1)}$ is at most $\frac{\text{\textgreek{p}}}{2}$,
we can bound for any $\text{\textgreek{j}}\in[0,2\text{\textgreek{p}}]$:
\begin{equation}
x_{2,\text{\textgreek{j}}}^{(1)}-x_{1,\text{\textgreek{j}}}^{(1)}\le\frac{\text{\textgreek{p}}C_{0}^{2}}{|L_{0}|}.\label{eq:TimeDifferenceOfPoints}
\end{equation}
Thus, $x_{2,\text{\textgreek{j}}}^{(1)}\in[x_{0},x_{0}+\frac{2\text{\textgreek{p}}}{\text{\textgreek{w}}}+\frac{\text{\textgreek{p}}C_{0}^{2}}{|L_{0}|})\subset(a,b)$. 

We have shown that the map $h_{2}:[0,2\text{\textgreek{p}}]\rightarrow\mathbb{C}\backslash\{0\}$,
\begin{equation}
h_{2}(\text{\textgreek{j}})=P_{2,\text{\textgreek{j}}}^{(1)}\label{eq:TangentPoint}
\end{equation}
is a smooth map satisfying $h_{2}(0)=h_{2}(2\text{\textgreek{p}})$,
and moreover $h_{2}([0,2\text{\textgreek{p}}])$ does not intersect
$\text{\textgreek{e}}_{0}$. Defining, now, the map $\tilde{h}_{2}:[0,2\text{\textgreek{p}}]\rightarrow\mathbb{C}\backslash\{0\}$,
\begin{equation}
\tilde{h}_{2}(\text{\textgreek{j}})=e^{i\text{\textgreek{j}}}u_{2}(x_{2,0}^{(1)}),\label{eq:EllipsePoint}
\end{equation}
we readily verify that $\tilde{h}_{2}$ is also smooth, satisfying
$\tilde{h}_{2}(0)=\tilde{h}_{2}(2\text{\textgreek{p}})$, and furthermore
the curve $\text{\textgreek{j}}\longmapsto\tilde{h}_{2}(\text{\textgreek{j}})$
has winding number (around the origin) equal to 1. Therefore, there
exists some $\text{\textgreek{j}}_{2}\in[0,2\text{\textgreek{p}})$
and some $\text{\textgreek{l}}>0$ such that 
\begin{equation}
h_{2}(\text{\textgreek{j}}_{2})=\text{\textgreek{l}}\tilde{h}_{2}(\text{\textgreek{j}}_{2}).\label{eq:ParallelVectors}
\end{equation}
Since $h_{2}(\text{\textgreek{j}}_{2}),\tilde{h}_{2}(\text{\textgreek{j}}_{2})\in e^{i\text{\textgreek{j}}_{2}}\mathcal{C}_{2}$
and each ray from the origin intersects the ellipse $e^{i\text{\textgreek{j}}_{2}}\mathcal{C}_{2}$
only once, we thus infer that $\text{\textgreek{l}}=1$, i.\,e. 
\begin{equation}
h_{2}(\text{\textgreek{j}}_{2})=\tilde{h}_{2}(\text{\textgreek{j}}_{2}).\label{eq:ParallelVectors-1}
\end{equation}

In view of (\ref{eq:TangentPoint}), (\ref{eq:EllipsePoint}) and
(\ref{eq:ParallelVectors-1}), we have:
\begin{equation}
\text{\textgreek{e}}_{\text{\textgreek{j}}_{2}}^{(1)}(x_{2,\text{\textgreek{j}}_{2}}^{(1)})=e^{i\text{\textgreek{j}}_{2}}u_{2}(x_{2,\text{\textgreek{j}}_{2}}^{(1)}).\label{eq:IntersectionPoint2}
\end{equation}
Since $\text{\textgreek{e}}_{\text{\textgreek{j}}_{2}}^{(1)}$ is
tangent to $e^{i\text{\textgreek{j}}_{2}}\mathcal{C}_{2}$ at $P_{2,\text{\textgreek{j}}_{2}}^{(1)}$,
we also have for some $\text{\textgreek{l}}_{2}\in\mathbb{R}\backslash\{0\}$:
\begin{equation}
\frac{d\text{\textgreek{e}}_{\text{\textgreek{j}}_{2}}^{(1)}}{dx}(x_{2,\text{\textgreek{j}}_{2}}^{(1)})=\text{\textgreek{l}}_{2}e^{i\text{\textgreek{j}}_{2}}\frac{du_{2}}{dx}(x_{2,\text{\textgreek{j}}_{2}}^{(1)}).\label{eq:TangencyPoint-2}
\end{equation}
Thus, the equalities (\ref{eq:EqualAngularMomentum}), (\ref{eq:TangencyPoint-2})
and (\ref{eq:TangencyPoint}) (in view of the fact that $\text{\textgreek{e}}$
is linear) readily implies that $\text{\textgreek{l}}_{2}=\text{\textgreek{l}}_{1}=1$. 

All in all, after setting (for notational simplicity) $x_{k}=x_{k,\text{\textgreek{j}}_{2}}^{(1)}$
for $k=1,2$ and $\text{\textgreek{e}}=\text{\textgreek{e}}_{\text{\textgreek{j}}_{2}}^{(1)}$,
the equalities (\ref{eq:IntersectionPoint}), (\ref{eq:TangencyPoint}),
(\ref{eq:IntersectionPoint2}) and (\ref{eq:TangencyPoint-2}) yield
\begin{equation}
\text{\textgreek{e}}(x_{1})=u_{1}(x_{1}),\mbox{ }\text{\textgreek{e}}(x_{2})=e^{i\text{\textgreek{j}}_{2}}u_{2}(x_{2})\label{eq:IntersectionPoint-1}
\end{equation}
 and 
\begin{equation}
\frac{d\text{\textgreek{e}}}{dx}(x_{1})=\frac{du_{1}}{dx}(x_{1}),\mbox{ }\frac{d\text{\textgreek{e}}}{dx}(x_{2})=e^{i\text{\textgreek{j}}_{2}}\frac{du_{2}}{dx}(x_{2})\label{eq:TangencyPoint-1}
\end{equation}
with $\text{\textgreek{j}}_{2}\in[0,2\text{\textgreek{p}})$ and 
\begin{equation}
x_{0}\le x_{1}<x_{2}<x_{0}+\frac{2\text{\textgreek{p}}}{\text{\textgreek{w}}}+\frac{\text{\textgreek{p}}C_{0}^{2}}{|L_{0}|}.\label{eq:IntervalLength}
\end{equation}

If we define the function $\tilde{u}:[a,b]\rightarrow\mathbb{C}$
as the unique $C^{1}$ and piecewise $C^{2}$ solution of the initial
value problem 
\begin{equation}
\begin{cases}
\frac{d^{2}\tilde{u}}{dx^{2}}+\big(\text{\textgreek{w}}^{2}-\tilde{V}(x)\big)\tilde{u}=0\\
\tilde{u}(a)=u_{1}(a)\\
\frac{d\tilde{u}}{dx}(a)=\frac{du_{1}}{dx}(a),
\end{cases}\label{eq:RoughProblem-1}
\end{equation}
where $\tilde{V}$ is defined by (\ref{eq:RoughPotential-1}) for
the chosen values of $x_{1},x_{2}$, then $\tilde{u}$ satisfies the
following properties (readily inferred in view of (\ref{eq:IntersectionPoint-1}),
(\ref{eq:TangencyPoint-1}) and the fact that the function $\text{\textgreek{e}}(x)$
satisfies $\frac{d^{2}\text{\textgreek{e}}}{dx^{2}}=0$): 

\begin{itemize}

\item $\tilde{u}(x)=u_{1}(x)$ for $x\in[a,x_{1}]$,

\item $\tilde{u}(x)=\text{\textgreek{e}}(x)$ for $x\in[x_{1},x_{2}]$,

\item $\tilde{u}(x)=e^{i\text{\textgreek{j}}_{2}}u_{2}(x)$ for $x\in[x_{2},b]$.

\end{itemize}

Thus, the proof of the lemma is complete.
\end{proof}
The following lemma will allow us to mollify the piecewise constant
potential $\tilde{V}$ of Lemma \ref{lem:RoughPotential}, as well
as extend Lemma \ref{lem:RoughPotential} in order to include smooth
perturbations of equation \ref{eq:BasicODE}.
\begin{lem}
\label{lem:PerturbationOfTheSimplerODE}Let $\text{\textgreek{w}},L_{0}\in\mathbb{R}\backslash0$,$C_{0}>0$,
$a<b$, $x_{0}\in(a,b)$ and $u_{1},u_{2}:[a,b]\rightarrow\mathbb{C}$
be as in the statement of Lemma \ref{lem:RoughPotential}, and let
$Z:\mathbb{R}^{3}\times[a,b]\rightarrow\mathbb{C}$ be a smooth function
such that $Z(\text{\textgreek{w}},0,v,x)=0$ for all $(v,x)\in,\mathbb{R}\times[a,b]$.
We will use the following notation for the absolute value of the Wronskian
of $u_{1},u_{2}$ (which is constant in $x\in[a,b]$): 
\begin{equation}
\mathcal{W}[u_{1},u_{2}]=\Big|\frac{du_{1}}{dx}u_{2}-u_{1}\frac{du_{2}}{dx}\Big|.
\end{equation}
Then, for any $0<\text{\textgreek{e}}_{1}\le1$, there exists a $\text{\textgreek{d}}_{0}$
sufficiently small in terms of $\text{\textgreek{w}}$, $C_{0}$,
$L_{0}$, $\text{\textgreek{e}}_{1}$, $\mathcal{Z}$ and $\mathcal{\mathcal{W}}[u_{1},u_{2}]$,
where 
\begin{equation}
\mathcal{Z}\doteq\sup_{|\text{\textgreek{w}}_{I}|\le1,|v|\le2\text{\textgreek{w}}^{2}}\sum_{j=0}^{2}\int_{a}^{b}\big|\partial_{\text{\textgreek{w}}_{I}}^{j}Z(\text{\textgreek{w}},\text{\textgreek{w}}_{I},v,x)\big|\, dx,\label{eq:BoundControlTerm}
\end{equation}
such that for any $\text{\textgreek{w}}_{I}\in(-\text{\textgreek{d}}_{0},\text{\textgreek{d}}_{0})$
and any initial data sets $(u_{a}^{(0)},u_{a}^{(1)}),(u_{b}^{(0)},u_{b}^{(1)})\in\mathbb{C}^{2}$
satisfying 
\begin{equation}
\big|u_{a}^{(0)}-u_{1}(a)\big|+\big|u_{a}^{(1)}-\frac{du_{1}}{dx}(a)\big|+\big|u_{b}^{(0)}-u_{2}(b)\big|+\big|u_{b}^{(1)}-\frac{du_{2}}{dx}(b)\big|<\text{\textgreek{d}}_{0},\label{eq:ClosenessData}
\end{equation}
there exists a smooth function $V:\mathbb{R}\rightarrow[0,(1+\text{\textgreek{e}}_{1})\text{\textgreek{w}}^{2}]$
supported in $[x_{0}-\text{\textgreek{d}}_{0},x_{0}+\text{\textgreek{d}}_{0}+\frac{2\text{\textgreek{p}}}{\text{\textgreek{w}}_{R}}+\frac{\text{\textgreek{p}}C_{0}^{2}}{|L_{0}|}]$
and a smooth solution $u$ to equation 
\begin{equation}
\frac{d^{2}u}{dx^{2}}+\big(\text{\textgreek{w}}^{2}+Z(\text{\textgreek{w}},\text{\textgreek{w}}_{I},V(x),x)-V(x)\big)u=0\label{eq:DistortedBasicODE}
\end{equation}
 such that $\big(u(a),\frac{du}{dx}(a)\big)=(u_{a}^{(0)},u_{a}^{(1)})$
and $\big(u(b),\frac{du}{dx}(b)\big)=(\text{\textgreek{l}}u_{b}^{(0)},\text{\textgreek{l}}u_{b}^{(1)})$
for some $\text{\textgreek{l}}\in\mathbb{C}\backslash\{0\}$.
\end{lem}
\emph{Remark on the proof of Lemma \ref{lem:PerturbationOfTheSimplerODE}.}
Setting $Z=0$, $(u_{a}^{(0)},u_{a}^{(1)})=(u_{1}(a),\frac{du_{1}}{dx}(a))$,
$(u_{b}^{(0)},u_{b}^{(1)})=(u_{1}(b),\frac{du_{1}}{dx}(b))$ and $\text{\textgreek{e}}_{1}=1$
in the statement of Lemma \ref{lem:PerturbationOfTheSimplerODE} does
not obscure any of the main difficulties associated to the proof of
the lemma, but the notation is substantially simplified. Thus, it
might be advisable for the reader to adopt this simplifying assumption
on $Z,(u_{a}^{(0)},u_{a}^{(1)}),(u_{b}^{(0)},u_{b}^{(1)})$ and $\text{\textgreek{e}}_{1}$
at first reading.

Under this simplification, the proof of Lemma \ref{lem:PerturbationOfTheSimplerODE}
proceeds by showing that, after suitably perturbing and then smoothing
out the potential $\tilde{V}:\mathbb{R}\rightarrow[0,\text{\textgreek{w}}^{2}]$
of Lemma \ref{lem:RoughPotential}, obtaining a new potential $V:\mathbb{R}\rightarrow[0,2\text{\textgreek{w}}^{2}]$,
the two solutions of equation (\ref{eq:DistortedBasicODE}) coinciding
with $\tilde{u}$ (of Lemma \ref{lem:RoughPotential}) around $x=a$
and $x=b$, respectively, have vanishing Wronskian on $[a,b]$ (and
hence differ only by a constant multiple). 

The aforementioned perturbation of $\tilde{V}$ is achieved through
small variations $\bar{x}_{1}$ of the point $x_{1}$ in the definition
of (\ref{eq:RoughPotential-1}), as well as an addition of a term
of the form $\text{\textgreek{w}}^{2}\text{\textgreek{q}}_{[\frac{x_{1}+x_{2}}{2},\frac{x_{1}+x_{2}}{2}+\text{\textgreek{h}}]}$
for some small $\text{\textgreek{h}}\in\mathbb{R}$ (where $\text{\textgreek{q}}_{[c,d]}$
equals the characteristic function of $[c,d]$ if $c\le d$, and minus
the characteristic function of $[d,c]$ if $c>d$). Denoting with
$\tilde{V}_{\bar{x}_{1}\text{\textgreek{h}}}$ the perturbed potential,
and with $w(\bar{x}_{1},\text{\textgreek{h}})$ the Wronskian of the
two solutions of 
\begin{equation}
\frac{d^{2}u}{dx^{2}}+(\text{\textgreek{w}}^{2}-\tilde{V}_{\bar{x}_{1}\text{\textgreek{h}}})u=0\label{eq:RoughEquationRemark}
\end{equation}
coinciding with $\tilde{u}$ near $x=a$ and $x=b$, respectively,
it is shown that the image of $w$ (as a function of the parameters
$\bar{x}_{1},\text{\textgreek{h}}$) contains an open neighborhood
of $0\in\mathbb{C}$.%
\footnote{That would not be always true had we chosen to perturb $\tilde{V}$
only by varying the two points $x_{1},x_{2}$ in (\ref{eq:RoughPotential-1}).%
} This fact is then shown to imply that $0$ belongs to the image of
the associated Wronskian for equation (\ref{eq:RoughEquationRemark})
after a suitable mollification of the rough potential $\tilde{V}_{\bar{x}_{1}\text{\textgreek{h}}}$,
yielding the proof of Lemma \ref{lem:PerturbationOfTheSimplerODE}
after a suitable choice of the parameters $\bar{x}_{1},\text{\textgreek{h}}$.
The rough dependence of the mollified potential on the mollifying
parameter poses the main difficulty in the last step. 

\medskip{}

\begin{proof}
Let $\tilde{V}:\mathbb{R}\rightarrow[0,\text{\textgreek{w}}^{2}]$,
$\text{\textgreek{j}}_{2}$ and $\tilde{u}:[a,b]\rightarrow\mathbb{C}$
be as in the statement of Lemma \ref{lem:RoughPotential}. Let $\text{\textgreek{h}}_{0}>0$
be sufficiently small in terms of $\text{\textgreek{w}},x_{1},a,b,C_{0},L_{0}$
and $\mathcal{W}[u_{1},u_{2}]$, and let us use the notation 
\begin{equation}
\mathcal{D}_{\text{\textgreek{h}}_{0}}\doteq(x_{1}-\text{\textgreek{h}}_{0},x_{1}+\text{\textgreek{h}}_{0})\times(-\text{\textgreek{h}}_{0},\text{\textgreek{h}}_{0}).
\end{equation}
For any $(\bar{x}_{1},\text{\textgreek{h}})\in\mathcal{D}_{\text{\textgreek{h}}_{0}}$,
we define the function $\tilde{V}_{\bar{x}_{1}\text{\textgreek{h}}}:\mathbb{R}\rightarrow[0,2\text{\textgreek{w}}^{2}]$
as follows:
\begin{equation}
\tilde{V}_{\bar{x}_{1}\text{\textgreek{h}}}(x)=\begin{cases}
0, & x\in(-\infty,\bar{x}_{1})\cup(x_{2},+\infty),\\
\text{\textgreek{w}}^{2}, & x\in[\bar{x}_{1},\frac{x_{1}+x_{2}}{2}+\text{\textgreek{h}}]\cup[\frac{x_{1}+x_{2}}{2},\bar{x}_{2}],\\
(1-\text{\textgreek{e}}_{1})\text{\textgreek{w}}^{2}, & x\in(\frac{x_{1}+x_{2}}{2}+\text{\textgreek{h}},\frac{x_{1}+x_{2}}{2}),
\end{cases}\label{eq:PerturbedFunction}
\end{equation}
when $\text{\textgreek{h}}\ge0$, and 
\begin{equation}
\tilde{V}_{\bar{x}_{1}\text{\textgreek{h}}}(x)=\begin{cases}
0, & x\in(-\infty,\bar{x}_{1})\cup(x_{2},+\infty),\\
\text{\textgreek{w}}^{2}, & x\in[\bar{x}_{1},\frac{x_{1}+x_{2}}{2}+\text{\textgreek{h}}]\cup[\frac{x_{1}+x_{2}}{2},\bar{x}_{2}],\\
(1+\text{\textgreek{e}}_{1})\text{\textgreek{w}}^{2}, & x\in(\frac{x_{1}+x_{2}}{2}+\text{\textgreek{h}},\frac{x_{1}+x_{2}}{2}),
\end{cases}\label{eq:PerturbedFunctionNegativeCase}
\end{equation}
when $\text{\textgreek{h}}\le0$. Notice that $\tilde{V}_{x_{1}0}\equiv\tilde{V}$. 

We will also define the functions $\tilde{u}_{\bar{x}_{1}\text{\textgreek{h}}}^{(a)},\tilde{u}_{\bar{x}_{1}\text{\textgreek{h}}}^{(b)}:[a,b]\rightarrow\mathbb{C}$
as solutions to the following initial value problems: 
\begin{equation}
\begin{cases}
\frac{d^{2}\tilde{u}_{\bar{x}_{1}\text{\textgreek{h}}}^{(a)}}{dx^{2}}+\big(\text{\textgreek{w}}^{2}-\tilde{V}_{\bar{x}_{1}\text{\textgreek{h}}}\big)\tilde{u}_{\bar{x}_{1}\text{\textgreek{h}}}^{(a)}=0\\
\tilde{u}_{\bar{x}_{1}\text{\textgreek{h}}}^{(a)}(a)=\tilde{u}(a)\\
\frac{d\tilde{u}_{\bar{x}_{1}\text{\textgreek{h}}}^{(a)}}{dx}(a)=\frac{d\tilde{u}}{dx}(a)
\end{cases}\label{eq:RoughProblemA}
\end{equation}
and 
\begin{equation}
\begin{cases}
\frac{d^{2}\tilde{u}_{\bar{x}_{1}\text{\textgreek{h}}}^{(b)}}{dx^{2}}+\big(\text{\textgreek{w}}^{2}-\tilde{V}_{\bar{x}_{1}\text{\textgreek{h}}}\big)\tilde{u}_{\bar{x}_{1}\text{\textgreek{h}}}^{(b)}=0\\
\tilde{u}_{\bar{x}_{1}\text{\textgreek{h}}}^{(b)}(b)=\tilde{u}(b)\\
\frac{d\tilde{u}_{\bar{x}_{1}\text{\textgreek{h}}}^{(b)}}{dx}(b)=\frac{d\tilde{u}}{dx}(b).
\end{cases}\label{eq:RoughProblemB}
\end{equation}
Notice that 
\begin{equation}
\tilde{u}_{x_{1}0}^{(a)}=\tilde{u}_{x_{1}0}^{(b)}=\tilde{u}.\label{eq:EqualityBaseCase}
\end{equation}
Furthermore, the map $(\bar{x}_{1},\text{\textgreek{h}})\rightarrow\big(\tilde{u}_{\bar{x}_{1}\text{\textgreek{h}}}^{(a)},\tilde{u}_{\bar{x}_{1}\text{\textgreek{h}}}^{(b)}\big)$
on $\mathcal{D}_{\text{\textgreek{h}}_{0}}$ is $C^{0}$ in the $\big(C^{1}([a,b])\cap C^{2}([a,b]\backslash I_{\text{\textgreek{h}}_{0}})\big)^{2}$
topology and $C^{1}$ in the $\big(C^{1}([a,b]\backslash I_{\text{\textgreek{h}}_{0}})\big)^{2}$
topology, where 
\begin{equation}
I_{\text{\textgreek{h}}_{0}}\doteq(x_{1}-\text{\textgreek{h}}_{0},x_{1}+\text{\textgreek{h}}_{0})\cup(\frac{x_{1}+x_{2}}{2}-\text{\textgreek{h}}_{0},\frac{x_{1}+x_{2}}{2}+\text{\textgreek{h}}_{0})\cup(x_{2}-\text{\textgreek{h}}_{0},x_{2}+\text{\textgreek{h}}_{0})\subset[a,b].
\end{equation}

Let us define the function $w:\mathcal{D}_{\text{\textgreek{h}}_{0}}\rightarrow\mathbb{C}$
as the Wronskian of the pair $\tilde{u}_{\bar{x}_{1}\text{\textgreek{h}}}^{(a)},\tilde{u}_{\bar{x}_{1}\text{\textgreek{h}}}^{(b)}$
for any $x_{3}\in[a,b]$: 
\begin{equation}
w(\bar{x}_{1},\text{\textgreek{h}})=\Bigg(\frac{d\tilde{u}_{\bar{x}_{1}\text{\textgreek{h}}}^{(a)}}{dx}\tilde{u}_{\bar{x}_{1}\text{\textgreek{h}}}^{(b)}-\tilde{u}_{\bar{x}_{1}\text{\textgreek{h}}}^{(a)}\frac{d\tilde{u}_{\bar{x}_{1}\text{\textgreek{h}}}^{(b)}}{dx}\Bigg)\Bigg|_{x=x_{3}}.\label{eq:Wronskian}
\end{equation}
Notice that the value of the right hand side of (\ref{eq:Wronskian})
is independent of the choice of $x_{3}\in[a,b]$. In view of the fact
that the map $(\bar{x}_{1},\text{\textgreek{h}})\rightarrow\big(\tilde{u}_{\bar{x}_{1}\text{\textgreek{h}}}^{(a)},\tilde{u}_{\bar{x}_{1}\text{\textgreek{h}}}^{(b)}\big)$
is $C^{1}$ in the $\big(C^{1}([a,b]\backslash I_{\text{\textgreek{h}}_{0}})\big)^{2}$
topology, we deduce that $w\in C^{1}(\mathcal{D}_{\text{\textgreek{h}}_{0}})$.
Furthermore, 
\begin{equation}
w(x_{1},0)=0\label{eq:ZeroAtOrigin}
\end{equation}
 in view of (\ref{eq:EqualityBaseCase}). 

For any $(\bar{x}_{1},\text{\textgreek{h}})\in\mathcal{D}_{\text{\textgreek{h}}_{0}}$,
with $\text{\textgreek{h}}>0$, the functions $\tilde{u}_{\bar{x}_{1}\text{\textgreek{h}}}^{(a)},\tilde{u}_{\bar{x}_{1}\text{\textgreek{h}}}^{(b)}$
are of the form (with $\text{\textgreek{g}}\in\{a,b\}$):
\begin{equation}
\tilde{u}_{\bar{x}_{1}\text{\textgreek{h}}}^{(\text{\textgreek{g}})}(x)=\begin{cases}
A_{\bar{x}_{1}\text{\textgreek{h}}}^{(1,\text{\textgreek{g}})}e^{i\text{\textgreek{w}}x}+B_{\bar{x}_{1}\text{\textgreek{h}}}^{(1,\text{\textgreek{g}})}e^{-i\text{\textgreek{w}}x}, & x\in[a,\bar{x}_{1}]\\
A_{\bar{x}_{1}\text{\textgreek{h}}}^{(2,\text{\textgreek{g}})}x+B_{\bar{x}_{1}\text{\textgreek{h}}}^{(2,\text{\textgreek{g}})}, & x\in[\bar{x}_{1},\frac{x_{1}+x_{2}}{2}]\\
A_{\bar{x}_{1}\text{\textgreek{h}}}^{(3,\text{\textgreek{g}})}e^{i\text{\textgreek{e}}_{1}^{1/2}\text{\textgreek{w}}x}+B_{\bar{x}_{1}\text{\textgreek{h}}}^{(3,\text{\textgreek{g}})}e^{-i\text{\textgreek{e}}_{1}^{1/2}\text{\textgreek{w}}x}, & x\in[\frac{x_{1}+x_{2}}{2},\frac{x_{1}+x_{2}}{2}+\text{\textgreek{h}}]\\
A_{\bar{x}_{1}\text{\textgreek{h}}}^{(4,\text{\textgreek{g}})}x+B_{\bar{x}_{1}\text{\textgreek{h}}}^{(4,\text{\textgreek{g}})}, & x\in[\frac{x_{1}+x_{2}}{2}+\text{\textgreek{h}},x_{2}]\\
A_{\bar{x}_{1}\text{\textgreek{h}}}^{(5,\text{\textgreek{g}})}e^{i\text{\textgreek{w}}x}+B_{\bar{x}_{1}\text{\textgreek{h}}}^{(5,\text{\textgreek{g}})}e^{-i\text{\textgreek{w}}x}, & x\in[x_{2},b],
\end{cases}\label{eq:GeneralSolution}
\end{equation}
where the constants $A_{\bar{x}_{1}\text{\textgreek{h}}}^{(j,\text{\textgreek{g}})}\in\mathbb{C}$,
$j\in\{1,\ldots,5\}$, $\text{\textgreek{g}}\in\{a,b\}$ are uniquely
determined by the initial conditions of (\ref{eq:RoughProblemA}),
(\ref{eq:RoughProblemB}) and the requirement that the functions $\tilde{u}_{\bar{x}_{1}\text{\textgreek{h}}}^{(\text{\textgreek{g}})}$
are $C^{1}$. Thus, using the identity (\ref{eq:EqualityBaseCase}),
we can readily calculate that 
\begin{gather}
\partial_{\bar{x}_{1}}w|_{(\bar{x}_{1},\text{\textgreek{h}})=(x_{1},0)}=-\text{\textgreek{w}}^{2}\big(\tilde{u}(x_{1})\big)^{2},\label{eq:WronskianFirstPoint}\\
\partial_{\text{\textgreek{h}}}w|_{(\bar{x}_{1},\text{\textgreek{h}})=(x_{1},0)}=-\text{\textgreek{e}}_{1}\text{\textgreek{w}}^{2}\big(\tilde{u}(\frac{x_{1}+x_{2}}{2})\big)^{2}.\label{eq:WronskianMiddlePoint}
\end{gather}
Since $\tilde{u}(x_{1})=u_{1}(x_{1})$ and $\tilde{u}(x_{2})=e^{i\text{\textgreek{j}}_{2}}u_{2}(x_{2})$
differ by a polar angle of magnitude less than $\text{\textgreek{p}}$,
and $\tilde{u}(\frac{x_{1}+x_{2}}{2})=\frac{u_{1}(x_{1})+e^{i\text{\textgreek{j}}_{2}}u_{2}(x_{2})}{2}$
is the middle-point of the line segment defined by $\tilde{u}(x_{1})$
and $\tilde{u}(x_{2})$ (because $\tilde{u}$ is linear on $[x_{1},x_{2}]$),
we deduce that $\tilde{u}(x_{1})$ and $\tilde{u}(\frac{x_{1}+x_{2}}{2})$
differ by a polar angle smaller than $\frac{\text{\textgreek{p}}}{2}$.
Hence, the linear span of $\big(\tilde{u}(x_{1})\big)^{2}$ and $\big(\tilde{u}(\frac{x_{1}+x_{2}}{2})\big)^{2}$
(viewed as vectors in $\mathbb{C}\simeq\mathbb{R}^{2}$) is the whole
plane, and thus, in view of (\ref{eq:WronskianFirstPoint})--(\ref{eq:WronskianMiddlePoint}),
we deduce that the derivative map $Dw:T\mathcal{D}_{\text{\textgreek{h}}_{0}}\rightarrow T\mathbb{C}$
is invertible at $(\bar{x}_{1},\text{\textgreek{h}})=(x_{1},0)$,
satisfying in particular the lower bound: 
\begin{equation}
\big|Dw|_{(\bar{x}_{1},\text{\textgreek{h}})=(x_{1},0)}\big|\ge c\text{\textgreek{e}}_{1}\text{\textgreek{w}}^{2}\big|Im\big(u_{1}(x_{1})\cdot e^{-i\text{\textgreek{j}}_{2}}\bar{u}_{2}(x_{2})\big)\big|>0.\label{eq:LowerBoundDerivativeOrigin}
\end{equation}
Notice also that, in view of the fact that $u_{1},u_{2}:[a,b]\rightarrow\mathbb{C}$
are expressed as (\ref{eq:Ellipses}) satisfying (\ref{eq:EqualAngularMomentumForEllipseCoefficients}),
and $x_{1},x_{2}$ define a common tangent of $u_{1}$ and $e^{i\text{\textgreek{j}}_{2}}u_{2}$,
from (\ref{eq:LowerBoundDerivativeOrigin}) we can estimate (for an
absolute constant $c>0$):
\begin{equation}
\big|Dw|_{(\bar{x}_{1},\text{\textgreek{h}})=(x_{1},0)}\big|\ge c\text{\textgreek{e}}_{1}\text{\textgreek{w}}\mathcal{W}[u_{1},u_{2}].\label{eq:LowerBoundDerivativeOriginWronskian}
\end{equation}

Let $K:\mathbb{R}\rightarrow[0,+\infty)$ be a smooth function supported
in $[-1,1]$ such that $\int_{\mathbb{R}}K(x)\, dx=1$. For any $(\bar{x}_{1},\text{\textgreek{h}})\in\mathcal{D}_{\text{\textgreek{h}}_{0}}$
and $\text{\textgreek{d}}>0$, let us define the function $V_{\bar{x}_{1}\text{\textgreek{h}}}^{(\text{\textgreek{d}})}:\mathbb{R}\rightarrow[0,(1+\text{\textgreek{e}}_{1})\text{\textgreek{w}}^{2}]$
as the convolution: 
\begin{equation}
V_{\bar{x}_{1}\text{\textgreek{h}}}^{(\text{\textgreek{d}})}(x)=\int_{-\infty}^{+\infty}\tilde{V}_{\bar{x}_{1}\text{\textgreek{h}}}(x-y)\cdot\text{\textgreek{d}}^{-1}K(\text{\textgreek{d}}^{-1}y)\, dy.\label{eq:DefinitionSmoothingWithDelta}
\end{equation}
 Notice that for any $\text{\textgreek{d}}>0$, $V_{\bar{x}_{1}\text{\textgreek{h}}}^{(\text{\textgreek{d}})}$
is a smooth function supported in $[x_{1}-\text{\textgreek{d}},x_{2}+\text{\textgreek{d}}]$
and, as $\text{\textgreek{d}}\rightarrow0$, the function $V_{\bar{x}_{1}\text{\textgreek{h}}}^{(\text{\textgreek{d}})}$
converges to $\tilde{V}_{\bar{x}_{1}\text{\textgreek{h}}}$ pointwise
everywhere except at the points $\mathscr{A}=\big\{\bar{x}_{1},x_{2},\frac{x_{1}+x_{2}}{2},\frac{x_{1}+x_{2}}{2}+\text{\textgreek{h}}\big\}$.
We will set $V_{\bar{x}_{1}\text{\textgreek{h}}}^{(0)}\doteq\tilde{V}_{\bar{x}_{1}\text{\textgreek{h}}}$.
Notice that, since $\tilde{V}_{\bar{x}_{1}\text{\textgreek{h}}}$
is piecewise constant, has bounded support and is uniformly bounded
by $2\text{\textgreek{w}}^{2}$, $|V_{\bar{x}_{1}\text{\textgreek{h}}}^{(\text{\textgreek{d}})}-V_{\bar{x}_{1}\text{\textgreek{h}}}^{(0)}|$
is bounded by $4\text{\textgreek{w}}^{2}$ and supported in the set
$\big\{ x:\, dist(x,\mathscr{A})\le\text{\textgreek{d}}\big\}$ for
any $0\le\text{\textgreek{d}}\le1$. In particular, for any $0\le\text{\textgreek{d}}_{2},\text{\textgreek{d}}_{3}\le1$,
we can bound (provided $\text{\textgreek{h}}_{0}$ is smaller than
some absolute constant $c_{0}$):
\begin{equation}
\sup_{(\bar{x}_{1},\text{\textgreek{h}})\in\mathcal{D}_{\text{\textgreek{h}}_{0}}}\int_{-\infty}^{\infty}\big|V_{\bar{x}_{1}\text{\textgreek{h}}}^{(\text{\textgreek{d}}_{2})}(x)-V_{\bar{x}_{1}\text{\textgreek{h}}}^{(\text{\textgreek{d}}_{3})}(x)\big|\, dx\le C\text{\textgreek{w}}^{2}|\text{\textgreek{d}}_{3}-\text{\textgreek{d}}_{2}|.\label{eq:L1DerivativeOfPotential}
\end{equation}

Let $\text{\textgreek{d}}_{0},\text{\textgreek{d}}_{1}>0$ be small
constants (their magnitude will be specified in more detail later).
Let us define the set $\mathcal{B}_{\text{\textgreek{d}}_{0}}^{(u_{1},u_{2})}\subset\mathbb{C}^{4}$
as the set of all $(u_{a}^{(0)},u_{a}^{(1)},u_{b}^{(0)},u_{b}^{(1)})\in\mathbb{C}^{4}$
satisfying (\ref{eq:ClosenessData}). For any $(\text{\textgreek{d}},\text{\textgreek{w}}_{I})\in[0,\text{\textgreek{d}}_{1})\times(-\text{\textgreek{d}}_{0},\text{\textgreek{d}}_{0})$,
$(u_{a}^{(0)},u_{a}^{(1)},u_{b}^{(0)},u_{b}^{(1)})\in\mathcal{B}_{\text{\textgreek{d}}_{0}}^{(u_{1},u_{2})}$
and $(\bar{x}_{1},\text{\textgreek{h}})\in\mathcal{D}_{\text{\textgreek{h}}_{0}}$,
we define the functions $u^{(a)},u^{(b)}:[a,b]\rightarrow\mathbb{C}$
as solutions to the following initial value problems: 
\begin{equation}
\begin{cases}
\frac{d^{2}u^{(a)}}{dx^{2}}+\big(\text{\textgreek{w}}^{2}+Z(\text{\textgreek{w}},\text{\textgreek{w}}_{I},V_{\bar{x}_{1}\text{\textgreek{h}}}^{(\text{\textgreek{d}})}(x),x)-V_{\bar{x}_{1}\text{\textgreek{h}}}^{(\text{\textgreek{d}})}(x)\big)u^{(a)}=0\\
u^{(a)}(a)=u_{a}^{(0)}\\
\frac{du^{(a)}}{dx}(a)=u_{a}^{(1)}
\end{cases}\label{eq:SmoothProblemA}
\end{equation}
and 
\begin{equation}
\begin{cases}
\frac{d^{2}u^{(b)}}{dx^{2}}+\big(\text{\textgreek{w}}^{2}+Z(\text{\textgreek{w}},\text{\textgreek{w}}_{I},V_{\bar{x}_{1}\text{\textgreek{h}}}^{(\text{\textgreek{d}})}(x),x)-V_{\bar{x}_{1}\text{\textgreek{h}}}^{(\text{\textgreek{d}})}(x)\big)u^{(b)}=0\\
u^{(b)}(b)=u_{b}^{(0)}\\
\frac{du^{(b)}}{dx}(b)=u_{b}^{(1)}
\end{cases}\label{eq:SmoothProblemB}
\end{equation}
From now on, we will use the shorthand notation $\mathscr{U}=(u_{a}^{(0)},u_{a}^{(1)},u_{b}^{(0)},u_{b}^{(1)})$
and $\mathscr{U}_{0}=(u_{1}(a),\frac{du_{1}}{dx}(a),u_{2}(b),\frac{du_{2}}{dx}(b))$.
We will also set 
\begin{equation}
\mathcal{U}(\text{\textgreek{d}},\text{\textgreek{w}}_{I},\mathscr{U},\bar{x}_{1},\text{\textgreek{h}})=(u^{(a)},u^{(b)}).\label{eq:PairMap}
\end{equation}

Notice that, in view of the fact that $Z(\text{\textgreek{w}},0,\cdot)=0$,
we have $\mathcal{U}(0,0,\mathscr{U}_{0},\bar{x}_{1},\text{\textgreek{h}})=(\tilde{u}_{\bar{x}_{1}\text{\textgreek{h}}}^{(a)},\tilde{u}_{\bar{x}_{1}\text{\textgreek{h}}}^{(b)})$.
Furthermore, by differentiating (\ref{eq:SmoothProblemA}) and (\ref{eq:SmoothProblemB})
with respect to $\text{\textgreek{w}}_{I},\mathscr{U},\bar{x}_{1},\text{\textgreek{h}}$,
we readily infer that the map $\mathcal{U}:[0,\text{\textgreek{d}}_{1})\times(-\text{\textgreek{d}}_{0},\text{\textgreek{d}}_{0})\times\mathcal{B}_{\text{\textgreek{d}}_{0}}^{(u_{1},u_{2})}\times\mathcal{D}_{\text{\textgreek{h}}_{0}}\rightarrow\big(C^{1}([a,b])\cap C^{2}([a,b]\backslash I_{\text{\textgreek{h}}_{0}})\big)^{2}$
(defined by (\ref{eq:PairMap})) has the following regularity properties:

\begin{enumerate}

\item $\mathcal{U}\in C^{0}\Big\{[0,\text{\textgreek{d}}_{1})\times(-\text{\textgreek{d}}_{0},\text{\textgreek{d}}_{0})\times\mathcal{B}_{\text{\textgreek{d}}_{0}}^{(u_{1},u_{2})}\times\mathcal{D}_{\text{\textgreek{h}}_{0}}\rightarrow\big(C^{1}([a,b])\cap C^{2}([a,b]\backslash I_{\text{\textgreek{h}}_{0}})\big)^{2}\Big\}.$

\item $\mathcal{U}\in C^{0}\Big\{[0,\text{\textgreek{d}}_{1})\rightarrow C^{2}\Big((-\text{\textgreek{d}}_{0},\text{\textgreek{d}}_{0})\times\mathcal{B}_{\text{\textgreek{d}}_{0}}^{(u_{1},u_{2})}\rightarrow C^{1}\big(\mathcal{D}_{\text{\textgreek{h}}_{0}}\rightarrow\big(C^{2}([a,b])\backslash I_{\text{\textgreek{h}}_{0}}\big)^{2}\big)\Big)\Big\}.$

\item \label{enu:Lipschitz} $\mathcal{U}\in C^{0,1}\Big\{[0,\text{\textgreek{d}}_{1})\rightarrow C^{2}\Big((-\text{\textgreek{d}}_{0},\text{\textgreek{d}}_{0})\times\mathcal{B}_{\text{\textgreek{d}}_{0}}^{(u_{1},u_{2})}\rightarrow C^{0}\big(\mathcal{D}_{\text{\textgreek{h}}_{0}}\rightarrow\big(C^{2}([a,b])\backslash I_{\text{\textgreek{h}}_{0}}\big)^{2}\big)\Big)\Big\}.$

\end{enumerate}

In the above, $C^{k}(\mathcal{A}\rightarrow\mathscr{V})$ (or $C^{k,a}(A\rightarrow\mathscr{V})$)
denotes the space of $C^{k}$ (or $C^{k,a}$, respectively) functions
defined on the manifold $A$ and taking values in the Banach space
$\mathscr{V}$. Note that \ref{enu:Lipschitz} is a consequence of
(\ref{eq:L1DerivativeOfPotential}).

Thus, extending (\ref{eq:Wronskian}) on the whole of $\mathscr{D}_{\text{\textgreek{d}}_{1},\text{\textgreek{d}}_{0},\text{\textgreek{h}}_{0}}=[0,\text{\textgreek{d}}_{1})\times(-\text{\textgreek{d}}_{0},\text{\textgreek{d}}_{0})\times\mathcal{B}_{\text{\textgreek{d}}_{0}}^{(u_{1},u_{2})}\times\mathcal{D}_{\text{\textgreek{h}}_{0}}$
as the Wronskian of the associated pair $u^{(a)},u^{(b)}$: 
\begin{equation}
w(\text{\textgreek{d}},\text{\textgreek{w}}_{I},\mathscr{U},\bar{x}_{1},\text{\textgreek{h}})=\Big(\frac{du^{(a)}}{dx}u^{(b)}-u^{(a)}\frac{du^{(b)}}{dx}\Big)\Big|_{x=x_{3}},\label{eq:ExtendedWronskian}
\end{equation}
we infer that 
\begin{align}
w\in C^{0}\Big\{[0,\text{\textgreek{d}}_{1})\rightarrow C^{2}\big( & (-\text{\textgreek{d}}_{0},\text{\textgreek{d}}_{0})\times\mathcal{B}_{\text{\textgreek{d}}_{0}}^{(u_{1},u_{2})}\rightarrow C^{1}(\mathcal{D}_{\text{\textgreek{h}}_{0}})\big)\Big\}\label{eq:RegularityW}\\
 & \bigcap C^{0,1}\Big\{[0,\text{\textgreek{d}}_{1})\rightarrow C^{2}\big((-\text{\textgreek{d}}_{0},\text{\textgreek{d}}_{0})\times\mathcal{B}_{\text{\textgreek{d}}_{0}}^{(u_{1},u_{2})}\rightarrow C^{0}(\mathcal{D}_{\text{\textgreek{h}}_{0}})\big)\Big\},\nonumber 
\end{align}
satisfying the following bounds in view of (\ref{eq:SmoothProblemA}),
(\ref{eq:SmoothProblemB}) and (\ref{eq:L1DerivativeOfPotential})
(provided $\text{\textgreek{d}}_{0},\text{\textgreek{d}}_{1},\text{\textgreek{h}}_{0}$
are smaller than some absolute constant $c_{0}>0$):
\begin{equation}
\sum_{j_{1}+j_{2}=0}^{2}\sum_{j_{3}+j_{4}=0}^{1}|\partial_{\text{\textgreek{w}}_{I}}^{j_{1}}D_{\mathscr{U}}^{j_{2}}\partial_{\bar{x}_{1}}^{j_{3}}\partial_{\text{\textgreek{h}}}^{j_{4}}w|\le C(\text{\textgreek{w}},\mathcal{Z},C_{0},L_{0})\label{eq:UpperBoundDerivatives}
\end{equation}
and 
\begin{equation}
\sup_{\text{\textgreek{d}}_{2},\text{\textgreek{d}}_{3}\in[0,\text{\textgreek{d}}_{1})}\sum_{j_{1}+j_{2}=0}^{2}\frac{|\partial_{\text{\textgreek{w}}_{I}}^{j_{1}}D_{\mathscr{U}}^{j_{2}}w(\text{\textgreek{d}}_{2},\cdot)-\partial_{\text{\textgreek{w}}_{I}}^{j_{1}}D_{\mathscr{U}}^{j_{2}}w(\text{\textgreek{d}}_{3},\cdot)|}{|\text{\textgreek{d}}_{2}-\text{\textgreek{d}}_{3}|}\le C(\text{\textgreek{w}},\mathcal{Z},C_{0},L_{0}),\label{eq:UpperBoundHolder}
\end{equation}
where $C(\text{\textgreek{w}},\mathcal{Z},C_{0},L_{0})>0$ depends
only on $\text{\textgreek{w}},\mathcal{Z},C_{0},L_{0}$. 

In view of (\ref{eq:LowerBoundDerivativeOrigin}), (\ref{eq:UpperBoundDerivatives})
and the fact that $w(0,0,\mathscr{U}_{0},x_{1},0)=0$, we infer that,
provided $\text{\textgreek{h}}_{0},\text{\textgreek{d}}_{0}$ are
sufficiently small in terms of $\text{\textgreek{w}},\mathcal{Z},C_{0},L_{0},\text{\textgreek{e}}_{1}$
and $\mathcal{W}[u_{1},u_{2}]$, then for any $(\text{\textgreek{w}}_{I},\mathscr{U})\in(-\text{\textgreek{d}}_{0},\text{\textgreek{d}}_{0})\times\mathcal{B}_{\text{\textgreek{d}}_{0}}^{(u_{1},u_{2})}$
the map 
\begin{equation}
w(0,\text{\textgreek{w}}_{I},\mathscr{U},\cdot):\mathcal{D}_{\text{\textgreek{h}}_{0}}\rightarrow\mathbb{C}\label{eq:InitialDiffeomorphism}
\end{equation}
is a diffeomorphism onto an open neighborhood of $0\in\mathbb{C}$,
with $w(\{0\}\times\{\text{\textgreek{w}}_{I}\}\times\{\mathscr{U}\}\times\mathcal{D}_{\text{\textgreek{h}}_{0}/2})$
containing a disk of radius at least $c=c(\text{\textgreek{w}},\mathcal{Z},C_{0},L_{0},\mathcal{W}[u_{1},u_{2}])$.
Therefore, in view of (\ref{eq:UpperBoundHolder}), provided $\text{\textgreek{d}}_{1}$
is sufficiently small in terms of $\text{\textgreek{w}},\mathcal{Z},C_{0},L_{0},\mathcal{W}[u_{1},u_{2}],\text{\textgreek{e}}_{1},\text{\textgreek{h}}_{0}$
and $\text{\textgreek{d}}_{0}$, we have for all $(\text{\textgreek{d}},\text{\textgreek{w}}_{I},\mathscr{U})\in[0,\text{\textgreek{d}}_{1})\times(-\text{\textgreek{d}}_{0},\text{\textgreek{d}}_{0})\times\mathcal{B}_{\text{\textgreek{d}}_{0}}^{(u_{1},u_{2})}$:
\begin{equation}
0\notin w(\{\text{\textgreek{d}}\}\times\{\text{\textgreek{w}}_{I}\}\times\{\mathscr{U}\}\times\mathcal{D}_{\text{\textgreek{h}}_{0}}\backslash\mathcal{D}_{\text{\textgreek{h}}_{0}/2}).\label{eq:NotInAnnulus}
\end{equation}
 Thus, (\ref{eq:UpperBoundHolder}), (\ref{eq:NotInAnnulus}) and
Lemma \ref{lem:TopologicalLemma} of the Appendix imply that:
\begin{equation}
0\in w(\{\text{\textgreek{d}}\}\times\{\text{\textgreek{w}}_{I}\}\times\{\mathscr{U}\}\times\mathcal{D}_{\text{\textgreek{h}}_{0}})\mbox{ for all }(\text{\textgreek{d}},\text{\textgreek{w}}_{I},\mathscr{U})\in[0,\text{\textgreek{d}}_{1})\times(-\text{\textgreek{d}}_{0},\text{\textgreek{d}}_{0})\times\mathcal{B}_{\text{\textgreek{d}}_{0}}^{(u_{1},u_{2})}.\label{eq:OpenImageInPerturbation}
\end{equation}

The relation (\ref{eq:OpenImageInPerturbation}) implies (in view
of the definition (\ref{eq:ExtendedWronskian}) of the Wronskian)
that for any $(\text{\textgreek{d}},\text{\textgreek{w}}_{I},\mathscr{U})\in[0,\text{\textgreek{d}}_{1})\times(-\text{\textgreek{d}}_{0},\text{\textgreek{d}}_{0})\times\mathcal{B}_{\text{\textgreek{d}}_{0}}^{(u_{1},u_{2})}$,
there exist some $(\bar{x}_{1},\text{\textgreek{h}})\in\mathcal{D}_{\text{\textgreek{h}}_{0}}$
and a $\text{\textgreek{l}}\in\mathbb{C}\backslash\{0\}$ so that
the pair $(u^{(a)},u^{(b)})$ associated to $(\text{\textgreek{d}},\text{\textgreek{w}}_{I},\mathscr{U},\bar{x}_{1},\text{\textgreek{h}})$
satisfies on $[a,b]$: 
\begin{equation}
u^{(a)}\equiv\text{\textgreek{l}}u^{(b)}.\label{eq:RescaledSolutions}
\end{equation}

We finally construct the required function $u$ solving equation (\ref{eq:DistortedBasicODE})
as follows: Fix a $0<\text{\textgreek{d}}<\text{\textgreek{d}}_{1}$,
and define $(\bar{x}_{1},\text{\textgreek{h}})\in\mathcal{D}_{\text{\textgreek{h}}_{0}}$
and $\text{\textgreek{l}}\in\mathbb{C}\backslash\{0\}$ in terms of
$(\text{\textgreek{d}},\text{\textgreek{w}}_{I},\mathscr{U})$ as
above, so that (\ref{eq:RescaledSolutions}) holds. Then, setting
$V=V_{\bar{x}_{1}\text{\textgreek{h}}}^{(\text{\textgreek{d}})}$
and $u=u^{(a)}$, we deduce that $u$ satisfies equation (\ref{eq:DistortedBasicODE}),
and moreover $\big(u(a),\frac{du}{dx}(a)\big)=(u_{a}^{(0)},u_{a}^{(1)})$
and $\big(u(b),\frac{du}{dx}(b)\big)=(\text{\textgreek{l}}u_{b}^{(0)},\text{\textgreek{l}}u_{b}^{(1)})$.
Thus, the proof of the lemma is complete. 
\end{proof}
In Section \ref{sec:DeformedSpacetimes}, we will also need the following
refinement of Lemma \ref{lem:PerturbationOfTheSimplerODE}, providing
an estimate on the change of the potential $V$ of Lemma \ref{lem:PerturbationOfTheSimplerODE}
under smooth variations of equation (\ref{eq:DistortedBasicODE})
and the associated initial data: 
\begin{lem}
\label{lem:DifferenceV} Let $\text{\textgreek{w}},L_{0}\in\mathbb{R}\backslash\{0\}$,
$C_{0}>0$, $a<b$, $x_{0}$, and $u_{1},u_{2}:[a,b]\rightarrow\mathbb{C}$
be as in the statement of Lemma \ref{lem:PerturbationOfTheSimplerODE}.
Let $Z^{(s)}:\mathbb{R}^{3}\times[a,b]\rightarrow\mathbb{C}$ be a
family of functions depending smoothly on $s\in[0,1]$, satisfying
for all $s\in[0,1]$ $Z^{(s)}(\text{\textgreek{w}},0,\cdot)=0$. Let
also $\text{\textgreek{d}}_{0}>0$ be sufficiently small in terms
of $\text{\textgreek{w}}$, $C_{0}$, $L_{0}$, $\max_{s\in[0,1]}\mathcal{Z}^{(s)}$
and $\mathcal{\mathcal{W}}[u_{1},u_{2}]$, where 
\begin{equation}
\mathcal{Z}^{(s)}\doteq\sup_{|v|\le2\text{\textgreek{w}}^{2}}\sup_{|\text{\textgreek{w}}_{I}|\le1}\sum_{j=0}^{2}\int_{a}^{b}\big|\partial_{\text{\textgreek{w}}_{I}}^{j}Z^{(s)}(\text{\textgreek{w}},\text{\textgreek{w}}_{I},v,x)\big|\, dx.\label{eq:BoundControlTermFamily}
\end{equation}
Then, for any $s\in[0,1]$, any $\text{\textgreek{w}}_{I}\in(-\text{\textgreek{d}}_{0},\text{\textgreek{d}}_{0})$
and any family of initial data sets $\mathscr{U}^{(s)}=(u_{a}^{(0,s)},u_{a}^{(1,s)},u_{b}^{(0,s)},u_{b}^{(1,s)})\in\mathbb{C}^{4}$
depending smoothly on $s\in[0,1]$ and satisfying 
\begin{equation}
\big|u_{a}^{(0,s)}-u_{1}(a)\big|+\big|u_{a}^{(1,s)}-\frac{du_{1}}{dx}(a)\big|+\big|u_{b}^{(0,s)}-u_{2}(b)\big|+\big|u_{b}^{(1,s)}-\frac{du_{2}}{dx}(b)\big|<\text{\textgreek{d}}_{0},\label{eq:ClosenessDataFamily}
\end{equation}
there exists a family of smooth functions $V^{(s)}:\mathbb{R}\rightarrow[0,2\text{\textgreek{w}}^{2}]$,
$s\in[0,1]$, supported in $[x_{0}-\text{\textgreek{d}}_{0},x_{0}+\text{\textgreek{d}}_{0}+\frac{2\text{\textgreek{p}}}{\text{\textgreek{w}}_{R}}+\frac{\text{\textgreek{p}}C_{0}^{2}}{|L_{0}|}]$,
satisfying for all $0\le s_{1}\le s_{2}\le1$ 
\begin{equation}
\int_{a}^{b}\big|V^{(s_{1})}(x)-V^{(s_{2})}(x)\big|\, dx\le C(\text{\textgreek{w}},C_{0},L_{0},\max_{s\in[0,1]}\mathcal{Z}^{(s)},\mathcal{\mathcal{W}}[u_{1},u_{2}])\cdot\int_{s_{1}}^{s_{2}}\Big(|\frac{d}{ds}\mathscr{U}^{(s)}|+\sup_{|\text{\textgreek{w}}_{I}|\le1}\sum_{j=0}^{2}\int_{a}^{b}\big|\partial_{s}\partial_{\text{\textgreek{w}}_{I}}^{j}Z^{(s)}(\text{\textgreek{w}},\text{\textgreek{w}}_{I},x)\big|\, dx\Big)\, ds,\label{eq:L1EstimatePotentials}
\end{equation}
and a smooth family of solutions $u^{(s)}$ to equation 
\begin{equation}
\frac{d^{2}u^{(s)}}{dx^{2}}+\big(\text{\textgreek{w}}^{2}+Z^{(s)}(\text{\textgreek{w}},\text{\textgreek{w}}_{I},V^{(s)}(x),x)-V^{(s)}(x)\big)u^{(s)}=0\label{eq:DistortedBasicODE-1}
\end{equation}
 such that $\big(u^{(s)}(a),\frac{du^{(s)}}{dx}(a)\big)=(u_{a}^{(0,s)},u_{a}^{(1,s)})$
and $\big(u^{(s)}(b),\frac{du^{(s)}}{dx}(b)\big)=(\text{\textgreek{l}}^{(s)}u_{b}^{(0,s)},\text{\textgreek{l}}^{(s)}u_{b}^{(1,s)})$
for some $\text{\textgreek{l}}^{(s)}\in\mathbb{C}\backslash\{0\}$.\end{lem}
\begin{proof}
Let $\text{\textgreek{d}}_{1},\text{\textgreek{h}}_{0}$ be sufficiently
small in terms of $\text{\textgreek{w}}$, $C_{0}$, $L_{0}$, $\max_{s\in[0,1]}\mathcal{Z}^{(s)}$
and $\mathcal{\mathcal{W}}[u_{1},u_{2}]$. We will use the same notations
and conventions as in the proof of Lemma \ref{lem:PerturbationOfTheSimplerODE}.

Let $V^{(0)}:\mathbb{R}\rightarrow[0,2\text{\textgreek{w}}^{2}]$
be the potential function associated to the pair of functions $(u_{1},u_{2})$
and the initial data tetrad $(u_{a}^{(0,0)},u_{a}^{(1,0)},u_{b}^{(0,0)},u_{b}^{(1,0)})$
as in Lemma \ref{lem:PerturbationOfTheSimplerODE}, yielding a smooth
solution $u^{(0)}$ to (\ref{eq:DistortedBasicODE-1}) for $s=0$.
In particular, following the proof of Lemma \ref{lem:PerturbationOfTheSimplerODE},
$V^{(0)}$ is of the form $V_{\bar{x}_{1}(0)\text{\textgreek{h}}(0)}^{(\text{\textgreek{d}})}$
(defined according to (\ref{eq:DefinitionSmoothingWithDelta})) for
some $\text{\textgreek{d}}\in(0,\text{\textgreek{d}}_{1})$ and some
$(\bar{x}_{1}(0),\text{\textgreek{h}}(0))\in\mathcal{D}_{\text{\textgreek{h}}_{0}}$
such that 
\begin{equation}
w(\text{\textgreek{d}},\text{\textgreek{w}}_{I},\mathscr{U}^{(0)},\bar{x}_{1}(0),\text{\textgreek{h}}(0))=0,
\end{equation}
where $w$ is defined as (\ref{eq:ExtendedWronskian}), with $(u^{(a)},u^{(b)})$
solving (\ref{eq:SmoothProblemA}) and (\ref{eq:SmoothProblemB})
with $\mathscr{U}^{(0)}$ replacing $\mathscr{U}$. 

In view of the estimates (\ref{eq:LowerBoundDerivativeOrigin}), (\ref{eq:UpperBoundDerivatives})
and (\ref{eq:UpperBoundHolder}) (as well as (\ref{eq:ClosenessDataFamily})),
provided $\text{\textgreek{d}}_{1},\text{\textgreek{d}}_{0},\text{\textgreek{h}}_{0}$
are sufficiently small in terms of $\text{\textgreek{w}}$, $C_{0}$,
$L_{0}$, $\max_{s\in[0,1]}\mathcal{Z}^{(s)}$ and $\mathcal{\mathcal{W}}[u_{1},u_{2}]$,
there exists a pair $(\bar{x}_{1}(s),\text{\textgreek{h}}(s))\in\mathcal{D}_{\text{\textgreek{h}}_{0}}$
for any $s\in[0,1]$ such that 
\begin{equation}
w(\text{\textgreek{d}},\text{\textgreek{w}}_{I},\mathscr{U}^{(s)},\bar{x}_{1}(s),\text{\textgreek{h}}(s))=0.\label{eq:WronskianZeroOnTheFamily}
\end{equation}
In particular, an application of the implicit function theorem (in
view again of (\ref{eq:LowerBoundDerivativeOrigin}), (\ref{eq:UpperBoundDerivatives})
and (\ref{eq:UpperBoundHolder})) implies that 
\begin{equation}
\big|\frac{d}{ds}\bar{x}_{1}(s)\big|+\big|\frac{d}{ds}\text{\textgreek{h}}(s)\big|\le C(\text{\textgreek{w}},C_{0},L_{0},\max_{s\in[0,1]}\mathcal{Z}^{(s)},\mathcal{\mathcal{W}}[u_{1},u_{2}])\cdot\Big(|\frac{d}{ds}\mathscr{U}^{(s)}|+\sup_{|v|\le2\text{\textgreek{w}}^{2}}\sup_{|\text{\textgreek{w}}_{I}|\le1}\sum_{j=0}^{2}\int_{a}^{b}\big|\partial_{s}\partial_{\text{\textgreek{w}}_{I}}^{j}Z^{(s)}(\text{\textgreek{w}},\text{\textgreek{w}}_{I},v,x)\big|\, dx\Big).\label{eq:DifferenceEstimateParameters}
\end{equation}

Setting $V^{(s)}=V_{\bar{x}_{1}(s)\text{\textgreek{h}}(s)}^{(\text{\textgreek{d}})}$
for $s\in[0,1]$, the existence of a smooth solution $u^{(s)}$ to
(\ref{eq:DistortedBasicODE-1}) satisfying the assumptions of the
lemma follows, in view of (\ref{eq:WronskianZeroOnTheFamily}), as
in the end of the proof of Lemma \ref{lem:PerturbationOfTheSimplerODE}.
Furthermore, in view of (\ref{eq:PerturbedFunction}), (\ref{eq:PerturbedFunctionNegativeCase}),
(\ref{eq:DefinitionSmoothingWithDelta}) and (\ref{eq:DifferenceEstimateParameters}),
inequality (\ref{eq:L1EstimatePotentials}) follows readily.
\end{proof}

\subsection{\label{sub:Proof}Proof of Proposition \ref{prop:StrongerProposition}}

Let $u_{inf}$ be the unique (up to multiplication by a complex constant)
solution of the ordinary differential equation 
\begin{equation}
u^{\prime\prime}+\big(\text{\textgreek{w}}_{R}^{2}-V_{\text{\textgreek{w}}_{R}ml}\big)u=0,\label{eq:SeperatedODE-1}
\end{equation}
 satisfying the outgoing condition 
\begin{equation}
u_{inf}^{\prime}-i\text{\textgreek{w}}_{R}u_{inf}\rightarrow0\label{eq:RadiationConditionInfinity}
\end{equation}
 as $r_{*}\rightarrow+\infty$, and let $u_{hor}$ be the solution
of (\ref{eq:SeperatedODE-1}) satisfying the smoothness condition
on the horizon 
\begin{equation}
u_{hor}^{\prime}+i\big(\text{\textgreek{w}}_{R}-\frac{am}{2Mr_{+}}\big)u_{hor}\rightarrow0\label{eq:SmoothnessConditionHorizon}
\end{equation}
 as $r_{*}\rightarrow-\infty$. Notice that in view of the form of
equation (\ref{eq:SeperatedODE-1}) and the conditions (\ref{eq:RadiationConditionInfinity})
and (\ref{eq:SmoothnessConditionHorizon}), the following limits are
well defined in $\mathbb{C}$: 
\begin{gather}
\lim_{r_{*}\rightarrow+\infty}\big(e^{-i\text{\textgreek{w}}_{R}r_{*}}u_{inf}(r_{*})\big)\doteq u_{inf}(+\infty)\label{eq:InfinityLimit}\\
\lim_{r_{*}\rightarrow-\infty}\big(e^{i(\text{\textgreek{w}}_{R}-\frac{am}{2Mr_{+}})r_{*}}u_{hor}(r_{*})\big)\doteq u_{hor}(-\infty)\label{eq:HorizonLimit}
\end{gather}
for some $u_{inf}(+\infty),u_{hor}(-\infty)\in\mathbb{C}\backslash\{0\}$.

The quantity $Im\big(u^{\prime}\cdot\bar{u}\big)$ is constant as
a function of $r_{*}$ for both $u_{inf}$ and $u_{hor}$, since they
both satisfy (\ref{eq:SeperatedODE-1}) and $\text{\textgreek{w}}_{R},V_{\text{\textgreek{w}}_{R}ml}\in\mathbb{R}$
(see the remark below Proposition \ref{prop:StrongerProposition}).
Thus, from (\ref{eq:RadiationConditionInfinity}), (\ref{eq:SmoothnessConditionHorizon}),
(\ref{eq:InfinityLimit}) and (\ref{eq:HorizonLimit}) we deduce that
\begin{equation}
Im\big(u_{inf}^{\prime}\cdot\bar{u}_{inf}\big)=|u_{inf}(+\infty)|^{2}\text{\textgreek{w}}_{R}\label{eq:AngularMomentumInfinity}
\end{equation}
and 
\begin{equation}
Im\big(u_{hor}^{\prime}\cdot\bar{u}_{hor}\big)=|u_{hor}(-\infty)|^{2}\big(\frac{am}{2Mr_{+}}-\text{\textgreek{w}}_{R}\big).\label{eq:AngularMomentumHorizon}
\end{equation}
Since $(\text{\textgreek{w}}_{R},m,l)$ lies in the superradiant regime
(\ref{eq:SuperradiantRegimeKerr}), the quantities (\ref{eq:AngularMomentumInfinity})
and (\ref{eq:AngularMomentumHorizon}) are of the same sign. Thus,
there exists an $a_{3}\in\mathbb{R}\backslash\{0\}$, such that by
rescaling $u_{inf}\rightarrow a_{3}u_{inf}$,%
\footnote{which is allowed since $u_{inf}$ was only defined up to multiplication
by a non zero complex constant%
} we have 
\begin{equation}
Im\big(u_{inf}^{\prime}\cdot\bar{u}_{inf}\big)=Im\big(u_{hor}^{\prime}\cdot\bar{u}_{hor}\big).\label{eq:EqualityAngularMomenta}
\end{equation}
From now on, we will assume that $u_{inf}$ has been rescaled like
this. 

For any $\text{\textgreek{w}}_{I}\ge0$, we will also define the functions
$u_{inf}^{(\text{\textgreek{w}}_{I})},u_{hor}^{(\text{\textgreek{w}}_{I})}:\mathbb{R}\rightarrow\mathbb{C}$
as the unique solutions of equation 
\begin{equation}
u^{\prime\prime}+\big((\text{\textgreek{w}}_{R}+i\text{\textgreek{w}}_{I})^{2}-V_{(\text{\textgreek{w}}_{R}+i\text{\textgreek{w}}_{I})ml}\big)u=0\label{eq:SeperatedODE-1-1}
\end{equation}
satisfying 
\begin{equation}
\lim_{r_{*}\rightarrow+\infty}\big(e^{-i(\text{\textgreek{w}}_{R}+i\text{\textgreek{w}}_{I})r_{*}}u_{inf}^{(\text{\textgreek{w}}_{I})}(r_{*})\big)=u_{inf}(+\infty)
\end{equation}
and 
\[
\lim_{r_{*}\rightarrow-\infty}\big(e^{i(\text{\textgreek{w}}_{R}+i\text{\textgreek{w}}_{I}-\frac{am}{2Mr_{+}})r_{*}}u_{hor}^{(\text{\textgreek{w}}_{I})}(r_{*})\big)=u_{hor}(-\infty)
\]
respectively.

We will assume that $C_{\text{\textgreek{w}}_{R}ml}$ has been chosen
large in terms of $\text{\textgreek{w}}_{R},m,l,M,a$, so that in
the region $r\ge C_{\text{\textgreek{w}}_{R}ml}$ we can bound 
\begin{equation}
0\le V_{\text{\textgreek{w}}_{R}ml}\le\frac{\text{\textgreek{w}}_{R}^{2}}{2}.\label{eq:SmallnessOfThePotential}
\end{equation}
This is possible since $V_{\text{\textgreek{w}}_{R}ml}\rightarrow0$
as $r_{*}\rightarrow+\infty$ and $V_{\text{\textgreek{w}}_{R}ml}\ge0$
for $r$ sufficiently large in terms of $\text{\textgreek{w}}_{R},m,l,M,a$.
For a $C_{\text{\textgreek{w}}_{R}ml}^{(0)}>0$ sufficiently large
in terms of $\text{\textgreek{w}}_{R},m,l,M,a$, we will fix $\text{\textgreek{q}}:\mathbb{R}\rightarrow[0,1]$
to be a smooth cut-off function such that $\text{\textgreek{q}}\equiv0$
on $(-\infty,r_{*}(r_{0})]\cup[r_{*}(r_{0})+C_{\text{\textgreek{w}}_{R}ml}^{(0)},+\infty)$
and $\text{\textgreek{q}}\equiv1$ on $[r_{*}(r_{0})+1,r_{*}(r_{0})+C_{\text{\textgreek{w}}_{R}ml}^{(0)}-1]$. 

The proof of Proposition \ref{prop:StrongerProposition} will follow
by constructing, for any $0\le\text{\textgreek{w}}_{I}\ll1$ (sufficiently
small in terms of $\text{\textgreek{w}}_{R},m,l,M,a,\text{\textgreek{e}}_{1}$
and $\mathcal{G}$), a smooth function $V_{f}:\mathbb{R}\rightarrow[0,(\frac{1}{4}+\text{\textgreek{e}}_{1})\text{\textgreek{w}}_{R}^{2}]$
supported on $[r_{*}(r_{0})+2,r_{*}(r_{0})+C_{\text{\textgreek{w}}_{R}ml}^{(0)}-2]$
such that, using the ansatz 
\begin{equation}
V=-\text{\textgreek{q}}\big(V_{\text{\textgreek{w}}_{R}ml}(r_{*})-\frac{3\text{\textgreek{w}}_{R}^{2}}{4}\big)+V_{f}(r_{*}),
\end{equation}
 the equation 
\begin{equation}
u^{\prime\prime}+\big((\text{\textgreek{w}}_{R}+i\text{\textgreek{w}}_{I})^{2}-V_{(\text{\textgreek{w}}_{R}+i\text{\textgreek{w}}_{I})ml}(r_{*})+\mathcal{G}(\text{\textgreek{w}}_{R},\text{\textgreek{w}}_{I},V(r_{*}),r_{*})-V(r_{*})\big)u=0\label{eq:FirtsStepPerturbedPotential}
\end{equation}
 admits a solution $u$ which satisfies $u\equiv u_{hor}^{(\text{\textgreek{w}}_{I})}$
for $r_{*}\le r_{*}(r_{0})$ and $u\equiv\text{\textgreek{l}}u_{inf}^{(\text{\textgreek{w}}_{I})}$
for $r_{*}\ge r_{*}(r_{0})+C_{\text{\textgreek{w}}ml}^{(0)}$ and
some $\text{\textgreek{l}}\in\mathbb{C}\backslash\{0\}$. To this
end, we will make use of Lemma \ref{lem:PerturbationOfTheSimplerODE}.

The ordinary differential equation 
\begin{equation}
u^{\prime\prime}+\big(\text{\textgreek{w}}_{R}^{2}-V_{\text{\textgreek{w}}_{R}ml}+\text{\textgreek{q}}(V_{\text{\textgreek{w}}_{R}ml}-\frac{3\text{\textgreek{w}}_{R}^{2}}{4})\big)u=0\label{eq:FirtsStepPerturbedPotential-1}
\end{equation}
 admits two unique solutions $u_{1},u_{2}:\mathbb{R}\rightarrow\mathbb{C}$
satisfying $u_{1}\equiv u_{hor}$ for $r_{*}\le r_{*}(r_{0})$ and
$u_{2}\equiv u_{inf}$ for $r_{*}\ge r_{*}(r_{0})+C_{\text{\textgreek{w}}_{R}ml}^{(0)}$.
For $r_{*}\in[r_{*}(r_{0})+1,r_{*}(r_{0})+C_{\text{\textgreek{w}}_{R}ml}^{(0)}-1]$,
$u_{1}$ and $u_{2}$ satisfy 
\begin{equation}
u^{\prime\prime}+\big(\frac{\text{\textgreek{w}}_{R}}{2}\big)^{2}u=0.\label{eq:FirtsStepPerturbedPotential-1-1}
\end{equation}
By perturbing $\text{\textgreek{q}}$ on the interval $\{r_{*}(r_{0})+C_{\text{\textgreek{w}}_{R}ml}^{(0)}-1\le r_{*}\le r_{*}(r_{0})+C_{\text{\textgreek{w}}_{R}ml}^{(0)}\}$,
if necessary, we will assume without loss of generality that $\frac{u_{1}}{u_{2}}$
is not constant on $\mathbb{R}$ (and thus also on any open interval
of $\mathbb{R}$), so that $\mathcal{W}[u_{1},u_{2}]$ satisfies the
lower bound:
\begin{equation}
\mathcal{W}[u_{1},u_{2}]=\big|u_{1}^{\prime}u_{2}-u_{1}u_{2}^{\prime}\big|\ge c_{\text{\textgreek{w}}_{R}mlr_{0}}>0.\label{eq:NonZeroWronskian}
\end{equation}

Notice that, by comparing (\ref{eq:FirtsStepPerturbedPotential-1})
to equation 
\begin{equation}
u^{\prime\prime}+\big(\text{\textgreek{w}}_{R}^{2}-\text{\textgreek{q}}_{R_{1}}(r_{*})V_{\text{\textgreek{w}}_{R}ml}(r_{+})\big)u=0,
\end{equation}
for some fixed $R_{*}\in\mathbb{R}$, where $\text{\textgreek{q}}_{R_{*}}:\mathbb{R}\rightarrow+\infty$
is a step function satisfying $\text{\textgreek{q}}_{R_{*}}\equiv0$
for $r_{*}\le R_{*}$ and $\text{\textgreek{q}}_{R_{*}}\equiv1$ for
$r_{*}>R_{*}$, we can bound, in view of (\ref{eq:FirtsStepPerturbedPotential})
and the boundary conditions (\ref{eq:InfinityLimit}) and (\ref{eq:HorizonLimit})
(see also \cite{Coppel1965}):
\begin{align}
\sup_{r_{*}\in\mathbb{R}}\big|u_{1}(r_{*})\big|\le C & (\text{\textgreek{w}}_{R},m)|u_{hor}(-\infty)|+C\int_{R_{*}}^{\infty}\big|V_{\text{\textgreek{w}}_{R}ml}+\text{\textgreek{q}}(\frac{\text{\textgreek{w}}_{R}^{2}}{4}-V_{\text{\textgreek{w}}_{R}ml})\big|\, dr_{*}+\label{eq:ExplicitUniformBoundHorizon}\\
 & +C\int_{-\infty}^{R_{*}}\big|V_{\text{\textgreek{w}}_{R}ml}-V_{\text{\textgreek{w}}_{R}ml}(r_{+})+\text{\textgreek{q}}(\frac{\text{\textgreek{w}}_{R}^{2}}{4}-V_{\text{\textgreek{w}}_{R}ml})\big|\, dr_{*}\nonumber 
\end{align}
 and 
\begin{align}
\sup_{r_{*}\in\mathbb{R}}\big|u_{2}(r_{*})\big|\le C & (\text{\textgreek{w}}_{R},m)|u_{inf}(+\infty)|+C\int_{R_{*}}^{\infty}\big|V_{\text{\textgreek{w}}_{R}ml}+\text{\textgreek{q}}(\frac{\text{\textgreek{w}}_{R}^{2}}{4}-V_{\text{\textgreek{w}}ml})\big|\, dr_{*}+\label{eq:ExplicitUniformBoundInfinity}\\
 & +C\int_{-\infty}^{R_{*}}\big|V_{\text{\textgreek{w}}_{R}ml}-V_{\text{\textgreek{w}}_{R}ml}(r_{+})+\text{\textgreek{q}}(\frac{\text{\textgreek{w}}_{R}^{2}}{4}-V_{\text{\textgreek{w}}_{R}ml})\big|\, dr_{*}.\nonumber 
\end{align}
Thus, choosing the constant $C_{\text{\textgreek{w}}_{R}ml}^{(0)}$
sufficiently large in terms of $\text{\textgreek{w}}_{R},m,l,M,a$
we can bound (recall (\ref{eq:EqualityAngularMomenta})):
\begin{equation}
\frac{\sup_{\mathbb{R}}\big(|u_{1}|+|u_{2}|\big)^{2}}{Im\big(u_{1}^{\prime}\cdot\bar{u}_{1}\big)}\ll C_{\text{\textgreek{w}}ml}^{(0)}\label{eq:LargeInterval}
\end{equation}
(notice that the left hand side of (\ref{eq:LargeInterval}) is invariant
under multiplication of $u_{1},u_{2}$ with the same non-zero constant). 

We will also define the functions $u_{1}^{(\text{\textgreek{w}}_{I})},u_{2}^{(\text{\textgreek{w}}_{I})}:\mathbb{R}\rightarrow\mathbb{C}$
as the unique solutions of equation 
\begin{equation}
u^{\prime\prime}+\big((\text{\textgreek{w}}_{R}+i\text{\textgreek{w}}_{I})^{2}+\mathcal{G}(\text{\textgreek{w}}_{R},\text{\textgreek{w}}_{I},-\text{\textgreek{q}}(V_{\text{\textgreek{w}}_{R}ml}(r_{*})-\frac{3\text{\textgreek{w}}_{R}^{2}}{4}),r_{*})-V_{(\text{\textgreek{w}}_{R}+i\text{\textgreek{w}}_{I})ml}(r_{*})+\text{\textgreek{q}}(V_{\text{\textgreek{w}}_{R}ml}(r_{*})-\frac{3\text{\textgreek{w}}_{R}^{2}}{4})\big)u=0\label{eq:FirtsStepPerturbedPotential-1-2}
\end{equation}
 satisfying $u_{1}^{(\text{\textgreek{w}}_{I})}\equiv u_{hor}^{(\text{\textgreek{w}}_{I})}$
for $r_{*}\le r_{*}(r_{0})$ and $u_{2}^{(\text{\textgreek{w}}_{I})}\equiv u_{inf}^{(\text{\textgreek{w}}_{I})}$
for $r_{*}\ge r_{*}(r_{0})+C_{\text{\textgreek{w}}_{R}ml}^{(0)}$.
Notice that $u_{1}^{(0)}=u_{1}$ and $u_{2}^{(0)}=u_{2}$.

Setting 
\begin{equation}
\mathcal{Z}(\text{\textgreek{w}}_{R},\text{\textgreek{w}}_{I},v,r_{*})\doteq2i\text{\textgreek{w}}_{R}\text{\textgreek{w}}_{I}-\text{\textgreek{w}}_{I}^{2}-V_{(\text{\textgreek{w}}_{R}+i\text{\textgreek{w}}_{I})ml}(r_{*})+V_{\text{\textgreek{w}}_{R}ml}(r_{*})+\mathcal{G}(\text{\textgreek{w}}_{R},\text{\textgreek{w}}_{I},v-\text{\textgreek{q}}(V_{\text{\textgreek{w}}_{R}ml}(r_{*})-\frac{3\text{\textgreek{w}}_{R}^{2}}{4}),r_{*}),
\end{equation}
equation (\ref{eq:FirtsStepPerturbedPotential}) restricted on $[r_{*}(r_{0})+1,r_{*}(r_{0})+C_{\text{\textgreek{w}}_{R}ml}^{(0)}-1]$
becomes: 
\begin{equation}
u^{\prime\prime}+\big(\big(\frac{\text{\textgreek{w}}_{R}}{2}\big)^{2}+\mathcal{Z}(\text{\textgreek{w}}_{R},\text{\textgreek{w}}_{I},V_{f}(r_{*}),r_{*})-V_{f}(r_{*})\big)u=0.\label{eq:ModelOfPerturbation}
\end{equation}
Thus, the existence of a function $V_{f}:\mathbb{R}\rightarrow[0,(\frac{1}{4}+\text{\textgreek{e}}_{1})\text{\textgreek{w}}_{R}^{2}\text{\textgreek{w}}_{R}^{2}]$
supported on $[r_{*}(r_{0})+2,r_{*}(r_{0})+C_{\text{\textgreek{w}}_{R}ml}^{(0)}-2]$
such that the equation (\ref{eq:ModelOfPerturbation}) admits a solution
$u$ coinciding with $u_{1}^{(\text{\textgreek{w}}_{I})}$ on $[r_{*}(r_{0})+1,r_{*}(r_{0})+2]$
and with $\text{\textgreek{l}}u_{2}^{(\text{\textgreek{w}}_{I})}$
on $[r_{*}(r_{0})+C_{\text{\textgreek{w}}_{R}ml}^{(0)}-2,r_{*}(r_{0})+C_{\text{\textgreek{w}}_{R}ml}^{(0)}-1]$
(for some $\text{\textgreek{l}}\in\mathbb{C}\backslash\{0\}$) is
guaranteed by Lemma \ref{lem:PerturbationOfTheSimplerODE} (setting
$\frac{\text{\textgreek{w}}_{R}}{2}$ in place of $\text{\textgreek{w}}$
there, as well as $\text{\textgreek{e}}_{1}=1$), in view of (\ref{eq:ExplicitUniformBoundHorizon}),
(\ref{eq:ExplicitUniformBoundInfinity}), (\ref{eq:LargeInterval})
and (\ref{eq:NonZeroWronskian}). Thus, the proof can be concluded
by extending $u$ on the whole of $\mathbb{R}$ under the requirment
that it coincides with $u_{1}^{(\text{\textgreek{w}}_{I})}$ for $r_{*}\le r_{*}(r_{0})+1$
and with $\text{\textgreek{l}}u_{2}^{(\text{\textgreek{w}}_{I})}$
for $r_{*}\ge r_{*}(r_{0})+C_{\text{\textgreek{w}}_{R}ml}^{(0)}-1$.
\qed

\section{\label{sec:ProofOfCorolary}Proof of Theorem \hyperref[Theorem 2]{2}}

In this section, we will provide a more detailed statement and the
proof of Theorem \hyperref[Theorem 2]{2}.

In particular, a more detailed statement of Theorem \hyperref[Theorem 2]{2}
is the following:

\begin{customTheorem}{2}[detailed version] \label{Theorem 2 detailed}
For any $0<|a|<M$ and any frequency triad $(\text{\textgreek{w}}_{0},m_{0},l_{0})\in(\mathbb{R}\backslash\{0\})\times\mathbb{Z}\times\mathbb{Z}_{\ge|m|}$
in the superradiant regime (\ref{eq:SuperradiantRegimeKerr}), there
exist constants $C_{\text{\textgreek{w}}_{0}m_{0}l_{0}}>r_{+}$ and
$C_{\text{\textgreek{w}}_{0}m_{0}l_{0}}^{(0)}>0$ depending only on
$\text{\textgreek{w}}_{0}m_{0}l_{0},a,M$, such that for any $r_{0}>C_{\text{\textgreek{w}}_{0}m_{0}l_{0}}$
and any $\text{\textgreek{w}}_{I}\ge0$ sufficiently small in terms
of $\text{\textgreek{w}}_{0}m_{0}l_{0},r_{0}$, there exists a Lorentzian
metric $g_{M,a}^{(def)}$ on $\mathcal{M}_{M,a}$ such that:

\begin{enumerate}

\item The vector fields $T,\text{\textgreek{F}}$ are Killing vector
fields for $g_{M,a}^{(def)}$. 

\item $g_{M,a}^{(def)}\equiv g_{M,a}$ for $\{r\le r_{0}\}$ and
$\{r\ge r_{0}+C_{\text{\textgreek{w}}_{R}ml}^{(0)}\}$.

\item $g_{M,a}^{(def)}(T,T)<0$ for $r\ge r_{0}$.

\item The wave equation 
\begin{equation}
\square_{g_{M,a}^{(def)}}\text{\textgreek{y}}=0\label{eq:WaveEquationOnRoughPerturbation}
\end{equation}
completely separates in the Boyer-Lindquist coordinate chart.

\item Equation (\ref{eq:WaveEquationOnRoughPerturbation}) admits
an outgoing mode solution with parameters $(\text{\textgreek{w}}_{0}+i\text{\textgreek{w}}_{I},m_{0},l_{0})$.

\end{enumerate}

\end{customTheorem} 

\medskip{}

\noindent \emph{Proof.} Let $C_{\text{\textgreek{w}}_{0}m_{0}l_{0}}$
and $C_{\text{\textgreek{w}}_{0}m_{0}l_{0}}^{(0)}$ be as in the proof
of Theorem \hyperref[Theorem 1 detailed]{1}. 

For a function $h:(2M,+\infty)\rightarrow(0,+\infty)$ satisfying
$h=1$ for $r\notin[r_{0},r_{0}+C_{\text{\textgreek{w}}_{0}m_{0}l_{0}}^{(0)}]$
(to be defined later), let us introduce the following metric on $\mathcal{M}_{M,a}$,
expressed in the Boyer--Lindquist coordinate chart: 
\begin{align}
g_{M,a}^{(def)}=-h(r)\big( & 1-\frac{2Mr}{\text{\textgreek{r}}^{2}}\big)dt^{2}-h(r)\frac{4Mar\sin^{2}\text{\textgreek{j}}}{\text{\textgreek{r}}^{2}}dtd\text{\textgreek{f}}+h^{3}(r)\frac{\text{\textgreek{r}}^{2}}{\text{\textgreek{D}}}dr^{2}+\label{eq:RoughMetric}\\
 & +h(r)\text{\textgreek{r}}^{2}d\text{\textgreek{j}}^{2}+h(r)\sin^{2}\text{\textgreek{j}}\frac{\text{\textgreek{P}}}{\text{\textgreek{r}}^{2}}d\text{\textgreek{f}}^{2}.\nonumber 
\end{align}
Notice that (\ref{eq:RoughMetric}) satisfies the following properties:
\begin{enumerate}
\item The vector fields $\partial_{t},\partial_{\text{\textgreek{f}}}$
are Killing vector fields for $g_{M,a}^{(def)}$.
\item We have $g_{M,a}^{(def)}=g_{M,a}$ for $r\notin[r_{0},r_{0}+C_{\text{\textgreek{w}}_{0}m_{0}l_{0}}^{(0)}]$,
since $h=1$ there.
\item We have $g_{M,a}^{(def)}(\partial_{t},\partial_{t})<0$ in the region
$\{r\ge r_{0}\}$ (provided $C_{\text{\textgreek{w}}_{0}m_{0}l_{0}}$
is sufficiently large). 
\item The wave equation 
\begin{equation}
\square_{g_{M,a}^{(def)}}\text{\textgreek{y}}=0\label{eq:WaveEquationRoughSpacetime}
\end{equation}
 on $(\mathcal{M}_{M,a},g_{M,a}^{(def)})$ completely separates in
the $(t,r,\text{\textgreek{j}},\text{\textgreek{f}})$ coordinate
chart.
\end{enumerate}
The latter property is deduced as follows: The wave operator on $(\mathcal{M}_{M,a},g_{M,a}^{(def)})$
has the following form: 
\begin{equation}
h(r)\text{\textgreek{r}}^{2}\square_{g_{M,a}^{(def)}}\text{\textgreek{y}}=h^{-2}(r)\cdot\partial_{r}\big(\text{\textgreek{D}}\partial_{r}\text{\textgreek{y}}\big)-\big(\frac{(r^{2}+a^{2})^{2}}{\text{\textgreek{D}}}+a^{2}\sin^{2}\text{\textgreek{j}}\big)\partial_{t}^{2}\text{\textgreek{y}}+\big(\frac{1}{\sin^{2}\text{\textgreek{j}}}-\frac{a^{2}}{\text{\textgreek{D}}}\big)\partial_{\text{\textgreek{f}}}^{2}\text{\textgreek{y}}-\frac{4Mar}{\text{\textgreek{D}}}\partial_{t}\partial_{\text{\textgreek{f}}}\text{\textgreek{y}}-\frac{1}{\sin\text{\textgreek{j}}}\partial_{\text{\textgreek{j}}}\big(\sin\text{\textgreek{j}}\partial_{\text{\textgreek{j}}}\text{\textgreek{y}}\big).\label{eq:RadialODE-2}
\end{equation}
Therefore, for any $(\text{\textgreek{w}},m)\in\mathbb{C}\times\mathbb{Z}$,
equation (\ref{eq:WaveEquationRoughSpacetime}) admits solutions of
the form (\ref{eq:ProductSolution}), with $S(\text{\textgreek{j}})$
satisfying (\ref{eq:FormalAngularOde}) and $R(r)$ satisfying 
\begin{equation}
h^{-2}(r)\text{\textgreek{D}}\frac{d}{dr}\big(\text{\textgreek{D}}\frac{dR}{dr}\big)+\big(a^{2}m^{2}-4Mar\text{\textgreek{w}}m+(r^{2}+a^{2})^{2}\text{\textgreek{w}}^{2}-\text{\textgreek{D}}(\text{\textgreek{l}}+a^{2}\text{\textgreek{w}}^{2})\big)R=0.
\end{equation}
 In particular, for any $(\text{\textgreek{w}},m,l)\in\mathbb{C}\times\mathbb{Z}\times\mathbb{Z}_{\ge|m|}$
with $|Im(\text{\textgreek{w}})|$ sufficiently small, equation (\ref{eq:WaveEquationRoughSpacetime})
admits solutions of the form 
\begin{equation}
\text{\textgreek{y}}(t,r,\text{\textgreek{j}},\text{\textgreek{f}})=e^{-i\text{\textgreek{w}}t}e^{im\text{\textgreek{f}}}S_{\text{\textgreek{w}}ml}(\text{\textgreek{j}})R_{\text{\textgreek{w}}ml}(r),
\end{equation}
 with $S_{\text{\textgreek{w}}ml}$ defined as in Section \ref{sub:Seperation}
and $u_{\text{\textgreek{w}}ml}$ (defined in terms of $R$ by (\ref{eq:RescaledRadialFunction}))
satisfying
\begin{equation}
u_{\text{\textgreek{w}}ml}^{\prime\prime}+\big(h^{2}\text{\textgreek{w}}^{2}-V_{\text{\textgreek{w}}ml;h}\big)u_{\text{\textgreek{w}}ml}=0,\label{eq:RoughSeperatedODE}
\end{equation}
where $^{\prime}$ denotes differentiation with respect to the variable
$r_{*}$ (defined by (\ref{eq:AuxiliaryR*})), and $V_{\text{\textgreek{w}}ml;h}$
is defined as: 
\begin{equation}
V_{\text{\textgreek{w}}ml;h}(r)=h^{2}(r)\frac{4Mram\text{\textgreek{w}}-a^{2}m^{2}+\text{\textgreek{D}}(\text{\textgreek{l}}_{\text{\textgreek{w}}ml}+a^{2}\text{\textgreek{w}}^{2})}{(r^{2}+a^{2})^{2}}+\frac{\text{\textgreek{D}}}{(r^{2}+a^{2})^{4}}\big(a^{2}\text{\textgreek{D}}+2Mr(r^{2}-a^{2})\big).\label{eq:RoughPotential}
\end{equation}

We will now show that (\ref{eq:WaveEquationRoughSpacetime}) admits
an outgoing mode solution with parameters $(\text{\textgreek{w}}_{0}+i\text{\textgreek{w}}_{I},m_{0},l_{0})$,
provided $0\le\text{\textgreek{w}}_{I}\ll1$ and the function $h$
is chosen appropriately. Let us set 
\begin{equation}
\text{\textgreek{e}}_{1}=\frac{1}{2}\inf_{r\in[r_{0},r_{0}+C_{\text{\textgreek{w}}_{R}ml}^{(0)}]}\Big(\frac{\text{\textgreek{D}}}{(r^{2}+a^{2})^{4}}\big(a^{2}\text{\textgreek{D}}+2Mr(r^{2}-a^{2})\Big)>0\label{eq:Epsilon1}
\end{equation}
and 
\begin{equation}
\mathcal{G}[\text{\textgreek{w}}_{0},\text{\textgreek{w}}_{I},v,r]\doteq iv\cdot\frac{Im\Big((\text{\textgreek{w}}_{0}+i\text{\textgreek{w}}_{I})^{2}-\frac{4Mram(\text{\textgreek{w}}_{0}+i\text{\textgreek{w}}_{I})-a^{2}m^{2}+\text{\textgreek{D}}(\text{\textgreek{l}}_{\text{\textgreek{w}}ml}+a^{2}(\text{\textgreek{w}}_{0}+i\text{\textgreek{w}}_{I})^{2})}{(r^{2}+a^{2})^{2}}\Big)}{Re\Big((\text{\textgreek{w}}_{0}+i\text{\textgreek{w}}_{I})^{2}-\frac{4Mram(\text{\textgreek{w}}_{0}+i\text{\textgreek{w}}_{I})-a^{2}m^{2}+\text{\textgreek{D}}(\text{\textgreek{l}}_{\text{\textgreek{w}}ml}+a^{2}(\text{\textgreek{w}}_{0}+i\text{\textgreek{w}}_{I})^{2})}{(r^{2}+a^{2})^{2}}\Big)}.\label{eq:PerturbingFunction}
\end{equation}
Let $V(r(\cdot)):\mathbb{R}\rightarrow[0,+\infty)$ be the potential
function provided by Proposition \ref{prop:StrongerProposition} for
the the function $\mathcal{G}$ above and the parameters $(\text{\textgreek{w}}_{0},m_{0},l_{0})$,
$r_{0}$, $\text{\textgreek{e}}_{1}$ and $\text{\textgreek{w}}_{I}$.
Finally, let us then define
\begin{equation}
h^{2}(r)=1-\Big(Re\big((\text{\textgreek{w}}_{0}+i\text{\textgreek{w}}_{I})^{2}-\frac{4Mram(\text{\textgreek{w}}_{0}+i\text{\textgreek{w}}_{I})-a^{2}m^{2}+\text{\textgreek{D}}(\text{\textgreek{l}}_{\text{\textgreek{w}}ml}+a^{2}(\text{\textgreek{w}}_{0}+i\text{\textgreek{w}}_{I})^{2})}{(r^{2}+a^{2})^{2}}\big)\Big)^{-1}V(r).\label{eq:DefinitionOfH}
\end{equation}

Notice that, in view of (\ref{eq:KerrPotential}), (\ref{eq:UpperBoundPrecise})
and \ref{eq:Epsilon1}, provided $\text{\textgreek{w}}_{I}$ is sufficiently
small  in terms of $\text{\textgreek{w}}_{0},m_{0},l_{0},M,a,r_{0},\text{\textgreek{e}}_{1}$,
the relation \ref{eq:DefinitionOfH} indeed guarantees that 
\begin{equation}
h^{2}>0,\label{eq:PositivityH}
\end{equation}
and hence $h$ is a smooth positive function. Since $V(r)$ is supported
in $[r_{0},r_{0}+C_{\text{\textgreek{w}}_{0}m_{0}l_{0}}^{(0)}]$,
from (\ref{eq:DefinitionOfH}) we have 
\begin{equation}
h\equiv1\mbox{ outside }[r_{0},r_{0}+C_{\text{\textgreek{w}}_{0}m_{0}l_{0}}^{(0)}].\label{eq:His1}
\end{equation}
Notice also that (\ref{eq:PerturbingFunction}) and (\ref{eq:DefinitionOfH})
yield: 
\begin{equation}
h^{2}(\text{\textgreek{w}}_{0}+i\text{\textgreek{w}}_{I})^{2}-V_{\text{\textgreek{w}}_{0}+i\text{\textgreek{w}}_{I},m_{0}l_{0};h}(r)=(\text{\textgreek{w}}_{0}+i\text{\textgreek{w}}_{I})^{2}-V_{\text{\textgreek{w}}_{0}+i\text{\textgreek{w}}_{I},m_{0}l_{0}}(r)+\mathcal{G}[\text{\textgreek{w}}_{0},\text{\textgreek{w}}_{I},V(r),r]-V(r)\label{eq:EqualityPotentials}
\end{equation}
(where $V_{\text{\textgreek{w}}ml}$ is defined by (\ref{eq:KerrPotential})). 

In view of (\ref{eq:EqualityPotentials}), equation (\ref{eq:RoughSeperatedODE})
for the frequency triad $(\text{\textgreek{w}}_{0}+i\text{\textgreek{w}}_{I},m_{0},l_{0})$
becomes 
\begin{equation}
u^{\prime\prime}+\big((\text{\textgreek{w}}_{0}+i\text{\textgreek{w}}_{I})^{2}-V_{(\text{\textgreek{w}}_{0}+i\text{\textgreek{w}}_{I})m_{0}l_{0}}(r_{*})+\mathcal{G}(\text{\textgreek{w}}_{0},\text{\textgreek{w}}_{I},V(r_{*}),r_{*}\big)u=0.\label{eq:TheOldODEPotential}
\end{equation}
In view of Proposition \ref{prop:StrongerProposition} (and our choice
of $V$), provided $\text{\textgreek{w}}_{I}$ is sufficiently small
 in terms of $\text{\textgreek{w}}_{0},m_{0},l_{0},M,a,r_{0},\text{\textgreek{e}}_{1}$,
equation (\ref{eq:TheOldODEPotential}) admits a solution $u$ satisfying
at $r_{*}=\pm\infty$ the boundary conditions (\ref{eq:ODEForMode}).
Thus, we finally deduce that:
\begin{enumerate}[resume]
\item With $h$ defined as above, equation (\ref{eq:WaveEquationRoughSpacetime})
admits an outgoing real mode solution with parameters $(\text{\textgreek{w}}_{0},m_{0},l_{0})$.
\qed
\end{enumerate}

\section{\label{sec:DeformedSpacetimes}Superradiant instabilities for the
free wave equation on spacetimes with normally hyperbolic trapping}

According to our discussion in Section \ref{sub:IntroAsymptoticallyConic},
in order to better clarify the absence of any connection between superradiance
related mode instabilities and the structure of trapping outside the
ergoregion, it would be preferable to replace the spacetimes $(\mathcal{M}_{M,a},g_{M,a}^{(def)})$
of Theorem \hyperref[Theorem 2 detailed]{2} with a family of spacetimes
admitting real or exponentially growing mode solutions to (\ref{eq:WaveEquationIntro})
and at the same time having a ``nice'' trapped set. In this section,
we will construct such a family of spacetimes $(\mathcal{M}_{0},g_{M,a}^{(h)})$.
We note already that the nature of our construction forces the spacetimes
$(\mathcal{M}_{0},g_{M,a}^{(h)})$ to be asymptotically conic instead
of asymptotically flat (see the remarks below Theorem \hyperref[Theorem 3]{3}).

\subsection{\label{sub:MainResultAsymptoticallyConic}Theorem 3: Statement and
remarks on the proof}

The main result of this section will be the following:

\begin{customTheorem}{3} \label{Theorem 3} For any $\text{\textgreek{w}}_{R}\in\mathbb{R}\backslash\{0\}$
and $0\le\text{\textgreek{w}}_{I}\ll1$ (small in terms of $\text{\textgreek{w}}_{R}$),
there exists a smooth family of globally hyperbolic Lorentzian metrics
$g_{M,a}^{(h)}$, $M>0$, $a\ge0$, on the ambient ``Schwarzschild
exterior'' manifold $\mathcal{M}_{0}=\mathbb{R}\times(2M,+\infty)\times\mathbb{S}^{2}$,
such that for any $M>0$, $R_{b}>2M$, $k\in\mathbb{N}$ and $\text{\textgreek{e}}>0$,
there exists an $a_{0}\in(0,M)$ so that for $0<a\le a_{0}$ the following
properties are satisfied:

\begin{enumerate}

\item In the $(t,r,\text{\textgreek{j}},\text{\textgreek{f}})$ coordinate
chart on $\mathcal{M}_{0}$, the vector fields $T=\partial_{t}$ and
$\text{\textgreek{F}}=\partial_{\text{\textgreek{f}}}$ are Killing
vector fields for $g_{M,a}^{(h)}$.

\item The spacetime $(\mathcal{M}_{0},g_{M,a}^{(h)})$ is asymptotically
conic.

\item The metric $g_{M,a}^{(h)}$ can be smoothly extended up to
$\mathcal{H}=\{r=2M\}$ in the ``Schwarzschild star'' coordinate
chart on $\mathcal{M}_{0}$. In this extension, $\mathcal{H}$ is
the event horizon associated to the asymptotically conic end of $(\mathcal{M}_{0},g_{M,a}^{(h)})$.
Furthermore, $\mathcal{H}$ is a Killing horizon with positive surface
gravity.

\item On the closed subset $\{2M\le r\le R_{b}\}$ of $\mathcal{M}_{0}\cup\mathcal{H}^{+}$,
the metrics $g_{M,a}^{(h)}$ and $g_{M,a}$ are $\text{\textgreek{e}}$-close
in the $C^{k}$ norm defined with respect to the ``Schwarzschild
star'' coordinate chart on $\{2M\le r\le R_{b}\}$. Furthermore,
 the ergoregion of $(\mathcal{M}_{0},g_{M,a}^{(h)})$ is non-empty,
contained in the set $\{r\le2M+\text{\textgreek{e}}\}$, and 
\begin{equation}
\sup_{\mathcal{M}_{M,0}}g_{M,a}^{(h)}(T,T)<\text{\textgreek{e}}.
\end{equation}

\item The wave equation 
\begin{equation}
\square_{g_{M,a}^{(h)}}\text{\textgreek{y}}=0\label{eq:WaveEquationAsymptoticallyConic}
\end{equation}
completely separates in the $(t,r,\text{\textgreek{j}},\text{\textgreek{f}})$
coordinate chart.

\item The wave equation (\ref{eq:WaveEquationAsymptoticallyConic})
admits an outgoing mode solution with frequency parameter $\text{\textgreek{w}}_{R}+i\text{\textgreek{w}}_{I}$.

\item The trapped set of $(\mathcal{M}_{0},g_{M,a}^{(h)})$ is normally
hyperbolic.

\end{enumerate}

\end{customTheorem}

The proof of Theorem \hyperref[Theorem 3]{3} will be presented in
Section  \ref{sub:FinalProofDeformed}.

The metric $g_{M,a}^{(h)}$ in Theorem \hyperref[Theorem 3]{3} is
constructed as an asymptotically conic perturbation of an asymptotically
flat metric $g_{M,a}^{(1)}$ on $\mathcal{M}_{0}$ with normally hyperbolic
trapped set - see Sections \ref{sub:AuxiliaryMetric}--\ref{sub:DeformedMetric}.
The metric $g_{M,a}^{(1)}$ bears many algebraic and geometric similarities
with the Kerr metric $g_{M,a}$, but it does not solve the vacuum
Einstein equations. 

The main difference between the spacetimes $(\mathcal{M}_{M,a},g_{M,a}^{(def)})$
of Theorem \hyperref[Theorem 2 detailed]{2} and $(\mathcal{M}_{0},g_{M,a}^{(h)})$
of Theorem \hyperref[Theorem 3]{3} lies exactly in the structure
of the trapped set: The spacetime $(\mathcal{M}_{M,a},g_{M,a}^{(def)})$
contains stable trapped null geodesics, while the trapped set of $(\mathcal{M}_{0},g_{M,a}^{(h)})$
is normally hyperbolic. The normal hyperbolicity of the trapped set
of $(\mathcal{M}_{0},g_{M,a}^{(h)})$ can actually be deduced from
the high frequency integrated local energy decay statement of Proposition
\ref{prop:HighFrequencyILED}.

At this point, we should remark the following:

\begin{enumerate}

\item The construction of the auxiliary metric $g_{M,a}^{(1)}$,
as an intermediate step for establishing Theorem \hyperref[Theorem 3]{3},
was motivated as follows: Attempting to modify the Kerr metric $g_{M,a}$
in the region $\{r\gg1\}$, so that the wave equation $\square_{\tilde{g}_{M,a}}\text{\textgreek{y}}=0$
for the new metric $\tilde{g}_{M,a}$ admits an outgoing mode solution
while at the same time remaining completely separable, one encounters
the following obstacle to controlling the structure of the trapped
set: The separability of $\square_{\tilde{g}_{M,a}}$, combined with
the requirement that $\tilde{g}_{M,a}=g_{M,a}$ in the region $\{r\lesssim1\}$
and the fact that the angular separation variable $\text{\textgreek{L}}$
for $\square_{g_{M,a}}$ depends on the time separation variable $\text{\textgreek{w}}$
(see Section \ref{sub:Seperation} for the separation of the wave
equation on Kerr spacetimes), imply a rigid relation between $\tilde{g}_{M,a}(\partial_{t},\partial_{t})$
and $\tilde{g}_{M,a}(\partial_{\text{\textgreek{f}}},\partial_{\text{\textgreek{f}}})$,
which leaves almost no room for deforming $g_{M,a}$ without introducing
stably trapped null geodesics. In view of these difficulties, we introduced
a novel metric $g_{M,a}^{(1)}$ on $\mathcal{M}_{0}$, which has many
algebraic and analytic similarities with the slowly rotating Kerr
metric (such as the separability of the wave operator and the polynomial
decay of solutions to the wave equation), but for which the angular
separation variable $\text{\textgreek{L}}$ does not depend on $\text{\textgreek{w}}$. 

\item The conic asymptotics of the metric $g_{M,a}^{(h)}$ in Theorem
\hyperref[Theorem 3]{3} are a technical necessity imposed by the
methods used in the proof, which, in view of the conditions imposed
on the structure of the trapped set, enforces a monotonicity condition
on the angular components of $g_{M,a}^{(h)}$.%
\footnote{This is simply the condition $h^{\prime}\ge0$ appearing in the proof
of Lemma \ref{lem:IledDeforemdMetric}.%
} It would be of particular interest to examine whether $g_{M,a}^{(h)}$
can be replaced by an asymptotically flat metric with similar properties.
However, since the $r^{p}$-weighted estimates of \cite{DafRod7,Moschidisc}
also apply in the asymptotically conic case, the conic asymptotics
of $g_{M,a}^{(h)}$ pose no additional difficulties (compared to the
asymptotically flat case) in the study of the decay properties of
solutions to (\ref{eq:WaveEquationAsymptoticallyConic}).

\end{enumerate}

As we remarked before, the normal hyperbolicity of the trapped set
of the spacetimes $(\mathcal{M}_{0},g_{M,a}^{(h)})$ of Theorem \hyperref[Theorem 3]{3}
can be viewed as a consequence of the statement that $(\mathcal{M}_{0},g_{M,a}^{(h)})$
satisfies a high frequency integrated local energy decay estimate:
\begin{prop}
\label{prop:HighFrequencyILED}Let $(\mathcal{M}_{0},g_{M,a}^{(h)})$
be the spacetimes of Theorem \hyperref[Theorem 3]{3}. Let $\bar{t}:\mathcal{M}_{0}\cup\mathcal{H}^{+}\rightarrow\mathbb{R}$
be a smooth function with spacelike level sets intersecting $\mathcal{H}^{+}$,
such that $T(\bar{t})=1$ and $\bar{t}\equiv t$ for $\{r\ge3M\}$.
Then, for any $\text{\textgreek{e}}>0$, there exists a (small) parameter
$a_{0}>0$ and (large) parameters $2M\ll R_{-}\ll R_{+}$, with $R_{-}$
independent of $\text{\textgreek{e}}$, so that the following statement
holds for any $0\le a\le a_{0}$: For any solution $\text{\textgreek{y}}$
to the inhomogeneous wave equation 
\begin{equation}
\square_{g_{M,a}^{(h)}}\text{\textgreek{y}}=F\label{eq:InhomogeneousWaveEquationAsymptoticallyConic}
\end{equation}
 on $(\mathcal{M}_{0},g_{M,a}^{(h)})$ which is smooth up to $\mathcal{H}^{+}$,
we can bound for any $\text{\textgreek{t}}_{1}\le\text{\textgreek{t}}_{2}$:
\begin{equation}
\begin{split}\int_{\{\text{\textgreek{t}}_{1}\le\bar{t}\le\text{\textgreek{t}}_{2}\}}\Big((1- & \frac{2M}{r})r^{-2}|\partial_{r}\text{\textgreek{y}}|^{2}+\text{\textgreek{q}}_{r\neq3M}(r)\cdot r^{-2}J_{\text{\textgreek{m}}}^{N}(\text{\textgreek{y}})N^{\text{\textgreek{m}}}+r^{-4}|\text{\textgreek{y}}|^{2}\Big)\, dg_{M,a}^{(h)}+\int_{\{\bar{t}=\text{\textgreek{t}}_{2}\}}J_{\text{\textgreek{m}}}^{N}(\text{\textgreek{y}})\bar{n}^{\text{\textgreek{m}}}\le\\
\le & C_{R_{-},R_{+}}\Big(\int_{\{\bar{t}=\text{\textgreek{t}}_{1}\}}J_{\text{\textgreek{m}}}^{N}(\text{\textgreek{y}})\bar{n}^{\text{\textgreek{m}}}+\mathcal{Z}_{\text{\textgreek{t}}_{1},\text{\textgreek{t}}_{2}}[F,\text{\textgreek{y}};R_{-}]+\int_{\{R_{-}\le r_{*}\le R_{+}\}\cap\{\text{\textgreek{t}}_{1}\le\bar{t}\le\text{\textgreek{t}}_{2}\}}\min\big\{|T\text{\textgreek{y}}|^{2},|\text{\textgreek{y}}|^{2}\big\}\, dg_{M,a}^{(h)}\Big),
\end{split}
\label{eq:DegenerateIled-1-2}
\end{equation}
\begin{equation}
\begin{split}\int_{\{\text{\textgreek{t}}_{1}\le\bar{t}\le\text{\textgreek{t}}_{2}\}}\Big(r^{-2} & J_{\text{\textgreek{m}}}^{N}(\text{\textgreek{y}})N^{\text{\textgreek{m}}}+r^{-4}|\text{\textgreek{y}}|^{2}\Big)\, dg_{M,a}^{(h)}+\sum_{j=0}^{1}\int_{\{\bar{t}=\text{\textgreek{t}}_{2}\}}J_{\text{\textgreek{m}}}^{N}(T^{j}\text{\textgreek{y}})\bar{n}^{\text{\textgreek{m}}}\le\\
\le & C_{R_{-},R_{+}}\Big(\sum_{j=0}^{1}\int_{\{\bar{t}=\text{\textgreek{t}}_{1}\}}J_{\text{\textgreek{m}}}^{N}(T^{j}\text{\textgreek{y}})\bar{n}^{\text{\textgreek{m}}}+\sum_{j=0}^{1}\mathcal{Z}_{\text{\textgreek{t}}_{1},\text{\textgreek{t}}_{2}}[T^{j}F,T^{j}\text{\textgreek{y}};R_{-}]+\\
 & \hphantom{C_{R_{-},R_{+}}\Big(}+\int_{\{R_{-}\le r_{*}\le R_{+}\}\cap\{\text{\textgreek{t}}_{1}\le\bar{t}\le\text{\textgreek{t}}_{2}\}}\min\big\{|T\text{\textgreek{y}}|^{2},|\text{\textgreek{y}}|^{2}\big\}\, dg_{M,a}^{(h)}\Big).
\end{split}
\label{eq:NonDegenerateIled-1-2}
\end{equation}
 In the above, the vector field $N$ is everywhere timelike on $\mathcal{M}_{0}\cup\mathcal{H}^{+}$
and satisfies $[T,N]=0$ and $N\equiv T$ for $r\ge R_{2}\gg1$, $\mathcal{Z}_{\text{\textgreek{t}}_{1},\text{\textgreek{t}}_{2}}[F,\text{\textgreek{y}};R_{-}]$
is defined as 
\[
\mathcal{Z}_{\text{\textgreek{t}}_{1},\text{\textgreek{t}}_{2}}[F,\text{\textgreek{y}};R_{-}]\doteq\int_{\{r\le R_{-}\}\cap\{\text{\textgreek{t}}_{1}\le\bar{t}\le\text{\textgreek{t}}_{2}\}}|F|^{2}+\int_{\{r\ge R_{-}\}\cap\{\text{\textgreek{t}}_{1}\le\bar{t}\le\text{\textgreek{t}}_{2}\}}|F|\cdot\big(|\partial_{r}\text{\textgreek{y}}|+r^{-1}|\text{\textgreek{y}}|\big)+\int_{\{\text{\textgreek{t}}_{1}\le\bar{t}\le\text{\textgreek{t}}_{2}\}}|F|\cdot|T\text{\textgreek{y}}|,
\]
$\bar{n}^{\text{\textgreek{m}}}$ is the future directed unit normal
to the hypersurfaces $\{\bar{t}=\text{\textgreek{t}}\}$, the cut-off
function $\text{\textgreek{q}}_{r\neq3M}:(2M,+\infty)\rightarrow[0,1]$
vanishes on $[3M-\text{\textgreek{e}},3M+\text{\textgreek{e}}]$ and
is identically $1$ outside $(3M-2\text{\textgreek{e}},3M+2\text{\textgreek{e}})$,
and the constant $C_{R_{-},R_{+}}$ depends only on the precise choice
of $R_{-}$ and $R_{+}$.
\end{prop}
The proof of Proposition \ref{prop:HighFrequencyILED} will be presented
in Section \ref{sub:FinalProofDeformed}, as a part of the proof of
Theorem \hyperref[Theorem 3]{3}.

The estimates (\ref{eq:DegenerateIled-1-2}) and (\ref{eq:NonDegenerateIled-1-2})
can be viewed as integrated local energy decay estimates (with loss
of derivatives) for solutions $\text{\textgreek{y}}$ to (\ref{eq:InhomogeneousWaveEquationAsymptoticallyConic})
in the case when the time frequency of $\text{\textgreek{y}}$ is
very high or very low (so that the last term in the right hand side
of (\ref{eq:DegenerateIled-1-2}) and (\ref{eq:NonDegenerateIled-1-2})
can be absorbed into the left hand side). These high frequency properties
of (\ref{eq:DegenerateIled-1-2}) and (\ref{eq:NonDegenerateIled-1-2})
are closely associated to the normal hyperbolicity of the trapped
set of $(\mathcal{M}_{0},g_{M,a}^{(h)})$ when $a\ll M$. Notice,
however, that the last term in the right hand side of (\ref{eq:DegenerateIled-1-2})
and (\ref{eq:NonDegenerateIled-1-2}) can not be dropped completely,
since this would imply that the wave equation (\ref{eq:WaveEquationAsymptoticallyConic})
does not admit an outgoing mode solution, contradicting the statement
of Theorem \hyperref[Theorem 3]{3}.

The rest of this section is organised as follows: In Sections \ref{sub:AuxiliaryMetric}--\ref{sub:ILEDAuxiliaryMetric}
we will construct an auxiliary family of asymptotically flat spacetimes
$(\mathcal{M}_{0},g_{M,a}^{(1)})$, and we will study its main properties.
These spacetimes will then be deformed into the asymptotically conic
spacetimes $(\mathcal{M}_{0},g_{M,a}^{(h)})$ in Section \ref{sub:DeformedMetric}.
In Section \ref{sub:GeneralisationMainTheorem}, we will extend Theorem
\hyperref[Theorem 1 detailed]{1} to the spacetimes $(\mathcal{M}_{0},g_{M,a}^{(h)})$.
Finally, in Section \ref{sub:FinalProofDeformed}, we will complete
the proof of Theorem \hyperref[Theorem 3]{3} and Proposition \ref{prop:HighFrequencyILED}.

\subsection{\label{sub:AuxiliaryMetric}Construction of the auxiliary spacetimes
$(\mathcal{M}_{0},g_{M,a}^{(1)})$}

For any $M>0$ and $a\ge0$, we define the following metric on $\mathcal{M}_{0}=\mathbb{R}\times(2M,+\infty)\times\mathbb{S}^{2}$
in the $(t,r,\text{\textgreek{j}},\text{\textgreek{f}})$ coordinate
system (where $t,r$ are the projections onto the first two factors
of $\mathcal{M}_{0}$ and $(\text{\textgreek{j}},\text{\textgreek{f}})$
are the usual polar coordinates on $\mathbb{S}^{2}$): 
\begin{align}
g_{M,a}^{(1)}=-\big( & 1-\frac{2M}{r}-\frac{a^{2}M^{2}\sin^{2}\text{\textgreek{j}}}{r^{4}}\big)dt^{2}-\frac{2Ma\sin^{2}\text{\textgreek{j}}}{r}dtd\text{\textgreek{f}}+\big(1-\frac{2M}{r}\big)^{-1}dr^{2}+\label{eq:NewMetricBlackHole}\\
 & +r^{2}\big(d\text{\textgreek{j}}^{2}+\sin^{2}\text{\textgreek{j}}d\text{\textgreek{f}}^{2}\big).\nonumber 
\end{align}
Notice that $(\mathcal{M}_{0},g_{M,a}^{(1)})$ is a smooth, globally
hyperbolic and asymptotically flat spacetime. Furthermore, $(\mathcal{M}_{0},g_{M,a}^{(1)})$
is stationary and axisymmetric, with stationary Killing field $T=\partial_{t}$
and axisymmetric Killing field $\text{\textgreek{F}}=\partial_{\text{\textgreek{f}}}$
(both coordinate vector fields defined in the fixed $(t,r,\text{\textgreek{j}},\text{\textgreek{f}})$
coordinate chart). When $a=0$, (\ref{eq:NewMetricBlackHole}) is
simply the Schwarzschild exterior metric.

When $a>0$, the Killing field $T$ becomes spacelike in the region
$\{r<r_{1}(\text{\textgreek{j}})\}$, where $r_{1}(\text{\textgreek{j}})$
is the positive solution of
\begin{equation}
1-\frac{2M}{r_{1}}-\frac{a^{2}M^{2}\sin^{2}\text{\textgreek{j}}}{r_{1}^{4}}=0
\end{equation}
(notice that $r_{1}(\text{\textgreek{j}})>2M$ for $a>0$ and $\text{\textgreek{j}}\neq0,\text{\textgreek{p}}$).
However, the span of the Killing fields $T,\text{\textgreek{F}}$
contains everywhere on $\mathcal{M}_{0}$ a timelike direction. This
follows from the fact that the determinant 
\begin{equation}
\mathfrak{D}=\det\left(\begin{array}{cc}
g_{M,a}^{(1)}(T,T) & g_{M,a}^{(1)}(T,\text{\textgreek{F}})\\
g_{M,a}^{(1)}(T,\text{\textgreek{F}}) & g_{M,a}^{(1)}(\text{\textgreek{F}},\text{\textgreek{F}})
\end{array}\right)=-\big(1-\frac{2M}{r}\big)r^{2}\sin^{2}\text{\textgreek{j}}
\end{equation}
is everywhere negative on $\mathcal{M}_{0}$, except on the axis $\text{\textgreek{j}}=0,\text{\textgreek{p}}$,
where $\text{\textgreek{F}}$ vanishes and $T$ is timelike.

The spacetime $(\mathcal{M}_{0},g_{M,a}^{(1)})$ does not contain
any black hole or white hole region, i.\,e.~the domain of outer
communications of the asymptotically flat region $\{r\gg1\}$ of $\mathcal{M}_{0}$
is the whole $\mathcal{M}_{0}$. This can be inferred as follows:
Let us fix a vector field $V$ in the span of $\{T,\text{\textgreek{F}}\}$,
such that $V$ is everywhere on $\mathcal{M}_{0}$ future pointing
and timelike, $V\equiv T$ in the asymptotic region $\{r\gg1\}$ and
$[T,V]=[\text{\textgreek{F}},V]=0$. Then for some $h:(2M,+\infty)\rightarrow(1,+\infty)$
(taking suitably large values), the vector fields 
\begin{equation}
V_{1}=\partial_{r}+h(r)V
\end{equation}
 and 
\begin{equation}
V_{2}=\partial_{r}-h(r)V
\end{equation}
are future directed and past directed timelike vector fields respectively.
Furthermore, $V_{1}r=V_{2}r=1$, and thus, starting from any point
$p\in\mathcal{M}_{0}$, the flow of $V_{1},V_{2}$ reaches the asymptotically
flat region $\{r\gg1\}$ in finite time. Thus, any point $p\in\mathcal{M}_{0}$
communicates with the asymptotically flat region through both future
and past directed causal curves.

In the limit $r\rightarrow2M$, the expression (\ref{eq:NewMetricBlackHole})
for the metric breaks down, however, the spacetime is future incomplete,
with incomplete null geodesics approaching $r=2M$. It turns out that
the spacetime can be smoothly extended to the future ``beyond''
$r=2M$, and this follows imediately after the following change of
coordinates: By introducing the new coordinate functions 
\begin{gather}
t^{*}\doteq t+\bar{t}(r)\\
\text{\textgreek{f}}^{*}\doteq\text{\textgreek{f}}+\bar{\text{\textgreek{f}}}(r),
\end{gather}
where 
\begin{gather}
\frac{d\bar{t}}{dr}=\big(1-\frac{2M}{r}\big)^{-1},\\
\frac{d\bar{\text{\textgreek{f}}}}{dr}=\big(1-\frac{2M}{r}\big)^{-1}aMr^{-3},
\end{gather}
the expression for $g_{M,a}^{(1)}$ in the $(t^{*},r,\text{\textgreek{j}},\text{\textgreek{f}}^{*})$
coordinate chart on $\mathcal{M}_{0}$ becomes: 
\begin{align}
g_{M,a}^{(1)}=-\big( & 1-\frac{2M}{r}-\frac{a^{2}M^{2}\sin^{2}\text{\textgreek{j}}}{r^{4}}\big)(dt^{*})^{2}-\frac{2Ma\sin^{2}\text{\textgreek{j}}}{r}dt^{*}d\text{\textgreek{f}}^{*}+2dt^{*}dr+\label{eq:NewMetricBlackHole-1}\\
 & +r^{2}\big(d\text{\textgreek{j}}^{2}+\sin^{2}\text{\textgreek{j}}(d\text{\textgreek{f}}^{*})^{2}\big).\nonumber 
\end{align}
Thus, $g_{M,a}^{(1)}$ can be smoothly extended beyond $r=2M$. In
such a smooth extension $(\widetilde{\mathcal{M}}_{0},\tilde{g}_{M,a}^{(1)})$
of $(\mathcal{M}_{0},g_{M,a}^{(1)})$, where $\widetilde{\mathcal{M}}_{0}=\mathbb{R}\times(2M-\text{\textgreek{d}},+\infty)\times\mathbb{S}^{2}$
for some $0<\text{\textgreek{d}}<2M$, $\mathcal{M}_{0}$ has a non-empty
boundary $\mathcal{H}^{+}$, on which $r$ extends continuously (with
$r|_{\mathcal{H}^{+}}\equiv2M$). It can be readily verified (in the
$(t^{*},r,\text{\textgreek{j}},\text{\textgreek{f}}^{*})$ coordinate
chart) that $\mathcal{H}^{+}$ is actually a smooth, null hypersurface,
and thus $g_{M,a}^{(1)}$ extends uniquely (independently of the particular
choice of the extension $(\widetilde{\mathcal{M}}_{0},\tilde{g}_{M,a}^{(1)})$)
as a smooth Lorentzian metric on the manifold with boundary 
\begin{equation}
\overline{\mathcal{M}}_{0}=\mathcal{M}_{0}\cup\mathcal{H}^{+}\label{eq:ExtendedManifold}
\end{equation}
(on which the $(t^{*},r,\text{\textgreek{j}},\text{\textgreek{f}}^{*})$
coordinate chart is smooth). Furthermore, the Killing fields $T,\text{\textgreek{F}}$
extend smoothly on $\mathcal{H}^{+}$. Notice that $T$ is spacelike
on $\mathcal{H}^{+}$ (except on the points $\text{\textgreek{j}}=0,\text{\textgreek{p}}$
of $\mathcal{H}^{+}$, where $T$ is null). 

It can be readily verified that in any such extension $(\widetilde{\mathcal{M}}_{0},\tilde{g}_{M,a}^{(1)})$
of $(\mathcal{M}_{0},g_{M,a}^{(1)})$, the domain of outer communications
of $(\widetilde{\mathcal{M}}_{0},\tilde{g}_{M,a}^{(1)})$ is precisely
$(\mathcal{M}_{0},g_{M,a}^{(1)})$, and $\mathcal{H}^{+}$ is the
future event horizon. Furthermore, $\mathcal{H}^{+}$ is also a Killing
horizon: Introducing the vector field 
\begin{equation}
K\doteq T+\frac{a}{8M^{2}}\text{\textgreek{F}},\label{eq:HawkingKillingFieldNew}
\end{equation}
we notice that $K$ is a Killing vector field for $g_{M,a}^{(1)}$
(as a linear combination of $T,\text{\textgreek{F}}$ with constant
coefficients), and furthermore $g_{M,a}^{(1)}(K,K)=0$ on $r=2M$. 
\begin{rem*}
Let us note that the existence of a Killing field parallel to the
null generator of $\mathcal{H}^{+}$ does not follow in this case
by Hawking's theorem (see \cite{Hawking1972}), since $(\mathcal{M}_{0},g_{M,a}^{(1)})$
does not satisfy the null energy condition (in general, one would
expect the null generator of $\mathcal{H}^{+}$ to be a $\text{\textgreek{j}}$-dependent
linear combination of $T,\text{\textgreek{F}}$).
\end{rem*}
Since the vector field $K$ satisfies 
\begin{equation}
g_{M,a}^{(1)}(K,K)=-\big(1-\frac{2M}{r}\big)\Big\{1-a^{2}\sin^{2}\text{\textgreek{j}}\frac{\big(1-\frac{2M}{r}\big)^{2}(r^{2}+2Mr+4M^{2})^{2}}{64M^{4}r^{2}}\Big\},
\end{equation}
and thus $\partial_{r}g_{M,a}^{(1)}(K,K)|_{\mathcal{H}^{+}}>0$ in
the $(t^{*},r,\text{\textgreek{j}},\text{\textgreek{f}}^{*})$ coordinate
chart, the Killing horizon $\mathcal{H}^{+}$ is a non-degenerate
horizon with positive surface gravity.

\subsection{Separation of the wave equation and frequency decomposition on $(\mathcal{M}_{0},g_{M,a}^{(1)})$}

The wave operator $\square_{g_{M,a}^{(1)}}$ on $(\mathcal{M}_{0},g_{M,a}^{(1)})$
takes the form: 
\begin{align}
\square_{g_{M,a}^{(1)}}\text{\textgreek{y}}= & r^{-2}\partial_{r}\big(r^{2}(1-\frac{2M}{r})\partial_{r}\text{\textgreek{y}}\big)+r^{-2}(\sin\text{\textgreek{j}})^{-1}\partial_{\text{\textgreek{j}}}\big(\sin\text{\textgreek{j}}\partial_{\text{\textgreek{j}}}\text{\textgreek{y}}\big)+\label{eq:WaveOperatorUnpurtubredBeforeSeperation}\\
 & +\frac{-\partial_{t}^{2}\text{\textgreek{y}}-2Mar^{-3}\partial_{t}\partial_{\text{\textgreek{f}}}\text{\textgreek{y}}+\big(\big(1-\frac{2M}{r}\big)(\sin\text{\textgreek{j}})^{-2}-a^{2}M^{2}r^{-4}\big)r^{-2}\partial_{\text{\textgreek{f}}}^{2}\text{\textgreek{y}}}{\big(1-\frac{2M}{r}\big)}.\nonumber 
\end{align}
Therefore, the wave equation $\square_{g_{M,a}^{(1)}}\text{\textgreek{y}}=0$
separates on $(\mathcal{M}_{0},g_{M,a}^{(1)})$, i.\,e.~it admits
solutions of the form 
\begin{equation}
\text{\textgreek{y}}(t,r,\text{\textgreek{j}},\text{\textgreek{f}})=e^{-i\text{\textgreek{w}}t}e^{im\text{\textgreek{f}}}S_{ml}(\text{\textgreek{j}})\cdot R_{\text{\textgreek{w}}ml}(r)\label{eq:ProductSolution-2}
\end{equation}
for $(\text{\textgreek{w}},m,l)\in\mathbb{C}\times\mathbb{Z}\times\mathbb{Z}_{\ge|m|}$,
where $e^{im\text{\textgreek{f}}}S_{ml}(\text{\textgreek{j}})$ are
the usual spherical harmonics on $\mathbb{S}^{2}$, with $\{S_{ml}\}_{(m,l)\in\mathbb{Z}\times\mathbb{Z}_{\ge|m|}}$
being the usual set of eigenfunctions of the Sturm--Liouville problem
\begin{equation}
\begin{cases}
-\frac{1}{\sin\text{\textgreek{j}}}\frac{d}{d\text{\textgreek{j}}}\big(\sin\text{\textgreek{j}}\frac{dS_{ml}}{d\text{\textgreek{j}}}\big)+\frac{m^{2}}{\sin^{2}\text{\textgreek{j}}}S_{ml}=l(l+1)S_{ml}\\
S_{ml}(\text{\textgreek{j}})\mbox{ is bounded at }\text{\textgreek{j}}=0,\text{\textgreek{p}},
\end{cases}\label{eq:AngularODE-1}
\end{equation}
and $R_{\text{\textgreek{w}}ml}(r)$ satisfies the ordinary differential
equation 
\begin{equation}
r^{-2}\big(1-\frac{2M}{r}\big)\frac{d}{dr}\Big(r^{2}\big(1-\frac{2M}{r}\big)\frac{dR_{\text{\textgreek{w}}ml}}{dr}\Big)+\Big\{\text{\textgreek{w}}^{2}-2Mar^{-3}\text{\textgreek{w}}m+a^{2}M^{2}r^{-6}-\big(1-\frac{2M}{r}\big)r^{-2}l(l+1)\Big\} R_{\text{\textgreek{w}}ml}=0.\label{eq:RadialODENewSpacetime}
\end{equation}

Using on $\mathcal{M}_{0}$ the auxiliary radial function $r_{*}:(2M,+\infty)\rightarrow(-\infty,+\infty)$
defined by 
\begin{equation}
\frac{dr_{*}}{dr}\doteq\big(1-\frac{2M}{r}\big)^{-1}\label{eq:NewAuxiliaryFunction}
\end{equation}
and setting 
\begin{equation}
u_{\text{\textgreek{w}}ml}\doteq r\cdot R_{\text{\textgreek{w}}ml},
\end{equation}
equation (\ref{eq:RadialODENewSpacetime}) becomes: 
\begin{equation}
u_{\text{\textgreek{w}}ml}^{\prime\prime}+\big(\text{\textgreek{w}}^{2}-V_{\text{\textgreek{w}}ml}\big)u_{\text{\textgreek{w}}ml}=0,\label{eq:SimplifiedRadialODE}
\end{equation}
where $^{\prime}$ denotes differentiation with respect to $r_{*}$,
and 
\begin{equation}
V_{\text{\textgreek{w}}ml}(r)=\frac{\big(1-\frac{2M}{r}\big)l(l+1)+2aMr^{-1}\text{\textgreek{w}}m-a^{2}M^{2}r^{-4}m^{2}}{r^{2}}+\big(1-\frac{2M}{r}\big)\frac{2M}{r^{3}}.\label{eq:PotentialODENewSpacetime}
\end{equation}

Any smooth function $\text{\textgreek{Y}}:\mathcal{M}_{0}\rightarrow\mathbb{C}$
which is square integrable in the $t$ variable can be uniquely decomposed
as 
\[
\text{\textgreek{Y}}(t,r,\text{\textgreek{j}},\text{\textgreek{f}})\doteq\sum_{(m,l)\in\mathbb{Z}\times\mathbb{Z}_{\ge|m|}}\int_{-\infty}^{\infty}e^{-i\text{\textgreek{w}}t}e^{im\text{\textgreek{f}}}S_{ml}(\text{\textgreek{j}})\text{\textgreek{Y}}_{\text{\textgreek{w}}ml}(r)\, d\text{\textgreek{w}}
\]
for some $\text{\textgreek{Y}}_{\text{\textgreek{w}}ml}:(r_{+},+\infty)\rightarrow\mathbb{C}$.
With $\overline{\mathcal{M}}_{0}$ defined as in (\ref{eq:ExtendedManifold}),
if $\text{\textgreek{y}}:\overline{\mathcal{M}}_{0}\rightarrow\mathbb{C}$
is a smooth function which is supported on a set of the form $\{t\ge t_{0}\}$
for some $t_{0}\in\mathbb{R}$ and is square integrable in the $t^{*}$
variable, satisfying
\begin{equation}
\square_{g_{M,a}^{(1)}}\text{\textgreek{y}}=F\label{eq:InhomogeneousWaveEquationNewSpacetime}
\end{equation}
for some smooth function $F:\overline{\mathcal{M}}_{0}\rightarrow\mathbb{C}$
which is square integrable in $t^{*}$, then $u_{\text{\textgreek{w}}ml}(r)\doteq r\text{\textgreek{y}}_{\text{\textgreek{w}}ml}(r)$
satisfies for all $(m,l)\in\mathbb{Z}\times\mathbb{Z}_{\ge|m|}$ and
almost all $\text{\textgreek{w}}\in\mathbb{R}$: 
\begin{equation}
\begin{cases}
u_{\text{\textgreek{w}}ml}^{\prime\prime}+\big(\text{\textgreek{w}}^{2}-V_{\text{\textgreek{w}}ml}\big)u_{\text{\textgreek{w}}ml}=(r-2M)\cdot F_{\text{\textgreek{w}}ml}\\
u_{\text{\textgreek{w}}ml}^{\prime}-i\text{\textgreek{w}}u_{\text{\textgreek{w}}ml}\rightarrow0\mbox{ as }r_{*}\rightarrow+\infty\\
u_{\text{\textgreek{w}}ml}^{\prime}+i\big(\text{\textgreek{w}}-\frac{am}{8M^{2}}\big)u_{\text{\textgreek{w}}ml}\rightarrow0\mbox{ as }r_{*}\rightarrow-\infty.
\end{cases}\label{eq:ODEForTheFourierTransformNewSpacetime}
\end{equation}
 The derivation of this ordinary differential equation for the Fourier
tranform of $\text{\textgreek{y}}$ follows in exactly the same way
as for subextremal Kerr spacetimes (see \cite{DafRodSchlap} for more
details). The last boundary condition of (\ref{eq:ODEForTheFourierTransformNewSpacetime})
is derived from the fact that $\text{\textgreek{y}}$ is smooth on
$\mathcal{H}^{+}$, and $\partial_{r_{*}}$extends smoothly on $\mathcal{H}^{+}$
as $\partial_{r_{*}}|_{\mathcal{H}^{+}}=K|_{\mathcal{H}^{+}}$. 

Let us also remark that for any smooth $V:[2M,+\infty)\rightarrow\mathbb{C}$
of compact support, equation 
\begin{equation}
\square_{g_{M,a}^{(1)}}\text{\textgreek{y}}-V(r)\text{\textgreek{y}}=F\label{eq:InhomogeneousWaveEquationNewSpacetime-1}
\end{equation}
(with $\text{\textgreek{y}},F$ having the same asymptotic behaviour
as before) also separates, leading to the following variant of (\ref{eq:ODEForTheFourierTransformNewSpacetime})
for $u_{\text{\textgreek{w}}ml}(r)\doteq r\text{\textgreek{y}}_{\text{\textgreek{w}}ml}(r)$
for all $(m,l)\in\mathbb{Z}\times\mathbb{Z}_{\ge|m|}$ and almost
all $\text{\textgreek{w}}\in\mathbb{R}$: 
\begin{equation}
\begin{cases}
u_{\text{\textgreek{w}}ml}^{\prime\prime}+\big(\text{\textgreek{w}}^{2}-V_{\text{\textgreek{w}}ml}-(1-\frac{2M}{r})\cdot V\big)u_{\text{\textgreek{w}}ml}=(r-2M)\cdot F_{\text{\textgreek{w}}ml}\\
u_{\text{\textgreek{w}}ml}^{\prime}-i\text{\textgreek{w}}u_{\text{\textgreek{w}}ml}\rightarrow0\mbox{ as }r_{*}\rightarrow+\infty\\
u_{\text{\textgreek{w}}ml}^{\prime}+i\big(\text{\textgreek{w}}-\frac{am}{8M^{2}}\big)u_{\text{\textgreek{w}}ml}\rightarrow0\mbox{ as }r_{*}\rightarrow-\infty.
\end{cases}\label{eq:ODEForTheFourierTransformNewSpacetime-1}
\end{equation}

The \emph{superradiant frequency regime }of $(\mathcal{M}_{0},g_{M,a}^{(1)})$\emph{,
}defined as for the Kerr exterior spacetime $(\mathcal{M}_{M,a},g_{M,a})$
(see Section \ref{sub:SuperradiantRegime}), consists of those frequency
triads $(\text{\textgreek{w}},m,l)\in\mathbb{R}\times\mathbb{Z}\times\mathbb{Z}_{\ge|m|}$
for which the limits 
\begin{equation}
\mathcal{F}_{\pm}[u_{\text{\textgreek{w}}ml}]=\lim_{r_{*}\rightarrow\pm\infty}\pm Im\big(\text{\textgreek{w}}u_{\text{\textgreek{w}}ml}^{-1}\cdot u_{\text{\textgreek{w}}ml}^{\prime}\big)
\end{equation}
 for any non-zero function $u_{\text{\textgreek{w}}ml}$ satisfying
(\ref{eq:ODEForTheFourierTransformNewSpacetime}) have opposite sign.
In view of the boundary conditions of (\ref{eq:ODEForTheFourierTransformNewSpacetime}),
it readily follows that the superradiant frequency regime has the
following form 
\begin{equation}
\mathfrak{A}_{sup}^{(a,M)}=\Big\{(\text{\textgreek{w}},m,l)\in\mathbb{R}\times\mathbb{Z}\times\mathbb{Z}_{\ge|m|}\,\big|\,\text{\textgreek{w}}\big(\text{\textgreek{w}}-\frac{am}{8M^{2}}\big)<0\Big\}\label{eq:SuperradiantRegimeNewSpacetime}
\end{equation}

The metric $g_{M,a}^{(1)}$ on $\overline{\mathcal{M}}_{0}$ approaches
the Schwarzschild exterior metric $g_{M,0}$ smoothly as $a\rightarrow0$.
Based on this fact, the following lemma can be readily inferred (the
proof of which will be omitted):
\begin{lem}
\label{lem:BehaviourOfPotentialNewSpacetime}There exists some (large)
$C_{0}>1$ such that the following statement holds: For any $\text{\textgreek{d}}_{0}>0$,
there exists some $0<a_{0}\ll M$, such that for any $0<a<a_{0}$
the potential $V_{\text{\textgreek{w}}ml}$ has the following properties:

\begin{enumerate}

\item For any $(\text{\textgreek{w}},m,l)\in\mathbb{R}\times\mathbb{Z}\times\mathbb{Z}_{\ge|m|}$,
we have 
\begin{equation}
\big|V_{\text{\textgreek{w}}ml}(r)-V_{ml}^{(Sch)}(r)\big|\le\text{\textgreek{d}}_{0}r^{-2}\big(\text{\textgreek{w}}^{2}+r^{-2}m^{2}\big),\label{eq:ClosednessSchwarzschildPotential}
\end{equation}
where $V_{ml}^{(Sch)}$ is the corresponding potential for the Schwarzschild
metric $g_{M,0}$: 
\begin{equation}
V_{ml}^{(Sch)}(r)=\frac{\big(1-\frac{2M}{r}\big)l(l+1)}{r^{2}}+\big(1-\frac{2M}{r}\big)\frac{2M}{r^{3}}.
\end{equation}

\item For any $(\text{\textgreek{w}},m,l)\in\mathbb{R}\times\mathbb{Z}\times(\mathbb{Z}_{\ge|m|}\backslash\{0\})$
such that $l\ge\max\{C_{0}^{-1}|\text{\textgreek{w}}|,C_{0}\}$, there
exists some $r_{\text{\textgreek{w}}ml}\in(3M-\text{\textgreek{d}}_{0},3M+\text{\textgreek{d}}_{0})$
such that $\frac{dV_{\text{\textgreek{w}}ml}}{dr}$ has only a simple
root at $r=r_{\text{\textgreek{w}}ml}$, and moreover 
\begin{equation}
\big(1-\frac{r}{r_{\text{\textgreek{w}}ml}}\big)\cdot\frac{dV_{\text{\textgreek{w}}ml}}{dr}\ge C_{0}^{-1}(l^{2}+\text{\textgreek{w}}^{2})r^{-3}\label{eq:TrappingEstimatePotential}
\end{equation}

\end{enumerate}
\end{lem}

\subsection{\label{sub:ILEDAuxiliaryMetric}An integrated local energy decay
estimate for $(\mathcal{M}_{0},g_{M,a}^{(1)})$ in the case $a\ll M$}

Using Lemma \ref{lem:BehaviourOfPotentialNewSpacetime}, we will establish
the following integrated local energy decay estimate in the case $a\ll M$,
along the lines of the corresponding proof in the case of a very slowly
rotating Kerr spacetimes in \cite{DafRod6}: 
\begin{prop}
\label{Prop:ILEDNewSpacetime}With the notation as in Section \ref{sub:AuxiliaryMetric},
there exists a (small) $a_{0},\text{\textgreek{d}}_{0}>0$ and (large)
$C>0$,$R_{0}>2M$, such that for any $0<a\ll a_{0}$ and any smooth
solution $\text{\textgreek{y}}$ of (\ref{eq:InhomogeneousWaveEquationNewSpacetime})
on $\overline{\mathcal{M}}_{0}$ supported on a set of the form $\{t\ge t_{0}\}$
for some $t_{0}\in\mathbb{R}$ with the property that $\text{\textgreek{y}}$,
$F$ are square integrable in the $t^{*}$ variable, the following
integrated local energy decay estimates hold: 
\begin{gather}
\int_{\mathcal{M}_{0}}\big((1-\frac{2M}{r})r^{-2}|\partial_{r}\text{\textgreek{y}}|^{2}+\text{\textgreek{q}}_{r\neq3M}(r)\cdot r^{-2}J_{\text{\textgreek{m}}}^{N}(\text{\textgreek{y}})N^{\text{\textgreek{m}}}+r^{-4}|\text{\textgreek{y}}|^{2}\big)\le C\cdot\mathcal{Z}[F,\text{\textgreek{y}};R_{0}],\label{eq:DegenerateIled}\\
\int_{\mathcal{M}_{0}}\big(r^{-2}J_{\text{\textgreek{m}}}^{N}(\text{\textgreek{y}})N^{\text{\textgreek{m}}}+r^{-4}|\text{\textgreek{y}}|^{2}\big)\le C\cdot\big(\mathcal{Z}[F,\text{\textgreek{y}};R_{0}]+\mathcal{Z}[TF,T\text{\textgreek{y}};R_{0}]\big),\label{eq:NonDegenerateIled}
\end{gather}
where 
\begin{align}
\mathcal{Z}[F,\text{\textgreek{y}};R_{0}]\doteq & \int_{\{r\le R_{0}\}}|F|^{2}+\Bigg|\int_{\{r\ge R_{0}\}}F\cdot\big(r^{-1}+O(r^{-2})\big)\partial_{r}(r\bar{\text{\textgreek{y}}})\,\Bigg|+\Bigg|\int_{\{r\ge R_{0}\}}F\cdot O(r^{-2})\bar{\text{\textgreek{y}}}\,\Bigg|+\Bigg|\int_{\mathcal{M}_{0}}F\cdot T\bar{\text{\textgreek{y}}}\,\Bigg|,\label{eq:InhomogeneousErrorTerms}
\end{align}
 $N$ is a $T$-invariant timelike vectorfield on $\overline{\mathcal{M}}_{0}$
such that $N\equiv T$ in the region $\{r\gg1\}$, and the cut-off
function $\text{\textgreek{q}}_{r\neq3M}:(2M,+\infty)\rightarrow[0,1]$
vanishes in $[3M-\text{\textgreek{d}}_{0},3M+\text{\textgreek{d}}_{0}]$
and is identically $1$ on $(2M,3M-2\text{\textgreek{d}}_{0}]\cup[3M+2\text{\textgreek{d}}_{0},+\infty)$.\end{prop}
\begin{rem*}
Note that for $a\ll M$, the uniform energy boundedness results of
\cite{DafRod5} apply on $(\overline{\mathcal{M}}_{0},g_{M,a}^{(1)})$.\end{rem*}
\begin{proof}
Let us assume first that the following estimates hold for the frequency
separated equation (\ref{eq:ODEForTheFourierTransformNewSpacetime})
for any $(\text{\textgreek{w}},m,l)\in\mathbb{R}\times\mathbb{Z}\times\mathbb{Z}_{\ge|m|}$,
any $R_{1}^{*}\gg1$ and any $\text{\textgreek{e}}_{1}>0$ sufficiently
small: 
\begin{align}
\int_{-R_{1}^{*}}^{R_{1}^{*}}\big(r^{-2}|u_{\text{\textgreek{w}}ml}^{\prime}|^{2}+ & r^{-2}(\text{\textgreek{w}}^{2}+l^{2}+r^{-2})|u_{\text{\textgreek{w}}ml}|^{2}\big)\, dr_{*}\le\label{eq:NonTrappingILED}\\
\le C_{R_{1}^{*}} & \text{\textgreek{d}}_{0}m^{2}|u_{\text{\textgreek{w}}ml}(-\infty)|^{2}+C\int_{R_{1}^{*}}^{\text{\textgreek{e}}_{1}R_{1}^{*}}\big(\text{\textgreek{e}}_{1}r^{-2}+r^{-3}\big)\text{\textgreek{w}}^{2}|u_{\text{\textgreek{w}}ml}|^{2}+\nonumber \\
 & +C_{R_{1}^{*}}\int_{-\infty}^{+\infty}Re\big\{(r-2M)F_{\text{\textgreek{w}}ml}\cdot\big(f_{\text{\textgreek{w}}ml}\bar{u}_{\text{\textgreek{w}}ml}^{\prime}+(r^{-1}h_{\text{\textgreek{w}}ml}+i\text{\textgreek{w}})\bar{u}_{\text{\textgreek{w}}ml}\big)\big\}\, dr_{*},\nonumber 
\end{align}
for $|\text{\textgreek{w}}|\gg l$ or $|\text{\textgreek{w}}|\ll l$,
and 
\begin{align}
\int_{-R_{1}^{*}}^{R_{1}^{*}}\big(r^{-2}|u_{\text{\textgreek{w}}ml}^{\prime}|^{2}+ & r^{-2}\{(1-\frac{r_{\text{\textgreek{w}}ml}}{r})^{2}(\text{\textgreek{w}}^{2}+l^{2})+r^{-2}\}|u_{\text{\textgreek{w}}ml}|^{2}\big)\, dr_{*}\le\label{eq:TrappingILED}\\
\le C_{R_{1}^{*}} & \text{\textgreek{d}}_{0}m^{2}|u_{\text{\textgreek{w}}ml}(-\infty)|^{2}+C\int_{R_{1}^{*}}^{\text{\textgreek{e}}_{1}R_{1}^{*}}\big(\text{\textgreek{e}}_{1}r^{-2}+r^{-3}\big)\text{\textgreek{w}}^{2}|u_{\text{\textgreek{w}}ml}|^{2}+\nonumber \\
 & +C_{R_{1}^{*}}\int_{-\infty}^{+\infty}Re\big\{(r-2M)F_{\text{\textgreek{w}}ml}\cdot\big(f_{\text{\textgreek{w}}ml}\bar{u}_{\text{\textgreek{w}}ml}^{\prime}+(r^{-1}h_{\text{\textgreek{w}}ml}+i\text{\textgreek{w}})\bar{u}_{\text{\textgreek{w}}ml}\big)\big\}\, dr_{*},\nonumber 
\end{align}
for $|\text{\textgreek{w}}|\sim l$, where the functions $f_{\text{\textgreek{w}}ml},h_{\text{\textgreek{w}}ml}$
depend on the precise choice of $\text{\textgreek{w}},m,l,R_{1}^{*},\text{\textgreek{e}}_{1}$,
and are bounded independently of $\text{\textgreek{w}},m,l$.

\medskip{}

\noindent \emph{Remark.} Notice that the constants in front of the
second terms of the right hand sides of (\ref{eq:NonTrappingILED}),
(\ref{eq:TrappingILED}) are independent of $R_{1}^{*}$.

\medskip{}

\noindent Then, combining (\ref{eq:NonTrappingILED}) and (\ref{eq:TrappingILED})
with the red shift type estimates of Section 7 of \cite{DafRod6}
(see also \cite{DafRod9}) associated to the $K$ vector field of
$\overline{\mathcal{M}}_{0}$ and the general $\partial_{r}$-Morawetz
type inequalities of \cite{Moschidisc} for the region $\{r\gtrsim R_{0}\}$,
one obtains both the integrated local energy decay estimates (\ref{eq:DegenerateIled})
and (\ref{eq:NonDegenerateIled}) (see e.\,g.~\cite{DafRod9,DafRodSchlap}). 

We will now proceed to establish (\ref{eq:NonTrappingILED}) and (\ref{eq:TrappingILED})
for all $(\text{\textgreek{w}},m,l)\in\mathbb{R}\times\mathbb{Z}\times\mathbb{Z}_{\ge|m|}$.
We will first deal with the very low frequency regime $|\text{\textgreek{w}}|\ll1$.
Fixing a small $\text{\textgreek{e}}_{0}>0$ (depending only on the
geometry of the family $(\mathcal{M}_{0},g_{M,a}^{(1)})$ for $a>0$),
inequalities (\ref{eq:DegenerateIled}) and (\ref{eq:NonDegenerateIled})
for solutions $\text{\textgreek{y}}$ to (\ref{eq:InhomogeneousWaveEquationNewSpacetime})
with frequency support contained in $\{|\text{\textgreek{w}}|\le\text{\textgreek{e}}_{0}\}$
follow readily by repeating the proof of Proposition 6.1 of \cite{Moschidisb}
for equation (\ref{eq:InhomogeneousWaveEquationNewSpacetime}). Thus,
it remains to obtain (\ref{eq:NonTrappingILED}) and (\ref{eq:TrappingILED})
in the case $|\text{\textgreek{w}}|\ge\text{\textgreek{e}}_{0}$.

In the case $|\text{\textgreek{w}}|\gg l$ or $\text{\textgreek{e}}_{0}\le|\text{\textgreek{w}}|\ll l$,
provided $a_{0}$ is sufficiently small, in view of (\ref{eq:ClosednessSchwarzschildPotential})
(and the boundary conditions of (\ref{eq:ODEForTheFourierTransformNewSpacetime})
for $u_{\text{\textgreek{w}}ml}$ at $r_{*}=\pm\infty$), it follows
immediately from the corresponding frequency separated estimates in
\cite{DafRod9} for Schwarzschild and very slowly rotating Kerr spacetimes
that 
\begin{align}
\int_{-R_{1}^{*}}^{R_{1}^{*}}\big(r^{-2}|u_{\text{\textgreek{w}}ml}^{\prime}|^{2}+ & r^{-2}(\text{\textgreek{w}}^{2}+l^{2}+r^{-2})|u_{\text{\textgreek{w}}ml}|^{2}\big)\, dr_{*}\le C_{R_{1}^{*}}\big((\text{\textgreek{w}}^{2}+\text{\textgreek{d}}_{0}m^{2})|u_{\text{\textgreek{w}}ml}(-\infty)|^{2}-\text{\textgreek{w}}^{2}|u_{\text{\textgreek{w}}ml}(+\infty)|^{2}\big)+\label{eq:NonTrappingILED-1}\\
+ & C_{R_{1}^{*}}\int_{-\infty}^{+\infty}Re\big\{(r-2M)F_{\text{\textgreek{w}}ml}\cdot\big(f_{\text{\textgreek{w}}ml}\bar{u}_{\text{\textgreek{w}}ml}^{\prime}+r^{-1}h_{\text{\textgreek{w}}ml}\bar{u}_{\text{\textgreek{w}}ml}\big)\big\}\, dr_{*}+C\int_{R_{1}^{*}}^{\text{\textgreek{e}}_{1}R_{1}^{*}}\big(\text{\textgreek{e}}_{1}r^{-2}+r^{-3}\big)\text{\textgreek{w}}^{2}|u_{\text{\textgreek{w}}ml}|^{2},\nonumber 
\end{align}
for suitable functions $f_{\text{\textgreek{w}}ml},h_{\text{\textgreek{w}}ml}$
with the properties described above. Therefore, inequality (\ref{eq:NonTrappingILED})
follows from (\ref{eq:NonTrappingILED-1}), in view of the the microlocal
$T$-energy identity 
\begin{equation}
\text{\textgreek{w}}^{2}|u_{\text{\textgreek{w}}ml}(+\infty)|^{2}-\text{\textgreek{w}}\big(\text{\textgreek{w}}-\frac{am}{8M^{2}}\big)|u_{\text{\textgreek{w}}ml}(-\infty)|^{2}=\int_{-\infty}^{+\infty}Im\big((r-2M)F_{\text{\textgreek{w}}ml}\cdot\text{\textgreek{w}}\bar{u}_{\text{\textgreek{w}}ml}\big)\, dr_{*},\label{eq:MicrolocalT-energy identity}
\end{equation}
where 
\begin{equation}
|u_{\text{\textgreek{w}}ml}(\pm\infty)|^{2}\doteq\lim_{r_{*}\rightarrow\pm\infty}|u_{\text{\textgreek{w}}ml}(r_{*})|^{2}.
\end{equation}

Finally, in the frequency regime $|\text{\textgreek{w}}|\sim l$ (where
trapping takes place), inequality (\ref{eq:TrappingILED}) follows
in a similar way as for the very slowly rotating Kerr spacetimes in
Section 5 of \cite{DafRod6}. Let us introduce a smooth function $f_{\text{\textgreek{w}}ml}:\mathbb{R}\rightarrow[-1,1]$
such that $f_{\text{\textgreek{w}}ml}(r_{*}(\cdot)):(2M,+\infty)\rightarrow[-1,1]$
extends smoothly up to $r=2M$, satisfying the following properties: 

\begin{enumerate}

\item $f_{\text{\textgreek{w}}ml}^{\prime}\ge0$ and $f_{\text{\textgreek{w}}ml}^{\prime}\ge c_{R_{1}^{*}}>0$
on $[-R_{1}^{*},R_{1}^{*}]$,

\item $f_{\text{\textgreek{w}}ml}<0$ for $r<r_{\text{\textgreek{w}}ml}$
and $f_{\text{\textgreek{w}}ml}>0$ for $r>r_{\text{\textgreek{w}}ml}$,

\item $-f_{\text{\textgreek{w}}ml}V_{\text{\textgreek{w}}ml}^{\prime}-\frac{1}{2}f_{\text{\textgreek{w}}ml}^{\prime\prime\prime}\ge0$
and $-f_{\text{\textgreek{w}}ml}V^{\prime}-\frac{1}{2}f_{\text{\textgreek{w}}ml}^{\prime\prime\prime}\ge c_{R_{1}^{*}}>0$
on $[-R_{1}^{*},R_{1}^{*}]$,

\item $f_{\text{\textgreek{w}}ml}$ is independent of $(\text{\textgreek{w}},m,l)$
for $r_{*}\ge r_{*}(R_{0})$.

\end{enumerate}

Such a function clearly exists, in view of the property (\ref{eq:TrappingEstimatePotential})
of $V_{\text{\textgreek{w}}ml}$ (see also the related constuction
in \cite{DafRodSchlap}). 

After multiplying (\ref{eq:ODEForTheFourierTransformNewSpacetime})
with $2f_{\text{\textgreek{w}}ml}\bar{u}^{\prime}+f_{\text{\textgreek{w}}ml}^{\prime}\bar{u}$
and integrating by parts (using also the boundary conditions of (\ref{eq:ODEForTheFourierTransformNewSpacetime})
for $u_{\text{\textgreek{w}}ml}$ at $r_{*}=\pm\infty$), we obtain:
\begin{align}
\int_{-\infty}^{+\infty}\big(2f_{\text{\textgreek{w}}ml}^{\prime} & |u^{\prime}|^{2}-(f_{\text{\textgreek{w}}ml}V_{\text{\textgreek{w}}ml}^{\prime}+f_{\text{\textgreek{w}}ml}^{\prime\prime\prime}\big)|u|^{2}\big)\, dr_{*}\le\label{eq:CurrentForTrapping}\\
\le & -\int_{-\infty}^{+\infty}Re\Big\{(r-2M)F\cdot\big(2f_{\text{\textgreek{w}}ml}\bar{u}^{\prime}+f_{\text{\textgreek{w}}ml}^{\prime}\bar{u}\big)\Big\}\, dr_{*}+2f_{\text{\textgreek{w}}ml}\cdot\Big(\big(\text{\textgreek{w}}-\frac{am}{8M^{2}}\big)^{2}|u_{\text{\textgreek{w}}ml}(-\infty)|^{2}-\text{\textgreek{w}}^{2}|u_{\text{\textgreek{w}}ml}(+\infty)|^{2}\Big).\nonumber 
\end{align}
Thus, from (\ref{eq:MicrolocalT-energy identity}), (\ref{eq:CurrentForTrapping})
and the properties of $f_{\text{\textgreek{w}}ml}$, we obtain (\ref{eq:TrappingILED})
in the case $|\text{\textgreek{w}}|\sim l$. Therefore, the proof
of Proposition \ref{Prop:ILEDNewSpacetime} is complete.\end{proof}
\begin{rem*}
Notice that, while the integrated local energy decay estimates of
Proposition \ref{Prop:ILEDNewSpacetime} were established for solutions
$\text{\textgreek{y}}$ of (\ref{eq:InhomogeneousWaveEquationNewSpacetime})
on $\overline{\mathcal{M}}_{0}$ which are square integrable in $t^{*}$,
Proposition \ref{Prop:ILEDNewSpacetime}, used as a black box (combined
with the fact that the uniform energy boundedness results of \cite{DafRod5}
hold on $(\overline{\mathcal{M}}_{0},g_{M,a}^{(1)})$), also implies
that an estimate of the form (\ref{eq:IledWithLossPhysical}) (with
$k=1$) holds for solutions $\text{\textgreek{y}}$ of (\ref{eq:InhomogeneousWaveEquationNewSpacetime})
which are not necessarily square integrable in $t^{*}$(arguing as
in the proof of Proposition \ref{prop:ILEDProductCase}, see Section
\ref{sec:SpectralConsequencesAppendix} of the Appendix). In particular,
equation (\ref{eq:InhomogeneousWaveEquationNewSpacetime}) for $F=0$
on $(\mathcal{M}_{0},g_{M,a}^{(1)})$ (for $a$ sufficiently small)
does not admit real outgoing mode solutions 

Furthermore, fixing $\bar{t}:\overline{\mathcal{M}}_{0}\rightarrow\mathbb{R}$
to be a smooth function satisfying $T(\bar{t})=1$, with level sets
which are spacelike hyperboloids terminating at $\mathcal{I}^{+}$
and intersecting $\mathcal{H}^{+}$, the results of \cite{Moschidisc}
imply (in view of (\ref{eq:NonDegenerateIled})) that smooth solutions
$\text{\textgreek{y}}$ of the wave equation on $(\overline{\mathcal{M}}_{0},g_{M,a}^{(1)})$
(for $a$ sufficiently small) with suitably decaying initial data
on a Cauchy hypersurface of $\mathcal{M}_{0}$ decay at a $\bar{t}^{-\frac{3}{2}}$
rate. 
\end{rem*}

\subsection{\label{sub:DeformedMetric}The deformed metric $g_{M,a}^{(h,R_{+})}$}

For any $M>0$, $a>0$, any positive constant $R_{+}>2M$ and any
smooth function $h:(2M,+\infty)\rightarrow[1,+\infty)$ such that
$h\equiv1$ for $r\le R_{+}$ and $h=O(1)$ as $r\rightarrow+\infty$,
we introduce the following metric on $\mathcal{M}_{0}$:

\begin{align}
g_{M,a}^{(h,R_{+})}=-\big( & 1-\frac{2M}{r}-\frac{a^{2}M^{2}\sin^{2}\text{\textgreek{j}}}{r^{4}}\big)dt^{2}-\frac{2Ma\cdot h(r)\sin^{2}\text{\textgreek{j}}}{r}dtd\text{\textgreek{f}}+\big(1-\frac{2M}{r}\big)^{-1}dr^{2}+\label{eq:DeformedMetricBlackHole}\\
 & +h^{2}(r)\cdot r^{2}\big(d\text{\textgreek{j}}^{2}+\sin^{2}\text{\textgreek{j}}d\text{\textgreek{f}}^{2}\big).\nonumber 
\end{align}
Since $h\equiv1$ on $\{r\le R_{+}\}$, we have $g_{M,a}^{(1)}\equiv g_{M,a}^{(h,R_{+})}$
on $\{2M<r<R_{+}\}$, and thus $g_{M,a}^{(h,R_{+})}$ also extends
as a smooth metric on $\overline{\mathcal{M}}_{0}$. 

Notice that 
\begin{equation}
g_{M,a}^{(1)}(T,T)=g_{M,a}^{(h,R_{+})}(T,T)
\end{equation}
everywhere on $\mathcal{M}_{0}$, and thus $(\mathcal{M}_{0},g_{M,a}^{(h,R_{+})})$
has the same ergoregion as $(\mathcal{M}_{0},g_{M,a}^{(1)})$. 

The wave operator associated to $g_{M,a}^{(h,R_{+})}$ takes the form:
\begin{align}
\square_{g_{M,a}^{(h,R_{+})}}\text{\textgreek{y}}= & r^{-2}h^{-2}\partial_{r}\big(r^{2}(1-\frac{2M}{r})h^{2}\partial_{r}\text{\textgreek{y}}\big)+r^{-2}h^{-2}(\sin\text{\textgreek{j}})^{-1}\partial_{\text{\textgreek{j}}}\big(\sin\text{\textgreek{j}}\partial_{\text{\textgreek{j}}}\text{\textgreek{y}}\big)+\label{eq:WaveOperatorDeformedBeforeSeperation}\\
 & +\frac{-\partial_{t}^{2}\text{\textgreek{y}}-2Mar^{-3}h^{-1}\partial_{t}\partial_{\text{\textgreek{f}}}\text{\textgreek{y}}+\big(\big(1-\frac{2M}{r}\big)(\sin\text{\textgreek{j}})^{-2}-a^{2}M^{2}r^{-4}\big)r^{-2}h^{-2}\partial_{\text{\textgreek{f}}}^{2}\text{\textgreek{y}}}{\big(1-\frac{2M}{r}\big)}.\nonumber 
\end{align}
Therefore, the wave equation 
\begin{equation}
\square_{g_{M,a}^{(h,R_{+})}}\text{\textgreek{y}}=0\label{eq:WaveEquationDeformedSpacetime}
\end{equation}
separates (as in the case of $g_{M,a}^{(1)}$): For any $(\text{\textgreek{w}},m,l)\in\mathbb{C}\times\mathbb{Z}\times\mathbb{Z}_{\ge|m|}$,
(\ref{eq:WaveEquationDeformedSpacetime}) admits solutions of the
form (\ref{eq:ProductSolution-2}), with $S_{ml}(\text{\textgreek{j}})$
satisfying (\ref{eq:AngularODE-1}), and $R_{\text{\textgreek{w}}ml}(r)$
satisfying the following ordinary differential equation (compare with
(\ref{eq:SimplifiedRadialODE})):
\begin{equation}
u_{\text{\textgreek{w}}ml}^{\prime\prime}+\big(\text{\textgreek{w}}^{2}-V_{\text{\textgreek{w}}ml;h}^{(h)}\big)u_{\text{\textgreek{w}}ml}=0,\label{eq:SimplifiedRadialODEDeformed}
\end{equation}
where $^{\prime}$ denotes differentiation with respect to the $r_{*}$
variable (defined by (\ref{eq:NewAuxiliaryFunction})), $u_{\text{\textgreek{w}}ml}$
is defined in terms of $R_{\text{\textgreek{w}}ml}$ as
\begin{equation}
u_{\text{\textgreek{w}}ml}(r)\doteq h(r)\cdot rR_{\text{\textgreek{w}}ml}(r)\label{eq:NewRescaledVariable}
\end{equation}
and $V_{\text{\textgreek{w}}ml;h}^{(h)}$ is defined as: 
\begin{equation}
V_{\text{\textgreek{w}}ml;h}^{(h)}(r)\doteq V_{\text{\textgreek{w}}ml}^{(h)}(r)+\frac{h^{\prime\prime}(r)}{h(r)}+2\big(1-\frac{2M}{r}\big)\frac{h^{\prime}(r)}{r\cdot h(r)}\label{eq:DeformedPotential}
\end{equation}
and 
\begin{equation}
V_{\text{\textgreek{w}}ml}^{(h)}(r)=\frac{\big(1-\frac{2M}{r}\big)\big(h(r)\big)^{-2}l(l+1)+2aMr^{-1}\big(h(r)\big)^{-1}\text{\textgreek{w}}m-a^{2}M^{2}r^{-4}\big(h(r)\big)^{-2}m^{2}}{r^{2}}+\big(1-\frac{2M}{r}\big)\frac{2M}{r^{3}}.\label{eq:ModelPotential}
\end{equation}
Similarly, the inhomogeneous wave equation 
\begin{equation}
\square_{g_{M,a}^{(h,R_{+})}}\text{\textgreek{y}}=F\label{eq:InhomogeneousWaveEquationDeformedSpacetime}
\end{equation}
separates as: 
\begin{equation}
u_{\text{\textgreek{w}}ml}^{\prime\prime}+\big(\text{\textgreek{w}}^{2}-V_{\text{\textgreek{w}}ml;h}^{(h)}\big)u_{\text{\textgreek{w}}ml}=h\cdot(r-2M)F_{\text{\textgreek{w}}ml}\label{eq:SeperatedInhomogeneousEquation}
\end{equation}

In addition to equation (\ref{eq:SimplifiedRadialODEDeformed}), in
the next sections we will also study the behaviour of solutions to
the following simplified version of (\ref{eq:SimplifiedRadialODEDeformed}):
\begin{equation}
u_{\text{\textgreek{w}}ml}^{\prime\prime}+\big(\text{\textgreek{w}}^{2}-V_{\text{\textgreek{w}}ml}^{(h)}\big)u_{\text{\textgreek{w}}ml}=0.\label{eq:ModelODEdeformation}
\end{equation}

\subsection{\label{sub:GeneralisationMainTheorem}A generalisation of Theorem
\hyperref[Theorem 1 detailed]{1}}

Let us fix a large parameter $R_{\infty}\gg1$, and let us introduce
the following class of real functions: 
\begin{equation}
\mathcal{B}_{R_{\infty}}=\big\{ h:\mathbb{R}\rightarrow[1,+\infty)\mbox{ continuous, such that }h\equiv1\mbox{ on }(-\infty,R_{\infty}]\big\}.\label{eq:MetricFunctionClass}
\end{equation}
In this Section, we will study for any $(\text{\textgreek{w}},m,l)\in(\mathbb{C}\backslash\{0\})\times\mathbb{Z}\times\mathbb{Z}_{\ge|m|}$
with $Im(\text{\textgreek{w}})\ge0$ and any function $h\in\mathcal{B}_{R_{\infty}}$
the ordinary differential equation (\ref{eq:ModelODEdeformation}).
Recall that the two variables $r_{*},r$ (which will both be used
as arguments for the functions considered in this section) are related
by (\ref{eq:NewAuxiliaryFunction}). We will also set $\text{\textgreek{w}}=\text{\textgreek{w}}_{R}+i\text{\textgreek{w}}_{I}$.

Let us define $u_{inf}^{(h)},u_{hor}^{(h)}:\mathbb{R}\rightarrow\mathbb{C}$
to be the unique solutions of (\ref{eq:ModelODEdeformation}) satisfying
the following asymptotic boundary conditions at $r_{*}=\pm\infty$
respectively: 
\begin{gather}
\lim_{r_{*}\rightarrow+\infty}\big(e^{-i\text{\textgreek{w}}r_{*}}u_{inf}^{(h)}(r_{*})\big)=u_{inf}(+\infty)\label{eq:InfinityLimit-1}\\
\lim_{r_{*}\rightarrow-\infty}\big(e^{i(\text{\textgreek{w}}-\frac{am}{8M^{2}})r_{*}}u_{hor}^{(h)}(r_{*})\big)=u_{hor}(-\infty)\label{eq:HorizonLimit-1}
\end{gather}
for some $u_{inf}(+\infty),u_{hor}(-\infty)\in\mathbb{C}\backslash\{0\}$
(fixed independently of the particular choice of the function $h$).
Then, the classical theory of ordinary differential equations (see
e.\,g.~Chapter 3, Section 3, and Chapter 4, Section 1, of \cite{Coppel1965})
yields the following stability result for equation (\ref{eq:ModelODEdeformation})
for any two functions $h_{1},h_{2}\in\mathcal{B}_{R_{\infty}}$ (using
also the fact that $h\ge1$ for any $h\in\mathcal{B}_{R_{\infty}}$):
\begin{gather}
e^{\text{\textgreek{w}}_{I}|r_{*}|}\big(\big|u_{hor}^{(h_{1})}-u_{hor}^{(h_{2})}\big|+\big|(u_{hor}^{(h_{1})})^{\prime}-(u_{hor}^{(h_{2})})^{\prime}\big|\big)(r_{*})\le C_{\text{\textgreek{w}}ml}\big|u_{hor}(-\infty)\big|\int_{R_{\infty}}^{\max\{r_{*},R_{\infty}\}}\big|V_{\text{\textgreek{w}}ml}^{(h_{1})}(r(x))-V_{\text{\textgreek{w}}ml}^{(h_{2})}(r(x))\big|\, dx,\label{eq:BundDifferenceHorizon}\\
e^{\text{\textgreek{w}}_{I}|r_{*}|}\big(\big|u_{inf}^{(h_{1})}-u_{inf}^{(h_{2})}\big|+\big|(u_{inf}^{(h_{1})})^{\prime}-(u_{inf}^{(h_{2})})^{\prime}\big|\big)(r_{*})\le C_{\text{\textgreek{w}}ml}\big|u_{inf}(+\infty)\big|\int_{\max\{r_{*},R_{\infty}\}}^{+\infty}\big|V_{\text{\textgreek{w}}ml}^{(h_{1})}(r(x))-V_{\text{\textgreek{w}}ml}^{(h_{2})}(r(x))\big|\, dx,\label{eq:BoundDifferenceInfinity}\\
C_{\text{\textgreek{w}}ml}^{-1}\min\big\{\big|u_{hor}(-\infty)\big|,\big|u_{inf}(-\infty)\big|\big\}\le e^{\text{\textgreek{w}}_{I}|r_{*}|}\Big(\big|u_{hor}^{(h_{1})}\big|+\big|(u_{hor}^{(h_{1})})^{\prime}\big|+\big|u_{inf}^{(h_{1})}\big|+\big|(u_{inf}^{(h_{1})})^{\prime}\big|\Big)\le C_{\text{\textgreek{w}}ml}\big(\big|u_{hor}(-\infty)\big|+\big|u_{inf}(-\infty)\big|\big).\label{eq:UniformBoundNotDifference}
\end{gather}

By repeating the proof of Theorem \hyperref[Theorem 1 detailed]{1}
for equation (\ref{eq:ModelODEdeformation}) in place of (\ref{eq:SeperatedODE}),
and with the boundary conditions (\ref{eq:InfinityLimit-1}) and (\ref{eq:HorizonLimit-1})
in place of (\ref{eq:InfinityLimit}) and (\ref{eq:HorizonLimit}),
and using Lemma \ref{lem:DifferenceV} in place of \ref{lem:PerturbationOfTheSimplerODE},
we readily deduce the following result (the details of the proof are
exactly the same and hence will be omitted):
\begin{lem}
\label{lem:CorollaryOfDeformedApplication}For any function $h\in\mathcal{B}_{R_{\infty}}$
and any frequency triad $(\text{\textgreek{w}}_{R},m,l)\in\mathfrak{A}_{sup}^{(a,M)}$,
there exist constants $C_{\text{\textgreek{w}}_{R}ml;h},C_{\text{\textgreek{w}}_{R}ml;h}^{(0)}\gg1$
such that for any $r_{0}\ge C_{\text{\textgreek{w}}_{R}ml;h}$ and
any $0\le\text{\textgreek{w}}_{I}\ll1$ sufficiently small in terms
of $\text{\textgreek{w}}_{R},m,l,r_{0}$ and $R_{\infty}$, there
exists a function $V^{(h,r_{0})}:\mathbb{R}\rightarrow[0,2\text{\textgreek{w}}_{R}^{2}]$
supported in $[r_{0},r_{0}+C_{\text{\textgreek{w}}_{R}ml;h}^{(0)}]$
so that equation 
\begin{equation}
(u^{(h)})^{\prime\prime}+\big((\text{\textgreek{w}}_{R}+i\text{\textgreek{w}}_{I})^{2}-V_{(\text{\textgreek{w}}_{R}+i\text{\textgreek{w}}_{I})ml}^{(h)}-V^{(h,r_{0})}\big)u^{(h)}=0\label{eq:DeformedOdePotential}
\end{equation}
admits a non-zero solution $u^{(h)}$ for which the limits $\lim_{r_{*}\rightarrow+\infty}\big(e^{-i(\text{\textgreek{w}}_{R}+i\text{\textgreek{w}}_{I})r_{*}}u^{(h)}(r_{*})\big)$
and $\lim_{r_{*}\rightarrow-\infty}\big(e^{i(\text{\textgreek{w}}_{R}+i\text{\textgreek{w}}_{I})-\frac{am}{8M^{2}})r_{*}}u^{(h)}(r_{*})\big)$
exist in $\mathbb{C}\backslash\{0\}$. 

Furthermore, for any $h_{1},h_{2}\in\mathcal{B}_{R_{\infty}}$ and
$\text{\textgreek{w}}_{I}$ as above, we can bound (provided $R_{\infty}$
has been fixed sufficiently large in terms of $\text{\textgreek{w}}_{R},m,l$
): 
\begin{equation}
\int_{r_{0}}^{r_{0}+C_{\text{\textgreek{w}}_{R}ml}^{(0)}}\big|V^{(h_{1},r_{0})}(r_{*})-V^{(h_{2},r_{0})}(r_{*})\big|\, dr_{*}\le C_{\text{\textgreek{w}}_{R}ml}\int_{R_{\infty}}^{+\infty}\big|V_{\text{\textgreek{w}}_{R}ml}^{(h_{1})}(r(r_{*}))-V_{\text{\textgreek{w}}_{R}ml}^{(h_{2})}(r(r_{*}))\big|\, dr_{*}.\label{eq:BoundPotentialDifference}
\end{equation}
\end{lem}
\begin{rem*}
The estimate (\ref{eq:BoundPotentialDifference}) follows from the
estimate (\ref{eq:L1EstimatePotentials}) (for the family $u^{(s)}\doteq u^{(sh_{1}+(1-s)h_{2})}$)
of Lemma \ref{lem:DifferenceV}, in view also of the bounds (\ref{eq:BundDifferenceHorizon})--(\ref{eq:UniformBoundNotDifference}).
\end{rem*}
By inspecting the proof of Theorem \hyperref[Theorem 1 detailed]{1},
we readily infer, in view of the estimates (\ref{eq:BundDifferenceHorizon})
and (\ref{eq:BoundDifferenceInfinity}) (as well as the fact that
$h\ge1$ for any $h\in\mathcal{B}_{R_{\infty}}$), that the constants
$C_{\text{\textgreek{w}}ml;h},C_{\text{\textgreek{w}}ml;h}^{(0)}$
in the statement of Lemma \ref{lem:CorollaryOfDeformedApplication}
can be chosen independently of $h\in\mathcal{B}_{R_{\infty}}$ (we
will thus use the notation $C_{\text{\textgreek{w}}ml},C_{\text{\textgreek{w}}ml}^{(0)}$
for $C_{\text{\textgreek{w}}ml;h},C_{\text{\textgreek{w}}ml;h}^{(0)}$
from now on). 

Finally, let us also note that $V^{(1,r_{0})}$ is not identically
$0$ (since that would imply that the wave equation on the spacetime
$(\mathcal{M}_{0},g_{0})$ admits an outgoing real or exponentially
growing mode solution, violating the statement of Proposition \ref{Prop:ILEDNewSpacetime}).
In particular, the proof of Theorem \hyperref[Theorem 1 detailed]{1}
yields the following lower bound for $V^{(1,r_{0})}$ (uniform in
$r_{0}$, $\text{\textgreek{w}}_{I}$) holds: 
\begin{equation}
\int_{r_{0}}^{r_{0}+C_{\text{\textgreek{w}}_{R}ml}^{(0)}}V^{(1,r_{0})}(r_{*})\, dr_{*}\ge c_{\text{\textgreek{w}}_{R}ml}>0.\label{eq:PositivityFlatCase}
\end{equation}
Moreover, since any $h\in\mathcal{B}_{R_{\infty}}$ satisfies $h\ge1$,
from (\ref{eq:ModelPotential}) we deduce that 
\begin{equation}
\limsup_{R_{\infty}\rightarrow+\infty}\int_{R_{\infty}}^{+\infty}\big|V_{(\text{\textgreek{w}}_{R}+i\text{\textgreek{w}}_{I})ml}^{(h)}(r(r_{*}))\big|\, dr_{*}=0\label{eq:DecayAtZeroForL1Potential}
\end{equation}
uniformly in $h\in\mathcal{B}_{R_{\infty}}$ (for fixed $\text{\textgreek{w}}_{R},m,l$
and any $\text{\textgreek{w}}_{I}$ sufficiently small  in terms of
$\text{\textgreek{w}}_{R},m,l$). Thus, (\ref{eq:BoundPotentialDifference}),
(\ref{eq:PositivityFlatCase}) and (\ref{eq:DecayAtZeroForL1Potential})
imply that, provided $R_{\infty}$ is sufficiently large in terms
of $\text{\textgreek{w}}_{R},m,l$, there exists some $c_{\text{\textgreek{w}}_{R}ml}>0$
depending only on $\text{\textgreek{w}}_{R},m,l$ so that for any
$h\in\mathcal{B}_{R_{\infty}}$, any $r_{0}>C_{\text{\textgreek{w}}_{R}ml}$
and any $\text{\textgreek{w}}_{I}$ is sufficiently small in terms
of $\text{\textgreek{w}}_{R},m,l,r_{0}$ and $R_{\infty}$:
\begin{align}
\int_{r_{0}}^{r_{0}+C_{\text{\textgreek{w}}_{R}ml}^{(0)}}V^{(h,r_{0})}(r_{*})\, dr_{*} & \ge\int_{r_{0}}^{r_{0}+C_{\text{\textgreek{w}}_{R}ml}^{(0)}}V^{(1,r_{0})}(r_{*})\, dr_{*}-C_{\text{\textgreek{w}}_{R}ml}\int_{R_{\infty}}^{+\infty}\big(V_{(\text{\textgreek{w}}_{R}+i\text{\textgreek{w}}_{I})ml}^{(h)}(r(r_{*}))+V_{(\text{\textgreek{w}}_{R}+i\text{\textgreek{w}}_{I})ml}^{(1)}(r(r_{*}))\big)\, dr_{*}>\label{eq:LowerBoundL1Norm}\\
 & >c_{\text{\textgreek{w}}_{R}ml}>0.\nonumber 
\end{align}

\subsection{\label{sub:FinalProofDeformed}Proof of Theorem \hyperref[Theorem 3]{3}
and Proposition \ref{prop:HighFrequencyILED}}

We will now proceed to establish Theorem \hyperref[Theorem 3]{3}
and Proposition \ref{prop:HighFrequencyILED}. In view of (\ref{eq:DeformedMetricBlackHole}),
(\ref{eq:WaveOperatorDeformedBeforeSeperation}), (\ref{eq:SimplifiedRadialODEDeformed})
and the corresponding properties of the spacetimes $(\mathcal{M}_{0},g_{M,a}^{(1)})$,
we immediately obtain the following properties of $(\mathcal{M}_{0},g_{M,a}^{(h,R_{+})})$
for any $M>0$, $a>0$ and $R_{+}\gg1$:
\begin{enumerate}
\item In the $(t,r,\text{\textgreek{j}},\text{\textgreek{f}})$ coordinate
chart on $\mathcal{M}_{0}$, the vector fields $T=\partial_{t}$ and
$\text{\textgreek{F}}=\partial_{\text{\textgreek{f}}}$ are Killing
vector fields for $g_{M,a}^{(h,R_{+})}$.
\item The spacetime $(\mathcal{M}_{0},g_{M,a}^{(h,R_{+})})$ is asymptotically
conic.
\item The metric $g_{M,a}^{(h,R_{+})}$ can be smoothly extended up to $\mathcal{H}=\{r=2M\}$
in the ``Schwarzschild star'' coordinate chart on $\mathcal{M}_{0}$.
In this extension, $\mathcal{H}$ is the event horizon associated
to the asymptotically conic end of $(\mathcal{M}_{0},g_{M,a}^{(h,R_{+})})$.
Furthermore, $\mathcal{H}$ is a Killing horizon with positive surface
gravity.
\item For any $\text{\textgreek{e}}>0$, $k\in\mathbb{N}$, there exists
an $a_{0}>0$ so that if $a\le a_{0}$, the metrics $g_{M,a}^{(h,R_{+})}$
and $g_{M,a}$ are $\text{\textgreek{e}}$-close in the $C^{k}$ norm
defined with respect to the ``Schwarzschild star'' coordinate chart
on the closed subset $\{2M\le r\le R_{b}\}$ of $\mathcal{M}_{0}\cup\mathcal{H}^{+}$.
The ergoregion of $(\mathcal{M}_{0},g_{M,a}^{(h)})$ is non-empty,
contained in the set $\{r\le2M+\text{\textgreek{e}}\}$, and 
\begin{equation}
\sup_{\mathcal{M}_{M,0}}g_{M,a}^{(h)}(T,T)<\text{\textgreek{e}}.
\end{equation}

\item The wave equation (\ref{eq:WaveEquationDeformedSpacetime}) completely
separates in the $(t,r,\text{\textgreek{j}},\text{\textgreek{f}})$
coordinate chart.
\end{enumerate}
Furthermore, given $(\text{\textgreek{w}}_{R},m,l)\in\mathfrak{A}_{sup}^{(a,M)}$,
provided $R_{+}$ is sufficiently large in terms of $\text{\textgreek{w}}_{R},m,l,a,M$,
and, additionally, $\text{\textgreek{w}}_{I}$ is sufficiently small
 in terms of $\text{\textgreek{w}}_{0},m_{0},l_{0},a,M,R_{+}$ and
$a\ll M$, the function $h$ can be chosen appropriately so that the
following properties also hold on $(\mathcal{M}_{0},g_{M,a}^{(h,R_{+})})$
(thus completing the proof of Theorem \hyperref[Theorem 3]{3} and
Proposition \ref{prop:HighFrequencyILED}):
\begin{enumerate}[resume]
\item \label{enu:RealMode}The wave equation (\ref{eq:WaveEquationDeformedSpacetime})
admits an outgoing mode solution with frequency parameter $\text{\textgreek{w}}_{R}+i\text{\textgreek{w}}_{I}$.
\item \label{enu:ILED}The trapped set of $(\mathcal{M}_{0},g_{M,a}^{(h)})$
is normally hyperbolic. Furthermore, any solution $\text{\textgreek{y}}$
to the inhomogeneous wave equation (\ref{eq:InhomogeneousWaveEquationDeformedSpacetime})
on $(\mathcal{M}_{0},g_{M,a}^{(h)})$ which is smooth up to $\mathcal{H}^{+}$
satisfies, for any $\text{\textgreek{t}}_{1}\le\text{\textgreek{t}}_{2}$,
the estimates (\ref{eq:DegenerateIled-1-2}) and (\ref{eq:NonDegenerateIled-1-2}).
\end{enumerate}
The statements \ref{enu:RealMode} and \ref{enu:ILED} above follow
directly as corollaries of Propositions \ref{prop:ModeDeformedMetric}
and \ref{prop:FromL^2ILEDtoActualILED} respectively:
\begin{prop}
\label{prop:ModeDeformedMetric}For any $a>0$, $M>0$ and any superradiant
frequency triad $(\text{\textgreek{w}}_{R},m,l)\in\mathfrak{A}_{sup}^{(a,M)}$,
there exists an $R_{+}>2M$ large in terms of $\text{\textgreek{w}}_{R},m,l,a,M$,
so that for any $0\le\text{\textgreek{w}}_{I}\ll1$ small in terms
of $\text{\textgreek{w}}_{R},m,l,R_{+}$, there exists a smooth and
increasing function $h:\mathbb{R}\rightarrow[1,+\infty)$ satisfying
$h\equiv1$ on $(-\infty,R_{+}]$ and $h(r_{*})=C_{1}+O(r_{*}^{-1})$
as $r_{*}\rightarrow+\infty$, for which the wave equation (\ref{eq:WaveEquationDeformedSpacetime})
on $(\mathcal{M}_{0},g_{M,a}^{(h,R_{+})})$ admits an outgoing mode
solution with frequency parameter $\text{\textgreek{w}}_{R}+i\text{\textgreek{w}}_{I}$. \end{prop}
\begin{proof}
In view of Lemma \ref{lem:CorollaryOfDeformedApplication} and the
form (\ref{eq:DeformedPotential}) of the potential for $g_{M,a}^{(h,R_{+})}$,
in order to construct the function $h$ so that (\ref{eq:WaveEquationDeformedSpacetime})
has a mode solution with frequency parameter $\text{\textgreek{w}}_{R}+i\text{\textgreek{w}}_{I}$,
it suffices to solve the following initial value problem on $\mathbb{R}$:
\begin{equation}
\begin{cases}
h^{\prime\prime}(r_{*})+2r^{-1}\big(1-\frac{2M}{r}\big)h^{\prime}(r_{*})=V^{(h,R_{+})}(r_{*})\cdot h(r_{*}),\\
h(R_{+})=1,\mbox{ }h^{\prime}(R_{+})=0,
\end{cases}\label{eq:NonLocalOde}
\end{equation}
 where $V^{(h,R_{+})}:\mathbb{R}\rightarrow[0,2\text{\textgreek{w}}_{R}^{2}]$
is the function provided by Lemma \ref{lem:CorollaryOfDeformedApplication}
for $R_{+}$ in place of $r_{0}$. 

\medskip{}

\noindent \emph{Remark.} Notice that the relation (\ref{eq:NonLocalOde})
is not an ordinary differential equation, since $V^{(h,R_{+})}$ is
a non-local (and non-linear) operator acting on the function $h$.
Notice also that, since $V^{(h,R_{+})}$ is supported in $\{r_{*}\ge R_{+}\}$,
a solution $h$ to (\ref{eq:NonLocalOde}) (if it exists) will be
identically equal to $1$ on $(-\infty,R_{+}]$. 

\medskip{}

We will solve (\ref{eq:NonLocalOde}) using a Picard-type iteration
scheme, assuming $R_{+}$ has been fixed sufficiently large. Let $h_{n}:\mathbb{R}\rightarrow\mathbb{R}$,
$n\in\mathbb{N}$ be defined by the recursive relation
\begin{equation}
\begin{cases}
h_{n}^{\prime\prime}(r_{*})+2r^{-1}\big(1-\frac{2M}{r}\big)h_{n}^{\prime}(r_{*})-V^{(h_{n-1},R_{+})}(r_{*})\cdot h_{n}(r_{*})=0,\\
h_{n}|_{r=R_{+}}=1,\mbox{ }h_{n}^{\prime}|_{r=R_{+}}=0,
\end{cases}\label{eq:RecursiveRelation}
\end{equation}
with $h_{0}\equiv1$. In view of (\ref{eq:NewAuxiliaryFunction}),
equation (\ref{eq:RecursiveRelation}) can be rewritten as 
\begin{equation}
\begin{cases}
r^{-2}\big(r^{2}h_{n}^{\prime}\big)^{\prime}=V^{(h_{n-1},R_{+})}h_{n},\\
h_{n}|_{r=R_{+}}=1,\mbox{ }h_{n}^{\prime}|_{r=R_{+}}=0.
\end{cases}\label{eq:RecursiveRelation-1}
\end{equation}

In order for (\ref{eq:RecursiveRelation-1}) to be well defined, we
first need to establish that for all $n\in\mathbb{N}$, $h_{n}$ is
a smooth function belonging to the class $\mathcal{B}_{R_{\infty}}$,
i.\,e.~(\ref{eq:MetricFunctionClass}). To this end, it suffices
to show that $h_{n}(r_{*})\ge1$ for all $r_{*}\ge R_{+}$. This follows
readily by induction: Assuming that $h_{n-1}$ is smooth and belongs
to the class $\mathcal{B}_{R_{\infty}}$ (this statement is true for
$h_{0}$), the function $V^{(h_{n-1},R_{+})}$ is well defined by
Lemma \ref{lem:CorollaryOfDeformedApplication}, and thus a unique
smooth solution $h_{n}:\mathbb{R}\rightarrow\mathbb{R}$ of the (inhomogeneous)
linear ordinary differential equation (\ref{eq:RecursiveRelation-1})
exists. Since $V^{(h_{n-1},R_{+})}\ge0$, integrating (\ref{eq:RecursiveRelation-1})
yields 
\begin{equation}
h_{n}^{\prime}\ge0,\label{eq:IncreasingH_n}
\end{equation}
and hence $h_{n}(r_{*})\ge h_{n}(R_{+})=1$ for all $r_{*}\ge R_{+}$.
Since $V^{(h_{n-1},R_{+})}$ is supported in $[R_{+},+\infty)$, $h_{n}$
is identically $1$ on $(-\infty,R_{+}]$, in view of the conditions
$h_{n}|_{r=R_{+}}=1$, $h_{n}^{\prime}|_{r=R_{+}}=0$, and thus $h_{n}\in\mathcal{B}_{R_{\infty}}$.

By integrating equation (\ref{eq:RecursiveRelation-1}), we obtain
the following implicit formula for $h_{n}$ for any $r_{*}\ge R_{+}$:
\begin{align}
h_{n}(r_{*}) & =1+\int_{R_{+}}^{r_{*}}\Big(\frac{1}{r^{2}(\text{\textgreek{r}})}\big(\int_{R_{+}}^{\text{\textgreek{r}}}r^{2}(\text{\textgreek{sv}})V^{(h_{n-1},R_{+})}(\text{\textgreek{sv}})h_{n}(\text{\textgreek{sv}})\, d\text{\textgreek{sv}}\big)\, d\text{\textgreek{r}}\Big)=\label{eq:ExpressionH_n}\\
 & =1+\int_{R_{+}}^{r_{*}}r^{2}(\text{\textgreek{sv}})\big(\int_{\text{\textgreek{sv}}}^{r_{*}}r^{-2}(\text{\textgreek{r}})\, d\text{\textgreek{r}}\big)V^{(h_{n-1},R_{+})}(\text{\textgreek{sv}})h_{n}(\text{\textgreek{sv}})\, d\text{\textgreek{sv}}.\nonumber 
\end{align}
Let us set 
\begin{equation}
R_{+}^{(1)}\doteq R_{+}+C_{\text{\textgreek{w}}_{R}ml}^{(0)},
\end{equation}
where $C_{\text{\textgreek{w}}_{R}ml}^{(0)}$ is the constant in the
statement of Lemma \ref{lem:CorollaryOfDeformedApplication} for $h\equiv1$
(see the remarks below Lemma \ref{lem:CorollaryOfDeformedApplication}).
Since $V^{(h_{n-1},R_{+})}$ is supported in the interval $[R_{+},R_{+}^{(1)}]$
and $0\le V^{(h_{n-1},R_{+})}\le2\text{\textgreek{w}}_{R}^{2}$, an
application of Gronwall's inequality on (\ref{eq:ExpressionH_n})
readily yields that there exists some $C_{\text{\textgreek{w}}_{R}ml}$
depending only on $\text{\textgreek{w}}_{R},m,l$ (and $a,M$), so
that for all $n\in\mathbb{N}$:
\begin{equation}
\max_{r_{*}\in[R_{+},R_{+}^{(1)}]}\big(|h_{n}^{\prime}(r_{*})|+|h_{n}(r_{*})|\big)\le C_{\text{\textgreek{w}}_{R}ml}.\label{eq:UniformBoundednessPotentialRegion}
\end{equation}
Furthermore, in view of (\ref{eq:ExpressionH_n}) and (\ref{eq:LowerBoundL1Norm})
(as well as the fact that $h_{n}\ge1$), we obtain the following lower
bound for $h_{n}^{\prime}(R_{+}^{(1)})$ for some $c_{\text{\textgreek{w}}_{R}ml}>0$
depending only on $\text{\textgreek{w}}_{R},m,l$ (and $a,M$): 
\begin{equation}
h_{n}^{\prime}(R_{+}^{(1)})\ge c_{\text{\textgreek{w}}_{R}ml}>0.\label{eq:LowerBoundDerivative}
\end{equation}

Since $V^{(h_{n-1},R_{+})}\equiv0$ for $r_{*}\ge R_{+}^{(1)}$, we
obtain the following expression for any $r_{*}>R_{+}^{(1)}$ by integrating
(\ref{eq:RecursiveRelation-1}) over $[R_{+}^{(1)},r_{*}]$: 
\begin{equation}
h_{n}(r_{*})=h_{n}(R_{+}^{(1)})+\big(R_{+}^{(1)}\big)^{2}\cdot h_{n}^{\prime}(R_{+}^{(1)})\int_{R_{+}^{(1)}}^{r_{*}}r^{-2}(\text{\textgreek{r}})\, d\text{\textgreek{r}}.\label{eq:ExpressionFarAwayRegion}
\end{equation}
Thus, in view of (\ref{eq:UniformBoundednessPotentialRegion}), (\ref{eq:LowerBoundDerivative}),
(\ref{eq:ExpressionFarAwayRegion}) and (\ref{eq:NewAuxiliaryFunction})
(as well as the fact that $h_{n}\ge1$), we obtain for any $r_{*}\ge R_{+}^{(1)}$:
\begin{equation}
1+c_{\text{\textgreek{w}}_{R}ml}R_{+}^{(1)}\big(\frac{r_{*}-R_{+}^{(1)}}{r_{*}}\big)\le h_{n}(r_{*})\le C_{\text{\textgreek{w}}_{R}ml}R_{+}^{(1)}.\label{eq:BehaviourH_nFarAway}
\end{equation}

The expression (\ref{eq:ExpressionH_n}) yields the following formula
for the differences $h_{n}-h_{n-1}$ for any $r_{*}\ge R_{+}$ and
any $n\ge2$ (recall that $h_{n}-h_{n-1}\equiv0$ for $r_{*}\le R_{+}$):
\begin{align}
h_{n}(r_{*})-h_{n-1}(r_{*})=\int_{R_{+}}^{r_{*}} & r^{2}(\text{\textgreek{sv}})\big(\int_{\text{\textgreek{sv}}}^{r_{*}}r^{-2}(\text{\textgreek{r}})\, d\text{\textgreek{r}}\big)V^{(h_{n-1},R_{+})}(\text{\textgreek{sv}})\big(h_{n}(\text{\textgreek{sv}})-h_{n-1}(\text{\textgreek{sv}})\big)\, d\text{\textgreek{sv}}+\label{eq:DifferenceExpression}\\
 & +\int_{R_{+}}^{r_{*}}r^{2}(\text{\textgreek{sv}})\big(\int_{\text{\textgreek{sv}}}^{r_{*}}r^{-2}(\text{\textgreek{r}})\, d\text{\textgreek{r}}\big)\big(V^{(h_{n-1},R_{+})}(\text{\textgreek{sv}})-V^{(h_{n-2},R_{+})}(\text{\textgreek{sv}})\big)h_{n-1}(\text{\textgreek{sv}})\, d\text{\textgreek{sv}}.\nonumber 
\end{align}
Notice that for any $r_{*}\in[R_{+},R_{+}^{(1)}]$ and any $\text{\textgreek{sv}}\in[R_{+},r_{*}]$,
we can bound, provided $R_{+}\gg1$ (in view of (\ref{eq:NewAuxiliaryFunction})):
\begin{equation}
r^{2}(\text{\textgreek{sv}})\big(\int_{\text{\textgreek{sv}}}^{r_{*}}r^{-2}(\text{\textgreek{r}})\, d\text{\textgreek{r}}\big)\le C(R_{+}^{(1)}-R_{+})\le C_{\text{\textgreek{w}}_{R}ml}.\label{eq:BoundForTrickyTerm}
\end{equation}
In view of (\ref{eq:UniformBoundednessPotentialRegion}), (\ref{eq:BoundForTrickyTerm})
and the fact that $V^{(h_{k},R_{+})}$ is supported in $[R_{+},R_{+}^{(1)}]$
and satisfies $0\le V^{(h_{k},R_{+})}\le2\text{\textgreek{w}}_{R}^{2}$
for any $k\in\mathbb{N}$, from (\ref{eq:DifferenceExpression}) we
obtain the following estimates for any $r_{*}\in[R_{+},R_{+}^{(1)}]$:
\begin{equation}
\big|h_{n}(r_{*})-h_{n-1}(r_{*})\big|\le C_{\text{\textgreek{w}}_{R}ml}\int_{R_{+}}^{r_{*}}\big|h_{n}(\text{\textgreek{sv}})-h_{n-1}(\text{\textgreek{sv}})\big|\, d\text{\textgreek{sv}}+C_{\text{\textgreek{w}}_{R}ml}\int_{R_{+}}^{r_{*}}\big|V^{(h_{n-1},R_{+})}(\text{\textgreek{sv}})-V^{(h_{n-2},R_{+})}(\text{\textgreek{sv}})\big|\, d\text{\textgreek{sv}}\label{eq:EstimateBeforeGronwall}
\end{equation}
and 
\begin{equation}
\big|h_{n}^{\prime}(r_{*})-h_{n-1}^{\prime}(r_{*})\big|\le C_{\text{\textgreek{w}}_{R}ml}\int_{R_{+}}^{r_{*}}\big|h_{n}(\text{\textgreek{sv}})-h_{n-1}(\text{\textgreek{sv}})\big|\, d\text{\textgreek{sv}}+C_{\text{\textgreek{w}}_{R}ml}\int_{R_{+}}^{r_{*}}\big|V^{(h_{n-1},R_{+})}(\text{\textgreek{sv}})-V^{(h_{n-2},R_{+})}(\text{\textgreek{sv}})\big|\, d\text{\textgreek{sv}}.\label{eq:EstimateBeforeGronwallDerivative}
\end{equation}
Thus, an application of Gronwall's inequality yields:
\begin{equation}
\sup_{r_{*}\in[R_{+},R_{+}^{(1)}]}\big(\big|h_{n}(r_{*})-h_{n-1}(r_{*})\big|+\big|h_{n}^{\prime}(r_{*})-h_{n-1}^{\prime}(r_{*})\big|\big)\le C_{\text{\textgreek{w}}_{R}ml}\int_{R_{+}}^{R_{+}^{(1)}}\big|V^{(h_{n-1},R_{+})}(\text{\textgreek{sv}})-V^{(h_{n-2},R_{+})}(\text{\textgreek{sv}})\big|\, d\text{\textgreek{sv}}.\label{eq:GronwallBeforePicard}
\end{equation}
 Moreover, the following bound is a consequence of (\ref{eq:ExpressionFarAwayRegion})
and the fact that $R_{+}^{(1)}\le2R_{+}$ (provided $R_{+}$ is sufficiently
large in terms of $\text{\textgreek{w}}_{R},m,l,M,a$):
\begin{equation}
\sup_{r_{*}\in[R_{+},+\infty)}\big(\big|h_{n}(r_{*})-h_{n-1}(r_{*})\big|+\big|h_{n}^{\prime}(r_{*})-h_{n-1}^{\prime}(r_{*})\big|\big)\le C_{\text{\textgreek{w}}_{R}ml}R_{+}\sup_{r_{*}\in[R_{+},R_{+}^{(1)}]}\big(\big|h_{n}(r_{*})-h_{n-1}(r_{*})\big|+\big|h_{n}^{\prime}(r_{*})-h_{n-1}^{\prime}(r_{*})\big|\big).\label{eq:EstimateDifferenceEverywhere}
\end{equation}

Fix some $\text{\textgreek{e}}_{\text{\textgreek{w}}_{R}ml}>0$ small
in terms of $\text{\textgreek{w}}_{R},m,l,M,a$. Provided $R_{+}$
is sufficiently large in terms of $\text{\textgreek{w}}_{R},m,l,M,a$
and the specific value of $\text{\textgreek{e}}_{\text{\textgreek{w}}_{R}ml}$,
we will establish the following bound for the right hand side of (\ref{eq:GronwallBeforePicard})
for any $n\in\mathbb{N}$: 
\begin{equation}
\int_{R_{+}}^{R_{+}^{(1)}}\big|V^{(h_{n},R_{+})}(r_{*})-V^{(h_{n-1},R_{+})}(r_{*})\big|\, dr_{*}\le\text{\textgreek{e}}_{\text{\textgreek{w}}_{R}ml}\sup_{r_{*}\in[R_{+},R_{+}^{(1)}]}\big(\big|h_{n}(r_{*})-h_{n-1}(r_{*})\big|+\big|h_{n}^{\prime}(r_{*})-h_{n-1}^{\prime}(r_{*})\big|\big).\label{eq:TheRequiredLipschitzBound}
\end{equation}
In view of (\ref{eq:BoundPotentialDifference}) and (\ref{eq:ModelPotential})
(using also the fact that $h_{n}\equiv h_{n-1}$ for $r\le R_{+}$),
we can estimate (recall that $h_{n}\ge1$ for all $n\in\mathbb{N}$):
\begin{align}
\int_{R_{+}}^{R_{+}^{(1)}}\big|V^{(h_{n},R_{+})}(r_{*})-V^{(h_{n-1},R_{+})}(r_{*})\big|\, dr_{*} & \le C_{\text{\textgreek{w}}_{R}ml}\int_{R_{+}}^{\infty}\Big(\big|h_{n}^{-2}(r_{*})-h_{n-1}^{-2}(r_{*})\big|+r_{*}^{-1}\big|h_{n}^{-1}(r_{*})-h_{n-1}^{-1}(r_{*})\big|\Big)r_{*}^{-2}\, dr_{*}\label{eq:DifferencePotentialsLipschitz}\\
 & \le C_{\text{\textgreek{w}}_{R}ml}\int_{R_{+}}^{\infty}\frac{\big|h_{n}(r_{*})-h_{n-1}(r_{*})\big|}{h_{n}(r_{*})h_{n-1}(r_{*})}r_{*}^{-2}\, dr_{*}.\nonumber 
\end{align}
Notice that we can trivially bound 
\begin{equation}
\int_{R_{+}}^{R_{+}^{(1)}}\frac{\big|h_{n}(r_{*})-h_{n-1}(r_{*})\big|}{h_{n}(r_{*})h_{n-1}(r_{*})}r_{*}^{-2}\, dr_{*}\le C_{\text{\textgreek{w}}_{R}ml}R_{+}^{-2}\sup_{r_{*}\in[R_{+},R_{+}^{(1)}]}\big|h_{n}(r_{*})-h_{n-1}(r_{*})\big|.\label{eq:BoundPotentialRegion}
\end{equation}
Furthermore, in the region $r_{*}\ge R_{+}^{(1)}$, in view of the
expression (\ref{eq:ExpressionFarAwayRegion}) we can bound 
\begin{equation}
\big|h_{n}(r_{*})-h_{n-1}(r_{*})\big|\le C\Big(1+R_{+}^{(1)}\big(\frac{r_{*}-R_{+}^{(1)}}{r_{*}}\big)\Big)\big(\big|h_{n}(R_{+}^{(1)})-h_{n-1}(R_{+}^{(1)})\big|+\big|h_{n}^{\prime}(R_{+}^{(1)})-h_{n-1}^{\prime}(R_{+}^{(1)})\big|\big).\label{eq:UsefulBoundDifferenceAway}
\end{equation}
Thus, (\ref{eq:UsefulBoundDifferenceAway}) and the left hand side
of (\ref{eq:BehaviourH_nFarAway}) imply that 
\begin{equation}
\int_{R_{+}^{(1)}}^{\infty}\frac{\big|h_{n}(r_{*})-h_{n-1}(r_{*})\big|}{h_{n}(r_{*})h_{n-1}(r_{*})}r_{*}^{-2}\, dr_{*}\le C_{\text{\textgreek{w}}_{R}ml}\Big(\int_{R_{+}^{(1)}}^{\infty}r_{*}^{-2}\, dr_{*}\Big)\sup_{r_{*}\in[R_{+},R_{+}^{(1)}]}\big(\big|h_{n}(r_{*})-h_{n-1}(r_{*})\big|+\big|h_{n}^{\prime}(r_{*})-h_{n-1}^{\prime}(r_{*})\big|\big).\label{eq:FinalBoundDifferenceAway}
\end{equation}

Adding (\ref{eq:BoundPotentialRegion}) and (\ref{eq:FinalBoundDifferenceAway})
and using (\ref{eq:DifferencePotentialsLipschitz}), we deduce: 
\begin{equation}
\int_{R_{+}}^{R_{+}^{(1)}}\big|V^{(h_{n},R_{+})}(r_{*})-V^{(h_{n-1},R_{+})}(r_{*})\big|\, dr_{*}\le C_{\text{\textgreek{w}}_{R}ml}R_{+}^{-1}\sup_{r_{*}\in[R_{+},R_{+}^{(1)}]}\big(\big|h_{n}(r_{*})-h_{n-1}(r_{*})\big|+\big|h_{n}^{\prime}(r_{*})-h_{n-1}^{\prime}(r_{*})\big|\big),
\end{equation}
and thus (\ref{eq:TheRequiredLipschitzBound}) follows, provided $R_{+}$
has been fixed sufficiently large in terms of $\text{\textgreek{w}}_{R},m,l,M,a$
and the specific value of $\text{\textgreek{e}}_{\text{\textgreek{w}}_{R}ml}$.

We will now proceed to show (using (\ref{eq:TheRequiredLipschitzBound}))
that the sequence $\{h_{n}\}_{n\in\mathbb{N}}$ converges in $C^{1}(\mathbb{R})$
to a solution $h$ of (\ref{eq:NonLocalOde}) with the desired properties.
Provided $\text{\textgreek{e}}_{\text{\textgreek{w}}_{R}ml}$ was
chosen sufficiently small  in terms of $\text{\textgreek{w}}_{R},m,l,M,a$,
from (\ref{eq:GronwallBeforePicard}) and (\ref{eq:TheRequiredLipschitzBound})
we infer that for some $\text{\textgreek{j}}<1$ and any $n\ge2$:
\begin{equation}
\sup_{r_{*}\in[R_{+},R_{+}^{(1)}]}\big(\big|h_{n}(r_{*})-h_{n-1}(r_{*})\big|+\big|h_{n}^{\prime}(r_{*})-h_{n-1}^{\prime}(r_{*})\big|\big)<\text{\textgreek{j}}\sup_{r_{*}\in[R_{+},R_{+}^{(1)}]}\big(\big|h_{n-1}(r_{*})-h_{n-2}(r_{*})\big|+\big|h_{n}^{\prime}(r_{*})-h_{n-1}^{\prime}(r_{*})\big|\big).\label{eq:CauchySequencePotentialRegion}
\end{equation}
Then (\ref{eq:CauchySequencePotentialRegion}) implies that $\{h_{n}|_{[R_{+},R_{+}^{(1)}]}\}_{n\in\mathbb{N}}$
is a Cauchy sequence in $C^{1}\big([R_{+},R_{+}^{(1)}]\big)$, and
thus (\ref{eq:EstimateDifferenceEverywhere}) (in view of the fact
that $h_{n}\equiv1$ for $r_{*}\le R_{+}$) yields that $\{h_{n}\}_{n\in\mathbb{N}}$
is a Cauchy sequence in $C^{1}(\mathbb{R})$, converging to a function
$h\in C^{1}(\mathbb{R})$. Notice that we immediately obtain that
$h\in\mathcal{B}_{R_{\infty}}$, since $\mathcal{B}_{R_{\infty}}$
is a closed subset of $C^{1}(\mathbb{R})$. Furthermore, in view of
(\ref{eq:BoundPotentialDifference}) and the fact that $h_{n}$ converges
to $h$ in the $C^{1}$ norm, we obtain that 
\begin{equation}
\limsup_{n\rightarrow+\infty}\int_{\mathbb{R}}\big|V^{(h_{n},R_{+})}-V^{(h,R_{+})}\big|=0,
\end{equation}
and thus, from the expression (\ref{eq:ExpressionH_n}) we obtain
for any $r_{*}\ge R_{+}$: 
\begin{equation}
h(r_{*})=1+\int_{R_{+}}^{r_{*}}\Big(\frac{1}{r^{2}(\text{\textgreek{r}})}\big(\int_{R_{+}}^{\text{\textgreek{r}}}r^{2}(\text{\textgreek{sv}})V^{(h,R_{+})}(\text{\textgreek{sv}})h(\text{\textgreek{sv}})\, d\text{\textgreek{sv}}\big)\, d\text{\textgreek{r}}\Big).\label{eq:ExpressionH_n-1}
\end{equation}
Thus, $h$ is a smooth solution of equation (\ref{eq:NonLocalOde}),
satisfying $h\equiv1$ for $r_{*}\le R_{+}$ and 
\begin{equation}
h(r_{*})=C_{1}+C_{2}\int_{r_{*}}^{\infty}r^{-2}\, dr_{*}\label{eq:ExpressionHFarAway}
\end{equation}
 for $r_{*}\ge R_{+}^{(1)}$ and some constants $C_{1},C_{2}$ depending
on $\text{\textgreek{w}}_{R},m,l,M,a$ (in view of (\ref{eq:ExpressionFarAwayRegion})).
Furthermore, $h^{\prime}\ge0$ on all of $\mathbb{R}$, in view of
(\ref{eq:IncreasingH_n}).
\end{proof}
The following lemma will be used in the proof of Proposition \ref{prop:FromL^2ILEDtoActualILED}:
\begin{lem}
\label{lem:IledDeforemdMetric}Let $h,R_{+}$ be as in the statement
of Proposition \ref{prop:ModeDeformedMetric}. Provided that $a$
has been fixed sufficiently small in terms of $M$ (so that Proposition
\ref{Prop:ILEDNewSpacetime} applies on $(\mathcal{M}_{0},g_{M,a}^{(1)})$)
and $R_{+}$ has been fixed sufficiently large in terms of the geometry
of $(\mathcal{M}_{0},g_{M,a}^{(1)})$, there exists a constant $1\ll R_{-}<R_{+}$,
independent of $a,R_{+}$ and large in terms of the geometry of the
undeformed spacetime $(\mathcal{M}_{0},g_{M,a}^{(1)})$, such that
the following integrated local energy decay estimates hold for solutions
$\text{\textgreek{y}}$ to the inhomogeneous wave equation (\ref{eq:InhomogeneousWaveEquationDeformedSpacetime})
on $(\mathcal{M}_{0},g_{M,a}^{(h,R_{+})})$ with $\text{\textgreek{y}}$,
$F$ as in the statement of Proposition \ref{Prop:ILEDNewSpacetime}:
\begin{equation}
\begin{split}\int_{\mathcal{M}_{0}}\Big(\text{\textgreek{q}}_{r\neq3M}(r)\cdot r^{-2}J_{\text{\textgreek{m}}}^{N}(\text{\textgreek{y}})N^{\text{\textgreek{m}}} & +\big(1-\text{\textgreek{q}}_{r\neq3M}(r)\big)|\partial_{r}\text{\textgreek{y}}|^{2}+r^{-4}|\text{\textgreek{y}}|^{2}\Big)\, dg_{M,a}^{(h,R_{+})}\le\\
\le & C_{R_{-}}\mathcal{Z}[F,\text{\textgreek{y}};R_{-}]+C_{R_{-},R_{+}^{(1)}}\int_{\{R_{-}\le r_{*}\le R_{+}^{(1)}\}}\min\big\{|T\text{\textgreek{y}}|^{2},|\text{\textgreek{y}}|^{2}\big\}\, dg_{M,a}^{(h,R_{+})},
\end{split}
\label{eq:DegenerateIled-1}
\end{equation}
\begin{equation}
\begin{split}\int_{\mathcal{M}_{0}}\Big(r^{-2}J_{\text{\textgreek{m}}}^{N}(\text{\textgreek{y}})N^{\text{\textgreek{m}}}+ & r^{-4}|\text{\textgreek{y}}|^{2}\Big)\, dg_{M,a}^{(h,R_{+})}\le\\
\le & C_{R_{-}}\mathcal{Z}[F,\text{\textgreek{y}};R_{-}]+C_{R_{-}}\mathcal{Z}[TF,T\text{\textgreek{y}};R_{-}]+C_{R_{-},R_{+}^{(1)}}\int_{\{R_{-}\le r_{*}\le R_{+}^{(1)}\}}\min\big\{|T\text{\textgreek{y}}|^{2},|\text{\textgreek{y}}|^{2}\big\}\, dg_{M,a}^{(h,R_{+})},
\end{split}
\label{eq:NonDegenerateIled-1}
\end{equation}
 where $\mathcal{Z}[F,\text{\textgreek{y}};R_{-}]$ is defined as
in the statement of Proposition \ref{Prop:ILEDNewSpacetime} and $R_{+}^{(1)}$
is defined in terms of $R_{+},h$ as in the proof of Proposition \ref{prop:ModeDeformedMetric}.
The constants $C_{R_{-}}$ and $C_{R_{-},R_{+}^{(1)}}$ in the right
hand side of $\eqref{eq:DegenerateIled-1}$ and (\ref{eq:NonDegenerateIled-1})
depend only on the precise choice of $R_{-}$ and $R_{-},R_{+}^{(1)}$
respectively, as well as the geometry of the spacetime $(\mathcal{M}_{0},g_{M,a}^{(1)})$.\end{lem}
\begin{proof}
Let $R_{1}^{*}>0$ be any constant sufficiently large in terms of
the geometry of $(\mathcal{M}_{0},g_{M,a}^{(1)})$ as in the statement
of Proposition \ref{Prop:ILEDNewSpacetime}, assuming without loss
of generality that $R_{+}$ is sufficiently large in terms of the
geometry of $(\mathcal{M}_{0},g_{M,a}^{(1)})$ so that $R_{+}\gg(R_{1}^{*})^{2}$.
Let also $R_{+}^{(1)}$ be defined as in the proof of Proposition
\ref{prop:ModeDeformedMetric}. 

For any $(\text{\textgreek{w}},m,l)\in\mathbb{R}\times\mathbb{Z}\times\mathbb{Z}_{\ge|m|}$,
the frequency separated equation (\ref{eq:SeperatedInhomogeneousEquation})
is the same as (\ref{eq:ODEForTheFourierTransformNewSpacetime}) in
the region $\{r_{*}\le(R_{1}^{*})^{2}\}$ (since $h\equiv1$ there).
Furthermore, choosing $\text{\textgreek{e}}_{1}=(R_{1}^{*})^{-1}$
in the proof of Proposition \ref{Prop:ILEDNewSpacetime}, if $f_{\text{\textgreek{w}}ml}$
and $h_{\text{\textgreek{w}}ml}$ are the seed functions appearing
in the proof of Proposition \ref{Prop:ILEDNewSpacetime}, then 
\begin{equation}
f_{\text{\textgreek{w}}ml}^{\prime}=h_{\text{\textgreek{w}}ml}=0
\end{equation}
 in the region $\{r_{*}\ge(R_{1}^{*})^{2}\}$. Therefore, by repeating
the proof of Proposition \ref{Prop:ILEDNewSpacetime} for $\text{\textgreek{e}}_{1}=(R_{1}^{*})^{-1}$,
using exactly the same functions $f_{\text{\textgreek{w}}ml}$ and
$h_{\text{\textgreek{w}}ml}$ as in that proof, we readily obtain
the following estimates for the frequency separated equation (\ref{eq:SeperatedInhomogeneousEquation})
for any $(\text{\textgreek{w}},m,l)\in\mathbb{R}\times\mathbb{Z}\times\mathbb{Z}_{\ge|m|}$
(assuming that the limits $\lim_{r_{*}\rightarrow+\infty}\big(e^{-i\text{\textgreek{w}}r_{*}}u_{\text{\textgreek{w}}ml}\big)$
and $\lim_{r_{*}\rightarrow-\infty}\big(e^{i(\text{\textgreek{w}}-\frac{am}{8M^{2}})r_{*}}u_{\text{\textgreek{w}}ml})$
exist): 
\begin{align}
\int_{-R_{1}^{*}}^{R_{1}^{*}}\big(r^{-2}|u_{\text{\textgreek{w}}ml}^{\prime}|^{2}+ & r^{-2}(\text{\textgreek{w}}^{2}+l^{2}+r^{-2})|u_{\text{\textgreek{w}}ml}|^{2}\big)\, dr_{*}\le\label{eq:NonTrappingILED-2}\\
\le C_{R_{1}^{*}} & \text{\textgreek{d}}_{0}m^{2}|u_{\text{\textgreek{w}}ml}(-\infty)|^{2}+C\int_{R_{1}^{*}}^{(R_{1}^{*})^{2}}\big((R_{1}^{*})^{-1}r^{-2}+r^{-3}\big)\text{\textgreek{w}}^{2}|u_{\text{\textgreek{w}}ml}|^{2}+C_{R_{+},R_{+}^{(1)}}\int_{R_{+}}^{R_{+}^{(1)}}|u_{\text{\textgreek{w}}ml}|^{2}+\nonumber \\
 & +C_{R_{1}^{*}}\int_{-\infty}^{+\infty}Re\big\{(r-2M)hF_{\text{\textgreek{w}}ml}\cdot\big(f_{\text{\textgreek{w}}ml}\bar{u}_{\text{\textgreek{w}}ml}^{\prime}+(r^{-1}h_{\text{\textgreek{w}}ml}+i\text{\textgreek{w}})\bar{u}_{\text{\textgreek{w}}ml}\big)\big\}\, dr_{*},\nonumber 
\end{align}
for $|\text{\textgreek{w}}|\gg l$ or $|\text{\textgreek{w}}|\ll l$,
and 
\begin{align}
\int_{-R_{1}^{*}}^{R_{1}^{*}}\big(r^{-2}|u_{\text{\textgreek{w}}ml}^{\prime}|^{2}+ & r^{-2}(\{(1-\frac{r_{\text{\textgreek{w}}ml}}{r})^{2}(\text{\textgreek{w}}^{2}+l^{2})+r^{-2}\}|u_{\text{\textgreek{w}}ml}|^{2}\big)\, dr_{*}\le\label{eq:TrappingILED-1}\\
\le C_{R_{1}^{*}} & \text{\textgreek{d}}_{0}m^{2}|u_{\text{\textgreek{w}}ml}(-\infty)|^{2}+C\int_{R_{1}^{*}}^{(R_{1}^{*})^{2}}\big((R_{1}^{*})^{-1}r^{-2}+r^{-3}\big)\text{\textgreek{w}}^{2}|u_{\text{\textgreek{w}}ml}|^{2}+C_{R_{+},R_{+}^{(1)}}\int_{R_{+}}^{R_{+}^{(1)}}|u_{\text{\textgreek{w}}ml}|^{2}+\nonumber \\
 & +C_{R_{1}^{*}}\int_{-\infty}^{+\infty}Re\big\{(r-2M)hF_{\text{\textgreek{w}}ml}\cdot\big(f_{\text{\textgreek{w}}ml}\bar{u}_{\text{\textgreek{w}}ml}^{\prime}+(r^{-1}h_{\text{\textgreek{w}}ml}+i\text{\textgreek{w}})\bar{u}_{\text{\textgreek{w}}ml}\big)\big\}\, dr_{*},\nonumber 
\end{align}
for $|\text{\textgreek{w}}|\sim l$. Notice that in obtaining (\ref{eq:NonTrappingILED-2})
and (\ref{eq:TrappingILED-1}), we used the following one sided bound
for the derivative of $V_{\text{\textgreek{w}}ml;h}^{(h)}$ in the
region $\{r_{*}\ge(R_{1}^{*})^{2}\}$ (following from the non-negativity
condition $h^{\prime}\ge0$), provided $R_{1}^{*}$ is sufficiently
large in terms of the geometry of $(\mathcal{M}_{0},g_{M,a}^{(1)})$:
\begin{equation}
\big(V_{\text{\textgreek{w}}ml;h}^{(h)}\big)^{\prime}\ge\begin{cases}
0, & r_{*}\in[(R_{1}^{*})^{2},R_{+}]\cup[R_{+}^{(1)},+\infty)\\
-C_{R_{+},R_{+}^{(1)}}, & r_{*}\in[R_{+},R_{+}^{(1)}].
\end{cases}\label{eq:BoundDerivativeVh}
\end{equation}
Observe that the right hand side of (\ref{eq:BoundDerivativeVh})
is independent of the frequency parameters $(\text{\textgreek{w}},m,l)$.

From (\ref{eq:NonTrappingILED-2}) and (\ref{eq:TrappingILED-1})
we obtain after summing in $m,l$, integrating in $\text{\textgreek{w}}$,
and using the red shift type estimates of Section 7 of \cite{DafRod6}
in the region $\{r_{*}\le-R_{1}^{*}\}$:
\begin{equation}
\begin{split}\int_{\{r_{*}\le R_{1}^{*}\}}\Big(\text{\textgreek{q}}_{r\neq3M}(r)\cdot r^{-2}J_{\text{\textgreek{m}}}^{N}(\text{\textgreek{y}})N^{\text{\textgreek{m}}}+ & \big(1-\text{\textgreek{q}}_{r\neq3M}(r)\big)|\partial_{r}\text{\textgreek{y}}|^{2}+r^{-4}|\text{\textgreek{y}}|^{2}\Big)\, dg_{M,a}^{(1)}\le\\
\le C_{R_{1}^{*}} & \mathcal{Z}[F,\text{\textgreek{y}};R_{1}^{*}]+C\int_{\{R_{1}^{*}\le r_{*}\le(R_{1}^{*})^{2}\}}\big((R_{1}^{*})^{-1}r^{-2}+r^{-3}\big)|T\text{\textgreek{y}}|^{2}\, dg_{M,a}^{(1)}+\\
 & +C_{R_{+},R_{+}^{(1)}}\int_{\{R_{+}\le r_{*}\le R_{+}^{(1)}\}}|\text{\textgreek{y}}|^{2}\, dg_{M,a}^{(h,R_{+})}
\end{split}
\label{eq:DegenerateIled-1-1}
\end{equation}
and 
\begin{equation}
\begin{split}\int_{\{r_{*}\le R_{1}^{*}\}}\Big(r^{-2} & J_{\text{\textgreek{m}}}^{N}(\text{\textgreek{y}})N^{\text{\textgreek{m}}}+r^{-4}|\text{\textgreek{y}}|^{2}\Big)\, dg_{M,a}^{(1)}\le\\
\le C_{R_{1}^{*}} & \Big(\mathcal{Z}[F,\text{\textgreek{y}};R_{1}^{*}]+\mathcal{Z}[TF,T\text{\textgreek{y}};R_{1}^{*}]\Big)+C\int_{\{R_{1}^{*}\le r_{*}\le(R_{1}^{*})^{2}\}}\big((R_{1}^{*})^{-1}r^{-2}+r^{-3}\big)|T\text{\textgreek{y}}|^{2}\, dg_{M,a}^{(1)}+\\
 & +C_{R_{+},R_{+}^{(1)}}\int_{\{R_{+}\le r_{*}\le R_{+}^{(1)}\}}|\text{\textgreek{y}}|^{2}\, dg_{M,a}^{(h,R_{+})}.
\end{split}
\label{eq:NonDegenerateIled-1-1}
\end{equation}

In the $(t,r_{*},\text{\textgreek{j}},\text{\textgreek{f}})$ coordinate
chart, equation (\ref{eq:InhomogeneousWaveEquationDeformedSpacetime})
reads: 
\begin{align}
h\cdot(r-2M)F=-\partial_{t}^{2} & \text{\textgreek{Y}}+\partial_{r_{*}}^{2}\text{\textgreek{Y}}+\big(1-\frac{2M}{r}\big)h^{-2}r^{-2}\text{\textgreek{D}}_{\mathbb{S}^{2}}\text{\textgreek{Y}}+2aMr^{-3}h^{-1}\partial_{t}\partial_{\text{\textgreek{f}}}\text{\textgreek{Y}}-\label{eq:RescaledInhomogeneousEquation}\\
 & -a^{2}M^{2}r^{-6}h^{-2}\partial_{\text{\textgreek{f}}}^{2}\text{\textgreek{Y}}+\Big(-\big(1-\frac{2M}{r}\big)\frac{2M}{r^{3}}+\frac{h^{\prime\prime}(r)}{h(r)}+2\big(1-\frac{2M}{r}\big)\frac{h^{\prime}(r)}{r\cdot h(r)}\Big)\text{\textgreek{Y}},\nonumber 
\end{align}
where
\begin{equation}
\text{\textgreek{Y}}\doteq h\cdot r\text{\textgreek{y}}.
\end{equation}
Let $\text{\textgreek{q}}_{1}:\mathbb{R}\rightarrow[0,1]$ be a smooth
function satisfying $\text{\textgreek{q}}\equiv0$ on $(-\infty,\frac{1}{2}]$
and $\text{\textgreek{q}}\equiv1$ on $[1,+\infty)$, and let us set
$\text{\textgreek{q}}_{R_{1}^{*}}(r_{*})\doteq\text{\textgreek{q}}_{1}\big(\frac{r_{*}}{R_{1}^{*}}\big)$.
Multiplying (\ref{eq:RescaledInhomogeneousEquation}) with 
\begin{equation}
\text{\textgreek{q}}_{R_{1}^{*}}(r_{*})\cdot\frac{r^{\text{\textgreek{h}}}}{1+r^{\text{\textgreek{h}}}}\partial_{r_{*}}\bar{\text{\textgreek{Y}}}
\end{equation}
 for some $0<\text{\textgreek{h}}<1$ and integrating by parts over
$\mathcal{M}_{0}$, we obtain, provided $R_{1}^{*}$ is sufficiently
large in terms of $a,M$ (using also the fact that $h\ge1$): 
\begin{equation}
\begin{split}\int_{-\infty}^{+\infty}\int_{-\infty}^{+\infty}\int_{0}^{\text{\textgreek{p}}}\int_{0}^{2\text{\textgreek{p}}}\frac{1}{2}\text{\textgreek{q}}_{R_{1}^{*}}(r_{*})\Bigg\{\big( & \text{\textgreek{h}}r^{-1-\text{\textgreek{h}}}+O_{\text{\textgreek{h}}}(r^{-2-\text{\textgreek{h}}})\big)\big(|\partial_{r_{*}}\text{\textgreek{Y}}|^{2}+(1+O_{\text{\textgreek{h}}}(r^{-1}))\big|\partial_{t}\text{\textgreek{Y}}\big|^{2}\big)+\big(6Mr^{-4}+O_{\text{\textgreek{h}}}(r^{-5})\big)\big|\text{\textgreek{Y}}|^{2}+\\
+2\big(h^{\prime}h^{-3}r^{-2}+ & h^{-2}r^{-3}+O_{\text{\textgreek{h}}}(r^{-4})\big)\big(\big|\partial_{\text{\textgreek{j}}}\text{\textgreek{Y}}\big|^{2}+\frac{1+O(r^{-1})}{\sin^{2}\text{\textgreek{j}}}\big|\partial_{\text{\textgreek{f}}}\text{\textgreek{Y}}\big|^{2}\big)\Bigg\}\,\sin\text{\textgreek{j}}d\text{\textgreek{f}}d\text{\textgreek{j}}dr_{*}dt=\\
=-\int_{-\infty}^{+\infty}\int_{-\infty}^{+\infty}\int_{0}^{\text{\textgreek{p}}}\int_{0}^{2\text{\textgreek{p}}}\Bigg\{ & \text{\textgreek{q}}_{R_{1}^{*}}(r_{*})\cdot\frac{r^{\text{\textgreek{h}}}}{1+r^{\text{\textgreek{h}}}}Re\big\{ F\cdot\partial_{r_{*}}\bar{\text{\textgreek{Y}}}\big\}+\text{\textgreek{q}}_{R_{1}^{*}}^{\prime}(r_{*})\Big(O(1)\big|\partial\text{\textgreek{Y}}\big|^{2}+O(r^{-2})\big|\text{\textgreek{Y}}\big|^{2}\Big)+\\
 & +\frac{1}{2}\Big(\frac{h^{\prime\prime}(r)}{h(r)}+2\big(1-\frac{2M}{r}\big)\frac{h^{\prime}(r)}{r\cdot h(r)}\Big)^{\prime}\big|\text{\textgreek{Y}}\big|^{2}\Bigg\}\,\sin\text{\textgreek{j}}d\text{\textgreek{f}}d\text{\textgreek{j}}dr_{*}dt,
\end{split}
\label{eq:ILEDfarAway}
\end{equation}
where 
\begin{equation}
\big|\partial\text{\textgreek{Y}}\big|^{2}\doteq\big|\partial_{r_{*}}\text{\textgreek{Y}}|^{2}+\big|\partial_{t}\text{\textgreek{Y}}\big|^{2}+\big|\partial_{\text{\textgreek{j}}}\text{\textgreek{Y}}\big|^{2}+\sin^{-2}\text{\textgreek{j}}\big|\partial_{\text{\textgreek{f}}}\text{\textgreek{Y}}\big|^{2}.
\end{equation}
Thus, in view of the fact that $h^{\prime}\ge0$ on $\mathbb{R}$
and 
\[
\frac{h^{\prime\prime}(r)}{h(r)}+2\big(1-\frac{2M}{r}\big)\frac{h^{\prime}(r)}{r\cdot h(r)}=0
\]
 for $r_{*}\notin[R_{+},R_{+}^{(1)}]$, from (\ref{eq:ILEDfarAway})
we obtain (provided $R_{1}^{*}$ is sufficiently large in terms of
$\text{\textgreek{h}}$): 
\begin{equation}
\int_{\{r_{*}\ge R_{1}^{*}\}}r^{-1-\text{\textgreek{h}}}\big|T\text{\textgreek{y}}|^{2}\, dg_{M,a}^{(h,R_{+})}\le C_{\text{\textgreek{h}}}\Big(\int_{\frac{1}{2}R_{1}^{*}}^{R_{1}^{*}}\big(r^{-1}|\partial\text{\textgreek{y}}|^{2}+r^{-3}|\text{\textgreek{y}}|^{2}\big)\, dg_{M,a}^{(1)}+\int_{\{R_{+}\le r_{*}\le R_{+}^{(1)}\}}|\text{\textgreek{y}}|^{2}\, dg_{M,a}^{(h,R_{+})}+\mathcal{Z}[F,\text{\textgreek{y}};R_{1}^{*}]\Big).\label{eq:IledToUseForLowerOrderterms}
\end{equation}

From (\ref{eq:DegenerateIled-1-1}), (\ref{eq:NonDegenerateIled-1-1})
and (\ref{eq:IledToUseForLowerOrderterms}), we readily obtain: 
\begin{equation}
\begin{split}\int_{\mathcal{M}_{0}}\Big(\text{\textgreek{q}}_{r\neq3M}(r)\cdot r^{-2}J_{\text{\textgreek{m}}}^{N}(\text{\textgreek{y}})N^{\text{\textgreek{m}}}+ & \big(1-\text{\textgreek{q}}_{r\neq3M}(r)\big)|\partial_{r}\text{\textgreek{y}}|^{2}+r^{-4}|\text{\textgreek{y}}|^{2}\Big)\, dg_{M,a}^{(h,R_{+})}\le\\
\le & C_{R_{1}^{*}}\mathcal{Z}[F,\text{\textgreek{y}};R_{1}^{*}]+C_{R_{1}^{*},R_{+}^{(1)}}\int_{\{\frac{1}{2}R_{1}^{*}\le r_{*}\le R_{+}^{(1)}\}}|\text{\textgreek{y}}|^{2}\, dg_{M,a}^{(h,R_{+})}
\end{split}
\label{eq:DegenerateIled-1-1-2}
\end{equation}
and 
\begin{equation}
\begin{split}\int_{\mathcal{M}_{0}}\Big( & r^{-2}J_{\text{\textgreek{m}}}^{N}(\text{\textgreek{y}})N^{\text{\textgreek{m}}}+r^{-4}|\text{\textgreek{y}}|^{2}\Big)\, dg_{M,a}^{(h,R_{+})}\le\\
\le & C_{R_{1}^{*}}\Big(\mathcal{Z}[F,\text{\textgreek{y}};R_{1}^{*}]+\mathcal{Z}[TF,T\text{\textgreek{y}};R_{1}^{*}]\Big)+C_{R_{1}^{*},R_{+}^{(1)}}\int_{\{\frac{1}{2}R_{1}^{*}\le r_{*}\le R_{+}^{(1)}\}}|\text{\textgreek{y}}|^{2}\, dg_{M,a}^{(h,R_{+})}.
\end{split}
\label{eq:NonDegenerateIled-1-1-2}
\end{equation}
 Therefore, inequalities (\ref{eq:DegenerateIled-1}) and (\ref{eq:NonDegenerateIled-1})
readily follow by combining (\ref{eq:DegenerateIled-1-1}), (\ref{eq:NonDegenerateIled-1-1}),
(\ref{eq:DegenerateIled-1-1-2}) and (\ref{eq:NonDegenerateIled-1-1-2}).
\end{proof}
The statement \ref{enu:ILED} will now follow as a consequence of
the following proposition:
\begin{prop}
\label{prop:FromL^2ILEDtoActualILED}Let $h,R_{+},R_{-}$ be as in
the statement of Lemma \ref{lem:IledDeforemdMetric}. Provided that
$a$ has been fixed sufficiently small in terms of $M$, the integrated
local energy decay estimates (\ref{eq:DegenerateIled-1-2}) and (\ref{eq:NonDegenerateIled-1-2})
hold for any $\text{\textgreek{t}}_{1}\le\text{\textgreek{t}}_{2}$
and any solution $\text{\textgreek{y}}$ to the inhomogeneous wave
equation (\ref{eq:InhomogeneousWaveEquationDeformedSpacetime}) on
$(\mathcal{M}_{0},g_{M,a}^{(h,R_{+})})$ which is smooth up to $\mathcal{H}^{+}$.
In particular, the trapped set of $(\mathcal{M}_{0},g_{M,a}^{(h)})$
is normally hyperbolic.\end{prop}
\begin{proof}
Let $\text{\textgreek{q}}_{hor}:(2M,+\infty)\rightarrow[0,1]$ be
a smooth function such that $\text{\textgreek{q}}_{hor}\equiv1$ on
$(2M,\frac{9M}{4}]$ and $\text{\textgreek{q}}_{hor}\equiv0$ on $[\frac{10M}{4},+\infty)$,
and let us introduce the vector field 
\begin{equation}
\tilde{K}\doteq T+\text{\textgreek{q}}_{hor}\cdot\frac{a}{8M^{2}}\text{\textgreek{F}}.\label{eq:PerturbedHawkingField}
\end{equation}
Notice that $\tilde{K}\equiv K$ for $\{r\le\frac{9M}{4}\}$ and $\tilde{K}\equiv T$
for $\{r\ge\frac{10M}{4}\}$, and furthermore $\tilde{K}$ is everywhere
future directed and timelike on $\mathcal{M}_{M,0}$ (but merely null
on $\mathcal{H}^{+}$), provided $a$ is sufficiently small. In particular,
$\tilde{K}$ is a Killing vector field on $\{r\le\frac{9M}{4}\}\cup\{r\ge\frac{10M}{4}\}$,
while on $\{\frac{9M}{4}\le r\le\frac{10M}{4}\}$ we can bound for
any fixed and $T$-inveriant reference Riemannian metric $g_{Rm}$
on $\mathcal{M}_{0}$ (provided $a$ is sufficiently small): 
\begin{equation}
\Big|{}^{(\tilde{K})}\text{\textgreek{p}}_{g_{M,a}^{(h,R_{+})}}\Big|_{g_{Rm}}\le C_{g_{Rm}}a,\label{eq:SmallnessDeformationTensor}
\end{equation}
where $^{(\tilde{K})}\text{\textgreek{p}}_{g_{M,a}^{(h,R_{+})}}$
is the deformation tensor of $\tilde{K}$ with respect to $g_{M,a}^{(h,R_{+})}$.
In view of the energy identity for the vector field $\tilde{K}$ (and
the fact that $\tilde{K}$ is everywhere causal), the bound (\ref{eq:SmallnessDeformationTensor})
implies that for any $\text{\textgreek{t}}_{1}\le\text{\textgreek{t}}_{2}$
we can bound (omitting the volume form notation for simplicity) 
\begin{equation}
\sup_{\text{\textgreek{t}}_{1}\le\text{\textgreek{t}}\le\text{\textgreek{t}}_{2}}\int_{\{\bar{t}=\text{\textgreek{t}}\}}J_{\text{\textgreek{m}}}^{\tilde{K}}(\text{\textgreek{y}})\bar{n}^{\text{\textgreek{m}}}\le\int_{\{\bar{t}=\text{\textgreek{t}}_{1}\}}J_{\text{\textgreek{m}}}^{\tilde{K}}(\text{\textgreek{y}})\bar{n}^{\text{\textgreek{m}}}+Ca\int_{\{\frac{9M}{4}\le r\le\frac{10M}{4}\}\cap\{\text{\textgreek{t}}_{1}\le\bar{t}\le\text{\textgreek{t}}_{2}\}}J_{\text{\textgreek{m}}}^{\tilde{K}}(\text{\textgreek{y}})\tilde{K}^{\text{\textgreek{m}}}+C\int_{\{\text{\textgreek{t}}_{1}\le\bar{t}\le\text{\textgreek{t}}_{2}\}}|\tilde{K}\text{\textgreek{y}}||F|.\label{eq:DegenerateEnergyIdentity}
\end{equation}

Let $N$ be a $T$-invariant, smooth and everywhere timelike on $\mathcal{M}_{0}\cup\mathcal{H}^{+}$,
such that $N\equiv T$ on $\{r\ge\frac{10M}{4}\}$, and furthermore
satisfying on $\{r\le2M+c\}$ for some small $c>0$ depending only
on the geometry of $(\mathcal{M}_{0},g_{M,a}^{(1)})$ (and independent
of $a$, provided $a$ is sufficiently small): 
\begin{equation}
K^{N}(\text{\textgreek{y}})\ge cJ_{\text{\textgreek{m}}}^{N}(\text{\textgreek{y}})N^{\text{\textgreek{m}}}.\label{eq:RedShiftEstimate}
\end{equation}
Such a vector field $N$ can always be constructed, in view of the
fact that the surface gravity of $\mathcal{H}^{+}$ is positive; see
Section 7 of \cite{DafRod6}. Notice also that $N$ can be chosen
so that in the region $\{2M+c\le r\le\frac{10M}{4}\}$ the following
bound holds: 
\begin{equation}
\big|K^{N}(\text{\textgreek{y}})\big|\le CJ_{\text{\textgreek{m}}}^{N}(\text{\textgreek{y}})N^{\text{\textgreek{m}}}
\end{equation}
for some constant $C$ independent of $a$ (provided $a$ is sufficiently
small), while in the region $\{r\ge\frac{10M}{4}\}$, we have $K^{N}=K^{T}=0$.
Thus, the energy identity for $N$ yields for any $\text{\textgreek{t}}_{1}\le\text{\textgreek{t}}_{2}$:
\begin{equation}
\begin{split}\sup_{\text{\textgreek{t}}_{1}\le\text{\textgreek{t}}\le\text{\textgreek{t}}_{2}}\int_{\{\bar{t}=\text{\textgreek{t}}\}}J_{\text{\textgreek{m}}}^{N}(\text{\textgreek{y}})\bar{n}^{\text{\textgreek{m}}} & +c\int_{\{\text{\textgreek{t}}_{1}\le\bar{t}\le\text{\textgreek{t}}_{2}\}\cap\{r\le2M+c\}}J_{\text{\textgreek{m}}}^{N}(\text{\textgreek{y}})N^{\text{\textgreek{m}}}\le\\
\le & \int_{\{\bar{t}=\text{\textgreek{t}}_{1}\}}J_{\text{\textgreek{m}}}^{N}(\text{\textgreek{y}})\bar{n}^{\text{\textgreek{m}}}+C\int_{\{2M+c\le r\le\frac{10M}{4}\}\cap\{\text{\textgreek{t}}_{1}\le\bar{t}\le\text{\textgreek{t}}_{2}\}}J_{\text{\textgreek{m}}}^{N}(\text{\textgreek{y}})n^{\text{\textgreek{m}}}+C\int_{\{\text{\textgreek{t}}_{1}\le\bar{t}\le\text{\textgreek{t}}_{2}\}}|N\text{\textgreek{y}}||F|.
\end{split}
\label{eq:DegenerateEnergyIdentity-1}
\end{equation}

Using (\ref{eq:DegenerateEnergyIdentity}) (integrated in $\bar{t}$)
to control the second term of the right hand side of (\ref{eq:DegenerateEnergyIdentity-1}),
we readily obtain for any $\text{\textgreek{t}}_{1}\le\text{\textgreek{t}}_{2}$:
\begin{equation}
\sup_{\text{\textgreek{t}}_{1}\le\text{\textgreek{t}}\le\text{\textgreek{t}}_{2}}E_{N}[\text{\textgreek{y}}](\text{\textgreek{t}})+c\int_{\text{\textgreek{t}}_{1}}^{\text{\textgreek{t}}_{2}}E_{N}[\text{\textgreek{y}}](\text{\textgreek{t}})\, d\text{\textgreek{t}}\le E_{N}[\text{\textgreek{y}}](\text{\textgreek{t}}_{1})+C\cdot(\text{\textgreek{t}}_{2}-\text{\textgreek{t}}_{1})E_{N}[\text{\textgreek{y}}](\text{\textgreek{t}}_{1})+Ca\int_{\text{\textgreek{t}}_{1}}^{\text{\textgreek{t}}_{2}}\mathcal{G}[\text{\textgreek{y}}](\text{\textgreek{t}}_{1};\text{\textgreek{t}})\, d\text{\textgreek{t}}+C\mathcal{F}[F,\text{\textgreek{t}}_{2},\text{\textgreek{t}}_{1}],\label{eq:DegenerateEnergyIdentity-1-1}
\end{equation}
where 
\begin{gather*}
E_{N}[\text{\textgreek{y}}](\text{\textgreek{t}})\doteq\int_{\{\bar{t}=\text{\textgreek{t}}\}}J_{\text{\textgreek{m}}}^{N}(\text{\textgreek{y}})\bar{n}^{\text{\textgreek{m}}},\\
\mathcal{G}[\text{\textgreek{y}}](\text{\textgreek{t}}_{1};\text{\textgreek{t}})\doteq\int_{\{\frac{9M}{4}\le r\le\frac{10M}{4}\}\cap\{\text{\textgreek{t}}_{1}\le\bar{t}\le\text{\textgreek{t}}\}}J_{\text{\textgreek{m}}}^{\tilde{K}}(\text{\textgreek{y}})\tilde{K}^{\text{\textgreek{m}}},\\
\mathcal{F}[F,\text{\textgreek{t}}_{2},\text{\textgreek{t}}_{1}]\doteq\int_{\{\text{\textgreek{t}}_{1}\le\bar{t}\le\text{\textgreek{t}}_{2}\}}|N\text{\textgreek{y}}||F|+\int_{\text{\textgreek{t}}_{1}}^{\text{\textgreek{t}}_{1}}\Big(\int_{\{\text{\textgreek{t}}_{1}\le\bar{t}\le\text{\textgreek{t}}\}}|\tilde{K}\text{\textgreek{y}}||F|\Big)\, d\text{\textgreek{t}}.
\end{gather*}
From (\ref{eq:DegenerateEnergyIdentity-1-1}) and Gronwall's inequality,
we thus infer for any $\text{\textgreek{t}}_{1}\le\text{\textgreek{t}}_{2}$:
\begin{equation}
\sup_{\text{\textgreek{t}}_{1}\le\text{\textgreek{t}}\le\text{\textgreek{t}}_{2}}\int_{\{\bar{t}=\text{\textgreek{t}}\}}J_{\text{\textgreek{m}}}^{N}(\text{\textgreek{y}})\bar{n}^{\text{\textgreek{m}}}\le C\int_{\{\bar{t}=\text{\textgreek{t}}_{1}\}}J_{\text{\textgreek{m}}}^{N}(\text{\textgreek{y}})\bar{n}^{\text{\textgreek{m}}}+Ca\int_{\{\frac{9M}{4}\le r\le\frac{10M}{4}\}\cap\{\text{\textgreek{t}}_{1}\le\bar{t}\le\text{\textgreek{t}}_{2}\}}J_{\text{\textgreek{m}}}^{\tilde{K}}(\text{\textgreek{y}})\tilde{K}^{\text{\textgreek{m}}}+C\int_{\{\text{\textgreek{t}}_{1}\le\bar{t}\le\text{\textgreek{t}}_{2}\}}\big(|\tilde{K}\text{\textgreek{y}}|+|N\text{\textgreek{y}}|\big)\cdot|F|.\label{eq:BoundednessEstimate}
\end{equation}

Let $\text{\textgreek{q}}:\mathbb{R}\rightarrow[0,1]$ be a smooth
cut-off function, such that $\text{\textgreek{q}}\equiv1$ on $(-\infty,-1]$
and $\text{\textgreek{q}}\equiv0$ on $[0,+\infty)$, and let us define
for any $\text{\textgreek{t}}_{1}\le\text{\textgreek{t}}_{2}$ the
function $\text{\textgreek{q}}_{\text{\textgreek{t}}_{1},\text{\textgreek{t}}_{2}}:\mathcal{M}_{M,0}\rightarrow[0,1]$:
\begin{equation}
\text{\textgreek{q}}_{\text{\textgreek{t}}_{1},\text{\textgreek{t}}_{2}}\doteq\text{\textgreek{q}}(\bar{t}-\text{\textgreek{t}}_{2})\text{\textgreek{q}}(\text{\textgreek{t}}_{1}-\bar{t}).
\end{equation}
Notice that $\text{\textgreek{q}}_{\text{\textgreek{t}}_{1},\text{\textgreek{t}}_{2}}\text{\textgreek{y}}$
satisfies 
\begin{equation}
\square_{g_{M,a}^{(h,R_{+})}}(\text{\textgreek{q}}_{\text{\textgreek{t}}_{1},\text{\textgreek{t}}_{2}}\text{\textgreek{y}})=\text{\textgreek{q}}_{\text{\textgreek{t}}_{1},\text{\textgreek{t}}_{2}}F+2\nabla^{\text{\textgreek{m}}}\text{\textgreek{q}}_{\text{\textgreek{t}}_{1},\text{\textgreek{t}}_{2}}\nabla_{\text{\textgreek{m}}}\text{\textgreek{y}}+\big(\square_{g_{M,a}^{(h,R_{+})}}\text{\textgreek{q}}_{\text{\textgreek{t}}_{1},\text{\textgreek{t}}_{2}}\big)\cdot\text{\textgreek{y}},\label{eq:NewInhomogeneousTerm}
\end{equation}
and both $\text{\textgreek{q}}_{\text{\textgreek{t}}_{1},\text{\textgreek{t}}_{2}}\text{\textgreek{y}}$
and $\square_{g_{M,a}^{(h,R_{+})}}(\text{\textgreek{q}}_{\text{\textgreek{t}}_{1},\text{\textgreek{t}}_{2}}\text{\textgreek{y}})$
have compact support in the $t^{*}$ variable. Hence, applying Lemma
\ref{lem:IledDeforemdMetric} for $\text{\textgreek{q}}_{\text{\textgreek{t}}_{1},\text{\textgreek{t}}_{2}}\text{\textgreek{y}}$
in place of $\text{\textgreek{y}}$, we readily obtain
\begin{equation}
\begin{split}\int_{\{\text{\textgreek{t}}_{1}\le\bar{t}\le\text{\textgreek{t}}_{2}\}}\Big(\text{\textgreek{q}}_{r\neq3M}(r) & \cdot r^{-2}J_{\text{\textgreek{m}}}^{N}(\text{\textgreek{y}})N^{\text{\textgreek{m}}}+r^{-4}|\text{\textgreek{y}}|^{2}\Big)\le\\
\le & C_{R_{-}}\mathcal{Z}_{\text{\textgreek{t}}_{1},\text{\textgreek{t}}_{2}}[F,\text{\textgreek{y}};R_{-}]+C_{R_{-}}\int_{\{\bar{t}=\text{\textgreek{t}}_{1}\}}J_{\text{\textgreek{m}}}^{N}(\text{\textgreek{y}})\bar{n}^{\text{\textgreek{m}}}+C_{R_{-}}\int_{\big\{\bar{t}=\max\{\text{\textgreek{t}}_{1},\text{\textgreek{t}}_{2}-1\}\big\}}J_{\text{\textgreek{m}}}^{N}(\text{\textgreek{y}})\bar{n}^{\text{\textgreek{m}}}+\\
 & \hphantom{C_{R_{-}}}+C_{R_{-},R_{+}^{(1)}}\int_{\{R_{-}\le r_{*}\le R_{+}^{(1)}\}\cap\{\text{\textgreek{t}}_{1}\le\bar{t}\le\text{\textgreek{t}}_{2}\}}\min\big\{|T\text{\textgreek{y}}|^{2},|\text{\textgreek{y}}|^{2}\big\}.
\end{split}
\label{eq:Cut-Off-ILED}
\end{equation}

\medskip{}

\noindent \emph{Remark.} Notice that, in obtaining (\ref{eq:Cut-Off-ILED})
from (\ref{eq:DegenerateIled-1}) and (\ref{eq:NewInhomogeneousTerm}),
the following estimate is used: 
\begin{equation}
\mathcal{Z}[2\nabla^{\text{\textgreek{m}}}\text{\textgreek{q}}_{\text{\textgreek{t}}_{1},\text{\textgreek{t}}_{2}}\nabla_{\text{\textgreek{m}}}\text{\textgreek{y}}+\big(\square_{g_{M,a}^{(h,R_{+})}}\text{\textgreek{q}}_{\text{\textgreek{t}}_{1},\text{\textgreek{t}}_{2}}\big)\cdot\text{\textgreek{y}},\text{\textgreek{q}}_{\text{\textgreek{t}}_{1},\text{\textgreek{t}}_{2}}\text{\textgreek{y}};R_{-}]\le C\Big\{\int_{\{\bar{t}=\text{\textgreek{t}}_{1}\}}J_{\text{\textgreek{m}}}^{N}(\text{\textgreek{y}})\bar{n}^{\text{\textgreek{m}}}+\int_{\big\{\bar{t}=\max\{\text{\textgreek{t}}_{1},\text{\textgreek{t}}_{2}-1\}\big\}}J_{\text{\textgreek{m}}}^{N}(\text{\textgreek{y}})\bar{n}^{\text{\textgreek{m}}}\Big\}.\label{eq:BoundForCutOff}
\end{equation}
Inequality (\ref{eq:BoundForCutOff}) is inferred using integrations
by parts in $\partial_{r_{*}}$ and $T$, combined with local-in-time
energy estimates for $\text{\textgreek{y}}$, in the spirit of \cite{DafRod9,DafRodSchlap}. 

\medskip{}

From (\ref{eq:BoundednessEstimate}) and (\ref{eq:Cut-Off-ILED})
we obtain:
\begin{align}
\sup_{\text{\textgreek{t}}_{1}\le\text{\textgreek{t}}\le\text{\textgreek{t}}_{2}}\int_{\{\bar{t}=\text{\textgreek{t}}\}}J_{\text{\textgreek{m}}}^{N}(\text{\textgreek{y}})\bar{n}^{\text{\textgreek{m}}}\le C & \int_{\{\bar{t}=\text{\textgreek{t}}_{1}\}}J_{\text{\textgreek{m}}}^{N}(\text{\textgreek{y}})\bar{n}^{\text{\textgreek{m}}}+C_{R_{-}}a\int_{\{\bar{t}=\text{\textgreek{t}}_{1}\}}J_{\text{\textgreek{m}}}^{N}(\text{\textgreek{y}})\bar{n}^{\text{\textgreek{m}}}+C_{R_{-}}a\int_{\big\{\bar{t}=\max\{\text{\textgreek{t}}_{1},\text{\textgreek{t}}_{2}-1\}\big\}}J_{\text{\textgreek{m}}}^{N}(\text{\textgreek{y}})\bar{n}^{\text{\textgreek{m}}}+\label{eq:BeforeBuondedness}\\
 & +C_{R_{-},R_{+}^{(1)}}\int_{\{R_{-}\le r_{*}\le R_{+}^{(1)}\}\cap\{\text{\textgreek{t}}_{1}\le\bar{t}\le\text{\textgreek{t}}_{2}\}}\min\big\{|T\text{\textgreek{y}}|^{2},|\text{\textgreek{y}}|^{2}\big\}+C_{R_{-}}\mathcal{Z}_{\text{\textgreek{t}}_{1},\text{\textgreek{t}}_{2}}[F,\text{\textgreek{y}};R_{-}].\nonumber 
\end{align}
Recall that the constant $R_{-}$ can be chosen independently of $a$
(provided $a$ is sufficiently small). Therefore, for $a$ sufficiently
small, the third term of the right hand side of (\ref{eq:BeforeBuondedness})
can be absobed into the left hand side, yielding: 
\begin{align}
\sup_{\text{\textgreek{t}}_{1}\le\text{\textgreek{t}}\le\text{\textgreek{t}}_{2}}\int_{\{\bar{t}=\text{\textgreek{t}}\}}J_{\text{\textgreek{m}}}^{N}(\text{\textgreek{y}})\bar{n}^{\text{\textgreek{m}}}\le C_{R_{-}} & \int_{\{\bar{t}=\text{\textgreek{t}}_{1}\}}J_{\text{\textgreek{m}}}^{N}(\text{\textgreek{y}})\bar{n}^{\text{\textgreek{m}}}+C_{R_{-},R_{+}^{(1)}}\int_{\{R_{-}\le r_{*}\le R_{+}^{(1)}\}\cap\{\text{\textgreek{t}}_{1}\le\bar{t}\le\text{\textgreek{t}}_{2}\}}\min\big\{|T\text{\textgreek{y}}|^{2},|\text{\textgreek{y}}|^{2}\big\}+\label{eq:BeforeBuondedness-1}\\
 & +C_{R_{-}}\mathcal{Z}_{\text{\textgreek{t}}_{1},\text{\textgreek{t}}_{2}}[F,\text{\textgreek{y}};R_{-}].\nonumber 
\end{align}
 From (\ref{eq:Cut-Off-ILED}) and (\ref{eq:BeforeBuondedness-1}),
we hence obtain the estimate (\ref{eq:DegenerateIled-1-2}). Inequality
(\ref{eq:NonDegenerateIled-1-2}) then follows by applying the same
procedure for $T\text{\textgreek{y}}$ in place of $\text{\textgreek{y}}$,
and using (\ref{eq:NonDegenerateIled-1}) instead of (\ref{eq:DegenerateIled-1}).

The normal hyperbolicity of the trapped set of $(\mathcal{M}_{0},g_{M,a}^{(h)})$
can be readily inferred by a simple computation (using the fact that
$h^{\prime}\ge0$).
\end{proof}

\section{\label{sec:Criterion}A zero-frequency continuity criterion for decay
and its failure in the presence of superradiance}

In Section \ref{sub:Superradiance}, we discussed some differences
between superradiant and non-superadiant spacetimes concerning the
behaviour of solutions to equation (\ref{eq:PotentialWaveGeneralIntro}).
In the context of this discussion, we will present a specific example
of such a difference: As a consequence of Theorem \hyperref[Theorem 1 detailed]{1},
we will show that a simple zero-frequency continuity criterion for
extending an integrated local energy decay estimate for equation 
\begin{equation}
\square_{g}\text{\textgreek{y}}=0\label{eq:WaveEquationGeneral}
\end{equation}
 on a non-superradiant spacetime $(\mathcal{M},g)$ to the family
of equations 
\begin{equation}
\square_{g}\text{\textgreek{y}}-V_{\text{\textgreek{l}}}\text{\textgreek{y}}=0,\label{eq:PotentialOperators}
\end{equation}
 where $V_{\text{\textgreek{l}}}$ is time independent and depends
smoothly on $\text{\textgreek{l}}\in[0,1]$ (with $V_{0}=0$), utterly
fails in the superradiant case.

The outline of this section is as follows: In Section \ref{sub:The-resolvent-operator},
we will introduce the definition of the resolvent operator associated
to the family (\ref{eq:PotentialOperators}) on a general class of
stationary and asymptotically flat spacetimes and derive some of its
main properties in the case when equation (\ref{eq:WaveEquationGeneral})
satisfies a suitable integrated local energy decay estimate. In Section
\ref{sub:A-zero-frequency-continuity}, we will state a zero-frequency
continuity criterion for integrated local energy decay for the family
(\ref{eq:PotentialOperators}), valid on non-superradiant spacetimes.
Finally, in Section \ref{sub:Failure-of-the}, we will show how the
analogue of this continuity criterion fails in the presence of superradiance.

\subsection{\label{sub:The-resolvent-operator}The resolvent operator}

Before stating the aforementioned continuity criterion, we will introduce
the notion of the resolvent operator for the family (\ref{eq:PotentialOperators})
on a general stationary and asymptotically flat spacetime, and we
will state the main properties satisfied by the resolvent operator
under the assumption that an integrated local energy decay estimate
holds for (\ref{eq:PotentialOperators}) when $\text{\textgreek{l}}=0$.

\begin{figure}[h] 
\centering 
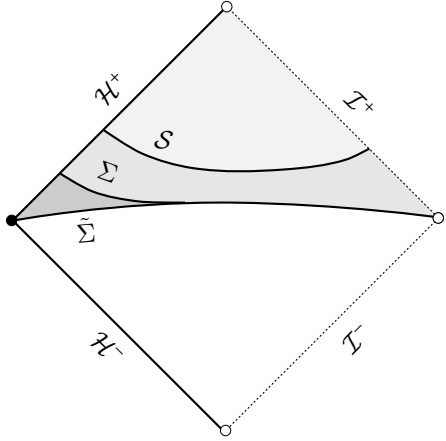 
\caption{In the case when $(\mathcal{M},g)$ is the subextremal Kerr exterior spacetime $(\mathcal{M}_{M,a},g_{M,a})$, the intersection of the hypersurfaces $\tilde{\text{\textgreek{S}}}$, $\text{\textgreek{S}}$ and $\mathcal{S}$ with the $1+1$ dimensional slice $(\theta,\phi)=(\pi/2,0)$ is schematically as depicted above.} 
\end{figure}

Let $(\mathcal{M}^{d+1},g)$, $d\ge3$, be a smooth, globally hyperbolic,
stationary and asymptotically flat spacetime,%
\footnote{The results of this section also apply on asymptotically conic spacetimes
without any change.%
} possibly bounded by an event horizon $\mathcal{H}$ (see \cite{Moschidisb}
for the relevant definition). Let $T$ be the stationary Killing field
of $(\mathcal{M},g)$, normalised so that it is future directed in
the asymptotically flat region of $(\mathcal{M},g)$, and let $\tilde{\text{\textgreek{S}}}$
be a Cauchy hypersurface of $(\mathcal{M},g)$. We will assume without
loss of generality that $\tilde{\text{\textgreek{S}}}$ can be chosen
so that $T$ is everywhere transversal to $\tilde{\text{\textgreek{S}}}$.
Let also $\text{\textgreek{S}}$ be a smooth, spacelike hypersurface
of $\mathcal{M}$, which is not necessarily a Cauchy hypersurface
of $\mathcal{M}$, such that $\text{\textgreek{S}}$ differs from
$\tilde{\text{\textgreek{S}}}$ only near $\mathcal{H}$ (and coincides
everywhere with $\text{\textgreek{S}}$ in the case $\mathcal{H}=\emptyset$)
and satisfies $\mathcal{H}\cap\text{\textgreek{S}}\subset I^{+}(\tilde{\text{\textgreek{S}}})\backslash\tilde{\text{\textgreek{S}}}$.
For any set $A\subset\mathcal{M}$, we will denote with $\mathcal{J}_{\text{\textgreek{t}}}(A)$
the image of $A$ under the flow of $T$ for time $\text{\textgreek{t}}$.
We will also fix a $T$-invariant globally timelike vector field $N$
on $(\mathcal{M},g)$, such that $N\equiv T$ in the asymptotically
flat region of $(\mathcal{M},g)$.

Let $t:\mathcal{M}\rightarrow\mathbb{R}$ be defined by the condition
$t|_{\text{\textgreek{S}}}=0$ and $T(t)=1$ (hence, $\{t=\text{\textgreek{t}}\}=\mathcal{J}_{\text{\textgreek{t}}}(\text{\textgreek{S}})$).
Notice that in the case of Kerr spacetime $(\mathcal{M}_{M,a},g_{M,a})$,
this is a different function than the Kerr $t$ coordinate (see Figure
4.1). Let $V_{\text{\textgreek{l}}}:\mathcal{M}\rightarrow\mathbb{R}$
be a family of smooth $T$-invariant functions supported in the same
set $\cup_{\text{\textgreek{t}}\in\mathbb{R}}\mathcal{J}_{\text{\textgreek{t}}}(K)$,
with $K\subset\text{\textgreek{S}}\backslash\mathcal{H}$ compact,
such that $V_{\text{\textgreek{l}}}$ depends continuously (in the
$C^{\infty}$ topology) on the parameter $\text{\textgreek{l}}\in[0,1]$
and $V_{0}=0$. Finally, let us fix a spacelike hyperboloidal hypersurface
$\mathcal{S}$ in the future of $\text{\textgreek{S}}$, terminating
at future null infinity $\mathcal{I}^{+}$ as in Section 3.1 of \cite{Moschidisc},
with $\mathcal{S}$ intersecting $\mathcal{H}$ transversally and
satisfying $\mathcal{H}\cap\mathcal{S}\subset I^{+}(\tilde{\text{\textgreek{S}}})$
in the case $\mathcal{H}\neq\emptyset$. We will define the function
$\bar{t}:\mathcal{M}\rightarrow\mathbb{R}$ associated to $\mathcal{S}$
by solving $T(\bar{t})=1$, $\bar{t}|_{\mathcal{S}}=0$. 

We will assume that for any smooth function $F:\mathcal{M}\rightarrow\mathbb{R}$
which is compacly supported when restricted to the $\{t=const\}$
hypersurfaces, the (unique) smooth solution $\text{\textgreek{y}}$
to 
\begin{equation}
\begin{cases}
\square_{g}\text{\textgreek{y}}=F,\\
(\text{\textgreek{y}},T\text{\textgreek{y}})|_{t=0}=(0,0),
\end{cases}\label{eq:WithoutPotential}
\end{equation}
on $\{t\ge0\}$ satisfies the following integrated local energy decay
estimate (for any $0<\text{\textgreek{h}}<\frac{1}{2}$) for some
$k\in\mathbb{N}$ and any $t_{f}\ge0$: 
\begin{equation}
\int_{\{0\le t\le t_{f}\}}(1+r)^{-1-2\text{\textgreek{h}}}\big(J_{\text{\textgreek{m}}}^{N}(\text{\textgreek{y}})N^{\text{\textgreek{m}}}+(1+r)^{-2}|\text{\textgreek{y}}|^{2}\big)\, dg\lesssim_{\text{\textgreek{h}}}\sum_{j=0}^{k}\int_{\{0\le t\le t_{f}\}}(1+r)^{1+2\text{\textgreek{h}}}|T^{j}F|^{2}\, dg.\label{eq:IledWithLossPhysical}
\end{equation}

The integrated local energy decay estimate (\ref{eq:IledWithLossPhysical})
allows us to define the free ``resolvent'' operator $R(\square_{g};\text{\textgreek{w}})$
for any $\text{\textgreek{w}}\in\mathbb{C}$ with $Im(\text{\textgreek{w}})\ge0$: 
\begin{prop}
\label{prop:DefinitionResolvent}Let $(\mathcal{M},g)$ be a spacetime
as above, so that any smooth solution $\text{\textgreek{y}}$ to (\ref{eq:WithoutPotential})
satisfies the integrated local energy decay estimate (\ref{eq:IledWithLossPhysical}).
For any $\text{\textgreek{w}}\in\mathbb{C}$ with $Im(\text{\textgreek{w}})\ge0$
and any function $\mathcal{F}\in L_{cp}^{2}(\text{\textgreek{S}})$,%
\footnote{Here, $L_{cp}^{2}(\text{\textgreek{S}})$ denotes the subspace of
$L^{2}(\text{\textgreek{S}})$ spanned by functions of compact support.%
} there exists a unique function $\text{\textgreek{f}}\in H_{loc}^{1}(\text{\textgreek{S}})$,
such that the functions $e^{-i\text{\textgreek{w}}t}\mathcal{F}$
and $e^{-i\text{\textgreek{w}}t}\text{\textgreek{f}}$ on $\mathcal{\cup_{\text{\textgreek{t}}\in\mathbb{R}}\mathcal{J}_{\text{\textgreek{t}}}}(\text{\textgreek{S}})\simeq\mathbb{R}\times\text{\textgreek{S}}$
(where $\mathcal{\cup_{\text{\textgreek{t}}\in\mathbb{R}}\mathcal{J}_{\text{\textgreek{t}}}}(\text{\textgreek{S}})\subseteq\mathcal{M}$
coincides with $\mathcal{M}$ in the case $\mathcal{H}=\emptyset$)
satisfy

\begin{equation}
\begin{cases}
\square_{g}(e^{-i\text{\textgreek{w}}t}\text{\textgreek{f}})=e^{-i\text{\textgreek{w}}t}\mathcal{F},\\
\int_{\mathcal{S}}J_{\text{\textgreek{m}}}^{N}(e^{-i\text{\textgreek{w}}t}\text{\textgreek{f}})n_{\mathcal{S}}^{\text{\textgreek{m}}}<+\infty\mbox{ and }\lim_{r\rightarrow+\infty}(e^{-i\text{\textgreek{w}}t}\text{\textgreek{f}})|_{\mathcal{S}}=0,
\end{cases}\label{eq:DefinitionResolventFree}
\end{equation}
where $n_{\mathcal{S}}$ is the future directed unit normal to the
hypersurface $\mathcal{S}$. The operator $R(\square_{g};\text{\textgreek{w}}):L_{cp}^{2}(\text{\textgreek{S}})\rightarrow H_{loc}^{1}(\text{\textgreek{S}})$
defined by 
\begin{equation}
R(\square_{g};\text{\textgreek{w}})\mathcal{F}\doteq\text{\textgreek{f}}\label{eq:ResolventDefinition}
\end{equation}
 is uniformly bounded on $\{\text{\textgreek{w}}\in\mathbb{C}\,|\, Im(\text{\textgreek{w}})\ge0\}$
with respect to the operator norm 
\begin{equation}
||R(\square_{g};\text{\textgreek{w}})||_{\mathcal{L},\text{\textgreek{h}}}\doteq\sup_{\mathcal{F}\in C_{0}^{\infty}(\text{\textgreek{S}})}\frac{||(1+r)^{-\frac{1}{2}-\text{\textgreek{h}}}\nabla_{g_{\text{\textgreek{S}}}}R(\square_{g};\text{\textgreek{w}})\mathcal{F}||_{L^{2}(\text{\textgreek{S}})}+||(1+r)^{-\frac{1}{2}-\text{\textgreek{h}}}(|\text{\textgreek{w}}|+(1+r)^{-1})R(\square_{g};\text{\textgreek{w}})\mathcal{F}||_{L^{2}(\text{\textgreek{S}})}}{(1+|\text{\textgreek{w}}|^{k})||(1+r)^{\frac{1}{2}+\text{\textgreek{h}}}\mathcal{F}||_{L^{2}(\text{\textgreek{S}})}}\label{eq:ResolventNorm-1}
\end{equation}
for any $0<\text{\textgreek{h}}<\frac{1}{2}$ (where $k\in\mathbb{N}$
is the same number appearing in the right hand side of (\ref{eq:IledWithLossPhysical})).
Furthermore, $R(\square_{g};\text{\textgreek{w}})$ is H\"older continuous
in $\text{\textgreek{w}}$ for $Im(\text{\textgreek{w}})\ge0$ with
respect to the norm 
\begin{equation}
||R(\square_{g};\text{\textgreek{w}})||_{\mathcal{L},\text{\textgreek{h}},\text{\textgreek{h}}_{0}}\doteq\sup_{\mathcal{F}\in C_{0}^{\infty}(\text{\textgreek{S}})}\frac{||(1+r)^{-\frac{1}{2}-\text{\textgreek{h}}}\nabla_{g_{\text{\textgreek{S}}}}R(\square_{g};\text{\textgreek{w}})\mathcal{F}||_{L^{2}(\text{\textgreek{S}})}+||(1+r)^{-\frac{1}{2}-\text{\textgreek{h}}}(|\text{\textgreek{w}}|+(1+r)^{-1})R(\square_{g};\text{\textgreek{w}})\mathcal{F}||_{L^{2}(\text{\textgreek{S}})}}{||(1+r)^{\frac{1}{2}+\text{\textgreek{h}}+\text{\textgreek{h}}_{0}}\mathcal{F}||_{L^{2}(\text{\textgreek{S}})}}\label{eq:ResolventNormHolder}
\end{equation}
 for any $0<\text{\textgreek{h}}<\frac{1}{2}$ and any $0<\text{\textgreek{h}}_{0}<\frac{1}{2}-\text{\textgreek{h}}$
and, for any $\mathcal{F},\text{\textgreek{f}}_{0}\in L_{cp}^{2}(\text{\textgreek{S}})$,
the inner product $\left\langle \text{\textgreek{f}}_{0},R(\square_{g};\text{\textgreek{w}})\mathcal{F}\right\rangle _{L^{2}(\text{\textgreek{S}})}$
is a holomorphic function of $\text{\textgreek{w}}$ for $Im(\text{\textgreek{w}})>0$.\end{prop}
\begin{rem*}
In the case of a product Lorentzian metric $g=-dt^{2}+g_{\text{\textgreek{S}}}$,
our definition of the free resolvent operator $R(\square_{g};\text{\textgreek{w}})$
coincides with the more standard definition of the free resolvent
operator $\mathcal{R}(\text{\textgreek{D}}_{g_{\text{\textgreek{S}}}};\text{\textgreek{w}})$
associated to the time independent operator $\text{\textgreek{D}}_{g_{\text{\textgreek{S}}}}$,
appearing for instance in \cite{Melrose1995}.
\end{rem*}
For the proof of Proposition \ref{prop:DefinitionResolvent}, see
Section \ref{sec:ResolventDefinitionAppendix} of the Appendix.

We will now turn to the family of operators (\ref{eq:PotentialOperators}).
It can be readily established, through a simple application of Gronwall's
inequality, that there exists a $C_{0}>0$ depending on the geometry
of $(\mathcal{M},g)$ and the family $V_{\text{\textgreek{l}}}$,
such that, for any $\text{\textgreek{l}}\in[0,1]$ and any solution
$\text{\textgreek{y}}:\{t\ge0\}\rightarrow\mathbb{C}$ to 
\begin{equation}
\begin{cases}
(\square_{g}-V_{\text{\textgreek{l}}})\text{\textgreek{y}}=0\\
(\text{\textgreek{y}},T\text{\textgreek{y}})|_{t=0}=(\text{\textgreek{y}}_{0},\text{\textgreek{y}}_{1}),
\end{cases}\label{eq:WavePotentialInitialData}
\end{equation}
where the initial data $(\text{\textgreek{y}}_{0},\text{\textgreek{y}}_{1}):\text{\textgreek{S}}\rightarrow\mathbb{C}^{2}$
are smooth and compactly supported, we can estimate for any $\text{\textgreek{t}}\ge0$
\begin{equation}
\int_{\{t=\text{\textgreek{t}}\}}J_{\text{\textgreek{m}}}^{N}(\text{\textgreek{y}})n^{\text{\textgreek{m}}}\le e^{C_{0}\text{\textgreek{t}}}\int_{\{t=0\}}J_{\text{\textgreek{m}}}^{N}(\text{\textgreek{y}})n^{\text{\textgreek{m}}}.\label{eq:ExponentialBoundedness}
\end{equation}
The bound (\ref{eq:ExponentialBoundedness}), combined with an application
of the Fourier--Laplace transformation in the $t$ variable on (\ref{eq:WavePotentialInitialData}),
readily yields that, for any $\mathcal{F}\in L_{cp}^{2}(\text{\textgreek{S}})$
and any $\text{\textgreek{w}}\in\mathbb{C}$ with $Im(\text{\textgreek{w}})>C_{0}$,
there exists a unique $\text{\textgreek{f}}\in H_{loc}^{1}(\text{\textgreek{S}})$
solving 
\begin{equation}
\begin{cases}
(\square_{g}-V_{\text{\textgreek{l}}})(e^{-i\text{\textgreek{w}}t}\text{\textgreek{f}})=e^{-i\text{\textgreek{w}}t}\mathcal{F},\\
\int_{\mathcal{S}}J_{\text{\textgreek{m}}}^{N}(e^{-i\text{\textgreek{w}}t}\text{\textgreek{f}})n_{\mathcal{S}}^{\text{\textgreek{m}}}<+\infty.
\end{cases}\label{eq:DefinitionResolventPotential}
\end{equation}
Thus, we can introduce the following definition:
\begin{defn*}
For any $\text{\textgreek{w}}\in\mathbb{C}$ with $Im(\text{\textgreek{w}})>C_{0}$
and any $\text{\textgreek{l}}\in[0,1]$, we define the operator $R(\square_{g}-V_{\text{\textgreek{l}}};\text{\textgreek{w}}):L_{cp}^{2}(\text{\textgreek{S}})\rightarrow H_{loc}^{1}(\text{\textgreek{S}})$
so that, for any $\mathcal{F}\in L_{cp}^{2}(\text{\textgreek{S}})$,
$\text{\textgreek{f}}=R(\square_{g}-V_{\text{\textgreek{l}}};\text{\textgreek{w}})\mathcal{F}$
is the unique solution of (\ref{eq:DefinitionResolventPotential}).\end{defn*}
\begin{rem*}
It can be readily shown that $R(\square_{g}-V_{\text{\textgreek{l}}};\text{\textgreek{w}})$
is holomorphic in $\text{\textgreek{w}}$ for $Im(\text{\textgreek{w}})>C_{0}$. 
\end{rem*}
We will now show the following:
\begin{lem}
\label{lem:PropertiesResolventNotFree}Let $(\mathcal{M},g)$ be as
in proposition \ref{prop:DefinitionResolvent}. For any $\text{\textgreek{l}}\in[0,1]$,
the operator $R(\square_{g}-V_{\text{\textgreek{l}}};\text{\textgreek{w}}):L_{cp}^{2}\rightarrow H_{loc}^{1}$,
defined by the relation (\ref{eq:DefinitionResolventPotential}),
can be extended as a meromorphic operator-valued function of $\text{\textgreek{w}}$
on the whole of $\{\text{\textgreek{w}}\in\mathbb{C}:Im(\text{\textgreek{w}})>0\}$,
bounded up to $\text{\textgreek{w}}\in\mathbb{R}$ with respect to
the norm (\ref{eq:ResolventNorm-1}), except, possibly, at $\text{\textgreek{w}}=0$
and at all the values $\text{\textgreek{w}}\in\mathbb{R}\backslash\{0\}$
for which equation (\ref{eq:PotentialOperators}) admits an outgoing
mode solution with frequency parameter $\text{\textgreek{w}}$ (see
the definition at the end of Section \ref{sub:Modes}). Furthermore,
the poles of $R(\square_{g}-V_{\text{\textgreek{l}}};\text{\textgreek{w}})$
for $Im(\text{\textgreek{w}})>0$ depend continuously on $\text{\textgreek{l}}$,
except at the values of $\text{\textgreek{l}}$ where they reach the
real axis.\end{lem}
\begin{proof}
As a consequence of Rellich's embedding theorem, in view also of the
boundedness of $R(\square_{g};\text{\textgreek{w}})$ with respect
to the norm (\ref{eq:ResolventNorm-1}) and the compact support in
space of $V_{\text{\textgreek{l}}}$, the operator $R(\square_{g};\text{\textgreek{w}})\circ V_{\text{\textgreek{l}}}:H_{\text{\textgreek{w}}}^{1,\frac{1}{2}+\text{\textgreek{h}}}(\text{\textgreek{S}})\rightarrow H_{\text{\textgreek{w}}}^{1,\frac{1}{2}+\text{\textgreek{h}}}(\text{\textgreek{S}})$,
where 
\begin{equation}
||\mathcal{G}||_{H_{\text{\textgreek{w}}}^{1,\frac{1}{2}+\text{\textgreek{h}}}(\text{\textgreek{S}})}\doteq||(1+r)^{-\frac{1}{2}-\text{\textgreek{h}}}\nabla_{g_{\text{\textgreek{S}}}}\mathcal{G}||_{L^{2}(\text{\textgreek{S}})}+||(1+r)^{-\frac{1}{2}-\text{\textgreek{h}}}(|\text{\textgreek{w}}|+(1+r)^{-1})\mathcal{G}||_{L^{2}(\text{\textgreek{S}})},
\end{equation}
is bounded and compact when $Im(\text{\textgreek{w}})\ge0$. Furthermore,
it also follows (in view of Lemma \ref{prop:DefinitionResolvent})
that $R(\square_{g};\text{\textgreek{w}})\circ V_{\text{\textgreek{l}}}$
is holomorphic in $\text{\textgreek{w}}$ when $Im(\text{\textgreek{w}})>0$. 

Therefore, the Fredholm alternative implies that the operator $\big(1-R(\square_{g};\text{\textgreek{w}})\circ V_{\text{\textgreek{l}}}\big)^{-1}$is
a meromorphic function of $\text{\textgreek{w}}$ for $Im(\text{\textgreek{w}})>0$,
bounded up to $\text{\textgreek{w}}\in\mathbb{R}$ with respect to
the operator norm $||\cdot||_{H_{\text{\textgreek{w}}}^{1,\frac{1}{2}+\text{\textgreek{h}}}(\text{\textgreek{S}})\rightarrow H_{\text{\textgreek{w}}}^{1,\frac{1}{2}+\text{\textgreek{h}}}(\text{\textgreek{S}})}$,
except at those frequencies $\text{\textgreek{w}}\in\mathbb{R}$ for
which 
\begin{equation}
\begin{cases}
\square_{g}(e^{-i\text{\textgreek{w}}t}\text{\textgreek{f}})=0,\\
\int_{\mathcal{S}}J_{\text{\textgreek{m}}}^{N}(e^{-i\text{\textgreek{w}}t}\text{\textgreek{f}})n_{\mathcal{S}}^{\text{\textgreek{m}}}<+\infty\mbox{ and }\lim_{r\rightarrow+\infty}(e^{-i\text{\textgreek{w}}t}\text{\textgreek{f}})|_{\mathcal{S}}=0
\end{cases}\label{eq:DefinitionResolventFree-2}
\end{equation}
admits a non-trivial solution $\text{\textgreek{f}}\in H_{loc}^{1}(\text{\textgreek{S}})$.
Furthermore, in view of the continuous dependence of $V_{\text{\textgreek{l}}}$,
the poles of $\big(1-R(\square_{g};\text{\textgreek{w}})\circ V_{\text{\textgreek{l}}}\big)^{-1}$
for $Im(\text{\textgreek{w}})>0$ depend continuously on $\text{\textgreek{l}}$
(except, of course, at the values of $\text{\textgreek{l}}$ where
they reach the real axis); see \cite{Kato2013}. 

In the region $Im(\text{\textgreek{w}})>C_{0}$ where $R(\square_{g}-V_{\text{\textgreek{l}}};\text{\textgreek{w}})$
was defined, the following relation holds:

\begin{equation}
R(\square_{g}-V_{\text{\textgreek{l}}};\text{\textgreek{w}})=\big(Id-R(\square_{g};\text{\textgreek{w}})\circ V_{\text{\textgreek{l}}}\big)^{-1}\circ R(\square_{g};\text{\textgreek{w}}).\label{eq:FormalEquality}
\end{equation}
Thus, using the relation (\ref{eq:FormalEquality}), the operator
$R(\square_{g}-V_{\text{\textgreek{l}}};\text{\textgreek{w}})$ can
be extended as a meromorphic operator-valued function of $\text{\textgreek{w}}$
on $\{\text{\textgreek{w}}\in\mathbb{C}:\, Im(\text{\textgreek{w}})>0\}$
with the required properties. Furthermore, on $\{\text{\textgreek{w}}\in\mathbb{C}:\, Im(\text{\textgreek{w}})>0\}\backslash\mathcal{P}_{\text{\textgreek{l}}}$,
where $\mathcal{P}_{\text{\textgreek{l}}}\subset\{\text{\textgreek{w}}\in\mathbb{C}:\, Im(\text{\textgreek{w}})>0\}$
are the poles of $R(\square_{g}-V_{\text{\textgreek{l}}};\text{\textgreek{w}})$,
the relation (\ref{eq:FormalEquality}) implies that, for any $\mathcal{F}\in L_{cp}^{2}(\text{\textgreek{S}})$,
the function $\text{\textgreek{f}}=R(\square_{g}-V_{\text{\textgreek{l}}};\text{\textgreek{w}})\mathcal{F}$
satisfies (\ref{eq:DefinitionResolventPotential}). Note that the
finiteness of $\int_{\mathcal{S}}J_{\text{\textgreek{m}}}^{N}(e^{-i\text{\textgreek{w}}t}\text{\textgreek{f}})n_{\mathcal{S}}^{\text{\textgreek{m}}}$
follows from the relation 
\begin{equation}
\big(Id-R(\square_{g};\text{\textgreek{w}})\circ V_{\text{\textgreek{l}}}\big)^{-1}\circ R(\square_{g};\text{\textgreek{w}})=R(\square_{g};\text{\textgreek{w}})\circ\Big(Id+V_{\text{\textgreek{l}}}\circ\big(1-R(\square_{g};\text{\textgreek{w}})\circ V_{\text{\textgreek{l}}}\big)^{-1}\circ R(\square_{g};\text{\textgreek{w}})\Big)
\end{equation}
and the fact that, for any $\mathcal{G}\in L_{cp}^{2}(\text{\textgreek{S}})$,
we have $\int_{\mathcal{S}}J_{\text{\textgreek{m}}}^{N}(e^{-i\text{\textgreek{w}}t}R(\square_{g};\text{\textgreek{w}})\mathcal{G})n_{\mathcal{S}}^{\text{\textgreek{m}}}<+\infty$. 
\end{proof}

\subsection{\label{sub:A-zero-frequency-continuity}A zero-frequency continuity
criterion for decay for the family (\ref{eq:PotentialOperators})
in the non-superradiant case}

Having introduced the resolvent family $R(\square_{g}-V_{\text{\textgreek{l}}};\text{\textgreek{w}})$
and stated its basic properties stemming from the integrated local
energy decay estimate (\ref{eq:IledWithLossPhysical}), we can now
proceed to examine sufficient zero-frequency conditions on the family
$V_{\text{\textgreek{l}}}$ (i.\,e.~conditions related to the boundedness
of $R(\square_{g}-V_{\text{\textgreek{l}}};\text{\textgreek{w}})$
near $\text{\textgreek{w}}=0$) so that (\ref{eq:IledWithLossPhysical})
also holds for equation (\ref{eq:PotentialOnceMore}). We will first
consider the case when $(\mathcal{M},g)$ is the product Lorentzian
manifold $(\mathbb{R}\times\text{\textgreek{S}},-dt^{2}+g_{\text{\textgreek{S}}})$,
where $(\text{\textgreek{S}},g_{\text{\textgreek{S}}})$ is a complete
asymptotically flat Riemannian manifold, with $-dt^{2}+g_{\text{\textgreek{S}}}$
having the asymptotics described by Assumption 1 of \cite{Moschidisb}.
In that case, the following result can be readily established:
\begin{prop}
\label{prop:ILEDProductCase}Let $(\mathcal{M},g)=(\mathbb{R}\times\text{\textgreek{S}},-dt^{2}+g_{\text{\textgreek{S}}})$,
with $-dt^{2}+g_{\text{\textgreek{S}}}$ having the asymptotics described
by Assumption 1 of \cite{Moschidisb}, and let $V_{\text{\textgreek{l}}}:\mathcal{M}\rightarrow\mathbb{R}$,
$\text{\textgreek{l}}\in[0,1]$, be as described above. Assume that
there exists some (small) $\text{\textgreek{e}}>0$, so that the near-zero
frequency bound 
\begin{equation}
\sup_{|\text{\textgreek{w}}|\le\text{\textgreek{e}}}||R(\square_{g}-V_{\text{\textgreek{l}}};\text{\textgreek{w}})||_{\mathcal{L},\text{\textgreek{h}}}<+\infty\label{eq:NearZeroFrequencyBound}
\end{equation}
 holds for all $\text{\textgreek{l}}\in[0,1]$. Then, for any $\text{\textgreek{l}}\in[0,1]$
and any smooth $F:\mathcal{M}\rightarrow\mathbb{R}$ which has compact
support when restricted on the $\{t=const\}$ hypersurfaces, the solution
$\text{\textgreek{y}}$ to 
\begin{equation}
\begin{cases}
\square_{g}\text{\textgreek{y}}-V_{\text{\textgreek{l}}}\text{\textgreek{y}}=F,\\
(\text{\textgreek{y}},T\text{\textgreek{y}})|_{t=0}=(0,0),
\end{cases}\label{eq:PotentialOnceMore}
\end{equation}
 satisfies the integrated local energy decay estimate (\ref{eq:IledWithLossPhysical}).\end{prop}
\begin{rem*}
The proof of Proposition \ref{prop:ILEDProductCase} consists of showing
that the spectrum of the operator $\text{\textgreek{D}}_{g_{\text{\textgreek{S}}}}-V_{\text{\textgreek{l}}}$
does not obtain a discrete component as $\text{\textgreek{l}}$ varies
in $[0,1]$. Note that, in case the potantials $V_{\text{\textgreek{l}}}$
are assumed to be small, Proposition \ref{prop:ILEDProductCase} holds
without assuming (\ref{eq:NearZeroFrequencyBound}), even under substantially
weaker assumptions on the regularity and decay properties of the potentials
$V_{\text{\textgreek{l}}}$, see e.\,g.~\cite{Kato1966,Rodnianski2004}). 
\end{rem*}
For the proof of Proposition \ref{prop:ILEDProductCase}, see Section
\ref{sec:SpectralConsequencesAppendix} of the Appendix. Let us remark
that Proposition \ref{prop:ILEDProductCase} also applies on stationary
spacetimes $(\mathcal{M},g)$, possibly bounded by an event horizon
$\mathcal{H},$ where the stationary vector field $T$ is everywhere
timelike on $\mathcal{M}\backslash\mathcal{H}$ and $\mathcal{H}$
has positive surface gravity (so that the red-shift estimates of Section
7 of \cite{DafRod6} apply).

\subsection{\label{sub:Failure-of-the}Failure of the zero-frequency continuity
criterion in the presence of superradiance}

Let us now examine the case of a stationary spacetime $(\mathcal{M},g)$
with a non-empty ergoregion, i.\,e.~$\{g(T,T)>0\}\neq\emptyset$.
There is no point in considering the case when either $\mathcal{H}^{+}=\emptyset$,
or $\mathcal{H}^{+}\neq\emptyset$ but $\mathcal{H}^{+}\cap\{g(T,T)>0\}=\emptyset$,
since, according to \cite{Friedman1978} (see also our forthcoming
\cite{Moschidis}), the integrated local energy decay estimate (\ref{eq:IledWithLossPhysical})
does not hold in this case. Therefore, we will only consider spacetimes
$(\mathcal{M},g)$ for which $\mathcal{H}^{+}\neq\emptyset$ and $\{g(T,T)>0\}\cap\mathcal{H}^{+}\neq\emptyset$. 

On such a spacetime $(\mathcal{M},g)$, we will show that the zero-frequency
condition of Proposition \ref{prop:ILEDProductCase} is no longer
sufficient to guarantee an integrated local energy decay estimate
of the form (\ref{eq:IledWithLossPhysical}) for the family (\ref{eq:PotentialOnceMore}),
$\text{\textgreek{l}}\in[0,1]$. In particular, as a corollary of
Theorem \hyperref[Theorem 1 detailed]{1}, we will establish the following: 
\begin{cor}
\label{cor:FailureCriterion}On Kerr exterior spacetime $(\mathcal{M}_{M,a},g_{M,a})$
with $0<|a|<M$, for any fixed $(\text{\textgreek{w}}_{R},m,l)\in(\mathbb{R}\backslash\{0\})\times\mathbb{Z}\times\mathbb{Z}_{\ge|m|}$
in the superradiant regime (\ref{eq:SuperradiantRegimeKerr}) and
$r_{0}\gg1$ as in the statement of Theorem \hyperref[Theorem 1 detailed]{1},
defining $V:[r_{+},+\infty)\rightarrow[0,+\infty)$ as the potential
function of Theorem \hyperref[Theorem 1 detailed]{1} for $\text{\textgreek{w}}_{I}=0$,
the family of equations 
\begin{equation}
\square_{g_{M,a}}\text{\textgreek{y}}-\text{\textgreek{l}}\frac{(r^{2}+a^{2})^{2}}{(r-r_{+})(r-r_{-})\text{\textgreek{r}}^{2}}V(r)\text{\textgreek{y}}=0,\label{eq:Family Kerr}
\end{equation}
$\text{\textgreek{l}}\in[0,1]$, has the following properites:

\begin{enumerate}

\item There exists a $\text{\textgreek{l}}_{0}\in(0,1]$ such that
(\ref{eq:Family Kerr}) satisfies an integrated local energy decay
estimate for all $\text{\textgreek{l}}\in[0,\text{\textgreek{l}}_{0})$.

\item There exists an $\text{\textgreek{e}}>0$ so that for all $\text{\textgreek{l}}\in[0,1]$
we can bound: 
\begin{equation}
\sup_{|\text{\textgreek{w}}|\le\text{\textgreek{e}}}||R(\square_{g_{M,a}}-\text{\textgreek{l}}\frac{(r^{2}+a^{2})^{2}}{(r-r_{+})(r-r_{-})\text{\textgreek{r}}^{2}}V(r);\text{\textgreek{w}})||_{\mathcal{L},\text{\textgreek{h}}}<+\infty.\label{eq:BoundResolventKerr}
\end{equation}

\item There exists a $\text{\textgreek{l}}_{1}\in(0,1]$ such that
(\ref{eq:Family Kerr}) does not admit an outgoing real or exponentially
growing mode solution for $\text{\textgreek{l}}<\text{\textgreek{l}}_{1}$,
but admits a real mode solution for $\text{\textgreek{l}}=\text{\textgreek{l}}_{1}$.
As a consequence, no integrated local energy decay estimate of the
form (\ref{eq:IledWithLossPhysical}) holds for (\ref{eq:Family Kerr})
when $\text{\textgreek{l}}=\text{\textgreek{l}}_{1}$.

\end{enumerate}

Therefore, the analogue of Proposition \ref{prop:ILEDProductCase}
does not hold on $(\mathcal{M}_{M,a},g_{M,a})$, $a\neq0$.\end{cor}
\begin{proof}
The robustness of the estimates of \cite{DafRodSchlap} guarantee
that (\ref{eq:Family Kerr}) satisfies an integrated local energy
decay estimate when $\text{\textgreek{l}}\in[0,\text{\textgreek{l}}_{0}]$
with $\text{\textgreek{l}}_{0}\ll1$, while the near-zero frequency
estimates of Section $8.7$ of \cite{DafRodSchlap} apply to (\ref{eq:Family Kerr})
for all $\text{\textgreek{l}}\in[0,1]$ (without any change in the
proof), in view of the sign condition $V\ge0$, yielding (\ref{eq:BoundResolventKerr}).
It thus remains to show the existence of a value $\text{\textgreek{l}}_{1}\in(0,1]$
such that (\ref{eq:Family Kerr}) does not admit an outgoing real
or exponentially growing mode solution for $\text{\textgreek{l}}<\text{\textgreek{l}}_{1}$,
but admits a real mode solution for $\text{\textgreek{l}}=\text{\textgreek{l}}_{1}$.
We will also show that this implies that no integrated local energy
decay estimate holds for (\ref{eq:Family Kerr}) when $\text{\textgreek{l}}=\text{\textgreek{l}}_{1}$.

Let us set for any $\text{\textgreek{l}}\in[0,1]$: 
\begin{equation}
V_{\text{\textgreek{l}}}\doteq\text{\textgreek{l}}\frac{(r^{2}+a^{2})^{2}}{(r-r_{+})(r-r_{-})\text{\textgreek{r}}^{2}}V(r).
\end{equation}
Let $\text{\textgreek{l}}_{1}$ be defined as the infimum of all $\text{\textgreek{l}}\in(0,1]$
for which equation (\ref{eq:Family Kerr}) admits an outgoing real
or exponentially growing mode solution. In view of our aforementioned
remarks, it is necessary that $\text{\textgreek{l}}_{1}>0$. The estimates
of \cite{DafRodSchlap} imply that, when $Re(\text{\textgreek{w}})\gg1$,
the operators $Id-R(\square_{g_{M,a}};\text{\textgreek{w}})\circ V_{\text{\textgreek{l}}}$
are invertible on the space $H_{\text{\textgreek{w}}}^{1,\frac{1}{2}+\text{\textgreek{h}}}(\text{\textgreek{S}})$
uniformly in $\text{\textgreek{l}}$ (see the proof of Lemma \ref{lem:PropertiesResolventNotFree}
for the relevant notation). Thus, the identity (\ref{eq:FormalEquality})
(see also the remarks in the proof of Lemma \ref{lem:PropertiesResolventNotFree}
on the compactness of $R(\square_{g};\text{\textgreek{w}})\circ V_{\text{\textgreek{l}}}$
and dependence of the poles of $R(\square_{g}-V_{\text{\textgreek{l}}};\text{\textgreek{w}})$
on $\text{\textgreek{l}}$) implies that, when $\text{\textgreek{l}}=\text{\textgreek{l}}_{1}$,
there exists at least one $\text{\textgreek{w}}\in\mathbb{R}$ such
that the operator $Id-R(\square_{g_{M,a}};\text{\textgreek{w}})\circ V_{\text{\textgreek{l}}}$
has non trivial kernel, i.\,e.~equation (\ref{eq:Family Kerr})
admits an outgoing real mode solution.

For $\text{\textgreek{l}}=\text{\textgreek{l}}_{1}$ and $\text{\textgreek{w}}\in\mathbb{R}$
as above, let $\text{\textgreek{f}}:\mathcal{M}_{M,a}\rightarrow\mathbb{C}$
be a non-trivial $T$-invariant function solving 
\begin{equation}
\begin{cases}
(\square_{g_{M,a}}-V_{\text{\textgreek{l}}_{1}})(e^{-i\text{\textgreek{w}}t}\text{\textgreek{f}})=0,\\
\int_{\mathcal{S}}J_{\text{\textgreek{m}}}^{N}(e^{-i\text{\textgreek{w}}t}\text{\textgreek{f}})n_{\mathcal{S}}^{\text{\textgreek{m}}}<+\infty\mbox{ and }\lim_{r\rightarrow+\infty}(e^{-i\text{\textgreek{w}}t}\text{\textgreek{f}})|_{\mathcal{S}}=0,
\end{cases}\label{eq:DefinitionResolventFree-2-1}
\end{equation}
where $\mathcal{S}\subset\mathcal{M}_{M,a}$ is a spacelike hyperboloidal
hypersurface terminating at $\mathcal{I}^{+}$ and intersecting $\mathcal{H}^{+}\backslash\mathcal{H}^{-}$
transversally, defined as in Section (\ref{sub:The-resolvent-operator}).
We will show that (\ref{eq:DefinitionResolventFree-2-1}) implies
that no estimate of the form (\ref{eq:IledWithLossPhysical}) can
hold for equation (\ref{eq:Family Kerr}) when $\text{\textgreek{l}}=\text{\textgreek{l}}_{1}$.
Assume, for the sake of contradiction, that (\ref{eq:IledWithLossPhysical})
holds for equation (\ref{eq:Family Kerr}). Then, the uniqueness part
of the proof of Lemma \ref{prop:DefinitionResolvent} (see Section
\ref{sec:ResolventDefinitionAppendix} of the Appendix) applies, without
any change, yielding that any function $\text{\textgreek{f}}$ satisfying
(\ref{eq:DefinitionResolventFree-2-1}) must be identically $0$;
thus, we obtain a contradiction.
\end{proof}
The fact that Proposition \ref{prop:ILEDProductCase} fails to apply
on $(\mathcal{M}_{M,a},g_{M,a})$ highlights that, generally, there
is no ``cheap'' way of controlling the resolvent operator $R(\square_{g}-V;\cdot)$
for frequencies in a neighborhood of the real axis on superradiant
spacetimes $(\mathcal{M},g)$, as is the case on spacetimes with an
everywhere timelike Killing field $T$. Therefore, controlling the
behaviour of the resolvent operator $R(\square_{g}-V_{\text{\textgreek{l}}};\text{\textgreek{w}})$
in the near zero frequency regime $\{|\text{\textgreek{w}}|\ll1\}$
is not enough to exclude the emergence of ``spectral'' instabilities
for the family (\ref{eq:PotentialOnceMore}) elsewhere on the real
axis as $\text{\textgreek{l}}$ varies in $[0,1]$.

\appendix

\section{\label{sec:ResolventDefinitionAppendix}Proof of Proposition \ref{prop:DefinitionResolvent}}

The proof of Proposition \ref{prop:DefinitionResolvent} will consist
of several steps. We will first show that for any $\mathcal{F}\in L_{cp}^{2}(\text{\textgreek{S}})$
and any $\text{\textgreek{w}}\in\mathbb{C}$ with $Im(\text{\textgreek{w}})\ge0$,
there exists a solution $\text{\textgreek{f}}\in H_{loc}^{1}(\text{\textgreek{S}})$
to (\ref{eq:DefinitionResolventFree}), which is also unique, establishing,
thus, that the free resolvent operator $R(\square_{g};\text{\textgreek{w}})$
is well defined. In this step, we will also obtain some useful estimates
for the operator $R(\square_{g};\text{\textgreek{w}})$. Then, we
will proceed to establish the H\"older continuity and holomorphicity
properties of the operator $R(\square_{g};\text{\textgreek{w}})$.

\subsubsection*{1. Existence}

We will first show that, for any $\mathcal{F}\in L_{cp}^{2}(\text{\textgreek{S}})$
and any $\text{\textgreek{w}}\in\mathbb{C}$ with $Im(\text{\textgreek{w}})\ge0$,
there exists a solution $\text{\textgreek{f}}\in H_{loc}^{1}(\text{\textgreek{S}})$
to (\ref{eq:DefinitionResolventFree}). Let $\text{\textgreek{q}}:\mathbb{R}\rightarrow[0,1]$
be a smooth function such that $\text{\textgreek{q}}(t)=0$ for $t\le0$
and $\text{\textgreek{q}}(t)=1$ for $t\ge1$, and let us define the
function $\text{\textgreek{y}}:\{t\ge0\}\rightarrow\mathbb{C}$ as
the unique $H_{loc}^{1}$ solution of the initial value problem
\begin{equation}
\begin{cases}
\square_{g}\text{\textgreek{y}}=\text{\textgreek{q}}(t)e^{-i\text{\textgreek{w}}t}\mathcal{F},\\
(\text{\textgreek{y}},T\text{\textgreek{y}})|_{t=0}=(0,0).
\end{cases}\label{eq:FunctionForExistenceResolvent}
\end{equation}
In view of the integrated local energy decay estimate (\ref{eq:IledWithLossPhysical}),
we can bound for any $t_{f}\ge2$ and any $0<\text{\textgreek{h}}<\frac{1}{2}$
after multiplying both sides of (\ref{eq:IledWithLossPhysical}) with
$\Big(\int_{1}^{t_{f}}e^{2Im(\text{\textgreek{w}})t}\, dt\Big)^{-1}$:
\begin{align}
\Big(\int_{1}^{t_{f}}e^{2Im(\text{\textgreek{w}})t}\, dt\Big)^{-1}\int_{\{0\le t\le t_{f}\}}(1+r)^{-1-2\text{\textgreek{h}}} & \big(J_{\text{\textgreek{m}}}^{N}(\text{\textgreek{y}})N^{\text{\textgreek{m}}}+(1+r)^{-2}|\text{\textgreek{y}}|^{2}\big)\, dg\le\label{eq:BasicILEDforexistence}\\
 & \le C_{\text{\textgreek{h}}\text{\textgreek{q}}}(1+|\text{\textgreek{w}}|^{k})\int_{\text{\textgreek{S}}}(1+r)^{1+2\text{\textgreek{h}}}|\mathcal{F}|^{2}\, dg_{\text{\textgreek{S}}},\nonumber 
\end{align}
where $C_{\text{\textgreek{h}}\text{\textgreek{q}}}$ depends on $\text{\textgreek{h}}$
and $\sum_{j=0}^{k}\sup_{t\in[0,1]}|\frac{d^{j}\text{\textgreek{q}}}{dt^{j}}(t)|$. 

Commuting (\ref{eq:FunctionForExistenceResolvent}) with $T+i\text{\textgreek{w}}$
and applying (\ref{eq:IledWithLossPhysical}) for the commuted equation
(after multiplying, again, both sides of (\ref{eq:IledWithLossPhysical})
with $\Big(\int_{1}^{t_{f}}e^{Im(\text{\textgreek{w}})t}\, dt\Big)^{-1}$),
we obtain: 
\begin{align}
\Big(\int_{1}^{t_{f}}e^{2Im(\text{\textgreek{w}})t}\, dt\Big)^{-1} & \int_{\{0\le t\le t_{f}\}}(1+r)^{-1-2\text{\textgreek{h}}}\big(J_{\text{\textgreek{m}}}^{N}(T\text{\textgreek{y}}+i\text{\textgreek{w}}\text{\textgreek{y}})N^{\text{\textgreek{m}}}+(1+r)^{-2}|T\text{\textgreek{y}}+i\text{\textgreek{w}}\text{\textgreek{y}}|^{2}\big)\, dg\le\label{eq:BasicEstimateForPeriodicity}\\
 & \le C_{\text{\textgreek{h}}}\Big(\int_{1}^{t_{f}}e^{2Im(\text{\textgreek{w}})t}\, dt\Big)^{-1}\sum_{j=0}^{k}\int_{\{0\le t\le t_{f}\}}(1+r)^{1+2\text{\textgreek{h}}}|T^{j}(T+i\text{\textgreek{w}})(\text{\textgreek{q}}(t)e^{-i\text{\textgreek{w}}t}\mathcal{F})|^{2}\, dg\le\nonumber \\
 & \le C_{\text{\textgreek{h}}\text{\textgreek{q}}}\Big(\int_{1}^{t_{f}}e^{2Im(\text{\textgreek{w}})t}\, dt\Big)^{-1}\int_{\text{\textgreek{S}}}(1+r)^{1+2\text{\textgreek{h}}}|\mathcal{F}|^{2}\, dg_{\text{\textgreek{S}}},\nonumber 
\end{align}
where we used the fact that $(T+i\text{\textgreek{w}})(e^{-i\text{\textgreek{w}}t}\mathcal{F})=0$
and $T\text{\textgreek{q}}$ is supported in $\{0\le t\le1\}$. 

Using the $r^{p}$-weighted energy estimate of Theorem 5.1 of \cite{Moschidisc}
for $p=1$, $\text{\textgreek{t}}_{1}=0$ and $\text{\textgreek{t}}_{2}=t_{f}$,
in conjuction with the integrated local energy decay estimate (\ref{eq:IledWithLossPhysical}),
we readily obtain
\begin{align}
\int_{0}^{t_{f}} & \Big(\int_{\{\bar{t}=\text{\textgreek{t}}\}}J_{\text{\textgreek{m}}}^{N}(\text{\textgreek{y}})n_{\mathcal{S}}^{\text{\textgreek{m}}}\Big)\, d\text{\textgreek{t}}\le\label{eq:BeforeEnergyFlux}\\
 & \le C_{\text{\textgreek{q}}}(1+|\text{\textgreek{w}}|^{k})\int_{\{0\le\bar{t}\le t_{f}\}}(1+r)^{2}|e^{-i\text{\textgreek{w}}t}\mathcal{F}|^{2}\, dg,\nonumber 
\end{align}
where $n_{\mathcal{S}}$ is the future directed unit normal to the
$\{\bar{t}=const\}$ hypersurfaces. Notice that, while $e^{-i\text{\textgreek{w}}t}$
is unbounded on the $\{\bar{t}=const\}$ hypersurfaces when $Im(\text{\textgreek{w}})>0$,
the right hand side of (\ref{eq:BasicEstimateForEnergyFlux}) is finite,
since $\mathcal{F}$ was assumed to have compact support on $\text{\textgreek{S}}$.
In particular, (\ref{eq:BeforeEnergyFlux}) implies that, for any
$t_{f}\ge2$: 
\begin{equation}
\Big(\int_{1}^{t_{f}}e^{2Im(\text{\textgreek{w}})t}\, dt\Big)^{-1}\int_{0}^{t_{f}}\Big(\int_{\{\bar{t}=\text{\textgreek{t}}\}}J_{\text{\textgreek{m}}}^{N}(\text{\textgreek{y}})n_{\mathcal{S}}^{\text{\textgreek{m}}}\Big)\, d\text{\textgreek{t}}\le C_{\text{\textgreek{q}}\text{\textgreek{w}}\mathcal{F}}\label{eq:BasicEstimateForEnergyFlux}
\end{equation}
where $C_{\text{\textgreek{q}}\text{\textgreek{w}}\mathcal{F}}$ depends
on $\text{\textgreek{q}},\text{\textgreek{w}},\mathcal{F}$. Similarly,
using Theorem 5.3 of \cite{Moschidisc} for $p=2\text{\textgreek{h}}$,
$t_{1}=0$ and $t_{2}=t_{f}$, combined with (\ref{eq:IledWithLossPhysical}),
we can estimate: 
\begin{align}
\Big(\int_{1}^{t_{f}}e^{2Im(\text{\textgreek{w}})t}\, dt\Big)^{-1}\int_{\{0\le t\le t_{f}\}}(1+r)^{-1+2\text{\textgreek{h}}} & \Big(\big(\nabla_{\text{\textgreek{m}}}\bar{t}\nabla^{\text{\textgreek{m}}}\text{\textgreek{y}}\big)^{2}+(1+r)^{-2}|\text{\textgreek{y}}|^{2}\Big)\, dg\le\label{eq:BasicEstimateForExistenceRp}\\
 & \le C_{\text{\textgreek{h}}\text{\textgreek{q}}}(1+|\text{\textgreek{w}}|^{k})\int_{\text{\textgreek{S}}}(1+r)^{1+2\text{\textgreek{h}}}|\mathcal{F}|^{2}\, dg_{\text{\textgreek{S}}}.\nonumber 
\end{align}
Adding (\ref{eq:BasicILEDforexistence}) and (\ref{eq:BasicEstimateForExistenceRp}),
we thus obtain: 
\begin{align}
\Big(\int_{1}^{t_{f}}e^{2Im(\text{\textgreek{w}})t}\, dt\Big)^{-1}\int_{\{0\le t\le t_{f}\}}(1+r)^{-1+2\text{\textgreek{h}}}\big((\nabla_{\text{\textgreek{m}}}\bar{t}\nabla^{\text{\textgreek{m}}}\text{\textgreek{y}})^{2}+ & (1+r)^{-4\text{\textgreek{h}}}J_{\text{\textgreek{m}}}^{N}(\text{\textgreek{y}})N^{\text{\textgreek{m}}}+(1+r)^{-2}|\text{\textgreek{y}}|^{2}\big)\, dg\le\label{eq:BasicEstimateForExistence}\\
 & \le C_{\text{\textgreek{h}}\text{\textgreek{q}}}(1+|\text{\textgreek{w}}|^{k})\int_{\text{\textgreek{S}}}(1+r)^{1+2\text{\textgreek{h}}}|\mathcal{F}|^{2}\, dg_{\text{\textgreek{S}}}.\nonumber 
\end{align}

For any integer $j\in\{0,\ldots,\lfloor t_{f}\rfloor-1\}$, let us
define the non-negative quantities 
\begin{align}
f_{j}= & \frac{\int_{\{j\le t\le j+1\}}(1+r)^{-1+2\text{\textgreek{h}}}\big((\nabla_{\text{\textgreek{m}}}\bar{t}\nabla^{\text{\textgreek{m}}}\text{\textgreek{y}})^{2}+(1+r)^{-4\text{\textgreek{h}}}J_{\text{\textgreek{m}}}^{N}(\text{\textgreek{y}})N^{\text{\textgreek{m}}}+(1+r)^{-2}|\text{\textgreek{y}}|^{2}\big)\, dg}{C_{\text{\textgreek{h}}\text{\textgreek{q}}}(1+|\text{\textgreek{w}}|^{k})\int_{\text{\textgreek{S}}}(1+r)^{1+2\text{\textgreek{h}}}|\mathcal{F}|^{2}\, dg_{\text{\textgreek{S}}}}+\label{eq:FirstPositiveQuantity}\\
 & +\frac{\int_{\{j\le t\le j+1\}}(1+r)^{-1-2\text{\textgreek{h}}}\big(J_{\text{\textgreek{m}}}^{N}(T\text{\textgreek{y}}+i\text{\textgreek{w}}\text{\textgreek{y}})N^{\text{\textgreek{m}}}+(1+r)^{-2}|T\text{\textgreek{y}}+i\text{\textgreek{w}}\text{\textgreek{y}}|^{2}\big)\, dg}{C_{\text{\textgreek{h}}\text{\textgreek{q}}}\Big(\int_{0}^{t_{f}}e^{2Im(\text{\textgreek{w}})t}\, dt\Big)^{-1}\int_{\text{\textgreek{S}}}(1+r)^{1+2\text{\textgreek{h}}}|\mathcal{F}|^{2}\, dg_{\text{\textgreek{S}}}}+\frac{\int_{j}^{j+1}\Big(\int_{\{\bar{t}=\text{\textgreek{t}}\}}J_{\text{\textgreek{m}}}^{N}(\text{\textgreek{y}})n_{\mathcal{S}}^{\text{\textgreek{m}}}\Big)\, d\text{\textgreek{t}}}{C_{\text{\textgreek{q}}\text{\textgreek{w}}\mathcal{F}}},\nonumber 
\end{align}
and 
\begin{equation}
g_{j}=\int_{j}^{j+1}e^{2Im(\text{\textgreek{w}})t}\, dt,
\end{equation}
where the denominators of the terms in the right hand side of (\ref{eq:FirstPositiveQuantity})
are precisely the right hand sides of inequalities (\ref{eq:BasicEstimateForExistence}),
(\ref{eq:BasicEstimateForPeriodicity}) and (\ref{eq:BasicEstimateForEnergyFlux})
respectively. Then, the inequalities (\ref{eq:BasicEstimateForExistence}),
(\ref{eq:BasicEstimateForPeriodicity}) and (\ref{eq:BasicEstimateForEnergyFlux})
imply that, for any $t_{f}\ge2$: 
\begin{equation}
\Big(\sum_{j=1}^{\lfloor t_{f}\rfloor-1}g_{j}\Big)^{-1}\sum_{j=1}^{\lfloor t_{f}\rfloor-1}f_{j}\le3.\label{eq:ForPidgeonHolePrinciple}
\end{equation}
An application of the pidgeonhole principle on the relation (\ref{eq:ForPidgeonHolePrinciple})
yields that there exists a $j_{0}\in\{1,\ldots,\lfloor t_{f}\rfloor-1\}$
such that 
\begin{equation}
\frac{f_{j_{0}}}{g_{j_{0}}}\le3.
\end{equation}
Therefore, we deduce that, for any $t_{f}\ge2$, there exists a $t_{in}=t_{in}(t_{f})\in[1,t_{f}-1]$
such that 
\begin{align}
\Big(\int_{t_{in}}^{t_{in}+1}e^{2Im(\text{\textgreek{w}})t}\, dt\Big)^{-1}\int_{\{t_{in}\le t\le t_{in}+1\}}\int_{\{0\le t\le t_{f}\}}(1+r)^{-1+2\text{\textgreek{h}}}\big((\nabla_{\text{\textgreek{m}}}\bar{t}\nabla^{\text{\textgreek{m}}}\text{\textgreek{y}})^{2}+ & (1+r)^{-4\text{\textgreek{h}}}J_{\text{\textgreek{m}}}^{N}(\text{\textgreek{y}})N^{\text{\textgreek{m}}}+(1+r)^{-2}|\text{\textgreek{y}}|^{2}\big)\, dg\le\label{eq:BasicEstimateForExistence-1}\\
 & \le C_{\text{\textgreek{h}}\text{\textgreek{q}}}(1+|\text{\textgreek{w}}|^{k})\int_{\text{\textgreek{S}}}(1+r)^{1+2\text{\textgreek{h}}}|\mathcal{F}|^{2}\, dg_{\text{\textgreek{S}}},\nonumber 
\end{align}
\begin{align}
\Big(\int_{t_{in}}^{t_{in}+1}e^{2Im(\text{\textgreek{w}})t}\, dt\Big)^{-1} & \int_{\{t_{in}\le t\le t_{in}+1\}}(1+r)^{-1-2\text{\textgreek{h}}}\big(J_{\text{\textgreek{m}}}^{N}(T\text{\textgreek{y}}+i\text{\textgreek{w}}\text{\textgreek{y}})N^{\text{\textgreek{m}}}+(1+r)^{-2}|T\text{\textgreek{y}}+i\text{\textgreek{w}}\text{\textgreek{y}}|^{2}\big)\, dg\le\label{eq:BasicEstimateForPeriodicity-1}\\
 & \le C_{\text{\textgreek{h}}\text{\textgreek{q}}}\Big(\int_{1}^{t_{f}}e^{2Im(\text{\textgreek{w}})t}\, dt\Big)^{-1}\int_{\text{\textgreek{S}}}(1+r)^{1+2\text{\textgreek{h}}}|\mathcal{F}|^{2}\, dg_{\text{\textgreek{S}}}\nonumber 
\end{align}
and 
\begin{equation}
\Big(\int_{t_{in}}^{t_{in}+1}e^{2Im(\text{\textgreek{w}})t}\, dt\Big)^{-1}\int_{t_{in}}^{t_{in}+1}\Big(\int_{\{\bar{t}=\text{\textgreek{t}}\}}J_{\text{\textgreek{m}}}^{N}(\text{\textgreek{y}})n_{\mathcal{S}}^{\text{\textgreek{m}}}\Big)\, d\text{\textgreek{t}}\le C_{\text{\textgreek{q}}\text{\textgreek{w}}\mathcal{F}}.\label{eq:BasicEstimateForEnergyFlux-1}
\end{equation}

Let us fix a sequence $\{t_{f,n}\}_{n\in\mathbb{N}}$ of positive
numbers such that $\lim_{n}t_{f,n}=+\infty$, and let us define the
sequence $\{t_{n}\}_{n\in\mathbb{N}}\ge1$ by the relation 
\begin{equation}
t_{n}=t_{in}(t_{f,n}).
\end{equation}
We also define the sequence of functions $\text{\textgreek{y}}_{n}:\{0\le t\le1\}\rightarrow\mathbb{C}$
by the relation 
\begin{equation}
\text{\textgreek{y}}_{n}(t,x)\doteq e^{i\text{\textgreek{w}}t_{n}}\text{\textgreek{y}}(t+t_{n},x),\label{eq:DefinitionSequenceConvergence}
\end{equation}
noting that (\ref{eq:FunctionForExistenceResolvent}) (and the fact
that $\text{\textgreek{q}}(t)=1$ for $t\ge1$) imply that $\text{\textgreek{y}}_{n}$
satisfies on $\{0\le t\le1\}$ 
\begin{equation}
\square_{g}\text{\textgreek{y}}_{n}=e^{-i\text{\textgreek{w}}t}\mathcal{F}.\label{eq:InhomogeneousWaveEquationSubsequence}
\end{equation}

In view of (\ref{eq:BasicEstimateForExistence-1}), (\ref{eq:BasicEstimateForPeriodicity-1})
and (\ref{eq:BasicEstimateForEnergyFlux-1}) for $t_{n},t_{f,n}$
in place of $t_{in},t_{f}$, as well as the relation (\ref{eq:DefinitionSequenceConvergence})
and the definition of the function $\bar{t}$, we can bound for any
$n\in\mathbb{N}$: 
\begin{align}
\int_{\{0\le t\le1\}}\int_{\{0\le t\le t_{f}\}}(1+r)^{-1+2\text{\textgreek{h}}}\big((\nabla_{\text{\textgreek{m}}}\bar{t}\nabla^{\text{\textgreek{m}}}\text{\textgreek{y}}_{n})^{2}+ & (1+r)^{-4\text{\textgreek{h}}}J_{\text{\textgreek{m}}}^{N}(\text{\textgreek{y}}_{n})N^{\text{\textgreek{m}}}+(1+r)^{-2}|\text{\textgreek{y}}_{n}|^{2}\big)\, dg\le\label{eq:BasicEstimateForExistence-1-1}\\
 & \le C_{\text{\textgreek{h}}\text{\textgreek{q}}}e^{2Im(\text{\textgreek{w}})}(1+|\text{\textgreek{w}}|^{k})\int_{\text{\textgreek{S}}}(1+r)^{1+2\text{\textgreek{h}}}|\mathcal{F}|^{2}\, dg_{\text{\textgreek{S}}},\nonumber 
\end{align}
\begin{align}
\int_{\{0\le t\le1\}} & (1+r)^{-1-2\text{\textgreek{h}}}\big(J_{\text{\textgreek{m}}}^{N}(T\text{\textgreek{y}}_{n}+i\text{\textgreek{w}}\text{\textgreek{y}}_{n})N^{\text{\textgreek{m}}}+(1+r)^{-2}|T\text{\textgreek{y}}_{n}+i\text{\textgreek{w}}\text{\textgreek{y}}_{n}|^{2}\big)\, dg\le\label{eq:BasicEstimateForPeriodicity-1-1}\\
 & \le C_{\text{\textgreek{h}}\text{\textgreek{q}}}e^{2Im(\text{\textgreek{w}})}\Big(\int_{1}^{t_{f,n}}e^{2Im(\text{\textgreek{w}})t}\, dt\Big)^{-1}\int_{\text{\textgreek{S}}}(1+r)^{1+2\text{\textgreek{h}}}|\mathcal{F}|^{2}\, dg_{\text{\textgreek{S}}}\nonumber 
\end{align}
and 
\begin{equation}
\int_{0}^{1}\Big(\int_{\{\bar{t}=\text{\textgreek{t}}\}}J_{\text{\textgreek{m}}}^{N}(\text{\textgreek{y}}_{n})n_{\mathcal{S}}^{\text{\textgreek{m}}}\Big)\, d\text{\textgreek{t}}\le C_{\text{\textgreek{q}}\text{\textgreek{w}}\mathcal{F}}.\label{eq:BasicEstimateForEnergyFlux-1-1}
\end{equation}

The bound (\ref{eq:BasicEstimateForExistence-1-1}) implies that,
as $n\rightarrow+\infty$, there exists a subsequence of $\{\text{\textgreek{y}}_{n}\}_{n\in\mathbb{N}}$
(assuming, without loss of generality, that this subsequence is in
fact the whole sequence $\{\text{\textgreek{y}}_{n}\}_{n\in\mathbb{N}}$)
converging weakly in $H_{loc}^{1}(\{0\le t\le1\})$ to a function
$\tilde{\text{\textgreek{y}}}$ satisfying the bound: 
\begin{align}
\int_{\{0\le t\le1\}}\int_{\{0\le t\le t_{f}\}}(1+r)^{-1+2\text{\textgreek{h}}}\big((\nabla_{\text{\textgreek{m}}}\bar{t}\nabla^{\text{\textgreek{m}}}\tilde{\text{\textgreek{y}}})^{2}+ & (1+r)^{-4\text{\textgreek{h}}}J_{\text{\textgreek{m}}}^{N}(\tilde{\text{\textgreek{y}}})N^{\text{\textgreek{m}}}+(1+r)^{-2}|\tilde{\text{\textgreek{y}}}|^{2}\big)\, dg\le\label{eq:BoundWeakLimit}\\
 & \le C_{\text{\textgreek{h}}\text{\textgreek{q}}}e^{2Im(\text{\textgreek{w}})}(1+|\text{\textgreek{w}}|^{k})\int_{\text{\textgreek{S}}}(1+r)^{1+2\text{\textgreek{h}}}|\mathcal{F}|^{2}\, dg_{\text{\textgreek{S}}}.\nonumber 
\end{align}
Since the functions $\text{\textgreek{y}}_{n}$ satisfy (\ref{eq:InhomogeneousWaveEquationSubsequence})
and $\text{\textgreek{y}}_{n}$ converge weakly in $H_{loc}^{1}$
to $\tilde{\text{\textgreek{y}}}$, we readily infer that $\tilde{\text{\textgreek{y}}}$
also satisfies (\ref{eq:InhomogeneousWaveEquationSubsequence}). In
view of (\ref{eq:BasicEstimateForPeriodicity-1-1}) and the fact that
$t_{f,n}\rightarrow+\infty$ as $n\rightarrow+\infty$, we infer that
\begin{equation}
T\tilde{\text{\textgreek{y}}}+i\text{\textgreek{w}}\tilde{\text{\textgreek{y}}}=0,
\end{equation}
i.\,e.~$\tilde{\text{\textgreek{y}}}$ is of the form 
\begin{equation}
\tilde{\text{\textgreek{y}}}=e^{-i\text{\textgreek{w}}t}\text{\textgreek{f}}\label{eq:Periodicity}
\end{equation}
for some $\text{\textgreek{f}}\in H_{loc}^{1}(\text{\textgreek{S}})$.
We will extend $\tilde{\text{\textgreek{y}}}$ on the whole of $\mathcal{\cup_{\text{\textgreek{t}}\in\mathbb{R}}\mathcal{J}_{\text{\textgreek{t}}}}(\text{\textgreek{S}})$
by the relation (\ref{eq:Periodicity}). 

Since the functions $\text{\textgreek{y}}_{n}$ satisfy (\ref{eq:InhomogeneousWaveEquationSubsequence})
and $\{\text{\textgreek{y}}_{n}\}_{n\in\mathbb{N}}$ converges weakly
in $H_{loc}^{1}(\{0\le t\le1\})$ to $\tilde{\text{\textgreek{y}}}$,
we readily infer that $\tilde{\text{\textgreek{y}}}$ also satisfies
(\ref{eq:InhomogeneousWaveEquationSubsequence}) on $\{0\le t\le1\}$.
Since $\tilde{\text{\textgreek{y}}}$ was extended on the whole of
$\mathcal{\cup_{\text{\textgreek{t}}\in\mathbb{R}}\mathcal{J}_{\text{\textgreek{t}}}}(\text{\textgreek{S}})$
by the relation (\ref{eq:Periodicity}) and the metric $g$ is $T$-invariant,
we infer that $\tilde{\text{\textgreek{y}}}$ satisfies (\ref{eq:InhomogeneousWaveEquationSubsequence})
on the whole of $\mathcal{\cup_{\text{\textgreek{t}}\in\mathbb{R}}\mathcal{J}_{\text{\textgreek{t}}}}(\text{\textgreek{S}})$.
Thus, the function $\text{\textgreek{f}}$ satisfies (\ref{eq:DefinitionResolventFree}),
with the condition that $e^{-i\text{\textgreek{w}}t}\text{\textgreek{f}}$
having finite energy flux through $\mathcal{S}$ and $\lim_{r\rightarrow+\infty}(e^{-i\text{\textgreek{w}}t}\text{\textgreek{f}})|_{\mathcal{S}}$
being a direct consequence of (\ref{eq:BasicEstimateForEnergyFlux-1-1}).
Furthermore, $\text{\textgreek{f}}$ satisfies 
\begin{align}
\int_{\text{\textgreek{S}}}(1+r)^{-1-2\text{\textgreek{h}}}\Big(|\nabla_{g_{\text{\textgreek{S}}}}\text{\textgreek{f}}|_{g_{\text{\textgreek{S}}}}^{2}+ & \big(|\text{\textgreek{w}}|^{2}+(1+r)^{-2}\big)|\text{\textgreek{f}}|^{2}\big)\, dg_{\text{\textgreek{S}}}\le\label{eq:BoundForResolventNorm}\\
 & \le C_{\text{\textgreek{h}}}(1+|\text{\textgreek{w}}|^{k})\int_{\text{\textgreek{S}}}(1+r)^{1+2\text{\textgreek{h}}}|\mathcal{F}|^{2}\, dg_{\text{\textgreek{S}}}\nonumber 
\end{align}
and 
\begin{align}
\int_{\text{\textgreek{S}}}(1+r)^{-1+2\text{\textgreek{h}}}\Big(\big(\nabla_{\text{\textgreek{m}}}\bar{t}\nabla^{\text{\textgreek{m}}}(e^{-i\text{\textgreek{w}}t}\text{\textgreek{f}} & )\big)^{2}|_{t=0}+(1+r)^{-2}|\text{\textgreek{f}}|^{2}\Big)\, dg_{\text{\textgreek{S}}}\le\label{eq:BoundFromNewMethod}\\
 & \le C_{\text{\textgreek{h}}}(1+|\text{\textgreek{w}}|^{k})\int_{\text{\textgreek{S}}}(1+r)^{1+2\text{\textgreek{h}}}|\mathcal{F}|^{2}\, dg_{\text{\textgreek{S}}}\nonumber 
\end{align}
 in view of (\ref{eq:BoundWeakLimit}) (and the fact that $T$ is
everywhere transversal to $\text{\textgreek{S}}$).

\subsubsection*{2. Uniqueness}

Having established the existence of a solution $\text{\textgreek{f}}\in H_{loc}^{1}(\text{\textgreek{S}})$
to (\ref{eq:DefinitionResolventFree}), we will now proceed to show
that this solution is unique. In particular, we will show that, for
any $\text{\textgreek{w}}\in\mathbb{C}$ with $Im(\text{\textgreek{w}})\ge0$,
if $\text{\textgreek{f}}\in H_{loc}^{1}(\text{\textgreek{S}})$ satisfies
\begin{equation}
\begin{cases}
\square_{g}(e^{-i\text{\textgreek{w}}t}\text{\textgreek{f}})=0,\\
\int_{\mathcal{S}}J_{\text{\textgreek{m}}}^{N}(e^{-i\text{\textgreek{w}}t}\text{\textgreek{f}})n_{\mathcal{S}}^{\text{\textgreek{m}}}<+\infty\mbox{ and }\lim_{r\rightarrow+\infty}(e^{-i\text{\textgreek{w}}_{1}t}\text{\textgreek{f}})|_{\mathcal{S}}=0,
\end{cases}\label{eq:DefinitionResolventFree-1}
\end{equation}
then $\text{\textgreek{f}}\equiv0$. 

For any $0<a\le1$, we will introduce the auxiliary function 
\begin{equation}
\text{\textgreek{f}}_{a}=\text{\textgreek{q}}(t)e^{a\bar{t}}e^{-i\text{\textgreek{w}}t}\text{\textgreek{f}}.\label{eq:AuxiliaryFunction}
\end{equation}
In view of (\ref{eq:DefinitionResolventFree-1}), as well as the fact
that $\text{\textgreek{q}}\equiv0$ for $t\le0$, $\text{\textgreek{f}}_{a}$
satisfies

\begin{equation}
\begin{cases}
\square_{g}\text{\textgreek{f}}_{a}=\mathcal{A}_{cut}+\mathcal{A}_{exp},\\
(\text{\textgreek{y}}_{a},T\text{\textgreek{y}}_{a})|_{t=0}=(0,0),
\end{cases}\label{eq:WaveForCutOffDifference}
\end{equation}
where 
\begin{equation}
\mathcal{A}_{cut}\doteq2\text{\textgreek{q}}^{\prime}(t)\nabla_{\text{\textgreek{m}}}t\nabla^{\text{\textgreek{m}}}(e^{a\bar{t}}e^{-i\text{\textgreek{w}}t}\text{\textgreek{f}})+\square_{g}(\text{\textgreek{q}}(t))e^{a\bar{t}}e^{-i\text{\textgreek{w}}t}\text{\textgreek{f}}\label{eq:ErrorsFromCutOff}
\end{equation}
and 
\begin{equation}
\mathcal{A}_{exp}\doteq2ae^{a\bar{t}}\nabla_{\text{\textgreek{m}}}\bar{t}\nabla^{\text{\textgreek{m}}}\big(e^{-i\text{\textgreek{w}}t}\text{\textgreek{f}}\big)+\big(a^{2}\nabla_{\text{\textgreek{m}}}\bar{t}\nabla^{\text{\textgreek{m}}}\bar{t}+a\square_{g}\bar{t}\big)e^{a\bar{t}}e^{-i\text{\textgreek{w}}t}\text{\textgreek{f}}.\label{eq:ErrorsExponential}
\end{equation}

In view of the fact that 
\begin{equation}
\bar{t}\le t-\frac{1}{2}r-C,\label{eq:LowerBoundTBar}
\end{equation}
we can estimate for any $\text{\textgreek{t}}\ge0$ and any $\text{\textgreek{b}}>0$:
\begin{equation}
\sup_{t=\text{\textgreek{t}}}(1+r)^{\text{\textgreek{b}}}e^{2a\bar{t}}\le C_{\text{\textgreek{b}}}a^{-\text{\textgreek{b}}}e^{2a\text{\textgreek{t}}},\label{eq:BoundExponentialOnTimeSlices}
\end{equation}
where $C_{\text{\textgreek{b}}}>0$ depends only on $\text{\textgreek{b}}$.
In view of the flat asymptotics of $g$ and the fact that $\mathcal{F}$
is compactly supported, the conditions $Im(\text{\textgreek{w}})\ge0$
and $\sum_{j=1}^{2}\int_{\mathcal{S}}J_{\text{\textgreek{m}}}^{N}(e^{i\text{\textgreek{w}}_{j}t}\text{\textgreek{f}}_{j})n^{\mathcal{S}}$
imply that there exists a $\text{\textgreek{b}}>0$ such that 
\begin{equation}
\int_{\text{\textgreek{S}}}\Big((1+r)^{-\text{\textgreek{b}}}|\nabla_{\text{\textgreek{S}}}\text{\textgreek{f}}|_{g_{\text{\textgreek{S}}}}^{2}+(1+r)^{-\text{\textgreek{b}}-2}|\text{\textgreek{f}}|^{2}\Big)<+\infty\label{eq:PolynomialBoundL2}
\end{equation}
(note that, in the case $Im(\text{\textgreek{w}})>0$, $\text{\textgreek{f}}$
in fact belongs to the space $H^{1}(\text{\textgreek{S}})$; the bound
(\ref{eq:PolynomialBoundL2}) becomes non-trivial only when $Im(\text{\textgreek{w}})=0$).
Therefore, (\ref{eq:ErrorsFromCutOff}), (\ref{eq:BoundExponentialOnTimeSlices})
and (\ref{eq:PolynomialBoundL2}), combined with the fact that $\text{\textgreek{q}}^{\prime}\equiv0$
for $t\ge1$, imply that, for any $\text{\textgreek{t}}\ge0$: 
\begin{equation}
\sum_{j=0}^{k}\int_{\{t=\text{\textgreek{t}}\}}(1+r)^{2}|T^{j}\mathcal{A}_{cut}|^{2}\, dg_{\text{\textgreek{S}}}\le\begin{cases}
C_{\text{\textgreek{q}}\text{\textgreek{w}}\text{\textgreek{f}}}a^{-\text{\textgreek{b}}-2}, & 0\le\text{\textgreek{t}}\le1,\\
0, & \text{\textgreek{t}}\ge1,
\end{cases}\label{eq:BoundForCutOffUniqueness}
\end{equation}
where $C_{\text{\textgreek{q}}\text{\textgreek{w}}\text{\textgreek{f}}}>0$
depends on $\text{\textgreek{q}},\text{\textgreek{w}},\text{\textgreek{f}}$.

The finiteness of the $J^{N}$-energy flux of $e^{-i\text{\textgreek{w}}t}\text{\textgreek{f}}$
through $\mathcal{S}$, combined with a Hardy type inequality, implies
that, for any $\text{\textgreek{t}}\ge0$ 
\begin{equation}
\int_{\{\bar{t}=\text{\textgreek{t}}\}}\Big(\big|\nabla_{\text{\textgreek{m}}}\bar{t}\nabla^{\text{\textgreek{m}}}(e^{-i\text{\textgreek{w}}t}\text{\textgreek{f}})\big|^{2}+(1+r)^{-2}|e^{-i\text{\textgreek{w}}t}\text{\textgreek{f}}|\Big)\le Ce^{2Im(\text{\textgreek{w}})\text{\textgreek{t}}}\int_{\mathcal{S}}J_{\text{\textgreek{m}}}^{N}(e^{-i\text{\textgreek{w}}t}\text{\textgreek{f}})n_{\mathcal{S}}^{\text{\textgreek{m}}}<+\infty.\label{eq:BoundFromEnergySlice}
\end{equation}
Since $Im(\text{\textgreek{w}})>0$ and $\mathcal{S}\subset J^{+}(\text{\textgreek{S}})$,
from (\ref{eq:BoundFromEnergySlice}) we can deduce the following
estimate on the slice $\{t=\text{\textgreek{t}}\}$ for any $\text{\textgreek{t}}\ge0$:
\begin{equation}
\int_{\{t=\text{\textgreek{t}}\}}\Big(\big|\nabla_{\text{\textgreek{m}}}\bar{t}\nabla^{\text{\textgreek{m}}}(e^{-i\text{\textgreek{w}}t}\text{\textgreek{f}})\big|^{2}+(1+r)^{-2}|e^{-i\text{\textgreek{w}}t}\text{\textgreek{f}}|\Big)\, dg_{\text{\textgreek{S}}}\le Ce^{2Im(\text{\textgreek{w}})\text{\textgreek{t}}}\int_{\mathcal{S}}J_{\text{\textgreek{m}}}^{N}(e^{-i\text{\textgreek{w}}t}\text{\textgreek{f}})n_{\mathcal{S}}^{\text{\textgreek{m}}}<+\infty.\label{eq:BoundFromEnergySlice-1}
\end{equation}
Therefore, on $\{t\ge0\}$, from (\ref{eq:ErrorsExponential}) and
(\ref{eq:BoundFromEnergySlice-1}), as well as the fact that 
\begin{equation}
|\square_{g}\bar{t}|\le C(1+r)^{-1}\label{eq:WaveTbar}
\end{equation}
(following from the flat asymptotics of $(\mathcal{M},g)$ and the
definition of the hyperboloidal hypersurface $\mathcal{S}$, see also
Section 3.1 of \cite{Moschidisc}), we infer that, for any $\text{\textgreek{t}}\ge0$
and any $0<\text{\textgreek{h}}<\frac{1}{2}$: 
\begin{equation}
\sum_{j=0}^{k}\int_{\{t=\text{\textgreek{t}}\}}(1+r)^{1+2\text{\textgreek{h}}}|T^{j}\mathcal{A}_{exp}|^{2}\, dg_{\text{\textgreek{S}}}\le C_{\text{\textgreek{h}}\text{\textgreek{w}}\text{\textgreek{f}}}a^{1-2\text{\textgreek{h}}}e^{2(a+Im(\text{\textgreek{w}}))\text{\textgreek{t}}}.\label{eq:BoundForExponentialUniqueness}
\end{equation}

Applying the integrated local energy decay estimate (\ref{eq:IledWithLossPhysical})
for (\ref{eq:WaveForCutOffDifference}), we obtain for any $t_{f}\ge0$
in view of (\ref{eq:BoundForCutOffUniqueness}) and (\ref{eq:BoundForExponentialUniqueness}):
\begin{equation}
\int_{\{0\le t\le t_{f}\}}(1+r)^{-1-2\text{\textgreek{h}}}\big(J_{\text{\textgreek{m}}}^{N}(\text{\textgreek{f}}_{a})N^{\text{\textgreek{m}}}+(1+r)^{-2}|\text{\textgreek{f}}_{a}|^{2}\big)\, dg\le C_{\text{\textgreek{h}}\text{\textgreek{q}}\text{\textgreek{w}}\text{\textgreek{f}}}\big(a^{-\text{\textgreek{b}}-2}+a^{1-2\text{\textgreek{h}}}\int_{0}^{t_{f}}e^{2(a+Im(\text{\textgreek{w}}))t}\, dt\big).\label{eq:IledWithLossPhysical-1}
\end{equation}
Thus, in view of the relation (\ref{eq:AuxiliaryFunction}), (\ref{eq:IledWithLossPhysical-1})
yields for any $0<a\le1$, any $t_{f}\ge0$ and any $0<\text{\textgreek{h}}<\frac{1}{2}$:
\begin{equation}
\int_{\text{\textgreek{S}}}(1+r)^{-3-2\text{\textgreek{h}}}|\text{\textgreek{f}}|^{2}\, dg_{\text{\textgreek{S}}}\le C_{\text{\textgreek{h}}\text{\textgreek{q}}\text{\textgreek{w}}\text{\textgreek{f}}}\Big\{ a^{-\text{\textgreek{b}}-2}\Big(\int_{0}^{t_{f}}e^{2(a+Im(\text{\textgreek{w}}))t}\, dt\Big)^{-1}+a^{1-2\text{\textgreek{h}}}\Big\}.\label{eq:BoundForZeroUniqueness}
\end{equation}
Choosing $t_{f}=a^{-1}$ in (\ref{eq:BoundForZeroUniqueness}) and
letting $a\rightarrow0$, we thus infer that $\text{\textgreek{f}}\equiv0$.
Hence, we have established the uniqueness for solutions to (\ref{eq:DefinitionResolventFree}),
and, thus, the resolvent operator $R(\square_{g};\text{\textgreek{w}})$
given by (\ref{eq:ResolventDefinition}) is well defined, with $R(\square_{g};\text{\textgreek{w}})$
being uniformaly bounded on $\{\text{\textgreek{w}}\in\mathbb{C}\,|\, Im(\text{\textgreek{w}})\ge0\}$
with respect to the norm (\ref{eq:ResolventNorm-1}) in view of (\ref{eq:BoundForResolventNorm}).
Furthermore, for any $\mathcal{F}\in L_{cp}^{2}(\text{\textgreek{S}})$,
$\text{\textgreek{f}}=R(\square_{g};\text{\textgreek{w}})\mathcal{F}$
satisfies the bound (\ref{eq:BoundFromNewMethod}).

\subsubsection*{3. H\"older continuity}

We will now show that the operator $R(\square_{g};\text{\textgreek{w}})$
is H\"older continuous as a function of $\text{\textgreek{w}}\in\mathbb{C}$
with $Im(\text{\textgreek{w}})\ge0$, with respect to the norm (\ref{eq:ResolventNormHolder}).
For any given $\mathcal{F}\in L_{cp}^{2}(\text{\textgreek{S}})$ and
any $\text{\textgreek{w}}_{1},\text{\textgreek{w}}_{2}\in$, let us
set $\text{\textgreek{f}}_{1}=R(\square_{g};\text{\textgreek{w}}_{1})$
and $\text{\textgreek{f}}_{2}=R(\square_{g};\text{\textgreek{w}}_{2})$.
We will show that, for any $0<\text{\textgreek{h}}<\frac{1}{2}$,
$0<\text{\textgreek{h}}_{0}<\frac{1}{2}-\text{\textgreek{h}}$ (assuming
without loss of generality that $|\text{\textgreek{w}}_{1}-\text{\textgreek{w}}_{2}|<1$):
\begin{align}
\int_{\text{\textgreek{S}}}(1+r)^{-1-2\text{\textgreek{h}}}\big(\big|\nabla_{g_{\text{\textgreek{S}}}}(\text{\textgreek{f}}_{1}-\text{\textgreek{f}}_{2})\big|_{g_{\text{\textgreek{S}}}}^{2} & +\big(|\text{\textgreek{w}}_{1}|^{2}+(1+r)^{-2}\big)|\text{\textgreek{f}}_{1}-\text{\textgreek{f}}_{2}|^{2}\big)\, dg_{\text{\textgreek{S}}}\le\label{eq:HolderDone-1}\\
 & \le C_{\text{\textgreek{h}}\text{\textgreek{h}}_{0}\text{\textgreek{q}}}|\text{\textgreek{w}}_{1}-\text{\textgreek{w}}_{2}|^{2\text{\textgreek{h}}_{0}}e^{2Im(\text{\textgreek{w}}_{1})}(1+|\text{\textgreek{w}}_{1}|^{k})\int_{\text{\textgreek{S}}}(1+r)^{1+2\text{\textgreek{h}}+\text{\textgreek{h}}_{0}}|\mathcal{F}|^{2}\, dg_{\text{\textgreek{S}}}.\nonumber 
\end{align}

For any $0<a\le1$, we define 
\begin{equation}
\text{\textgreek{y}}_{a}=\text{\textgreek{q}}(t)e^{a\bar{t}}\big(e^{-i\text{\textgreek{w}}_{1}t}\text{\textgreek{f}}_{1}-e^{-i\text{\textgreek{w}}_{2}t}\text{\textgreek{f}}_{2}\big),\label{eq:AuxiliaryFunction-1}
\end{equation}
noticing that $\text{\textgreek{y}}_{a}$ satisfies

\begin{equation}
\begin{cases}
\square_{g}\text{\textgreek{y}}_{a}=\text{\textgreek{q}}(t)e^{a\bar{t}}(e^{-i\text{\textgreek{w}}_{1}t}-e^{i\text{\textgreek{w}}_{2}t})\mathcal{F}+\mathcal{B}_{cut}+\mathcal{B}_{exp},\\
(\text{\textgreek{y}}_{a},T\text{\textgreek{y}}_{a})|_{t=0}=(0,0),
\end{cases}\label{eq:WaveForCutOffDifference-1}
\end{equation}
where 
\begin{equation}
\mathcal{B}_{cut}\doteq2\text{\textgreek{q}}^{\prime}(t)\nabla_{\text{\textgreek{m}}}t\nabla^{\text{\textgreek{m}}}\big(e^{a\bar{t}}(e^{-i\text{\textgreek{w}}_{1}t}\text{\textgreek{f}}_{1}-e^{-i\text{\textgreek{w}}_{2}t}\text{\textgreek{f}}_{2})\big)+\square_{g}(\text{\textgreek{q}}(t))\big(e^{a\bar{t}}(e^{-i\text{\textgreek{w}}_{1}t}\text{\textgreek{f}}_{1}-e^{-i\text{\textgreek{w}}_{2}t}\text{\textgreek{f}}_{2})\big)\label{eq:ErrorsFromCutOff-1}
\end{equation}
and 
\begin{equation}
\mathcal{B}_{exp}\doteq2ae^{a\bar{t}}\nabla_{\text{\textgreek{m}}}\bar{t}\nabla^{\text{\textgreek{m}}}\big(e^{-i\text{\textgreek{w}}_{1}t}\text{\textgreek{f}}_{1}-e^{-i\text{\textgreek{w}}_{2}t}\text{\textgreek{f}}_{2}\big)+\big(a^{2}\nabla_{\text{\textgreek{m}}}\bar{t}\nabla^{\text{\textgreek{m}}}\bar{t}+a\square_{g}\bar{t}\big)e^{a\bar{t}}\big(e^{-i\text{\textgreek{w}}_{1}t}\text{\textgreek{f}}_{1}-e^{-i\text{\textgreek{w}}_{2}t}\text{\textgreek{f}}_{2}\big).\label{eq:ErrorsExponential-1}
\end{equation}

Since $\text{\textgreek{f}}_{1},\text{\textgreek{f}}_{2}$ satisfy
the bound (\ref{eq:BoundForResolventNorm}) for $\text{\textgreek{w}}_{1},\text{\textgreek{w}}_{2}$
in place of $\text{\textgreek{w}}$, respectively, in view of (\ref{eq:BoundExponentialOnTimeSlices})
we can estimate for any $\text{\textgreek{t}}\ge0$ and any $0<\text{\textgreek{h}}<\frac{1}{2}$:
\begin{equation}
\sum_{j=1}^{k}\int_{\{t=\text{\textgreek{t}}\}}(1+r)^{1+2\text{\textgreek{h}}}|T^{j}\mathcal{B}_{cut}|\, dg_{\text{\textgreek{S}}}\le\begin{cases}
C_{\text{\textgreek{h}}\text{\textgreek{q}}}a^{-3-\text{\textgreek{h}}}\sum_{j=1}^{2}e^{2Im(\text{\textgreek{w}}_{j})}(1+|\text{\textgreek{w}}_{j}|^{k})\int_{\text{\textgreek{S}}}(1+r)^{1+2\text{\textgreek{h}}}|\mathcal{F}|^{2}\, dg_{\text{\textgreek{S}}}, & 0\le\text{\textgreek{t}}\le1\\
0, & \text{\textgreek{t}}\ge1.
\end{cases}\label{eq:BoundForCutOffHolder}
\end{equation}
Furthermore, in view of the bound (\ref{eq:BoundFromNewMethod}) for
$\text{\textgreek{f}}_{1},\text{\textgreek{f}}_{2}$ (with $\text{\textgreek{h}}+\text{\textgreek{h}}_{0}$
in place of $\text{\textgreek{h}}$ there) and (\ref{eq:WaveTbar}),
we can estimate for any $\text{\textgreek{t}}\ge0$, any $0<\text{\textgreek{h}}<\frac{1}{2}$
and any $0<\text{\textgreek{h}}_{0}<\frac{1}{2}-\text{\textgreek{h}}$:
\begin{equation}
\sum_{j=1}^{k}\int_{\{t=\text{\textgreek{t}}\}}(1+r)^{1+2\text{\textgreek{h}}}|T^{j}\mathcal{B}_{exp}|^{2}\, dg_{\text{\textgreek{S}}}\le C_{\text{\textgreek{h}}\text{\textgreek{h}}_{0}\text{\textgreek{q}}}\sum_{j=1}^{2}a^{2\text{\textgreek{h}}_{0}}e^{2(Im(\text{\textgreek{w}}_{j})+a)(\text{\textgreek{t}}+1)}(1+|\text{\textgreek{w}}_{j}|^{k})\int_{\text{\textgreek{S}}}(1+r)^{1+2\text{\textgreek{h}}+2\text{\textgreek{h}}_{0}}|\mathcal{F}|^{2}\, dg_{\text{\textgreek{S}}}.\label{eq:BoundForExponentialHolder}
\end{equation}

Applying the integrated local energy decay estimate (\ref{eq:IledWithLossPhysical})
for (\ref{eq:WaveForCutOffDifference-1}), we obtain in view of (\ref{eq:BoundForCutOffHolder})
and (\ref{eq:BoundForExponentialHolder}) for any $t_{f}\ge0$ and
any $0<\text{\textgreek{h}}<\frac{1}{2},0<\text{\textgreek{h}}_{0}<\frac{1}{2}-\text{\textgreek{h}}$:
\begin{align}
\int_{\{0\le t\le t_{f}\}}(1+r)^{-1-2\text{\textgreek{h}}} & \big(J_{\text{\textgreek{m}}}^{N}(\text{\textgreek{y}}_{a})N^{\text{\textgreek{m}}}+(1+r)^{-2}|\text{\textgreek{y}}_{a}|^{2}\big)\, dg\le\label{eq:IledWithLossPhysical-1-1}\\
\le C_{\text{\textgreek{h}}\text{\textgreek{h}}_{0}\text{\textgreek{q}}}\Big\{ & \int_{\{0\le t\le t_{f}\}}\text{\textgreek{q}}(t)(1+r)^{1+2\text{\textgreek{h}}}e^{a\bar{t}}\big|(e^{-i\text{\textgreek{w}}_{1}t}-e^{i\text{\textgreek{w}}_{2}t})\mathcal{F}\big|^{2}\, dg+\nonumber \\
 & +\sum_{j=1}^{2}\Big(a^{-3-\text{\textgreek{h}}}+a^{2\text{\textgreek{h}}_{0}}\int_{0}^{t_{f}}e^{2(a+Im(\text{\textgreek{w}}_{j}))(t+1)}\, dt\Big)(1+|\text{\textgreek{w}}_{j}|^{k})\int_{\text{\textgreek{S}}}(1+r)^{1+2\text{\textgreek{h}}+2\text{\textgreek{h}}_{0}}|\mathcal{F}|^{2}\, dg_{\text{\textgreek{S}}}\Big\}.\nonumber 
\end{align}
The relation (\ref{eq:AuxiliaryFunction-1}) implies that 
\begin{equation}
\text{\textgreek{y}}_{a}=\text{\textgreek{q}}(t)e^{a\bar{t}}e^{-i\text{\textgreek{w}}_{1}t}(\text{\textgreek{f}}_{1}-\text{\textgreek{f}}_{2})+\text{\textgreek{q}}(t)e^{a\bar{t}}(e^{-i\text{\textgreek{w}}_{1}t}-e^{-i\text{\textgreek{w}}_{2}t})\text{\textgreek{f}}_{2}.\label{eq:DifferenceExpanded}
\end{equation}
Thus, (\ref{eq:IledWithLossPhysical-1-1}) and (\ref{eq:DifferenceExpanded})
yield 
\begin{align}
\int_{\{1\le t\le t_{f}\}}(1+r)^{-1-2\text{\textgreek{h}}} & e^{2a\bar{t}}\big(J_{\text{\textgreek{m}}}^{N}(e^{-i\text{\textgreek{w}}_{1}t}(\text{\textgreek{f}}_{1}-\text{\textgreek{f}}_{2}))N^{\text{\textgreek{m}}}+(1+r)^{-2}|e^{-i\text{\textgreek{w}}_{1}t}(\text{\textgreek{f}}_{1}-\text{\textgreek{f}}_{2})|^{2}\big)\, dg\le\label{eq:IledWithLossPhysical-1-1-1}\\
\le & C_{\text{\textgreek{q}}}\int_{\{0\le t\le t_{f}\}}(1+r)^{-1-2\text{\textgreek{h}}}e^{a\bar{t}}\big(J_{\text{\textgreek{m}}}^{N}((e^{-i\text{\textgreek{w}}_{1}t}-e^{-i\text{\textgreek{w}}_{2}t})\text{\textgreek{f}}_{2})N^{\text{\textgreek{m}}}+(1+r)^{-2}|(e^{-i\text{\textgreek{w}}_{1}t}-e^{-i\text{\textgreek{w}}_{2}t})\text{\textgreek{f}}_{2}|^{2}\big)\, dg+\nonumber \\
 & +C_{\text{\textgreek{h}}\text{\textgreek{h}}_{0}\text{\textgreek{q}}}\Big\{\sum_{l=0}^{k}\int_{\{0\le t\le t_{f}\}}\text{\textgreek{q}}(t)(1+r)^{1+2\text{\textgreek{h}}}e^{a\bar{t}}\big|(\text{\textgreek{w}}_{1}^{l}e^{-i\text{\textgreek{w}}_{1}t}-\text{\textgreek{w}}_{2}^{l}e^{i\text{\textgreek{w}}_{2}t})\mathcal{F}\big|^{2}\, dg+\nonumber \\
 & \hphantom{+C_{\text{\textgreek{h}}\text{\textgreek{h}}_{0}\text{\textgreek{q}}}\Big\{}+\sum_{j=1}^{2}\Big(a^{-3-\text{\textgreek{h}}}+a^{2\text{\textgreek{h}}_{0}}\int_{0}^{t_{f}}e^{2(a+Im(\text{\textgreek{w}}_{j}))(t+1)}\, dt\Big)(1+|\text{\textgreek{w}}_{j}|^{k})\int_{\text{\textgreek{S}}}(1+r)^{1+2\text{\textgreek{h}}+2\text{\textgreek{h}}_{0}}|\mathcal{F}|^{2}\, dg_{\text{\textgreek{S}}}\Big\}.\nonumber 
\end{align}
 From (\ref{eq:IledWithLossPhysical-1-1-1}), we readily obtain using
the bound (\ref{eq:BoundForResolventNorm}) for $\text{\textgreek{f}}_{2}$
for the first term in the right hand side: 
\begin{align}
\Big(\int_{1}^{t_{f}}e^{2(a+Im(\text{\textgreek{w}}_{1}))\text{\textgreek{t}}} & \, d\text{\textgreek{t}}\Big)\int_{\text{\textgreek{S}}}(1+r)^{-1-2\text{\textgreek{h}}}\big(\big|\nabla_{g_{\text{\textgreek{S}}}}(\text{\textgreek{f}}_{1}-\text{\textgreek{f}}_{2})\big|_{g_{\text{\textgreek{S}}}}^{2}+\big(|\text{\textgreek{w}}_{1}|^{2}+(1+r)^{-2}\big)|\text{\textgreek{f}}_{1}-\text{\textgreek{f}}_{2}|^{2}\big)\, dg_{\text{\textgreek{S}}}\le\label{eq:AlmostDoneWithHolder}\\
\le & C_{\text{\textgreek{q}}\text{\textgreek{h}}}\Big(\sum_{j=1}^{2}\int_{0}^{t_{f}}\big(|e^{(a-i\text{\textgreek{w}}_{1})\text{\textgreek{t}}}-e^{(a-i\text{\textgreek{w}}_{2})\text{\textgreek{t}}}|^{2}+\frac{|\text{\textgreek{w}}_{1}-\text{\textgreek{w}}_{2}|}{1+|\text{\textgreek{w}}_{j}|}e^{2(a+Im(\text{\textgreek{w}}_{1}))\text{\textgreek{t}}}\big)\, d\text{\textgreek{t}}\Big)e^{2Im(\text{\textgreek{w}}_{j})}(1+|\text{\textgreek{w}}_{j}|^{k})\int_{\text{\textgreek{S}}}(1+r)^{1+2\text{\textgreek{h}}}|\mathcal{F}|^{2}\, dg_{\text{\textgreek{S}}}\Big\}+\nonumber \\
 & +C_{\text{\textgreek{h}}\text{\textgreek{h}}_{0}\text{\textgreek{q}}}\Big\{\sum_{l=0}^{k}\int_{0}^{t_{f}}|\text{\textgreek{w}}_{1}^{l}e^{(a-i\text{\textgreek{w}}_{1})\text{\textgreek{t}}}-\text{\textgreek{w}}_{2}^{l}e^{(a-i\text{\textgreek{w}}_{2})\text{\textgreek{t}}}|^{2}\, d\text{\textgreek{t}}\int_{\text{\textgreek{S}}}(1+r)^{1+2\text{\textgreek{h}}}|\mathcal{F}|^{2}\, dg_{\text{\textgreek{S}}}+\nonumber \\
 & \hphantom{+C_{\text{\textgreek{h}}\text{\textgreek{h}}_{0}\text{\textgreek{q}}}\Big\{}+\sum_{j=1}^{2}\Big(a^{-3-\text{\textgreek{h}}}+a^{2\text{\textgreek{h}}_{0}}\int_{0}^{t_{f}}e^{2(a+Im(\text{\textgreek{w}}_{j}))(\text{\textgreek{t}}+1)}\, d\text{\textgreek{t}}\Big)(1+|\text{\textgreek{w}}_{j}|^{k})\int_{\text{\textgreek{S}}}(1+r)^{1+2\text{\textgreek{h}}+2\text{\textgreek{h}}_{0}}|\mathcal{F}|^{2}\, dg_{\text{\textgreek{S}}}\Big\}\le\nonumber \\
\le & C_{\text{\textgreek{h}}\text{\textgreek{h}}_{0}\text{\textgreek{q}}}\sum_{j=1}^{2}\Big\{\Big(\int_{0}^{t_{f}}e^{2(a+Im(\text{\textgreek{w}}_{j}))(\text{\textgreek{t}}+1)}\big(|1-e^{i(\text{\textgreek{w}}_{2}-\text{\textgreek{w}}_{1})\text{\textgreek{t}}}|^{2}+\frac{|\text{\textgreek{w}}_{1}-\text{\textgreek{w}}_{2}|}{1+|\text{\textgreek{w}}_{j}|}+a^{2\text{\textgreek{h}}_{0}}+\frac{a^{-3-\text{\textgreek{h}}}}{e^{2(a+Im(\text{\textgreek{w}}_{j}))(\text{\textgreek{t}}+1)}}\big)\, d\text{\textgreek{t}}\Big)\times\nonumber \\
 & \hphantom{C_{\text{\textgreek{h}}\text{\textgreek{h}}_{0}\text{\textgreek{q}}}\sum_{j=1}^{2}\Big\{\Big(\int_{0}^{t_{f}}e^{2(a+Im(\text{\textgreek{w}}_{j}))(\text{\textgreek{t}}+1)}\big(|1-e^{i(\text{\textgreek{w}}_{2}-\text{\textgreek{w}}_{1})\text{\textgreek{t}}}|^{2}+}\times e^{2Im(\text{\textgreek{w}}_{j})}(1+|\text{\textgreek{w}}_{j}|^{k})\int_{\text{\textgreek{S}}}(1+r)^{1+2\text{\textgreek{h}}+\text{\textgreek{h}}_{0}}|\mathcal{F}|^{2}\, dg_{\text{\textgreek{S}}}\Big\}\nonumber 
\end{align}
and thus: 
\begin{align}
\int_{\text{\textgreek{S}}} & (1+r)^{-1-2\text{\textgreek{h}}}\big(\big|\nabla_{g_{\text{\textgreek{S}}}}(\text{\textgreek{f}}_{1}-\text{\textgreek{f}}_{2})\big|_{g_{\text{\textgreek{S}}}}^{2}+\big(|\text{\textgreek{w}}_{1}|^{2}+(1+r)^{-2}\big)|\text{\textgreek{f}}_{1}-\text{\textgreek{f}}_{2}|^{2}\big)\, dg_{\text{\textgreek{S}}}\le\label{eq:AlmostDoneWithHolder-1}\\
\le & C_{\text{\textgreek{h}}\text{\textgreek{h}}_{0}\text{\textgreek{q}}}\Big(\int_{1}^{t_{f}}e^{2(a+Im(\text{\textgreek{w}}_{1}))\text{\textgreek{t}}}\, d\text{\textgreek{t}}\Big)^{-1}\sum_{j=1}^{2}\Big\{\Big(\int_{0}^{t_{f}}e^{2(a+Im(\text{\textgreek{w}}_{j}))(\text{\textgreek{t}}+1)}\big(|1-e^{i(\text{\textgreek{w}}_{2}-\text{\textgreek{w}}_{1})\text{\textgreek{t}}}|^{2}+\frac{|\text{\textgreek{w}}_{1}-\text{\textgreek{w}}_{2}|}{1+|\text{\textgreek{w}}_{j}|}+a^{2\text{\textgreek{h}}_{0}}+\frac{a^{-3-\text{\textgreek{h}}}}{e^{2(a+Im(\text{\textgreek{w}}_{j}))(\text{\textgreek{t}}+1)}}\big)\, d\text{\textgreek{t}}\Big)\times\nonumber \\
 & \hphantom{C_{\text{\textgreek{h}}\text{\textgreek{h}}_{0}\text{\textgreek{q}}}\sum_{j=1}^{2}\Big\{\Big(\int_{0}^{t_{f}}e^{2(a+Im(\text{\textgreek{w}}_{j}))(\text{\textgreek{t}}+1)}\big(|1-e^{i(\text{\textgreek{w}}_{2}-\text{\textgreek{w}}_{1})\text{\textgreek{t}}}|^{2}+++++++++}\times e^{2Im(\text{\textgreek{w}}_{j})}(1+|\text{\textgreek{w}}_{j}|^{k})\int_{\text{\textgreek{S}}}(1+r)^{1+2\text{\textgreek{h}}+\text{\textgreek{h}}_{0}}|\mathcal{F}|^{2}\, dg_{\text{\textgreek{S}}}\Big\}\nonumber 
\end{align}
Therefore, choosing $t_{f}=a^{-1-\text{\textgreek{h}}_{0}}$ and $a=|\text{\textgreek{w}}_{1}-\text{\textgreek{w}}_{2}|^{1-4\text{\textgreek{h}}_{0}}$,
assuming without loss of generality that $|\text{\textgreek{w}}_{1}-\text{\textgreek{w}}_{2}|<1$,
(\ref{eq:AlmostDoneWithHolder-1}) readily yields the required H\"older
continuity estimate: 
\begin{align}
\int_{\text{\textgreek{S}}}(1+r)^{-1-2\text{\textgreek{h}}}\big(\big|\nabla_{g_{\text{\textgreek{S}}}}(\text{\textgreek{f}}_{1}-\text{\textgreek{f}}_{2})\big|_{g_{\text{\textgreek{S}}}}^{2} & +\big(|\text{\textgreek{w}}_{1}|^{2}+(1+r)^{-2}\big)|\text{\textgreek{f}}_{1}-\text{\textgreek{f}}_{2}|^{2}\big)\, dg_{\text{\textgreek{S}}}\le\label{eq:HolderDone}\\
 & \le C_{\text{\textgreek{h}}\text{\textgreek{h}}_{0}\text{\textgreek{q}}}|\text{\textgreek{w}}_{1}-\text{\textgreek{w}}_{2}|^{2\text{\textgreek{h}}_{0}}e^{2Im(\text{\textgreek{w}}_{1})}(1+|\text{\textgreek{w}}_{1}|^{k})\int_{\text{\textgreek{S}}}(1+r)^{1+2\text{\textgreek{h}}+\text{\textgreek{h}}_{0}}|\mathcal{F}|^{2}\, dg_{\text{\textgreek{S}}}.\nonumber 
\end{align}

\subsubsection*{4. Holomorphicity}

Finally, we will show that for any $0<\text{\textgreek{h}}<\frac{1}{2}$
and any $\mathcal{F},\text{\textgreek{f}}_{0}\in L_{cp}^{2}(\text{\textgreek{S}})$,
the inner product $\left\langle \text{\textgreek{f}}_{0},R(\square_{g};\text{\textgreek{w}})\mathcal{F}\right\rangle _{L^{2}(\text{\textgreek{S}})}$
is a holomorphic function of $\text{\textgreek{w}}$ when $Im(\text{\textgreek{w}})>0$.
This can be readily established using the classical Morera's theorem,
since, as we showed, $\left\langle \text{\textgreek{f}}_{0},R(\square_{g};\text{\textgreek{w}})\mathcal{F}\right\rangle _{L^{2}(\text{\textgreek{S}})}$
is continuous in $\text{\textgreek{w}}$, and the right hand side
of (\ref{eq:DefinitionResolventFree}) vanishes upon complex integration
over any piecewise smooth closed loop $\text{\textgreek{g}}\subset\big\{\text{\textgreek{w}}\in\mathbb{C}:\, Im(\text{\textgreek{w}}\}>0\big\}$
(and, hence, the same arguments leading to the uniqueness of solutions
to (\ref{eq:DefinitionResolventFree}) show that $\int_{\text{\textgreek{g}}}R(\square_{g};\text{\textgreek{w}})\mathcal{F}\, d\text{\textgreek{w}}$
also vanishes).

\section{\label{sec:SpectralConsequencesAppendix}Proof of Proposition \ref{prop:ILEDProductCase}}

\noindent In order to establish the integrated local energy decay
estimate (\ref{eq:IledWithLossPhysical}) for equation (\ref{eq:PotentialOnceMore}),
we will first show that the condition (\ref{eq:NearZeroFrequencyBound})
implies that, for all $\text{\textgreek{l}}\in[0,1]$, $R(\square_{g}-V_{\text{\textgreek{l}}};\cdot)$
has no poles in the half plane $\{\text{\textgreek{w}}:\, Im(\text{\textgreek{w}})>0\}$,
and the following bound holds on the strip $\{0\le Im(\text{\textgreek{w}})\le\frac{\text{\textgreek{e}}}{2}\}$:
\begin{equation}
\sup_{0\le Im(\text{\textgreek{w}})\le\frac{\text{\textgreek{e}}}{2}}||R(\square_{g}-V_{\text{\textgreek{l}}};\text{\textgreek{w}})||_{\mathcal{L},\text{\textgreek{h}}}<+\infty.\label{eq:ILEDResolventBoundedness}
\end{equation}

The effective limiting absorption principles established in \cite{Rodnianski2011}
imply that, provided $\text{\textgreek{e}}>0$ is sufficiently small
depending on the precise choice of the family $V_{\text{\textgreek{l}}}$,
the following non-zero real frequency bound holds for all $\text{\textgreek{l}}\in[0,1]$:
\begin{equation}
\sup_{\{0\le Im(\text{\textgreek{w}})\le\frac{\text{\textgreek{e}}}{2}\}\cap\{|\text{\textgreek{w}}|\ge\text{\textgreek{e}}\}}||R(\square_{g}-V_{\text{\textgreek{l}}};\text{\textgreek{w}})||_{\mathcal{L},\text{\textgreek{h}}}<+\infty.\label{eq:BoundFromTaoRodnianski}
\end{equation}
Thus, (\ref{eq:NearZeroFrequencyBound}) and (\ref{eq:BoundFromTaoRodnianski})
yield (\ref{eq:ILEDResolventBoundedness}) for all $\text{\textgreek{l}}\in[0,1]$.
The non-existence of poles for $R(\square_{g}-V_{\text{\textgreek{l}}};\cdot)$
in the half plane $\{\text{\textgreek{w}}:\, Im(\text{\textgreek{w}})>0\}$
follows readily from the following facts:

\begin{enumerate}

\item The poles of $R(\square_{g}-V_{\text{\textgreek{l}}};\cdot)$
in $\{\text{\textgreek{w}}:\, Im(\text{\textgreek{w}})>0\}$ vary
continuously with $\text{\textgreek{l}}$, except when reaching the
real axis (see Lemma \ref{lem:PropertiesResolventNotFree}).

\item $R(\square_{g}-V_{0};\cdot)=R(\square_{g};\cdot)$ has no poles
in $\{\text{\textgreek{w}}:\, Im(\text{\textgreek{w}})>0\}$. 

\item The bound (\ref{eq:ILEDResolventBoundedness}) guarantees that
no poles of $R(\square_{g}-V_{\text{\textgreek{l}}};\cdot)$ exist
in $\{0\le Im(\text{\textgreek{w}})\le\text{\textgreek{e}}\}$ for
all $\text{\textgreek{l}}\in[0,1]$.

\item There exists some $C\gg1$ depending on the family $V_{\text{\textgreek{l}}}$
so that $R(\square_{g}-V_{\text{\textgreek{l}}};\cdot)$ has no poles
in the region $\{|\text{\textgreek{w}}|\ge C\}$. This follows from
the fact that all the poles of $R(\square_{g}-V_{\text{\textgreek{l}}};\cdot)$
in $\{\text{\textgreek{w}}\in\mathbb{C}:\, Im(\text{\textgreek{w}})>0\}$
must lie on the imaginary semi-axis $\{\text{\textgreek{w}}=ia,a>0\}$
(in view of the fact that $\text{\textgreek{D}}_{g_{\text{\textgreek{S}}}}-V_{\text{\textgreek{l}}}$
is essentially self-adjoint), combined with the fact that $R(\square_{g}-V_{\text{\textgreek{l}}};\cdot)$
is holomorphic in the region $Im(\text{\textgreek{w}})\gg1$ (see
the definition and the remark below (\ref{eq:DefinitionResolventPotential})).

\end{enumerate}

We will now proceed to establish the integrated local energy decay
estimate (\ref{eq:IledWithLossPhysical}). The bound (\ref{eq:ILEDResolventBoundedness})
readily implies (after an application of the Fourier transform in
the $t$-variable) the following bound for any smooth $\tilde{\text{\textgreek{y}}}:\mathcal{M}\rightarrow\mathbb{C}$
such that the restriction of $\tilde{\text{\textgreek{y}}}$ on the
$\{t=const\}$ hypersurfaces is compactly supported and both $\tilde{\text{\textgreek{y}}}$
and its first derivatives are square integrable in $t$: 
\begin{equation}
\int_{\mathcal{M}}(1+r)^{-1-2\text{\textgreek{h}}}\big(J_{\text{\textgreek{m}}}^{T}(\tilde{\text{\textgreek{y}}})T^{\text{\textgreek{m}}}+r_{+}^{-2}|\tilde{\text{\textgreek{y}}}|^{2}\big)\lesssim_{\text{\textgreek{l}}}\sum_{j=0}^{k}\int_{\mathcal{M}}(1+r)^{1+2\text{\textgreek{h}}}|T^{j}(\square_{g}\tilde{\text{\textgreek{y}}})|^{2}.\label{eq:IledWithLossInhomogeneous}
\end{equation}
Furthermore, the bound (\ref{eq:ILEDResolventBoundedness}) combined
with the absence of poles for $R(\square_{g}-V_{\text{\textgreek{l}}};\cdot)$
in the upper half plane imply that for all $\text{\textgreek{l}}\in[0,1]$:
\begin{equation}
\inf_{\text{\textgreek{f}}\in C_{0}^{\infty}(\text{\textgreek{S}})}\frac{\int_{\text{\textgreek{S}}}\big(|\nabla_{g_{\text{\textgreek{S}}}}\text{\textgreek{f}}|^{2}+V_{\text{\textgreek{l}}}|\text{\textgreek{f}}|^{2}\big)\, dg_{\text{\textgreek{S}}}}{\int_{supp(V_{\text{\textgreek{l}}})}|\text{\textgreek{f}}|^{2}\, dg_{\text{\textgreek{S}}}}>0.\label{eq:PositivityEnergy}
\end{equation}
The lower bound (\ref{eq:PositivityEnergy}), in turn, implies that
the $T$-energy flux (associated to the problem (\ref{eq:PotentialOnceMore}),
$\text{\textgreek{l}}\in[0,1]$) of any smooth and suitably decaying
function $\text{\textgreek{y}}:\mathcal{M}\rightarrow\mathbb{R}$
, is positive definite on the $\{\bar{t}=const\}$ hypersurfaces,
i.\,e. for any $s\in\mathbb{R}$:
\begin{equation}
\int_{\{\bar{t}=s\}}\big(J_{\text{\textgreek{m}}}^{T}(\text{\textgreek{y}})\bar{n}^{\text{\textgreek{m}}}+r_{+}^{-2}|\text{\textgreek{y}}|^{2}\big)\lesssim_{\text{\textgreek{l}}}\int_{\{\bar{t}=s\}}\big(J_{\text{\textgreek{m}}}^{T}(\text{\textgreek{y}})\bar{n}^{\text{\textgreek{m}}}+V_{\text{\textgreek{l}}}|\text{\textgreek{y}}|^{2}\big).\label{eq:Positivity of energyFlux}
\end{equation}

\begin{rem*}
\noindent The lower bound (\ref{eq:PositivityEnergy}) can be obtained
as follows: Assume that (\ref{eq:PositivityEnergy}) fails to hold
for some $\text{\textgreek{l}}\in[0,1]$, then there exists a smooth
and compactly supported function $V_{\text{\textgreek{l}},2}:\text{\textgreek{S}}\rightarrow(-\infty,0]$
with $V_{\text{\textgreek{l}},2}=-1$ on $supp(V_{\text{\textgreek{l}}})$,
such that for all $\text{\textgreek{d}}>0$:
\begin{equation}
\inf_{\text{\textgreek{f}}\in C_{0}^{\infty}(\text{\textgreek{S}})}\frac{\int_{\text{\textgreek{S}}}\big(|\nabla_{g_{\text{\textgreek{S}}}}\text{\textgreek{f}}|^{2}+(V_{\text{\textgreek{l}}}+\text{\textgreek{d}}V_{\text{\textgreek{l}},2})|\text{\textgreek{f}}|^{2}\big)\, dg_{\text{\textgreek{S}}}}{\int_{supp(V_{\text{\textgreek{l}}})}|\text{\textgreek{f}}|^{2}\, dg_{\text{\textgreek{S}}}}\le-\text{\textgreek{d}}<0.\label{eq:ContradictionForNegativity}
\end{equation}
Thus, in view of the compactness of the support of $V_{\text{\textgreek{l}}}+\text{\textgreek{d}}V_{\text{\textgreek{l}},2}$,
a standard minimization argument (see e.\,g.~\cite{Reed1972}) yields
that for any $\text{\textgreek{d}}>0$, there exists a $\text{\textgreek{l}}_{\text{\textgreek{d}}}>0$
and an $L^{2}$ solution $\text{\textgreek{f}}_{\text{\textgreek{d}}}$
to the eigenvalue problem:
\begin{equation}
\text{\textgreek{D}}_{g_{\text{\textgreek{S}}}}\text{\textgreek{f}}_{\text{\textgreek{d}}}-(V_{\text{\textgreek{l}}}+\text{\textgreek{d}}V_{\text{\textgreek{l}},2})\text{\textgreek{f}}_{\text{\textgreek{d}}}=\text{\textgreek{l}}_{\text{\textgreek{d}}}^{2}\text{\textgreek{f}}_{\text{\textgreek{d}}},\label{eq:Eigenvalue}
\end{equation}
i.\,e.~$R(\square_{g}-V_{\text{\textgreek{l}}}-\text{\textgreek{d}}V_{\text{\textgreek{l}},2};\cdot)$
has a pole at $\text{\textgreek{w}}=i\text{\textgreek{l}}_{\text{\textgreek{d}}}$
for any $\text{\textgreek{d}}>0$. However, for $\text{\textgreek{d}}>0$
sufficiently small, the estimate (\ref{eq:ILEDResolventBoundedness})
also holds for $R(\square_{g}-V_{\text{\textgreek{l}}}-\text{\textgreek{d}}V_{\text{\textgreek{l}},2};\cdot)$
in place of $R(\square_{g}-V_{\text{\textgreek{l}}};\cdot)$,%
\footnote{In view of the fact that $R(\square_{g}-V_{\text{\textgreek{l}}}-\text{\textgreek{d}}V_{\text{\textgreek{l}},2})=\big(1-R(\square_{g}-V_{\text{\textgreek{l}}})\circ\text{\textgreek{d}}V_{\text{\textgreek{l}},2}\big)^{-1}R(\square_{g}-V_{\text{\textgreek{l}}})$%
} and thus our previous analysis establishing the absence of resonances
in the upper half plane for the family $R(\square_{g}-V_{\text{\textgreek{l}}};\cdot)$
(with parameter $\text{\textgreek{l}}$) also applies for the family
$R(\square_{g}-V_{\text{\textgreek{l}}}-\text{\textgreek{d}}V_{\text{\textgreek{l}},2};\cdot)$
(with parameter $\text{\textgreek{d}}$), yielding a contradiction.
\end{rem*}
Let, now, $\text{\textgreek{y}}$ be a smooth solution to (\ref{eq:PotentialOnceMore})
for some $\text{\textgreek{l}}\in[0,1]$ and some smooth function
$F:\mathcal{M}\rightarrow\mathbb{C}$, such that $supp(F)\subset\{t\ge0\}$
(so that $\text{\textgreek{y}}\equiv0$ for $t\le0$). We will assume
without loss of generality that $F$ is compactly supported in $\mathcal{M}$,
and we will show that 
\begin{equation}
\int_{\{t\ge0\}}(1+r)^{-1-2\text{\textgreek{h}}}\big(J_{\text{\textgreek{m}}}^{N}(\text{\textgreek{y}})N^{\text{\textgreek{m}}}+(1+r)^{-2}|\text{\textgreek{y}}|^{2}\big)\, dg\lesssim_{\text{\textgreek{h}}}\sum_{j=0}^{k}\int_{\{t\ge0\}}(1+r)^{1+2\text{\textgreek{h}}}|T^{j}F|^{2}\, dg.\label{eq:IledWithLossPhysicalCompactSupport}
\end{equation}

The estimate (\ref{eq:IledWithLossPhysical}) for any $t_{f}\ge0$
and any function $F$ which does not necessarily have compact support
in the $t$ variable (but with $supp(F)\cap\{t=\text{\textgreek{t}}\}$
being compact for any $\text{\textgreek{t}}\ge0$) can be obtained
from (\ref{eq:IledWithLossPhysicalCompactSupport}) as follows: Fixing
a smooth function $\text{\textgreek{q}}_{1}:\mathbb{R}\rightarrow[0,1]$
such that $\text{\textgreek{q}}_{1}\equiv1$ on $(-\infty,-1]$ and
$\text{\textgreek{q}}_{1}\equiv0$ on $[0,+\infty)$, let us define
for any $t_{f}\ge0$ the function $\text{\textgreek{y}}_{t_{f}}:\mathcal{M}\rightarrow\mathbb{C}$
by solving 
\begin{equation}
\begin{cases}
(\square_{g}-V_{\text{\textgreek{l}}})\text{\textgreek{y}}_{t_{f}}=\text{\textgreek{q}}_{1}(t-t_{f})F,\\
(\text{\textgreek{y}}_{t_{f}},T\text{\textgreek{y}}_{t_{f}})|_{t=0}=(0,0).
\end{cases}\label{eq:ModifiedCauchyProblem}
\end{equation}
Note that $\text{\textgreek{y}}\equiv\text{\textgreek{y}}_{t_{f}}$
on $\{0\le t\le t_{f}-1\}$ in view of (\ref{eq:PotentialOnceMore}).
Since $\text{\textgreek{q}}_{1}(t-t_{f})F$ has compact support in
$\mathcal{M}$, an application of (\ref{eq:IledWithLossPhysicalCompactSupport})
for $\text{\textgreek{y}}_{t_{f}}$ yields: 
\begin{equation}
\int_{\{t\ge0\}}(1+r)^{-1-2\text{\textgreek{h}}}\big(J_{\text{\textgreek{m}}}^{N}(\text{\textgreek{y}}_{t_{f}})N^{\text{\textgreek{m}}}+(1+r)^{-2}|\text{\textgreek{y}}_{t_{f}}|^{2}\big)\, dg\lesssim_{\text{\textgreek{h}}}\sum_{j=0}^{k}\int_{\{0\le t\le t_{f}\}}(1+r)^{1+2\text{\textgreek{h}}}|T^{j}F|^{2}\, dg.\label{eq:IledWithLossPhysicalCompactSupport-1}
\end{equation}
Furthermore, the domain of dependence property for (\ref{eq:PotentialOnceMore})
combined with local-in-time energy estimates and a Cauchy--Schwarz
inequality readily yield: 
\begin{align}
\int_{\{t_{f}-1\le t\le t_{f}\}}(1+r)^{-1-2\text{\textgreek{h}}} & \big(J_{\text{\textgreek{m}}}^{N}(\text{\textgreek{y}})N^{\text{\textgreek{m}}}+(1+r)^{-2}|\text{\textgreek{y}}|^{2}\big)\, dg\lesssim\label{eq:IledWithLossPhysicalCompactSupport-1-1}\\
\lesssim & \int_{\{t_{f}-2\le t\le t_{f}-1\}}(1+r)^{-1-2\text{\textgreek{h}}}\big(J_{\text{\textgreek{m}}}^{N}(\text{\textgreek{y}})N^{\text{\textgreek{m}}}+(1+r)^{-2}|\text{\textgreek{y}}|^{2}\big)\, dg+\nonumber \\
 & +\int_{\{t_{f}-2\le t\le t_{f}\}}(1+r)^{1+2\text{\textgreek{h}}}|F|^{2}\, dg.\nonumber 
\end{align}
Therefore, since $\text{\textgreek{y}}=\text{\textgreek{y}}_{t_{f}}$
on $\{0\le t\le t_{f}-1\}$, (\ref{eq:IledWithLossPhysical}) can
be readily obtained from (\ref{eq:IledWithLossPhysicalCompactSupport-1})
and (\ref{eq:IledWithLossPhysicalCompactSupport-1-1}).

We will now proceed to establish (\ref{eq:IledWithLossPhysicalCompactSupport})
when $F$ is compactly supported in $\mathcal{M}$. Let $\text{\textgreek{q}}:\mathbb{R}\rightarrow(0,1]$
be a smooth function such that $\text{\textgreek{q}}(x)=1$ for $x\le0$
and $\text{\textgreek{q}}(x)=e^{-x}$ for $x\ge1$, and let us define
for any $\text{\textgreek{d}}>0$ the function $\text{\textgreek{q}}_{\text{\textgreek{d}}}:\mathcal{M}\rightarrow(0,1]$
by the relation:
\begin{equation}
\text{\textgreek{q}}_{\text{\textgreek{d}}}=\text{\textgreek{q}}(\text{\textgreek{d}}\bar{t}).\label{eq:Cut-Off-Function}
\end{equation}
Commuting (\ref{eq:PotentialOnceMore}) with $T^{j}$, $j\le\lceil\frac{d-1}{2}\rceil$,
the energy flux identity (for any $\text{\textgreek{t}}>0$) 
\begin{equation}
\int_{\{\bar{t}=\text{\textgreek{t}}\}}\big(J_{\text{\textgreek{m}}}^{T}(T^{j}\text{\textgreek{y}})\bar{n}^{\text{\textgreek{m}}}+V_{\text{\textgreek{l}}}|T^{j}\text{\textgreek{y}}|^{2}\big)=-2\int_{\{0\le\bar{t}\le\text{\textgreek{t}}\}}T^{j+1}\text{\textgreek{y}}\cdot T^{j}F\label{eq:EnergyFlux}
\end{equation}
combined with elliptic estimates (using equation (\ref{eq:PotentialOnceMore}))
and the Sobolev embedding theorem, implies that $\text{\textgreek{y}}$
is uniformly bounded on $\mathcal{M}$. Therefore, the function $\text{\textgreek{q}}_{\text{\textgreek{d}}}\text{\textgreek{y}}$
is square integrable in the $t$ variable. In view of this fact and
the relation 
\begin{equation}
\square_{g}(\text{\textgreek{q}}_{\text{\textgreek{d}}}\text{\textgreek{y}})-V_{\text{\textgreek{l}}}\text{\textgreek{q}}_{\text{\textgreek{d}}}\text{\textgreek{y}}=\text{\textgreek{q}}_{\text{\textgreek{d}}}F+2\partial^{\text{\textgreek{m}}}\text{\textgreek{q}}_{\text{\textgreek{d}}}\cdot\partial_{\text{\textgreek{m}}}\text{\textgreek{y}}+\square_{g}\text{\textgreek{q}}_{\text{\textgreek{d}}}\cdot\text{\textgreek{y}},\label{eq:InhomogeneousCutOff}
\end{equation}
from (\ref{eq:IledWithLossInhomogeneous}) (for $\text{\textgreek{q}}_{\text{\textgreek{d}}}\text{\textgreek{y}}$
in place of $\tilde{\text{\textgreek{y}}}$ there) and (\ref{eq:Cut-Off-Function})
we obtain:%
\footnote{Notice the bound $|\square_{g}\text{\textgreek{q}}_{\text{\textgreek{d}}}|\lesssim\text{\textgreek{d}}^{2}(1+r)^{-1}\min\{e^{-\text{\textgreek{d}}\bar{t}},1\}$,
which follows from the fact that $\square_{g}\text{\textgreek{q}}_{\text{\textgreek{d}}}=-2\partial_{\bar{t}}\partial_{r}\text{\textgreek{q}}_{\text{\textgreek{d}}}+\partial_{r}^{2}\text{\textgreek{q}}_{\text{\textgreek{d}}}+(d-1)r^{-1}\partial_{r}\text{\textgreek{q}}_{\text{\textgreek{d}}}+r^{-2}\text{\textgreek{D}}_{g_{\mathbb{S}^{d-1}}}\text{\textgreek{q}}_{\text{\textgreek{d}}}+O(r^{-1})\big\{\partial^{2}\text{\textgreek{q}}_{\text{\textgreek{d}}},\partial\text{\textgreek{q}}_{\text{\textgreek{d}}}\big\}$
in the polar coordinate chart $(\bar{t},r,\text{\textgreek{sv}})$
in the region $r\gg1$.%
}
\begin{align}
\int_{\mathcal{M}}(1+r)^{-1-2\text{\textgreek{h}}}\big( & J_{\text{\textgreek{m}}}^{T}(\text{\textgreek{q}}_{\text{\textgreek{d}}}\text{\textgreek{y}})T^{\text{\textgreek{m}}}+(1+r)^{-2}|\text{\textgreek{q}}_{\text{\textgreek{d}}}\text{\textgreek{y}}|^{2}\big)\lesssim_{\text{\textgreek{l}},\text{\textgreek{h}}}\label{eq:IledWithLossInhomogeneous-2}\\
\lesssim_{\text{\textgreek{l}},\text{\textgreek{h}}} & \sum_{j=0}^{k}\int_{\mathcal{M}}(1+r)^{1+2\text{\textgreek{h}}}|T^{j}F|^{2}+\sum_{j=0}^{k}\int_{0}^{+\infty}\text{\textgreek{d}}^{2}e^{-\text{\textgreek{d}}s}\big(\mathcal{E}^{(1+2\text{\textgreek{h}})}[T^{j}\text{\textgreek{y}}](s)+\mathcal{E}_{en}[T^{j}\text{\textgreek{y}}]\big)\, ds,\nonumber 
\end{align}
where, for some fixed $R\gg1$ and any $0<p\le2$: 
\begin{gather}
\mathcal{E}^{(p)}[\text{\textgreek{f}}](\text{\textgreek{t}})\doteq\int_{\{\bar{t}=\text{\textgreek{t}}\}\cap\{r\ge R\}}\big(r^{p}|\partial_{r}\text{\textgreek{f}}|^{2}+r^{p-2}|\text{\textgreek{f}}|^{2}\big),\label{eq:Penergy}\\
\mathcal{E}_{en}[\text{\textgreek{f}}](\text{\textgreek{t}})\doteq\int_{\{\bar{t}=\text{\textgreek{t}}\}}\big(J_{\text{\textgreek{m}}}^{T}(\text{\textgreek{f}})\bar{n}^{\text{\textgreek{m}}}+(1+r)^{-2}|\text{\textgreek{f}}|^{2}\big),
\end{gather}
the $\partial_{r}$-derivative in (\ref{eq:Penergy}) being considered
with respect to the (polar) coordinate chart $(\bar{t},r,\text{\textgreek{sv}})$
in the region $\{r\ge R\}$.

In order to obtain (\ref{eq:IledWithLossPhysicalCompactSupport})
for $\text{\textgreek{y}}$ from (\ref{eq:IledWithLossInhomogeneous-2}),
it suffices to show that the second term of the right hand side of
(\ref{eq:IledWithLossInhomogeneous-2}) converges to $0$ as $\text{\textgreek{d}}\rightarrow0$.
The $r^{p}$-weighted estimates of Section 5 of \cite{Moschidisc},
for $p=2$, yield for any $\text{\textgreek{t}}>0$ (provided $R$
is sufficiently large): 
\begin{equation}
\mathcal{E}^{(2)}[\text{\textgreek{y}}](\text{\textgreek{t}})+\int_{-\infty}^{\text{\textgreek{t}}}\mathcal{E}^{(1)}[\text{\textgreek{y}}](s)\, ds\lesssim\int_{\mathcal{M}}(1+r)^{3}|F|^{2}+\int_{\{\bar{t}\le\text{\textgreek{t}}\}\cap\{r\le R\}}\big(J_{\text{\textgreek{m}}}^{T}(\text{\textgreek{y}})T^{\text{\textgreek{m}}}+|\text{\textgreek{y}}|^{2}\big),\label{eq:NewMethodP=00003D2}
\end{equation}
 In view of the $T$-energy flux identity (\ref{eq:EnergyFlux}) and
the lower bound (\ref{eq:Positivity of energyFlux}), we obtain from
(\ref{eq:NewMethodP=00003D2}):
\begin{equation}
\mathcal{E}^{(2)}[\text{\textgreek{y}}](\text{\textgreek{t}})+\int_{-\infty}^{\text{\textgreek{t}}}\mathcal{E}^{(1)}[\text{\textgreek{y}}](s)\, ds\le C(F)(1+\text{\textgreek{t}}),\label{eq:LinearGrowth-1}
\end{equation}
where $C(F)>0$ depends on the precise choice of $F$. From (\ref{eq:LinearGrowth-1})
and the compact support of $F$ (yielding $\mathcal{E}^{(p)}[\text{\textgreek{y}}](\text{\textgreek{t}})=0$
for $\text{\textgreek{t}}\ll-1$) we deduce that for any $\text{\textgreek{t}}>0$
\begin{equation}
\int_{-\infty}^{\text{\textgreek{t}}}\mathcal{E}^{(2)}[\text{\textgreek{y}}](s)\, ds+\text{\textgreek{t}}\int_{-\infty}^{\text{\textgreek{t}}}\mathcal{E}^{(1)}[\text{\textgreek{y}}](s)\, ds\le C(F)(1+\text{\textgreek{t}})^{2},
\end{equation}
and thus a standard interpolation argument yields the following qualitative
bound for $\text{\textgreek{y}}$:
\begin{equation}
\int_{-\infty}^{\text{\textgreek{t}}}\mathcal{E}^{(1+2\text{\textgreek{h}})}[\text{\textgreek{y}}](s)\le C(F)(1+\text{\textgreek{t}})^{1+2\text{\textgreek{h}}}.\label{eq:SublinearGrowth}
\end{equation}
Therefore, for any integer $0\le j\le k$, setting for simplicity
\begin{equation}
f_{j}(\text{\textgreek{t}})\doteq\mathcal{E}^{(1+2\text{\textgreek{h}})}[T^{j}\text{\textgreek{y}}](s)+\mathcal{E}_{en}[T^{j}\text{\textgreek{y}}],
\end{equation}
in view of (\ref{eq:SublinearGrowth}) and (\ref{eq:EnergyFlux}),
we can bound for any integer $0\le j\le k$, any $\text{\textgreek{d}}>0$
and any $s_{0}\gg1$:
\begin{align}
\int_{0}^{+\infty}\text{\textgreek{d}}^{2}e^{-\text{\textgreek{d}}s}f_{j}(s)\, ds & \le\int_{0}^{s_{0}}\text{\textgreek{d}}^{2}f_{j}(s)\, ds+\int_{s_{0}}^{+\infty}\text{\textgreek{d}}^{2}e^{-\text{\textgreek{d}}s}f_{j}(s)\, ds\le\label{eq:BoundForInterpolation}\\
 & \le C(F)\Big\{\text{\textgreek{d}}^{2}s_{0}^{1+2\text{\textgreek{h}}}+\int_{s_{0}}^{+\infty}e^{-\text{\textgreek{d}}s}\text{\textgreek{d}}^{2}s^{2}\, ds\Big\}\le\nonumber \\
 & \le C(F)\Big\{\text{\textgreek{d}}^{2}s_{0}^{1+2\text{\textgreek{h}}}+e^{-\text{\textgreek{d}}s_{0}}(\text{\textgreek{d}}s_{0}^{2}+\text{\textgreek{d}}^{-1})\Big\}.\nonumber 
\end{align}
Choosing $s_{0}=\text{\textgreek{d}}^{-1-\frac{2\text{\textgreek{h}}(1-2\text{\textgreek{h}})}{1+2\text{\textgreek{h}}}}$
in (\ref{eq:BoundForInterpolation}), we obtain as $\text{\textgreek{d}}\rightarrow0$
(since $2\text{\textgreek{h}}<1$)
\begin{equation}
\lim_{\text{\textgreek{d}}\rightarrow0}\int_{0}^{+\infty}\text{\textgreek{d}}^{2}e^{-\text{\textgreek{d}}s}f_{j}(s)\, ds=0.
\end{equation}
Thus, letting $\text{\textgreek{d}}\rightarrow0$, (\ref{eq:IledWithLossInhomogeneous-2})
yields the desired integrated local energy decay estimate (\ref{eq:IledWithLossPhysicalCompactSupport}).
\qed

\section{\label{sec:A-topological-lemma}A topological lemma}

We will establish the following lemma on the image of a continuous
family of maps from the unit ball to itself:
\begin{lem}
\label{lem:TopologicalLemma}Let $\mathcal{F}:[0,1]\times B_{1}^{n}\rightarrow B_{1}^{n}$
be a continuous map, where $B_{\text{\textgreek{r}}}^{n}$ is the
open ball of radius $\text{\textgreek{r}}$ in $\mathbb{R}^{n}$.
Assume also that $\mathcal{F}\big(\{0\}\times\cdot\big):B_{1}^{n}\rightarrow B_{1}^{n}$
is a homeomorphism onto an open neighborhood of $0_{\mathbb{R}^{n}}$
and that for any $t\in[0,1]$ we have 
\begin{equation}
\mathcal{F}\big(\{t\}\times(B_{1}^{n}\backslash B_{1/2}^{n})\big)\subseteq B_{1}^{n}\backslash\{0\}.\label{eq:ForAllBallsMappingOnTheSphere}
\end{equation}
 Then for any $t\in[0,1]$:
\begin{equation}
0_{\mathbb{R}^{n}}\in\mathcal{F}\big(\{t\}\times B_{1}^{n}\big).\label{eq:0BelongsToAllBalls}
\end{equation}
\end{lem}
\begin{proof}
Because of (\ref{eq:ForAllBallsMappingOnTheSphere}), for any $t\in[0,1]$
the map $\mathcal{F}(\{t\}\times\cdot):B_{1}^{n}\rightarrow B_{1}^{n}$
induces a well defined group homeomorphism 
\begin{equation}
\mathcal{F}_{hom}(t):H_{n}(B_{1}^{n},B_{1}^{n}\backslash B_{1/2}^{n})\rightarrow H_{n}(B_{1}^{n},B_{1}^{n}\backslash\{0_{\mathbb{R}^{n}}\}),\label{eq:GroupHomeomorphism}
\end{equation}
where $H_{n}(A,B)$ is the $n$-th reduced homology group of $A$
relative to $B\subset A$ (see \cite{Hatcher2002}). In this case,
$H_{n}(B_{1}^{n},B_{1}^{n}\backslash B_{1/2}^{n})\simeq\mathbb{Z}\simeq H_{n}(B_{1}^{n},B_{1}^{n}\backslash\{0_{\mathbb{R}^{n}}\})$. 

Because $\mathcal{F}$ is continuous, the map (\ref{eq:GroupHomeomorphism})
is continuous in $t$ and, hence, since its domain and range are discrete,
it is constant in $t$. Because $\mathcal{F}\big(\{0\}\times\cdot\big):B_{1}^{n}\rightarrow B_{1}^{n}$
is a homeomorphism onto an open neighborhood of $0_{\mathbb{R}^{n}}$
, $\mathcal{F}_{hom}(0)$ is non-trivial, and thus (\ref{eq:GroupHomeomorphism})
is also non-trivial for any $t\in[0,1]$. This implies that (\ref{eq:0BelongsToAllBalls})
holds, since if $0_{\mathbb{R}^{n}}\notin\mathcal{F}\big(\{t\}\times B_{1}^{n}\big)$
then $\mathcal{F}_{hom}(t)$ is identically $0$. Thus, the proof
of the Lemma is complete.
\end{proof}
\bibliographystyle{plain}
\bibliography{DatabaseExample}

\end{document}